\documentclass{amsart}

\usepackage{latexsym}
\usepackage{linearb}
\usepackage[LGR, OT1]{fontenc}
\usepackage{amsmath,amsthm}
\usepackage{amssymb}
\usepackage{upgreek}
\usepackage[greek,english]{babel}
\usepackage{teubner}
\usepackage{MnSymbol, wasysym}
\usepackage{graphicx}
\usepackage{marvosym}
\usepackage[margin=1.25in,dvips]{geometry}
\usepackage{color}
\usepackage{comment}
\usepackage{mathrsfs}
\usepackage{mathtools}

\usepackage{multicol}

\newcommand\lessflat{{{\mbox{$\flat  \mkern-12mu {}^{\_}$}}}}

\usepackage[toc,page]{appendix}

\usepackage{physics}

\usepackage{enumitem}

\usepackage{amsaddr}

\usepackage{xcolor}

\usepackage{hyperref}
\hypersetup{%
  colorlinks = true,
  linkcolor  = black
}
\usepackage{slashed}

\newtheorem{innercustomTheorem}{Theorem}
\newenvironment{customTheorem}[1]
  {\renewcommand\theinnercustomTheorem{#1}\innercustomTheorem}
  {\endinnercustomTheorem}
\newtheorem{innercustomCorollary}{Corollary}
\newenvironment{customCorollary}[1]
  {\renewcommand\theinnercustomCorollary{#1}\innercustomCorollary}
  {\endinnercustomCorollary}
\newtheorem{innercustomLemma}{Lemma}
\newenvironment{customLemma}[1]
{\renewcommand\theinnercustomLemma{#1}\innercustomLemma}
{\endinnercustomLemma}

\DeclareMathOperator\supp{supp}

\newcommand{\nabb}{\mbox{$\nabla \mkern-13mu /$\,}}

\makeindex

\numberwithin{equation}{section}

\newtheorem{definition}{Definition}[section]

\newtheorem{remark}{Remark}[section]
\newtheorem{lemma}{Lemma}[section]
\newtheorem{theorem}{Theorem}[section]
\newtheorem{proposition}{Proposition}[section]

\newtheorem*{rough version}{Rough Version}

\newtheorem*{theorem*}{Theorem}

\newtheorem*{corollary*}{Corollary}

\newenvironment{sketch proof}{\proof}{\endproof}

\title[Boundedness and Morawetz estimates on subextremal Kerr--de~Sitter]{Boundedness and Morawetz estimates\\ on subextremal Kerr--de~Sitter}

\author{Georgios Mavrogiannis}
\address{Department of Mathematics, Rutgers University, New Brunswick, NJ 08903 USA.}
\email{gm758@math.rutgers.edu}

\date\today

\begin{document}

\begin{abstract}
We study the Klein--Gordon equation~$\Box_{g_{a,M,l}}\psi-\mu^2_{\textit{KG}}\psi=0$ on subextremal Kerr--de~Sitter black hole backgrounds with parameters~$(a,M,l)$, where~$l^2=\frac{3}{\Lambda}$. We prove boundedness and Morawetz estimates assuming an appropriate mode stability statement for real frequency solutions of Carter's radial ode. Our results in particular apply in the very slowly rotating case~$|a|\ll M,l$, and in the case where the solution~$\psi$ is axisymmetric. This generalizes the work of Dafermos--Rodnianski~\cite{DR3} on Schwarzschild--de~Sitter.

The boundedness and Morawetz results of the present paper will be used in our companion~\cite{mavrogiannis2} to prove a `relatively non-degenerate integrated estimate' for subextremal Kerr--de~Sitter black holes~(and as a consequence exponential decay). In a forthcoming paper~\cite{mavrogiannis3}, this will immediately yield nonlinear stability results for quasilinear wave equations on subextremal Kerr--de~Sitter backgrounds. 
\end{abstract}

\maketitle

{
  \hypersetup{linkcolor=black}
  \tableofcontents
}

\section{Introduction}\label{sec: intro}

A fundamental class of solutions of Einstein's equation with a positive cosmological constant~$\Lambda>0$, in the absence of matter
\begin{equation}\label{eq: Einstein equation}
	Ric[g] -\Lambda g=0,\qquad \Lambda>0
\end{equation}
is the Kerr--de~Sitter family of black hole solutions
\begin{equation}\label{eq: sec: intro, eq 0}
(\mathcal{M},g_{a,M,l}),
\end{equation}
see~\cite{Carter1,Carter2}. These are the analogues of the Kerr black holes, which solve~\eqref{eq: Einstein equation} with~$\Lambda=0$. The parameters~$(a,M,l)$ represent respectively the angular momentum per unit mass~$a$, the mass~$M$ of the black hole and a quantity related to the cosmological constant $\Lambda>0$ via
\begin{equation}\label{eq: the definition of l}
    l\:\dot{=}\:\sqrt{\frac{3}{\Lambda}}.
\end{equation}
The spacetime metric~\eqref{eq: sec: intro, eq 0}, in Boyer--Lindquist coordinates $(t,r,\theta,\varphi)$, takes the form
\begin{equation}\label{eq: prototype metric in BL coordinates}
    \begin{aligned}
        g_{a,M,l} &	= \frac{\rho^2}{\Delta}dr^{2}+\frac{\rho^2}{\Delta_{\theta}}d \theta^{2}+\frac{\Delta_{\theta}(r^{2}+a^{2})^{2}-\Delta a^{2}\sin^{2}\theta }{\Xi ^{2}\rho^2}\sin^{2}\theta d \varphi^{2} -2\frac{\Delta_{\theta}(r^{2}+a^{2})-\Delta}{\Xi\rho^2} a \sin^{2}\theta d\varphi dt\\
        &	\qquad\qquad-\frac{\Delta -\Delta_{\theta}a^{2}\sin^{2}\theta}{\rho^2}dt^{2},
    \end{aligned}
\end{equation}
see already Section~\ref{subsec: boyer Lindquist coordinates}, with~$\Delta_\theta=1+\frac{a^2}{l^2}\cos^2\theta,~\rho^2=r^2+a^2\cos^2\theta,~\Xi=1+\frac{a^2}{l^2}$, where
\begin{equation}\label{eq: prototype Delta}
    \Delta(r)=(r^2+a^2)\left(1-\frac{r^2}{l^2}\right)-2Mr.
\end{equation}
Note that in the non-rotating case~$a=0$ the metric~\eqref{eq: prototype metric in BL coordinates} reduces to the Schwarzschild--de~Sitter metric
\begin{equation}\label{eq: prototype metric of SdS}
	g_{M,\Lambda}=- \left(1-\frac{2M}{r}-\frac{\Lambda}{3}r^2\right)dt^2 +\frac{1}{1-\frac{2M}{r}-\frac{\Lambda}{3}r^2}dr^2 + r^2 d\sigma_{\mathbb{S}^2},
\end{equation}
where~$d\sigma_{\mathbb{S}^2}$ is the standard metric of the unit sphere.

In this paper we revisit the study of the Klein--Gordon equation 
\begin{equation}\label{eq: kleingordon}
    \Box_{g_{a,M,l}} \psi -\mu_{\textit{KG}}^2\psi=0,
\end{equation}
with mass~$\mu_{\textit{KG}}^2\geq 0$ (the case~$\mu_{\textit{KG}}=0$ is called the `wave equation' and the case~$\mu^2_{KG}=\frac{2\Lambda}{3}$ the conformally coupled case) on a subextremal Kerr--de~Sitter spacetime. We call the Kerr--de~Sitter black hole triad~$(a,M,l)$ \texttt{subextremal}, when the polynomial~$\Delta$, see~\eqref{eq: prototype Delta}, attains four distinct real roots, see already Definition~\ref{def: subextremality, and roots of Delta}. The Klein--Gordon equation~\eqref{eq: kleingordon} has been extensively studied in the case~$\Lambda\geq 0$ in recent years~\cite{Shlapentokh_Rothman_2014,DR2,Dyatlov1,Dyatlov2,Dyatlov4,melr-barr-vasy,Volker,barreto,mavrogiannis,vasy1,vasy2,zworski}. For the study of the Klein--Gordon equation~\eqref{eq: kleingordon} on~$\Lambda<0$ spacetimes, see~\cite{holzegel1,holzegel2,holzegel3,holzegel4,vasy6}.

We prove boundedness and degenerate Morawetz estimates~(Theorem~\ref{rough: theorems 1}) for solutions of the Klein--Gordon equation~\eqref{eq: kleingordon} with mass~$\mu^2_{KG}\geq 0$ on a subextremal Kerr--de~Sitter spacetime, with parameters~$(a,M,l)$, under the following condition:
\begin{equation*}
	\begin{aligned}
		&\text{(MS):~mode stability on the real axis for Carter's radial ode holds on a curve in subextremal}\\
		&	\qquad\quad \text{parameter space connecting}~(a,M)~\text{to the subextremal Schwarzschild--de~Sitter family.}
	\end{aligned}
\end{equation*}
We will show moreover that~(MS) holds in the very slowly rotating case~$|a|\ll M,l$. This generalizes the work of Dafermos--Rodnianski~\cite{DR3} on Schwarzschild--de~Sitter~$(a=0)$ and is based on the framework of~\cite{DR2} which proved boundedness and Morawetz estimates for Kerr~$(\Lambda=0)$. In the special case where~$\psi$ is axisymmetric our results hold for~$(a,M,l)$ in the full subextremal range.

To prove Theorem~\ref{rough: theorems 1} we follow the strategy of~\cite{DR2}, namely we reduce the problem to constructing suitable fixed frequency energy multipliers for Carter's ode, see already~\eqref{eq: subsec: sec: intro, subsec 2, eq 1}. Two major difficulties are superradiance and trapping. We note that in the Kerr--de~Sitter case, for certain black hole parameters, the spacetime admits trapping that corresponds to~$\omega=0$,~$|m|\gg 1$ frequencies, unlike the asymptotically flat Kerr case where trapping occurs only for frequencies~$|\omega|\gg 1$. To reduce the problem to an analysis of real frequencies, we also appeal to a continuity argument in the parameters~$a,M$ starting from Schwarzschild--de~Sitter~$a=0$, see already Section~\ref{sec: continuity argument}~(this explains the precise form of the (MS) assumption).
 In contrast to the~$\Lambda=0$ case, it is not known if~(MS) holds for all subextremal black hole parameters~$(a,M,l)$, say for the important special cases~$\mu_{KG}^2=0$ or~$\mu_{KG}^2=\frac{2\Lambda}{3}$. For general~$\mu^2_{KG}>0$, then~(MS) is not expected to hold for all subextremal~$(a,M,l)$, see~\cite{Shlapentokh_Rothman_2014} for a mode instability result in the~$\Lambda=0$ case.

In our companion~\cite{mavrogiannis2}, we prove a `relatively non-degenerate' integrated energy estimate for solutions of~\eqref{eq: kleingordon} under the condition~(MS). For this, we introduce a novel operator~$\mathcal{G}$, for the entire subextremal range of Kerr--de~Sitter black holes, thus generalizing our previous work~\cite{mavrogiannis} on Schwarzschild--de~Sitter, and we use this commutation in conjunction with the result of Theorem~\ref{main theorem 1}. Note that the commutation with~$\mathcal{G}$ had been previously studied by Holzegel--Kauffman in the~$\Lambda=0$ Schwarzschild case, see~\cite{gustav}, to treat the wave equation with small first order error terms. Moreover, see the recent work~\cite{gustav2} where the authors define a related~$\mathcal{G}$ commutation for the asymptotically flat~$\Lambda=0$ Kerr black hole. An immediate Corollary of the~`relatively non-degenerate' integrated estimate of our~\cite{mavrogiannis2} is exponential decay for solutions of equation~\eqref{eq: kleingordon}.

In our forthcoming~\cite{mavrogiannis3}, we use~\cite{mavrogiannis2} to give an elementary proof of stability results for quasilinear wave equations on Kerr--de~Sitter black hole backgrounds, assuming only~(MS).

\subsection{Previous results on Kerr(--de~Sitter) backgrounds}\label{subsec: sec: intro: subsec 0}

We briefly mention some previous results for linear and non-linear waves on black hole backgrounds. For a more complete discussion, see our previous~\cite{mavrogiannis,mavrogiannis1}. 

Concerning linear results, in addition to~\cite{DR3} to be discussed extensively in Section~\ref{subsec: sec: intro: subsec 1}, note the proof of exponential decay for solutions of the wave equation on Schwarzschild--de~Sitter away from the event horizon~$\mathcal{H}^+$ and the cosmological horizon~$\bar{\mathcal{H}}^+$, by Bony--H\"afner~\cite{bony}, and the proof of exponential decay for the solutions of the wave equation on Kerr--de~Sitter up to and including the event horizon~$\mathcal{H}^+$ and the cosmological horizon~$\bar{\mathcal{H}}^+$, by Dyatlov~\cite{Dyatlov1,Dyatlov2}. Moreover, we refer the reader to the following additional linear results~\cite{bony,zworski,Dyatlov1,Dyatlov2,vasy1,vasy2,DR2,whiting,Shlapentokh_Rothman_2014_mode_stability,tatarutohaneanuKerr,LindbladTohaneanuKerr,AllenFangLinear} and references therein. See especially the more recent work of Petersen--Vasy~\cite{petersen2021wave}, as well as the partial mode stability result on Kerr--de~Sitter of Casals--Teixeira~da~Costa~\cite{casals2021hidden} and the recent work of Hintz~\cite{hintz2021mode}, which will be mentioned later in the paper. For a numerical study of modes on Kerr--de~Sitter spacetime see the work~\cite{yoshida}, where in fact the authors did not find any growing mode solutions in the conformal case~$\mu^2_{KG}=\frac{2\Lambda}{3}$ for the full subextremal range.

Concerning non-linear results, note first the remarkable proof by Hintz--Vasy of non-linear stability of the slowly rotating Kerr--de~Sitter black hole~\cite{hintz2} and the more recent result of Fang~\cite{AllenFang}. Moreover, for the study of non-linear wave equations on~$\Lambda>0$ spacetimes we refer the reader to~\cite{hintz2,hintz3,hintz5,hintz6,mavrogiannis1}, and for non-linear wave equations on~$\Lambda=0$ spacetimes we refer the reader to~\cite{Chr-Klain,lindbladrodnianski,lindbladquasilinear,dafermos2021nonlinear,GiorgiKlainermanSzeftel} and references therein.

Finally, for the study of the wave equation on extremal black holes see the works of Aretakis and Gajic~\cite{aretakis,aretakis2,gajic}. Moreover, for mode stability on extremal Kerr spacetimes see work of Teixeira~da Costa~\cite{rita2}.

\subsection{Review of~\cite{DR3} and~\cite{DR2}}\label{subsec: sec: intro: subsec 1}

The present paper is a generalization of the results of Dafermos--Rodnianski~\cite{DR3} on Schwarzschild--de~Sitter and also uses the strategy of proof of Dafermos--Rodnianski--Shlapentokh-Rothman on Kerr~\cite{DR2}. We review these results here.

\subsubsection{The Morawetz~(degenerate integrated decay) estimate of~\cite{DR3} on Schwarzschild--de~Sitter}

We define the following system of `special' hyperboloidal coordinates
\begin{equation}\label{eq: sec: intro: eq 3, coordinates}
(\bar{t},r,\theta,\phi)
\end{equation}
that relate to the usual Schwarzschild--de~Sitter coordinates~$(t,r,\theta,\phi)$ as follows:
\begin{equation}
\bar{t}=t-\int_{3M}^r\frac{\xi(r)}{1-\frac{2M}{r}-\frac{\Lambda}{3}r^2},\qquad \xi(r)=\frac{1-\frac{3M}{r}}{\sqrt{1-9M^2\Lambda}}\sqrt{1+\frac{6M}{r}}.
\end{equation}
For the definitions of the coordinates~\eqref{eq: sec: intro: eq 3, coordinates} also see our previous~\cite{mavrogiannis}. Moreover, note that~$\{\bar{t}=\tau\}$, for~$\tau\geq 0$, are spacelike hypersurfaces that connect the event horizon~$\mathcal{H}^+$ with the cosmological horizon~$\bar{\mathcal{H}}^+$ and foliate the causal past of the horizons~$\mathcal{H}^+,\bar{\mathcal{H}}^+$. The specific choice of the coordinates~\eqref{eq: sec: intro: eq 3, coordinates} is also intimately connected to the nature of trapping at~$r=3M$. Note that at~$r=3M$ we have~$g(\partial_{\bar{t}},\partial_r)=0$.

In~\cite{DR3}, Dafermos--Rodnianski studied the solutions~$\psi$ of the wave equation~(\eqref{eq: kleingordon} with~$\mu^2_{\textit{KG}}=0$) and proved the following Morawetz type~(degenerate integrated decay) energy estimate
\begin{equation}\label{eq: prototype morawetz schwarzschild de Sitter}
\int_{\tau_1}^{\tau_2} d\tau^\prime \int_{\{\bar{t}=\tau^\prime\}} (\partial_{r}\psi)^2+ \left(1-\frac{3M}{r}\right)^2\left((\partial_{\bar{t}}\psi)^2+|\slashed{\nabla}\psi|^2\right)\lesssim \int_{\{\bar{t}=\tau_1\}} (\partial_{\bar{t}} \psi)^2+(\partial_{r}\psi)^2+|\slashed{\nabla}\psi|^2,
\end{equation}
for all~$0\leq \tau_1\leq \tau_2$, where $\slashed{\nabla}$ is the covariant derivative of the fixed $r$-spheres. The degeneration of \eqref{eq: prototype morawetz schwarzschild de Sitter} at the photon sphere $r=3M$ is necessitated by the obstructions of~\cite{ralston,sbierski}. Remarkably, the Morawetz estimate~\eqref{eq: prototype morawetz schwarzschild de Sitter} does not degenerate at the horizon, due to the redshift effect, which can be captured by a multiplier current, see~\cite{DR4}.

Faster than any polynomial decay for solutions of the wave equation on a Schwarzschild--de~Sitter background is an immediate corollary of the Morawetz estimate~\eqref{eq: prototype morawetz schwarzschild de Sitter} suitable commutations and an iterative argument, see~\cite{DR2}.~(We note, however, that exponential decay for solutions in Sobolev spaces does not follow immediately from~\eqref{eq: prototype morawetz schwarzschild de Sitter}.)

\subsubsection{The Morawetz estimate of~\cite{DR2} on \texorpdfstring{$\Lambda=0$}{g} Kerr}

We here review the proof of the Morawetz estimate for the wave equation on subextremal Kerr exteriors of~\cite{DR2}.

In~\cite{DR2} the authors used Carter's separation of variables for the wave operator on Kerr, see~\cite{Carter}. Specifically, let~$r^\star$ be the tortoise coordinate, see~\cite{DR2,DR5}, and~$^\prime=\frac{d}{dr^\star}$. Carter's radial ode reads
\begin{equation}\label{eq: prototype ode}
u^{\prime\prime}+\left(\omega^2-V(r^\star,\omega,m,\lambda^{(a\omega)}_{m\ell})\right)u=H,
\end{equation}
where~$\omega$ denotes the time frequency,~$m$ denotes the azimuthal frequency and~$\lambda^{(a\omega)}_{m\ell}$ denotes a frequency associated with the eigenvalues of the angular part of Carter's separation, where~$\ell\geq |m|$. Note that in~\cite{DR2} the error term~$H$ is associated with a time cut-off.

The paper~\cite{DR2} divided the frequency space~$\{(\omega,m,\lambda^{(a\omega)}_{m\ell})\}$ into frequency regimes and constructed energy currents for the ode~\eqref{eq: prototype ode} in each of these regimes in order to prove frequency localised energy estimates. To treat the non-stationary bounded frequencies~($|\omega|,\lambda^{(a\omega)}_{m\ell}\sim 1$), the authors utilized a flux bound on the horizon, which holds for the entire subextremal family of Kerr black holes. They inferred this bound from a quantitative version of mode stability on the real axis~$\Im\omega=0$, proved by Shlapentokh-Rothman~\cite{Shlapentokh_Rothman_2014_mode_stability}. (See already Proposition~\ref{prop: subsec: summing in the redshift estimate, prop 1} for such an estimate in the setting of the present paper, given the assumption~(MS).) For the original proof of mode stability on the upper half plane~$\Im\omega>0$ see the seminal work of Whiting~\cite{whiting}.

To justify the frequency analysis of~\cite{DR2}, based entirely on real frequencies~$\omega\in\mathbb{R}$, the authors appealed to a novel continuity argument in the rotation parameter~$a$. Namely, the authors first proved their Morawetz estimate for `\texttt{future integrable}' solutions of the wave equation, for which the Fourier inversion formula holds. Then, by utilizing a continuity argument in the rotation parameter~$a$~(starting from~$a=0$, for which the result indeed holds) they proved that all solutions of the wave equation that arise from compactly supported initial data are indeed future integrable and therefore the Morawetz estimate of~\cite{DR2} holds for all solutions, in the full subextremal range~$|a|<M$.

Similarly to the estimate~\eqref{eq: prototype morawetz schwarzschild de Sitter}, the Morawetz estimate of~\cite{DR2} degenerates appropriately at an interval around~$r=3M$, but does not degenerate on the event horizon~$\mathcal{H}^+$, exploiting the redshift effect~\cite{DR4}. However, in contrast to estimate~\eqref{eq: prototype morawetz schwarzschild de Sitter}, the Morawetz estimate of~\cite{DR2} exhibits appropriate weights in~$r$ in the fluxes and spacetime integrals to capture the correct behavior at null infinity~$\mathcal{I}^+$.

\subsection{Preliminaries for the statement and proof of Theorem~\ref{main theorem 1}}

Before presenting the rough version of our main Theorem~\ref{main theorem 1}, and discussing the proof, see already Section~\ref{subsec: sec: intro: subsec 2}, we first need several preliminaries.

\subsubsection{Carter's radial ode in Kerr--de~Sitter}\label{subsubsec: subsec: sec: intro, subsec 2, subsubsec 1}

Following~\cite{Carter,holzegel3}, we apply a cut-off to the solution~$\psi$ of the Klein--Gordon equation~\eqref{eq: kleingordon} with mass~$\mu^2_{KG}\geq 0$ and then decompose into real Fourier modes and then into spheroidal harmonics. As in the~$\Lambda=0$ case, Carter's separation of variables for the wave operator implies that the radial part~$u$ of the separated solution, for the Klein--Gordon equation~\eqref{eq: kleingordon}, satisfies the ode
\begin{equation}\label{eq: subsec: sec: intro, subsec 2, eq 1}
u^{\prime\prime}+(\omega^2-V)u=H
\end{equation}
where~$^\prime=\frac{d}{dr^\star}$, and~$r^\star$ is the tortoise coordinate, see Section~\ref{subsec: tortoise coordinate}. Here
\begin{equation}
\omega\in\mathbb{R},\qquad m\in\mathbb{Z}
\end{equation}
are the Fourier frequencies with respect to the Boyer--Lindquist coordinates~$t,\varphi$ respectively, see already Section~\ref{subsec: sec: carter separation, radial}, and the potential~$V$ is real and takes the form
\begin{equation}\label{eq: subsec: sec: intro, subsec 2, eq 2}
V=V_0+V_{\textit{SL}}+V_{\mu_{\textit{KG}}},\qquad V_0-\omega^2=\frac{\Delta}{(r^2+a^2)^2}\left(\lambda^{(a\omega)}_{m\ell}+(a\omega)^2-2m\omega a\Xi\right)-\left(\omega-\frac{am\Xi}{r^2+a^2}\right)^2,
\end{equation}
where~$V_{\textit{SL}},V_{\mu_{\textit{KG}}}$ are frequency independent terms (the latter is associated with the Klein--Gordon mass) see Section~\ref{subsec: sec: carter separation, radial}, and~$\lambda^{(a\omega)}_{m\ell}\in\mathbb{R}$ are the eigenvalues of the spheroidal harmonics, which for convenience are indexed by~$\ell$, where~$\ell\geq |m|$. The inhomogeneous term~$H$ on the RHS of~\eqref{eq: subsec: sec: intro, subsec 2, eq 1} comes from the cut-off.

As we shall see below the Fourier space decomposition allows us to capture the phenomena of trapping and superadiance at the level of modes and moreover allows us to rigorously formulate our mode stability condition~(MS).

\subsubsection{An energy identity for the solution~\texorpdfstring{$u$}{u}}

Let
\begin{equation}
	r_+,\qquad \bar{r}_+
\end{equation}
be the two largest roots of the polynomial~$\Delta$, see~\eqref{eq: prototype Delta}, which correspond to the event horizon~$\mathcal{H}^+$ and cosmological horizon~$\bar{\mathcal{H}}^+$ of Kerr--de~Sitter respectively, see already the Penrose--Carter diagram of Figure~\ref{fig: penrose diagram}. We have that~$\lim_{r\rightarrow r_+} r^\star (r)=-\infty$ and~$\lim_{r\rightarrow\bar{r}_+} r^\star(r)=+\infty$, see already Section~\ref{subsec: tortoise coordinate}.

We multiply the radial ode~\eqref{eq: subsec: sec: intro, subsec 2, eq 1} with~$\bar{u}$ and after integration by parts we obtain
\begin{equation}\label{eq: subsec: sec: intro, subsec 2, eq 3}
\left(\omega-\frac{am\Xi}{\bar{r}_+^2+a^2}\right)|u|^2(r^\star=+\infty)+\left(\omega-\frac{am\Xi}{r_+^2+a^2}\right)|u|^2(r^\star=-\infty)=\Im (\bar{u}H).
\end{equation}
To derive the above we are using that the smoothness of the solution~$\psi$ and the geometry of the cut-off implies the following
\begin{equation}\label{eq: subsec: sec: intro, subsec 2, eq 3.1}
u^\prime=-i\left(\omega-\frac{am\Xi}{r_+^2+a^2}\right) u,\qquad u^\prime=i\left(\omega-\frac{am\Xi}{\bar{r}_+^2+a^2}\right) u,
\end{equation}
at~$r^\star=-\infty,+\infty$ respectively, see already~\eqref{eq: lem: sec carters separation, subsec boundary behaviour of u, boundary beh. of u, eq 3}.

\subsubsection{Superradiant frequencies and mode stability}\label{subsec: sec: intro: subsec 1.1}

Let
\begin{equation}
	K_+=\partial_t+ \frac{a\Xi}{r_+^2+a^2}\partial_{\varphi},\qquad \bar{K}_+ = \partial_t + \frac{a\Xi}{\bar{r}_+^2+a^2}\partial_{\varphi}
\end{equation}
be the Hawking--Reall vector fields associated with the event horizon~$\mathcal{H}^+$ and the cosmological horizon~$\bar{\mathcal{H}}^+$ respectively, see already Section~\ref{subsec: Hawking-Reall v.f.}. (Note that in the~$\Lambda=0$ Kerr limit we have that~$\bar{K}_+~=~\partial_t$ since in the Kerr limit~$\bar{r}_+=+\infty$, see already Lemma~\ref{lem: sec: properties of Delta, lem 2, a,M,l}.)

An explicit computation shows that the~$\bar{K}_+$ energy flux along the event horizon~$\mathcal{H}^+$ and the~$\bar{K}^+$ flux along the cosmological horizon~$\bar{H}^+$ are respectively given by
\begin{equation}\label{eq: subsec: sec: intro: subsec 4, eq 0}
\int_{\mathcal{H}^+} K_+\psi \overline{\bar{K}_+\psi} =\int_{\mathcal{H}^+} \left(\partial_t\psi+\omega_+\partial_{\varphi}\psi\right) \cdot  \overline{\left(\partial_t\psi+\bar{\omega}_+\partial_{\varphi}\psi\right)},\qquad \int_{\bar{\mathcal{H}}^+} |\bar{K}_+\psi|^2  =\int_{\bar{\mathcal{H}}^+} \left|\partial_t\psi+\bar{\omega}_+\partial_{\varphi}\psi\right|^2.
\end{equation}
In particular, if we consider a solution of the form~$\psi=e^{-i\omega t}e^{im\varphi} \psi_0(r,\theta)$ then the expressions~\eqref{eq: subsec: sec: intro: subsec 4, eq 0} have different signs when the following expression is negative
\begin{equation}
\left(\omega-\frac{am\Xi}{r_+^2+a^2}\right) \left(\omega-\frac{am\Xi}{\bar{r}_+^2+a^2}\right).
\end{equation}

Therefore, we define the superradiant frequencies on Kerr--de~Sitter to be 
\begin{equation}\label{eq: subsec: sec: intro: subsec 4, eq 1}
	\mathcal{SF}= \{(\omega,m):~am\omega\in \left(\frac{a^2m^2\Xi}{\bar{r}_+^2+a^2},\frac{a^2m^2\Xi}{r_+^2+a^2}\right)\}
\end{equation}
or equivalently~$\mathcal{SF}= \{(\omega,m):~ \left(\omega-\frac{am\Xi}{r_+^2+a^2}\right)\left(\omega-\frac{am\Xi}{\bar{r}_+^2+a^2}\right)< 0\}$. We note that if~$a=0$ then~$\mathcal{SF}=\emptyset$. Moreover, we note that for any~$\omega\in\mathbb{R}$ we have~$(\omega,0)\notin \mathcal{SF}$; i.e.~axisymmetric frequencies are not superradiant.

For the superradiant frequencies~\eqref{eq: subsec: sec: intro: subsec 4, eq 1}, the left hand side of the energy identity~\eqref{eq: subsec: sec: intro, subsec 2, eq 3} is not coercive. As a result, for such frequencies, we cannot conclude that ``mode stability holds on the real axis", i.e.~the absence of non trivial real frequency solutions of~\eqref{eq: subsec: sec: intro, subsec 2, eq 1} with~$H=0$ and satisfying the boundary conditions~\eqref{eq: subsec: sec: intro, subsec 2, eq 3.1}. To make this more precise, we denote as~$\mathcal{W}(\omega,m,\ell,\mu^2_{KG})$ the Wronskian associated with ingoing and outgoing solutions, at the event~$\mathcal{H^+}$ and cosmological horizon~$\bar{\mathcal{H}}^+$ respectively, of the homogeneous~$(H=0)$ radial ode~\eqref{eq: subsec: sec: intro, subsec 2, eq 1} of Kerr--de~Sitter, see already Section~\ref{subsec: sec: Carter's separation, wronskian}. Note that
\begin{equation}\label{eq: subsec: sec: intro: subsec 4, eq 1.1}
	\mathcal{W}(a,M,l,\mu_{KG},\omega,m,\ell)=0
\end{equation}
where~$\omega\neq \frac{a\Xi}{r_+^2+a^2}m,~\omega\neq \frac{a\Xi}{\bar{r}_+^2+a^2}m,~m\in\mathbb{Z},\ell\geq |m|$, precisely when there exists a nontrivial mode solution with freqencies~$\omega,m,\ell$. Thus, mode stability on the real axis is equivalent to the statement that~\eqref{eq: subsec: sec: intro: subsec 4, eq 1.1} does not hold for any~$\omega\neq \frac{a\Xi}{r_+^2+a^2}m,~\omega\neq \frac{a\Xi}{\bar{r}_+^2+a^2}m,~m\in\mathbb{Z},\ell\geq |m|$. Note however that for the case~$\omega=m=0$ there do exist trivial mode solutions that correspond to constant solutions of the wave equation~\eqref{eq: kleingordon}~(with~$\mu_{KG}=0$). (Furthermore, we treat the near borderline superradiant frequencies~$0\leq|\omega-\frac{a\Xi}{r_+^2+a^2}m|\ll 1$ and~$0\leq|\omega-\frac{a\Xi}{\bar{r}_+^2+a^2}m|\ll 1$ with the estimate of Proposition~\ref{prop: energy estimate in the bounded non stationary frequency regime}, where we immediately obtain that there are no real mode solutions.)

It is an open problem to prove mode stability for the wave equation, $\mu^2_{KG}=0$, or the conformal wave equation,~$\mu^2_{KG}=\frac{2\Lambda}{3}$, on Kerr--de~Sitter in the spirit of~\cite{whiting,Shlapentokh_Rothman_2014_mode_stability}. One does not expect mode stability to hold for general mass~$\mu^2_{KG}>0$, see the work on the Klein--Gordon equation on asymptotically flat Kerr~\cite{Shlapentokh_Rothman_2014}. However, note the partial mode stability result of Casals--Teixeira~da~Costa~\cite{casals2021hidden} and the recent result of Hintz~\cite{hintz2021mode}, which we will refer to later.

\subsection{The rough version of Theorem~\ref{main theorem 1}}\label{subsec: sec: intro: subsec 2}

The reader familiar with the Carter--Penrose diagrammatic representation may wish to refer to Figure~\ref{fig: penrose diagram}. 
\begin{figure}[htbp]
	\centering
	\includegraphics[scale=0.8]{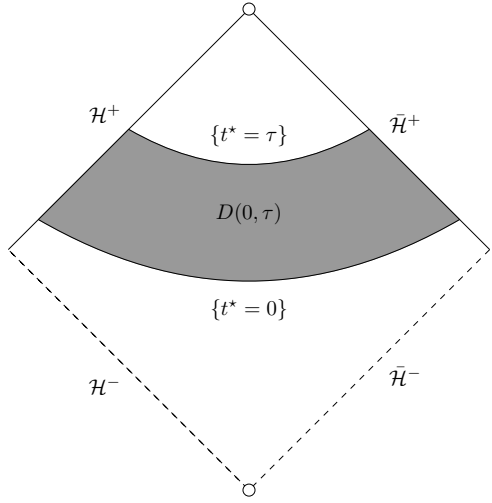}
	\caption{Carter--Penrose diagram of Kerr--de~Sitter}
	\label{fig: penrose diagram}
\end{figure}

In the present paper we are concerned with the region
\begin{equation}\label{eq: subsec: sec: intro: subsec 2, eq 0}
	\{r_+\leq r\leq \bar{r}_+\}.
\end{equation}
We will fix a system of coordinates
\begin{equation}
(t^\star,r,\theta^\star,\varphi^\star)\in \mathbb{R}_{t^\star}\times [r_+,\bar{r}_+]_r\times \mathbb{S}^2_{(\theta^\star,\varphi^\star)}
\end{equation}
which we term modified Kerr--de~Sitter star coordinates~(see already Section \ref{sec: the kerr de sitter spacetime} for their definition) covering~\eqref{eq: subsec: sec: intro: subsec 2, eq 0}. Moreover, for any $\tau\geq 0$ the set
\begin{equation}
\{t^\star=\tau\}
\end{equation}
will be a Cauchy hypersurface for~$\{t^\star\geq\tau\}$ that connects the event horizon~$\mathcal{H}^+$ and the cosmological horizon~$\bar{\mathcal{H}}^+$. Also, let
\begin{equation}
    n
\end{equation}
denote the future unit normal to the foliation $\{t^\star=\tau\}$.

We present the rough version of our boundedness and Morawetz estimate, where for the detailed version of Theorem~\ref{main theorem 1} see already Section~\ref{sec: main theorems}.

\begin{customTheorem}{\ref{main theorem 1}}[rough version]\label{rough: theorems 1}

Let~$l>0$,~$\mu^2_{KG}\geq 0$ and let~$(a,M,l)$ be subextremal Kerr--de~Sitter black hole parameters. Moreover, assume that the  following condition holds
\begin{equation*}
	\begin{aligned}
		&\text{(MS):~mode stability on the real axis for Carter's radial ode~\eqref{eq: subsec: sec: intro, subsec 2, eq 1} holds on a curve in subextremal}\\
		&	\qquad\quad \text{parameter space connecting}~(a,M)~\text{to the subextremal Schwarzschild--de~Sitter family}.
	\end{aligned}
\end{equation*}
Let $\psi$ be a solution of the Klein--Gordon equation~\eqref{eq: kleingordon}. Then, we have the following energy estimates
\begin{equation}\label{eq: rough: theorems 1, eq 1}
    \begin{aligned}
        \int_{\tau_1}^{\tau_2} d\tau^\prime\int_{\{t^\star=\tau^\prime\}} \mu^2_{KG}|\psi|^2+  |Z^\star\psi|^2+\zeta_{\textit{trap}}(r)\big( |\slashed{\nabla}\psi|^2 + (\partial_{t^\star}\psi)^2 \big) 
        &   \lesssim \int_{\{t^\star=\tau_1\}}  |Z^\star\psi|^2+|\partial_t\psi|^2+|\slashed{\nabla}\psi|^2+\mu^2_{KG}|\psi|^2, 
    \end{aligned}
\end{equation}
\begin{equation}\label{eq: rough: theorems 1, eq 2}
    \begin{aligned}
        \int_{\{t^\star=\tau_2\}} |Z^\star\psi|^2+|\partial_{t^\star}\psi|^2+|\slashed{\nabla}\psi|^2 +\mu^2_{KG}|\psi|^2&	\lesssim \int_{\{t^\star=\tau_1\}}  |Z^\star\psi|^2+|\partial_t\psi|^2+|\slashed{\nabla}\psi|^2+\mu^2_{KG}|\psi|^2,\\
        \int_{\{t^\star=\tau_2\}} |Z^\star\psi|^2+|\partial_t\psi|^2+|\slashed{\nabla}\psi|^2+|\psi|^2	&	\lesssim \int_{\{t^\star=\tau_1\}} |Z^\star\psi|^2+|\partial_t\psi|^2+|\slashed{\nabla}\psi|^2 +|\psi|^2,
    \end{aligned}
\end{equation}
for all $0\leq\tau_1\leq \tau_2$, where~$Z^\star$ is a regular spacelike vector field on~$\{t^\star\geq 0\}$ that extends smoothly on the horizons where it is linearly independent from the generators, and~$\zeta_{\textit{trap}}(r)=\left(1-1_{[R^-,R^+]}\right)\left(1-\frac{\frac{R^-+R^+}{2}}{r}\right)^2$, where~$[R^-,R^+]\subset (r_+,\bar{r}_+)$ and~$R^-=R^+=3M$ for~$a=0$.

Moreover, if the solution of the Klein--Gordon equation~\eqref{eq: kleingordon} is axisymmetric~$\partial_{\varphi^\star}\psi=0$, then we also have the estimates~\eqref{eq: rough: theorems 1, eq 1},~\eqref{eq: rough: theorems 1, eq 2}, with~$R^-=R^+$, without assuming the condition~(MS). 
\end{customTheorem}

Note the following remarks

\begin{remark}
	We emphasize that the condition~(MS) concerns only Carter's homogeneous fixed frequency radial ode~\eqref{eq: subsec: sec: intro, subsec 2, eq 1} for real frequencies~$(\omega,m,\ell)\in\mathbb{R}\times\bigcup_{m\in\mathbb{Z}}\{m\}\times\mathbb{Z}_{\geq |m|}$. It does not depend thus on the completeness properties of the angular eigenvalues for complex frequencies. 
\end{remark}

\begin{remark}
By Proposition~\ref{prop: subsec: energy identity, prop 2}, for any fixed~$\mu^2_{KG}\geq 0$, the set of black hole parameters for which~(MS) holds includes~$a=0$. Moreover, by Proposition~\ref{prop: subsec: sec: continuity argument, subsec 1, prop 1}, this set is open. Therefore,~(MS) in particular holds in the slowly rotating case~$|a|\ll M,l$, where the constant implicit in~$\ll$ depends only on~$\mu_{KG}$. Furthermore, in view of the recent work of Hintz~\cite{hintz2021mode} for~$\mu^2_{KG}=0$, we have that~(MS) also holds in the case~$0\leq |a|<M\ll l$, where the constant depends on~$M-|a|$. 
\end{remark}

\begin{remark}
We do not expect that the condition~(MS) holds for all Klein--Gordon masses~$\mu^2_{KG}>0$, see for example the construction~\cite{Shlapentokh_Rothman_2014} of real mode solutions for the Klein--Gordon equation in subextremal Kerr. 
\end{remark}

\subsection{Discussion of the proof}

We give some comments about the proof, highlighting especially some differences from~\cite{DR2}.

\subsubsection{The de Sitter frequencies and trapping}

If the ergoregion is connected and non-empty, namely when
\begin{equation}
	\max_{r\in[r_+,\bar{r}_+]}\frac{\Delta}{a^2}\leq 1,\quad |a|\not{=} 0, 
\end{equation}
where for~$\Delta$ see~\eqref{eq: prototype Delta}, then there exists a complete~$\partial_t$ orthogonal trapped null geodesic, and it never intersects the event horizon~$\mathcal{H}^+$ or the cosmological horizon~$\bar{\mathcal{H}}^+$~(see Appendix~\ref{sec: geodesics}). This geometric phenomenon is not present in the~$\Lambda=0$ case, where note that in~\cite{DR2} the authors prove that in fact there can be no $\partial_t$ orthogonal trapped null geodesic, for any subextremal black hole parameters.

This phenomenon can be seen at the fixed frequency analysis of~\eqref{eq: subsec: sec: intro, subsec 2, eq 1} near~$\omega=0$, for high azimuthal frequencies~$|m|\gg 1$, where trapping again corresponds to a maximum of the potential~$V$, see~\eqref{eq: subsec: sec: intro, subsec 2, eq 2}, such that~$\max V \approx \omega^2$. This is in sharp contrast to the Kerr case, see~\cite{DR2}, where trapping takes place only for large time frequencies~$|\omega|\gg 1$. Specifically, the existence of a~$\partial_t$ orthogonal geodesic for appropriate Kerr--de~Sitter black hole parameters, see Section~\ref{sec: geodesics}, necessitates that our fixed frequency estimates in the `de~Sitter frequencies' :
\begin{equation}\label{eq: subsec: sec: intro: subsec 4, eq 2}
	\mathcal{DSF}=\Biggl\{(\omega,m):~	|\omega|\in \left[0,\frac{|am|\Xi}{\bar{r}_+^2+a^2}\right]\Biggr\}
\end{equation}
will degenerate appropriately. See Section~\ref{subsec: de Sitter frequency regime} for estimates in this frequency range.

The interior of the frequency set~\eqref{eq: subsec: sec: intro: subsec 4, eq 2} is empty in the~$\Lambda=0$ Kerr case. Moreover, the interior of the frequency set~\eqref{eq: subsec: sec: intro: subsec 4, eq 2} is empty in the Schwarzschild de~Sitter case if~$a=0$. In the very slowly rotating case~$|a|\ll M,l$, we note that there is no trapping in this range.

\subsubsection{The continuity argument}

As mentioned earlier, since we rely entirely on real frequency analysis~$\omega\in\mathbb{R}$, the proof of Theorem~\ref{main theorem 1} relies on a continuity argument in the black hole parameters, see Section~\ref{sec: continuity argument}, which is inspired by the relevant continuity argument of~\cite{DR2}.

Note that the set of subextremal black hole parameters for which mode stability on the real axis holds may fail to be connected for given~$\mu^2_{KG}\geq 0$, since mode stability for the full subextremal family has not been proven in the Kerr-de~Sitter case. Therefore, in the condition~(MS) we must restrict to black hole parameters for which mode stability on the real axis holds on an entire curve in parameter space connecting~$(a,M)$ to the Schwarzschild de~Sitter family. See Definition~\ref{def: sec: main theorems, def 1}.

\subsubsection{Constant solutions when~\texorpdfstring{$\mu^2_{KG}=0$}{PDFstring}}

In the case~$\mu^2_{KG}=0$ constant functions are solutions of the wave equation~\eqref{eq: kleingordon}. Thus, in this case one does not have integrated decay for zeroth order terms, see~\eqref{eq: rough: theorems 1, eq 1}, and fixed frequency estimates near~$\omega=0$~(when~$m=0$) take on a different form, see Section~\ref{sec: proof: prop: sec: proofs of the main theorems}, from the relevant ode estimate of~\cite{DR2}. 

To generate zeroth order boundary terms on future hypersurfaces, see the second estimate of~\eqref{eq: rough: theorems 1, eq 2}, we use currents for twisted derivatives, see Section~\ref{subsec: sec: preliminaries, subsec 2}.

\subsection{Outline}

We here present an outline of the structure of our paper.

In Section~\ref{sec: the kerr de sitter spacetime} we give all the necessary background for the Kerr--de~Sitter spacetime, including the coordinates we use, the ergoregion of Kerr--de~Sitter, volume forms, spacetime domains and hypersurfaces, and a useful global causal vector field.

In Section~\ref{sec: delta polynomial} we gather properties of the~$\Delta$ polynomial, which we use throughout the paper. 

In Section~\ref{sec: preliminaries} we give all the necessary background on the Klein--Gordon equation~\eqref{eq: kleingordon} on a Kerr--de~Sitter black hole. 

In Section~\ref{sec: redshift and superradiance} we construct the redshift vector fields and prove redshift estimates.

In Section~\ref{sec: carter separation} we present Carter's separation of variables in Kerr--de~Sitter. Specifically, we discuss in detail the angular ode, the eigenvalues of the angular ode, the radial ode, the necessary Parseval identities, and the Wronskian of the radial ode. This allows us to formulate~(MS).

In Section~\ref{sec: main theorems} we present the detailed version of our main Theorem~\ref{main theorem 1}.

In Section~\ref{sec: the main theorem, no cases} we prove Theorem~\ref{thm: subsec: summing in the redshift estimate, thm -1}, which concerns an inhomogeneous Morawetz estimate for sufficiently integrable functions in the full subextremal parameter range, without the restriction~(MS), but with an extra horizon flux term on the right hand side. Essential in the proof of Theorem~\ref{thm: subsec: summing in the redshift estimate, thm -1} is the fixed frequency Theorem~\ref{thm: sec: proofs of the main theorems, thm 3}. The proof of the fixed frequency Theorem~\ref{thm: sec: proofs of the main theorems, thm 3} is deferred, howerer, to Section~\ref{sec: proof: prop: sec: proofs of the main theorems}.

In Section~\ref{sec: continuity argument} we prove Theorem~\ref{prop: sec: continuity argument, prop 1} with the use of a continuity argument. Specifically, Theorem~\ref{prop: sec: continuity argument, prop 1} states that, under the restriction~(MS), all solutions of the Klein--Gordon equation~\eqref{eq: kleingordon} are in fact `future integrable'.

In Section~\ref{sec: proof of Theorem 1} we prove the Morawetz estimate of our main Theorem~\ref{main theorem 1}, by combining Corollary~\ref{cor: subsec: summing in the redshift estimate, cor 1} and Theorem~\ref{prop: sec: continuity argument, prop 1}. 

In Section~\ref{sec: proof of boundedness} we prove the boundedness estimates of our main Theorem~\ref{main theorem 1}. 

In Section~\ref{sec: proof of main theorem in axisymmetry} we prove the Morawetz and boundedness estimates of our main Theorem~\ref{main theorem 1}, for axisymmetric solutions of the Klein--Gordon equation~\eqref{eq: kleingordon}, for the full subextremal range without the restriction~(MS).

The Sections~\ref{sec: trapping},~\ref{sec: frequency localized multiplier estimates},~\ref{sec: frequencies} are auxilliary in the proof of the fixed frequency Theorem~\ref{thm: sec: proofs of the main theorems, thm 3}, which is proved in Section~\ref{sec: proof: prop: sec: proofs of the main theorems}. Specifically, in Section~\ref{sec: trapping} we study the critical points of the potential~$V_0$. In Section~\ref{sec: frequency localized multiplier estimates} we present fixed frequency currents. In Section~\ref{sec: frequencies} we define the frequency regimes. In Section~\ref{sec: proof: prop: sec: proofs of the main theorems} we finally prove the fixed frequency Theorem~\ref{thm: sec: proofs of the main theorems, thm 3}, as a Corollary of a more general statement, see Theorem~\ref{thm: sec: proofs of the main theorems}. Specifically, to prove Theorem~\ref{thm: sec: proofs of the main theorems}, we integrate appropriately chosen currents as in Section~\ref{sec: frequency localized multiplier estimates} in the frequency regimes of Section~\ref{sec: frequencies}. The estimates of Section~\ref{sec: proof: prop: sec: proofs of the main theorems} should be thought of as the heart of our argument.

In Appendix~\ref{sec: appendix, sec 1} we prove that the set of subextremal Kerr--de~Sitter black hole parameters~$\mathcal{B}_l$, for any~$l>0$, is connected.

In Appendix~\ref{sec: geodesics} we prove that for appropriate subextremal Kerr--de~Sitter black hole parameters~$(a,M,l)$ there exist trapped null and~$\partial_t$-orthogogal geodesics.

\subsection{Acknowledgments} The author would like to thank his PhD supervisor M. Dafermos for introducing him to this problem and for his invaluable suggestions and comments throughout the author's PhD studies.

A preliminary version of the results of the current paper were outlined in my thesis~\cite{mavrogiannisthesis}, while the author was a PhD student at DPMMS Cambridge and a resident of Fitzwilliam College Cambridge.

\section{The Kerr--de~Sitter spacetime}\label{sec: the kerr de sitter spacetime}

We define the background manifold and the black hole horizons. 

\begin{definition}\label{def: sec: the kerr de sitter spacetime, def 1}
We define the `fixed' coordinates 
\begin{equation}
(y^\star,t^\star,\theta^\star,\varphi^\star)
\end{equation}
on the manifold $\mathbb{R}\times\mathbb{R}\times\mathbb{S}^2$. We define the event horizon and the cosmological horizon respectively as  
\begin{equation}
	\begin{aligned}
		&   \mathcal{H}^{+}\:\dot{=}\:\{ y^{\star}=0\},\qquad \bar{\mathcal{H}}^{+}\:\dot{=}\:\{y^{\star}=1\}.\\
	\end{aligned}
\end{equation}
We define the following Kerr--de~Sitter background manifold with boundary
\begin{equation}
    \mathcal{M}\:\dot{=}\:[0,1]_{y^{\star}}\times\mathbb{R}_{t^\star}\times\mathbb{S}^{2}_{(\theta^\star,\varphi^\star)}.
\end{equation}
\end{definition}

\subsection{The $\Delta$ polynomial and the subextremal parameters}\label{subsec: delta polynomial}

Let $(a,M,l)\in\mathbb{R}\times\mathbb{R}_+\times\mathbb{R}_+$. We define the function
\begin{equation}\label{eq: delta polynomial for the horizons}
    \Delta(r)\:\dot{=}\:\left(r^2+a^{2}\right)\left(1-\frac{r^{2}}{l^{2}}\right)-2Mr,
\end{equation}
viewed as a polynomial in $r\in\mathbb{R}$, also see Figure~\ref{fig: Delta}.

\begin{figure}[htbp]
    \centering
    \includegraphics[scale=0.9]{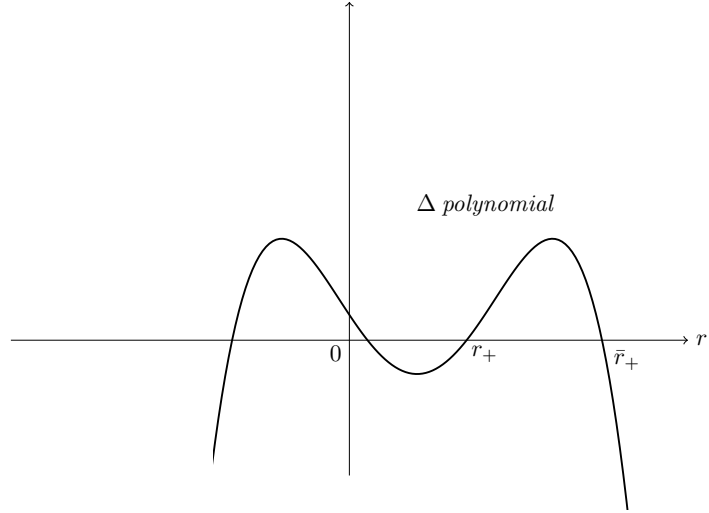}
    \caption{The $\Delta$ polynomial with 4 real roots}
    \label{fig: Delta}
\end{figure}

We next define the subextremal set of Kerr--de~Sitter black hole parameters.

\begin{definition}\label{def: subextremality, and roots of Delta}
Let~$l>0$. We define the open set
\begin{equation}\label{eq: subextremal set}
    \mathcal{B}_{l}\:\dot{=}\:\{ (a,M)\in \mathbb{R}\times(0,\infty): ~ \Delta ~\textrm{attains four distinct real roots} \}.
\end{equation}

If $l>0$ and $(a,M)\in\mathcal{B}_l$ we say that the black hole parameters $(a,M,l)$ correspond to a subextremal Kerr--de~Sitter black hole and denote as 
\begin{equation}
    \bar{r}_-(a,M,l)\quad <\quad 0 \quad \leq\quad r_-(a,M,l)\quad<\quad r_+(a,M,l)\quad<\quad \bar{r}_+(a,M,l)
\end{equation}
the four distinct real roots of the $\Delta$ polynomial. Note that in the Schwarzschild--de~Sitter case~$(a=0)$ we have~$r_-=0$. 
\end{definition}

Note the following lemma

\begin{lemma}\label{lem: subsec: delta polynomial, lem 1}
Let~$l>0$. Then, the set~$\mathcal{B}_l$, see Definition~\ref{def: subextremality, and roots of Delta}, is connected. 
\end{lemma}

\begin{proof}
	See Appendix~\ref{sec: appendix, sec 1}.
\end{proof}

\subsection{Modified Kerr--de~Sitter-star coordinates}\label{subsec: Kerr dS star coordinates}

We define the modified Kerr--de~Sitter star coordinates, which cover the entire background manifold~$\mathcal{M}$ of Definition~\ref{def: sec: the kerr de sitter spacetime, def 1}.

\begin{definition}\label{def: kerr star coordinates, manifold}
Let $l>0$ and~$(a,M)\in\mathcal{B}_l$. We choose a smooth map
\begin{equation}
	\begin{aligned}
		r		: \mathbb{R}_{>0}\times \mathcal{B}_l \times [0,1] \:\: &\rightarrow [r_{+}(a,M,l),\bar{r}_{+}(a,M,l)] \\
		 (l,a,M,y^\star) &\mapsto r (a,M,l,y^\star) \in [r_{+}(a,M,l),\bar{r}_{+}(a,M,l)],
	\end{aligned}
\end{equation}
such that for fixed~$l,a,M$ the map is a diffeomorphism~$[0,1]\rightarrow [r_+,\bar{r}_+]$. Then, we define the modified Kerr--de~Sitter star coordinates to be
\begin{equation}
(r,t^\star,\theta^\star,\varphi^\star).
\end{equation} 
\end{definition}

The Kerr--de~Sitter manifold~$\mathcal{M}$, of Definition~\ref{def: sec: the kerr de sitter spacetime, def 1}, can now be written as 
\begin{equation}
    \mathcal{M}=[r_+,\bar{r}_+]_r \times \mathbb{R}_{t^\star}\times \mathbb{S}^2_{(\theta^\star,\varphi^\star)},
\end{equation}
where
\begin{equation}
\begin{aligned}
&   \mathcal{H}^{+}\:\dot{=}\:\{ r=r_+\},\qquad \bar{\mathcal{H}}^{+}\:\dot{=}\:\{r=\bar{r}_+\}.\\
\end{aligned}
\end{equation}

We have the following remarks

\begin{remark}\label{rem: sec: the kerr de sitter spacetime, rem 1}
	Note from Definition~\ref{def: sec: the kerr de sitter spacetime, def 1} that the Kerr--de~Sitter star vector fields~$\partial_{t^\star},\partial_{\varphi^\star}$ are independent of the black hole parameters~$(a,M,l)$.
\end{remark}

\begin{remark}\label{rem: sec: the kerr de sitter spacetime, rem 2}
	Note that the leaves
	\begin{equation}
		\{t^\star=c\}
	\end{equation}
	of the modified Kerr--de~Sitter star coordinates of Definition~\ref{def: kerr star coordinates, manifold} connect the event horizon~$\mathcal{H}^+$ with the cosmological horizon~$\bar{\mathcal{H}}^+$. Moreover, note that the leaves~$\{t^\star=c\}$ are independent of the black hole parameters~$(a,M,l)$. This will be useful later in the continuity argument, see already Section~\ref{sec: continuity argument}.
\end{remark}

\subsection{Boyer--Lindquist coordinates}\label{subsec: boyer Lindquist coordinates}
We need the following definition

\begin{definition}\label{def: delta+, and other polynomials}
Let~$l>0$ and~$(a,M)\in \mathcal{B}_{l}$. Then, we define
\begin{equation}
    \begin{aligned}
        \rho^2 \:\dot{=}&\:r^{2}+a^{2}\cos^{2}\theta,\qquad \Delta_{\theta}\:\dot{=}\:1+\frac{a^{2}}{l^{2}} \cos^{2}\theta,\qquad \Xi\:\dot{=}\:1+\frac{a^{2}}{l^{2}}. 
    \end{aligned}
\end{equation}
\end{definition}

We define the Boyer--Lindquist coordinates

\begin{definition}\label{def: boyer lindquist}
Let~$l>0$ and~$(a,M)\in\mathcal{B}_{l}$. Moreover, let~$\epsilon_{BL}(a,M,l)>0$ be sufficiently small. We define the Boyer--Lindquist coordinates 
\begin{equation}
(t,r,\theta,\varphi)
\end{equation}
on $\mathcal{M}^o$ as follows
\begin{equation}
    t=t^\star+A(r;l,a,M),\qquad \varphi=\varphi^\star+B(r;l,a,M)\mod{2\pi}, \qquad \theta=\theta^\star.
\end{equation}
Here,~$A(r),B(r)$ satisfy
\begin{equation}\label{eq: def: boyer lindquist, eq 2}
    \begin{aligned}
     \frac{dA}{dr}    =-\frac{r^2+a^2}{\Delta},\qquad\frac{dB}{dr} &=-\frac{a\Xi}{\Delta},\qquad r\in (r_+,r_++\epsilon_{BL}),\\
     \frac{dA}{dr} =\frac{r^2+a^2}{\Delta},\qquad\frac{d B}{dr} &= \frac{a\Xi}{\Delta},\qquad r \in (\bar{r}_+-\epsilon_{BL},\bar{r}_+),
     \end{aligned}
\end{equation}
and we extend~$A(r(y^\star);l,a,M),B(r(y^\star);l,a,M)$ to~$(0,1)\times \bigcup_{l>0}\left(\{l\}\times\mathcal{B}_l\right)$ in a smooth manner such that the hypersurfaces~$\{t^\star=c\}$ are spacelike with respect to the metric~$g_{a,M,l}$, see already Definition~\ref{def: subsec: boyer Lindquist coordinates, def 1}. 
\end{definition}

We now define the Kerr--de~Sitter metric in Boyer--Lindquist coordinates to be 

\begin{definition}\label{def: subsec: boyer Lindquist coordinates, def 1}
	Let~$l>0$ and~$(a,M)\in\mathcal{B}_l$. We define the Kerr--de~Sitter metric as 
	\begin{equation}\label{eq: metric in boyer lindquist coordinates}
	\begin{aligned}
	g_{a,M,l} &	= \frac{\rho^2}{\Delta}dr^{2}+\frac{\rho^2}{\Delta_{\theta}}d \theta^{2}+\frac{\Delta_{\theta}(r^{2}+a^{2})^{2}-\Delta a^{2}\sin^{2}\theta }{\Xi ^{2}\rho^2}\sin^{2}\theta d \varphi^{2}   -2\frac{\Delta_{\theta}(r^{2}+a^{2})-\Delta}{\Xi\rho^2} a \sin^{2}\theta d\varphi dt\\
	&\qquad\qquad -\frac{\Delta -\Delta_{\theta}a^{2}\sin^{2}\theta}{\rho^2}dt^{2},
	\end{aligned}
	\end{equation}
	where $\Delta$ has been defined in~\eqref{eq: prototype Delta}, and  also see Definition~\ref{def: delta+, and other polynomials} for~$\Xi,\rho^2,\Delta_\theta$. The metric~\eqref{eq: metric in boyer lindquist coordinates} extends smoothly to~$\mathcal{M}$ and moreover the metric depends smoothly on~$a,M$ for fixed~$l$.

	Note that the vector field~$W=\partial_t+\frac{a\Xi}{r^2+a^2}\partial_{\varphi}$ is timelike away from the horizons, see already Lemma~\ref{lem: causal vf E,1}. We time orient~$\mathcal{M}$ by the vector~$W$.  
\end{definition}

In view of the form of the metric~$g_{a,M,l}$, we may deduce that the event horizon~$\mathcal{H}^+$ and the cosmological horizon~$\bar{\mathcal{H}}^+$ are null hypersurfaces and moreover are the future boundaries of the manifold~$\mathcal{M}$.

\begin{remark}
Note the identities $\partial_t=\partial_{t^\star},\partial_{\varphi}=\partial_{\varphi^\star}$.
\end{remark}

\subsection{The ergoregion in Kerr--de~Sitter}\label{subsec: ergoregion}

We define the ergoregion of the Kerr--de~Sitter spacetime to be 
\begin{equation}\label{eq: subsec: ergoregion, eq 1}
	\textit{Ergo}\:\dot{=}\: \{g(\partial_t,\partial_t)>0\}= \{\Delta-\Delta_\theta a^2\sin^2\theta< 0\}.
\end{equation}

\begin{remark}
In Section~\ref{sec: geodesics} we will prove that~$\partial_t$ orthogonal trapping takes place in Kerr--de~Sitter if and only if the ergoregion is connected and non empty, namely when
	\begin{equation}
		\max_{r\in[r_+,\bar{r}_+]}\frac{\Delta}{a^2}\leq 1,\quad |a|\not{=}0.
	\end{equation}
\end{remark}

\subsection{Tortoise coordinate}\label{subsec: tortoise coordinate}

Let~$l>0$ and~$(a,M)\in \mathcal{B}_l$. We define the tortoise coordinate~$r^\star:(r_+,\bar{r}_+)\rightarrow (-\infty,+\infty)$ as follows 
\begin{equation}\label{eq: tortoise coordinate}
    \frac{dr^\star}{dr}=\frac{r^{2}+a^{2}}{\Delta(r)},
\end{equation}
where for~$\Delta(r)$ see~\eqref{eq: delta polynomial for the horizons}, with~$r^\star(r_{\Delta,\textit{frac}})=0$, where 
\begin{equation}
r_{\Delta,\textit{frac}}\in (r_+,\bar{r}_+)
\end{equation}
is the unique local maximum of 
\begin{equation}
\frac{\Delta}{(r^2+a^2)^2},
\end{equation}
see already Lemma~\ref{lem: sec: general properties of Delta, lem 3}. We note that~$\lim_{r\rightarrow r_+} r^\star (r)=-\infty$ and~$\lim_{r\rightarrow\bar{r}_+} r^\star(r)=+\infty$, and moreover note that~$r^\star(r)$ is a diffeomorphism mapping~$(r_+,\bar{r}_+)\rightarrow (-\infty,+\infty)$.

We will often use the primed notation
\begin{equation}
	^\prime=\frac{d}{dr^\star}
\end{equation}
for the~$r^\star$-derivative. Finally, for a value~$\alpha\in (r_+,\bar{r}_+)$, we will often use the notation 
\begin{equation}
	\alpha^\star=r^\star (\alpha).
\end{equation}

\subsection{The regular vector field \texorpdfstring{$Z^\star$}{Z}}\label{subsec: boldsymbol partial r}

Let~$l>0$ and~$(a,M)\in\mathcal{B}_l$. Let~$\partial_r|_{star}$ be the coordinate vector field of the modified Kerr--de~Sitter star coordinates~$(t^\star,r,\theta^\star,\varphi^\star)$, see Definition~\ref{def: kerr star coordinates, manifold}, and~$\partial_r|_{\textit{BL}}$ be the coordinate vector field of the Boyer--Lindquist coordinates~$(t,r,\theta,\varphi)$, see Definition~\ref{def: boyer lindquist}. We will define here a hybrid vector field.

For a sufficiently small
\begin{equation}
\epsilon_{\textit{hyb}}(a,M,l)>0,
\end{equation}
we choose a smooth cut-off that satisfies 
\begin{equation}
\chi_{\textit{hyb}}(r)=
\begin{cases}
0,\:& r\in (r_+,r_++\epsilon_{\textit{hyb}})\cup (\bar{r}_+-\epsilon_{\textit{hyb}},\bar{r}_+),\\
1,\: & r\in(r_++2\epsilon_{\textit{hyb}},\bar{r}_+-2\epsilon_{\textit{hyb}}). 
\end{cases}    
\end{equation}

Then, we define the hybrid vector field 
\begin{equation}\label{eq: subsec: boldsymbol partial r, eq 1}
    Z^\star= (1-\chi_{\textit{hyb}}(r))\partial_r|_{\textit{star}}+\chi_{\textit{hyb}}(r) \partial_{r}|_{\textit{BL}}.
\end{equation}
It is easy to see from the Definitions~\ref{def: kerr star coordinates, manifold},~\ref{def: boyer lindquist} of the Kerr--de~Sitter star coordinates and the Boyer--Lindquist coordinates respectively and from the definition of the Kerr--de~Sitter metric, see Definition~\ref{def: subsec: boyer Lindquist coordinates, def 1}, that for a sufficiently small~$\epsilon_{hyb}(a,M,l)>0$, the vector field~$Z^\star$ is spacelike. Also see the relevant computation of~\cite{Dyatlov1}.

The vector field $Z^\star$ coincides with the tortoise coordinate vector field $\partial_{r}|_{BL}$, of Boyer--Lindquist coordinates, in the region~$r\in[R^-,R^+]$, where for~$R^\pm$ see Theorem~\ref{main theorem 1}.

The hybrid vector field~\eqref{eq: subsec: boldsymbol partial r, eq 1} is translation invariant and regular on the event horizon~$\mathcal{H}^+$ and on the cosmological horizon~$\bar{\mathcal{H}}^+$. The significance of~$Z^\star$ will be that the derivatives~$(Z^\star\psi)^2$ do \texttt{not} degenerate in the integrand of the Morawetz estimate~\eqref{eq: rough: theorems 1, eq 1} of Theorem~\ref{main theorem 1}.

\subsection{Spacelike hypersurfaces and their normal vector fields}\label{subsec: admissible hypersurfaces}

A prototype spacelike hypersurface that connects the event horizon~$\mathcal{H}^+$ with the cosmological horizon~$\bar{\mathcal{H}}^+$ is
\begin{equation}
\{t^\star=0\}.
\end{equation}

For any~$\tau\geq 0$ the hypersurfaces 
\begin{equation}
    \{t^\star=\tau\}
\end{equation}
are Cauchy for their future and connect the event horizon~$\mathcal{H}^+$ with the cosmological horizon~$\bar{\mathcal{H}}^+$. We denote as 
\begin{equation}
    n_{\{t^\star=\tau\}} 
\end{equation}
the future directed unit normal of the hypersurfaces~$\{t^\star=\tau\}$, which we often simply denote as 
\begin{equation}\label{eq: Hawking vector field on the event horizon, -1}
	n.
\end{equation}

\subsection{Hawking--Reall vector fields}\label{subsec: Hawking-Reall v.f.}

We define the null normals of the event horizon~$\mathcal{H}^+$ and the cosmological horizon~$\bar{\mathcal{H}}^+$ respectively as the Hawking--Reall vector fields
\begin{equation}\label{eq: Hawking vector field on the event horizon}
K=\partial_{t^\star}+\frac{a\Xi}{r_{+}^{2}+a^{2}}\partial_{\varphi^\star},\qquad \bar{K}=\partial_{t^\star}+\frac{a\Xi}{\bar{r}_{+}^{2}+a^{2}}\partial_{\varphi^\star}.
\end{equation}

When there is no risk of confusion with~\eqref{eq: Hawking vector field on the event horizon, -1} we will denote the normal vector fields~\eqref{eq: Hawking vector field on the event horizon} simply as 
\begin{equation}
	n.
\end{equation}

\subsection{Causal domains}\label{subsec: causal domains}

We define certain causal spacetime domains.

\begin{definition}\label{def: causal domain, def 1}
We define
\begin{equation}
    D(\tau_1,\tau_2)\:\dot{=}\: \{\tau_1\leq t^\star\leq \tau_2\},\qquad     D(\tau_1,\infty)\:\dot{=}\: \{t^\star\geq \tau_1\},
\end{equation}
where for the hypersurfaces~$\{t^\star=\tau\}$ see Section~\ref{subsec: admissible hypersurfaces}.
\end{definition}

We here extend the domain of Definition~\ref{def: causal domain, def 1}. 

\begin{definition}\label{def: causal domain, def 2}
	Let~$l>0$ and~$(a,M)\in\mathcal{B}_l$. Let~$\delta>0$ be sufficiently small. We define the extended manifold 
	\begin{equation}
		\mathcal{M}_{\delta}=[r_+-\delta,\bar{r}_++\delta]_r\times \mathbb{R}_{t^\star}\times\mathbb{S}^2_{(\theta^\star,\varphi^\star)},
	\end{equation}
	with respect to the Kerr--de~Sitter star coordinates~$(t^\star,r,\theta^\star,\varphi^\star)$ see Section~\ref{subsec: sec: preliminaries, subsec 2}. The Kerr--de~Sitter metric~\eqref{eq: metric in boyer lindquist coordinates} extends smoothly to~$\mathcal{M}_\delta$ by its analytic expression near~$r_+,\bar{r}_+$ in Kerr--de~Sitter star coordinates. Note that the hypersurfaces~$\{t^\star=c\}$ in~$\mathcal{M}_\delta$ are spacelike for sufficiently small~$\delta>0$. We define the spacelike boundaries of the manifold~$\mathcal{M}_\delta$ to be 
	\begin{equation}
		\mathcal{H}^+_\delta=\{r=r_+-\delta,\:t^\star> -\infty\},\qquad \bar{\mathcal{H}}^+_\delta=\{r=\bar{r}_++\delta,\: t^\star> -\infty\}.
	\end{equation}
	Now, we define the extended causal domains
	\begin{equation}
		D_\delta(\tau_1,\tau_2)\:\dot{=}\: \{\tau_1 \leq t^\star\leq \tau_2\}\subset \mathcal{M}_\delta,\qquad  D_\delta(\tau_1,+\infty)\:\dot{=}\: \{t^\star\geq \tau_1\}\subset\mathcal{M}_\delta. 
	\end{equation}
\end{definition}

\subsection{The induced metric on the $r$-sphere and its covariant derivative}\label{subsec: unit sphere}
We use
\begin{equation}
\slashed{g},\qquad \slashed{\nabla},\qquad d\sigma_{\mathbb{S}^2}
\end{equation}
to denote respectively the induced~$g_{a,M,l}$ metric on the~$\mathbb{S}^2$ factors of~$\mathcal{M}$, the covariant derivative with respect to~$\slashed{g}$ and the standard metric of the unit sphere.

\subsection{Volume forms and notation}\label{subsec: volume forms of spacelike hypersurfaces}

The spacetime volume form is
\begin{equation}\label{eq: subsec: volume forms of spacelike hypersurfaces, eq 0}
dg= v(r,\theta) dt^\star drd\slashed{g}=\frac{\Delta}{r^{2}+a^{2}} v(r,\theta)dt^\star dr^\star d\slashed{g},
\end{equation}
and~$d\slashed{g}= v(r,\theta)\frac{\rho^2}{\Xi}  \sin\theta d\theta d\varphi$ with~$v(r,\theta)\sim 1$, where for~$\Delta,\:\rho^2=r^2+a^2\cos^2\theta$ see~\eqref{eq: delta polynomial for the horizons} and~Definition~\ref{def: delta+, and other polynomials} respectively.

We denote the induced volume form of the spacelike hypersurface~$\{t^\star=\tau\}$ as  
\begin{equation}\label{eq: subsec: volume forms of spacelike hypersurfaces, eq 1}
	dg_{\{t^\star=\tau\}}.
\end{equation}
It can be characterised as the unique 3-form such that 
\begin{equation}
    dg= n_{\{t^\star=\tau\}} \wedge dg_{\{t^\star=\tau\}},
\end{equation}
where~$n_{\{t^\star=\tau\}}$ is the unit normal of the leaf~$\{t^\star=\tau\}$ see Section~\ref{subsec: admissible hypersurfaces}. Note that the volume form~\eqref{eq: subsec: volume forms of spacelike hypersurfaces, eq 1}, can be written as 
\begin{equation}
	dg_{\{t^\star=\tau\}}= v(r,\theta) dr d\theta^\star d\varphi^\star,
\end{equation}
with~$v(r,\theta)\sim r\sin\theta$, where the constant in the similarity depends only on the black hole parameters~$(a,M,l)$ and does not degenerate in the Kerr limit~$l\rightarrow+\infty$.

We denote the volume forms of the null hypersurfaces~$\mathcal{H}^+,\bar{\mathcal{H}}^+$ respectively as 
\begin{equation}\label{eq: subsec: volume forms of spacelike hypersurfaces, eq 2}
	dg_{\mathcal{H}^+},\qquad dg_{\bar{\mathcal{H}}^+}
\end{equation}
where these are defined to be the unique 3-forms such that
\begin{equation}
	dg=(K^+)^\flat\wedge dg_{\mathcal{H}^+},\qquad dg=(\bar{K}^+)^\flat \wedge dg_{\bar{\mathcal{H}}^+},
\end{equation}
where~$\flat$ here is the flat-musical isomorphism with respect to the Kerr--de~Sitter metric, see~\eqref{eq: metric in boyer lindquist coordinates}, and~$K,\bar{K}$ are the Hawking--Reall vector fields of Section~\ref{subsec: Hawking-Reall v.f.}.

From this point onward, when we integrate over the hypersurfaces 
\begin{equation}
	\{t^\star=\tau\},\qquad \mathcal{H}^+,\qquad \bar{\mathcal{H}}^+
\end{equation}
the volume forms~\eqref{eq: subsec: volume forms of spacelike hypersurfaces, eq 1},~\eqref{eq: subsec: volume forms of spacelike hypersurfaces, eq 2} respectively are to be understood, if no volume form is explicitly denoted. Similarly, when we integrate over spacetime the volume form~\eqref{eq: subsec: volume forms of spacelike hypersurfaces, eq 0} is to be understood, if no volume form is explicitly denoted.

\subsection{Coarea formula}\label{subsec: coarea formula}

Let~$f$ be a continuous non-negative function. Then, note that the following holds 
\begin{equation}\label{eq: subsec: coarea formula, eq 1}
\int\int_{D(\tau_1,\tau_2)} \frac{1}{r}\: f dg \sim  \int _{\tau_1}^{\tau_2}d\tau \int_{\{t^\star=\tau\}} f dg_{\{t^\star=\tau\}},
\end{equation}
where for the volume forms of the integrals of~\eqref{eq: subsec: coarea formula, eq 1} see Section~\ref{subsec: volume forms of spacelike hypersurfaces}. (Note that the constants in the similarity of inequality~\eqref{eq: subsec: coarea formula, eq 1}, with the inclusion of the~$\frac{1}{r}$ factor, do not degenerate in the Kerr limit~$l=\infty$.)

\subsection{The future oriented vector field~\texorpdfstring{$W$}{g}}

Note the following Lemma 

\begin{lemma}\label{lem: causal vf E,1}
	Let~$l>0$ and~$(a,M)\in\mathcal{B}_l$. The vector field
	\begin{equation}
	W\:\dot{=}\:\partial_{t^\star}+\frac{a\Xi}{r^2+a^2} \partial_{\varphi^\star},  
	\end{equation}
	is timelike in~$\{r_+<r<\bar{r}_+\}$ and null on the event horizon~$\mathcal{H}^+$ and on the cosmological horizon~$\bar{\mathcal{H}}^+$.
\end{lemma}
\begin{proof}
	Consider a vector field of the form
	\begin{equation}
	\partial_{t^\star}+f(r)\partial_{\varphi^\star}, 
	\end{equation}
	for a smooth~$f(r)$. We compute the following 
	\begin{equation}\label{eq: proof lem: causal vf E, eq 1}
	\begin{aligned}
	&   g(\partial_t+f(r)\partial_{\varphi},\partial_t+f(r)\partial_{\varphi})\\
	&   \: =\frac{1}{\rho^2}\left(-(\Delta-\Delta_\theta a^2\sin^2\theta)-\frac{2}{\Xi}f(r)(\Delta_\theta(r^2+a^2)-\Delta)a\sin^2\theta +\frac{1}{\Xi^2}f^2(r)\left(\Delta_\theta (r^2+a^2)^2-\Delta a^2\sin^2\theta\right)\sin^2\theta\right)\\
	&   \: =\frac{1}{\rho^2}\left(\Delta +\sin^2\theta \left(\Delta_\theta \left(a-\frac{r^2+a^2}{\Xi}f(r)\right)^2+\frac{2}{\Xi}f(r)\Delta a -\frac{1}{\Xi^2}f^2(r)\Delta a^2\sin^2\theta\right)\right)\\
	&   \: =\frac{1}{\rho^2}\left(-\Delta +\sin^2\theta \frac{2}{\Xi}f(r)\Delta a +\sin^2\theta \left(\Delta_\theta\left(a-\frac{r^2+a^2}{\Xi}f(r)\right)^2-\frac{1}{\Xi^2}f^2(r)\Delta a^2\sin^2\theta\right) \right),\\
	\end{aligned}
	\end{equation}
	where~$\rho^2=r^2+a^2\cos^2\theta$,~$\Delta_{\theta}\:\dot{=}\:1+\frac{a^{2}}{l^{2}} \cos^{2}\theta$, see Definition~\ref{def: delta+, and other polynomials}.
	Therefore, for 
	\begin{equation}
	f(r)=\frac{a\Xi}{r^2+a^2},
	\end{equation}
	equation \eqref{eq: proof lem: causal vf E, eq 1} implies 
	\begin{equation}\label{eq: proof lem: causal vf E, eq 2}
	\begin{aligned}
	g(\partial_t+\frac{a\Xi}{r^2+a^2} \partial_{\varphi},\partial_t+\frac{a\Xi}{r^2+a^2} \partial_{\varphi})	=\frac{\Delta}{\rho^2}\left(-1 +\sin^2\theta  \frac{2a^2}{r^2+a^2} -\sin^2\theta \left(\frac{1}{\Xi^2}\left(\frac{a\Xi}{r^2+a^2}\right)^2 a^2\sin^2\theta\right) \right).
	\end{aligned}
	\end{equation}
	We note
	\begin{equation}
		r_+>|a|,
	\end{equation}
	see already Lemma~\ref{lem: sec: properties of Delta, lem 2, a,M,l}, and therefore, for~$r\in[r_+,\bar{r}_+]$ the following holds
	\begin{equation}
		-1+\sin^2\theta \frac{2}{\Xi}\frac{a^2\Xi}{r^2+a^2}=-1+\frac{2a^2\sin^2\theta}{r^2+a^2}<0.
	\end{equation}	
	 Therefore, indeed the right hand side of equation~\eqref{eq: proof lem: causal vf E, eq 2} is negative in~$(r_+,\bar{r}_+)$ and zero where~$r=r_+,\bar{r}_+$.
\end{proof}

\begin{remark}
	Note that our proof of Lemma~\ref{lem: causal vf E,1} also holds for any~$\Lambda=0$ Kerr black hole, simply by substituting~$\Xi=1$~$(l=+\infty)$ in the proof.  
\end{remark}

Recall that in Definition~\ref{def: subsec: boyer Lindquist coordinates, def 1} we chose the time orientation so that the vector field~$W$, of the above Lemma~\ref{lem: causal vf E,1}, is future oriented.

\section{Properties of the \texorpdfstring{$\Delta$}{Delta} polynomial}\label{sec: delta polynomial}

We need the following Lemma.

\begin{lemma}\label{lem: sec: properties of Delta, lem 1, derivatives of Delta}
Let~$l>0$ and~$(a,M)\in\mathcal{B}_l$. Then, the following hold:
    \begin{equation}\label{eq: lem: sec: properties of Delta, lem 1, derivatives of Delta, eq 1}
        \begin{aligned}
            &   \frac{d\Delta}{dr}=2r-\frac{4}{l^{2}}r^{3}-a^{2}\frac{2}{l^{2}}r-2M,\qquad \frac{d^2\Delta}{dr^2}=2-a^{2}\frac{2}{l^{2}}-\frac{12}{l^{2}}r^{2},\qquad \frac{d^3\Delta}{dr^3}=-\frac{24}{l^2}r, \\  
        \end{aligned}
    \end{equation}
    \begin{equation}\label{eq: lem: sec: properties of Delta, lem 1, derivatives of Delta, eq 3}
    	\frac{d\Delta}{dr}(r_+)>0,\quad \frac{d\Delta}{dr}(\bar{r}_+)<0,\quad \frac{d}{dr}\left(\frac{\Delta}{(r^2+a^2)^2}\right)(r_+)>0,\quad \frac{d}{dr}\left(\frac{\Delta}{(r^2+a^2)^2}\right)(\bar{r}_+)<0,
    \end{equation}
where for~$\Delta$ see~\eqref{eq: delta polynomial for the horizons}. 
\end{lemma}
\begin{proof}

The identities~\eqref{eq: lem: sec: properties of Delta, lem 1, derivatives of Delta, eq 1} follow in a straightforward manner. The inequalities~\eqref{eq: lem: sec: properties of Delta, lem 1, derivatives of Delta, eq 3} follow immediately from the subextremality condition of Definition~\ref{def: subextremality, and roots of Delta}.
\end{proof}

We need the following lemma

\begin{lemma}\label{lem: sec: properties of Delta, lem 2, a,M,l}
Let $l>0$ and $(a,M)\in\mathcal{B}_{l}$. We have the following
\begin{equation}\label{eq: lem: sec: properties of Delta, lem 2, a,M,l, eq 1}
\bar{r}_-<0<r_-<r_+<\bar{r}_+,
\end{equation}
\begin{equation}\label{eq: lem: sec: properties of Delta, lem 2, a,M,l, eq 2}
M<r_+<\bar{r}_+<l,
\end{equation}
\begin{equation}\label{eq: lem: sec: properties of Delta, lem 2, a,M,l, eq 3}
r_+>|a|,
\end{equation}
\begin{equation}\label{eq: lem: sec: properties of Delta, lem 2, a,M,l, eq 5}
\frac{a^2}{l^2}<\frac{1}{4},
\end{equation}
	\begin{equation}\label{eq: lem: sec: properties of Delta, lem 2, a,M,l, eq 6}
\Big|\frac{a}{M}\Big|<\frac{12}{10}.
\end{equation}

Finally, 
\begin{equation}\label{eq: lem: sec: properties of Delta, lem 2, a,M,l, eq 4}
\bar{r}_+^2> \frac{1}{7}l^2
\end{equation}
and therefore~$\bar{r}_+\rightarrow \infty$ for $l\rightarrow+\infty$ ($\Lambda\rightarrow 0$). 
\end{lemma}
\begin{proof}
We prove~\eqref{eq: lem: sec: properties of Delta, lem 2, a,M,l, eq 1}. Note from Lemma~\ref{lem: sec: properties of Delta, lem 1, derivatives of Delta}, that the following hold
\begin{equation}\label{eq: lem: sec: properties of Delta, lem 2, a,M,l,/ eq 0}
    \begin{aligned}
         &  \Delta (0) =a^2>0, \qquad  \frac{d\Delta}{dr}(0)=-2M<0, \qquad \frac{d^2\Delta}{dr^2}(0)=2-2\frac{a^2}{l^2}>0.
    \end{aligned}
\end{equation}
Now, we note that $\frac{d^2\Delta}{dr^2}$ attains exactly two roots, see Lemma \ref{lem: sec: properties of Delta, lem 1, derivatives of Delta}. For a point $r_1\in (r_+,\bar{r}_+)$ such that 
\begin{equation}
\frac{d\Delta}{dr} (r_1)<0
\end{equation}
we obtain 
\begin{equation}
\frac{d^2\Delta}{dr^2} (r_1)\leq 0. 
\end{equation}
Therefore, 
\begin{equation}
r_->0
\end{equation}
follows from~\eqref{eq: lem: sec: properties of Delta, lem 2, a,M,l,/ eq 0}. Therefore, we obtain
\begin{equation}
\bar{r}_-=-\left(r_-+r_++\bar{r}_+\right)<0. 
\end{equation}

We now prove~\eqref{eq: lem: sec: properties of Delta, lem 2, a,M,l, eq 2}. Note that the inequality $\frac{d\Delta}{dr}(r_+)>0$ implies
\begin{equation}
    2(r_+-M)>\frac{4}{l^2}r_+^3+a^2\frac{2}{l^2}r_+>0
\end{equation}
from which we obtain
\begin{equation}
	r_+> M,
\end{equation}
since $r_+>0$. Finally, the following holds
\begin{equation}
    (\bar{r}_+^2+a^{2})(1-\frac{\bar{r}_+^2}{l^{2}})-2M\bar{r}_+=0
\end{equation}
from which the following holds~$	(\bar{r}_+^2+a^{2})(1-\frac{\bar{r}_+^2}{l^{2}})>0$ and therefore
\begin{equation}
|\bar{r}_+|<l,
\end{equation}
which concludes the result.

Now, we prove~\eqref{eq: lem: sec: properties of Delta, lem 2, a,M,l, eq 3}. We write the polynomial $\Delta$ in the form 
\begin{equation}
    \Delta=-\frac{1}{l^2}(r-\bar{r}_-)(r-r_-)(r-r_+)(r-\bar{r}_+),
\end{equation}
then by using the definition of~$\Delta$, see~\eqref{eq: delta polynomial for the horizons}, and by matching the relevant powers of~$r$, we obtain the following
\begin{equation}\label{eq: lem: sec: properties of Delta, lem 2, a,M,l,/ eq 1}
    \bar{r}_-+r_-+r_++\bar{r}_+=0,
\end{equation}
\begin{equation}\label{eq: lem: sec: properties of Delta, lem 2, a,M,l,/ eq 2}
    \bar{r}_-r_-r_+\bar{r}_+=-l^2a^2,
\end{equation} 
\begin{equation}\label{eq: lem: sec: properties of Delta, lem 2, a,M,l,/ eq 3}
    r_+\bar{r}_+ +r_-r_+ + r_-\bar{r}_+ +\bar{r}_-r_+ +\bar{r}_-\bar{r}_++r_-\bar{r}_-=-l^2+a^2.
\end{equation}
We use~\eqref{eq: lem: sec: properties of Delta, lem 2, a,M,l,/ eq 1} in conjunction with~\eqref{eq: lem: sec: properties of Delta, lem 2, a,M,l,/ eq 3} to obtain
\begin{equation}\label{eq: lem: sec: properties of Delta, lem 2, a,M,l,/ eq 4}
    a^2+r_+^2+\bar{r}_+^2+r_-^2+r_-r_++r_+\bar{r}_++\bar{r}_+r_-=l^2.
\end{equation}
By using~\eqref{eq: lem: sec: properties of Delta, lem 2, a,M,l,/ eq 1} and~\eqref{eq: lem: sec: properties of Delta, lem 2, a,M,l,/ eq 2} we get 
\begin{equation}
    r_-r_+\bar{r}_+(r_-+r_++\bar{r}_+)=l^2a^2,
\end{equation}
which, in conjunction with~\eqref{eq: lem: sec: properties of Delta, lem 2, a,M,l,/ eq 4}, implies
\begin{equation}
    l^2a^2<r_-r_+ l^2<l^2 r_+^2,
\end{equation}
from which~\eqref{eq: lem: sec: properties of Delta, lem 2, a,M,l, eq 3} is immediate.

Now, we prove~\eqref{eq: lem: sec: properties of Delta, lem 2, a,M,l, eq 5}. We use equation~\eqref{eq: lem: sec: properties of Delta, lem 2, a,M,l,/ eq 4}, in conjunction with the bounds 
\begin{equation}
\bar{r}_+>r_+>|a|,\qquad r_->0
\end{equation}
to obtain 
\begin{equation}
    a^2+a^2+a^2+a^2+\left(r_-^2+r_-r_++\bar{r}_+r_-\right)<l^2
\end{equation}
which indeed implies~\eqref{eq: lem: sec: properties of Delta, lem 2, a,M,l, eq 5}.

Now, we prove~\eqref{eq: lem: sec: properties of Delta, lem 2, a,M,l, eq 6}. The necessary and sufficient conditions on the coefficients of a quartic polynomial to attain four distinct real roots have been studied extensively, see for example~\cite{polynomials}. From~\cite{polynomials} we note that~$\Delta$ attains exactly 4 distinct real roots if and only if the following two inequalities hold
\begin{equation}\label{eq: lem: sec: properties of Delta, lem 2, a,M,l,/ eq 5}
    \begin{aligned}
         &-256\left(\frac{a^6}{l^6}\right) 
         -128\left(\frac{a^4}{l^4}\right)\left(1-\frac{a^2}{l^2}\right)^2+576\left(\frac{M^2a^2}{l^4}\right)\left(1-\frac{a^2}{l^2}\right)\\
         &	\qquad\qquad\qquad-432\frac{M^4}{l^4}-16\frac{a^2}{l^2}\left(1-\frac{a^2}{l^2}\right)^4+16\frac{M^2}{l^2}\left(1-\frac{a^2}{l^2}\right)^3>0,
    \end{aligned}
\end{equation}
\begin{equation}\label{eq: lem: sec: properties of Delta, lem 2, a,M,l,/ eq 5.0}
|a|<l.
\end{equation}
We want to prove that if the black hole parameters satisfy
\begin{equation}\label{eq: lem: sec: properties of Delta, lem 2, a,M,l,/ eq 5.1}
a=\alpha M,\qquad \alpha \geq \frac{12}{10}
\end{equation}
then the $\Delta$ polynomial attains at least one non-real root, and therefore the parameters do not correspond to a subexrtremal black hole. We substitute~\eqref{eq: lem: sec: properties of Delta, lem 2, a,M,l,/ eq 5.1} in the inequality~\eqref{eq: lem: sec: properties of Delta, lem 2, a,M,l,/ eq 5} and obtain  
\begin{equation}\label{eq: lem: sec: properties of Delta, lem 2, a,M,l,/ eq 6}
    \begin{aligned}
       &-16\alpha^{10} z^{10}-16(-1+\alpha^2)z^2-48\alpha^4(11+2\alpha^2)z^6-16\alpha^6(1+4\alpha^2)z^8-16(27-33\alpha^2+4\alpha^4)z^4>0
    \end{aligned}
\end{equation}
where 
\begin{equation}
z=\frac{M}{l}.
\end{equation}
We will prove that~\eqref{eq: lem: sec: properties of Delta, lem 2, a,M,l,/ eq 6} cannot be satisfied for any~$z\in\mathbb{R}$. It suffices to prove that the following part
\begin{equation}\label{eq: lem: sec: properties of Delta, lem 2, a,M,l,/ eq 7}
    \begin{aligned}
        &   -16(-1+\alpha^2)z^2-16(27-33\alpha^2+4\alpha^4)z^4-48\alpha^4(11+2\alpha^2)z^6\\
        &   \quad =z^2\left(-16(-1+\alpha^2)-16(27-33\alpha^2+4\alpha^4)z^2-48\alpha^4(11+2\alpha^2))z^4\right),
    \end{aligned}
\end{equation}
of inequality~\eqref{eq: lem: sec: properties of Delta, lem 2, a,M,l,/ eq 6}, is negative for all~$z\in \mathbb{R}$, since the remaining terms are negative. To prove that~\eqref{eq: lem: sec: properties of Delta, lem 2, a,M,l,/ eq 7} is always negative, it suffices to prove that the polynomial 
\begin{equation}\label{eq: lem: sec: properties of Delta, lem 2, a,M,l,/ eq 8}
    -16(-1+\alpha^2)-16(27-33\alpha^2+4\alpha^4)z^2-48\alpha^4(11+2\alpha^2)z^4
\end{equation}
is always negative. We first note that 
\begin{equation}
    256(AE)^3-128(ACE)^2+16AC^4E>0,\qquad 64A^3E-16(AC)^2>0
\end{equation}
where
\begin{equation}
A=-48\alpha^4(11+2\alpha^2),\qquad C=-16(27-33\alpha^2+4\alpha^4),\qquad E=-16(-1+\alpha^2),
\end{equation}
from which, by again appealing to~\eqref{eq: lem: sec: properties of Delta, lem 2, a,M,l,/ eq 5},~\eqref{eq: lem: sec: properties of Delta, lem 2, a,M,l,/ eq 5.0} (note the conditions of~\cite{polynomials}) we obtain that the roots of the quartic polynomial~\eqref{eq: lem: sec: properties of Delta, lem 2, a,M,l,/ eq 8} are never real. Moreover, we readily obtain that the polynomial~\eqref{eq: lem: sec: properties of Delta, lem 2, a,M,l,/ eq 8} is negative at~$0$, therefore it is always negative. Now, (since the remaining terms of~\eqref{eq: lem: sec: properties of Delta, lem 2, a,M,l,/ eq 6} are also negative), we conclude that for a subextremal Kerr--de~Sitter black hole we obtain 
\begin{equation}
	|a|<\frac{12}{10}M.
\end{equation}

Finally, we prove~\eqref{eq: lem: sec: properties of Delta, lem 2, a,M,l, eq 4}. We use equation~\eqref{eq: lem: sec: properties of Delta, lem 2, a,M,l,/ eq 4}, together with the bounds~$r_-<r_+<\bar{r}_+$ to obtain
\begin{equation}
l^2<a^2+6\bar{r}_+^2<r_+^2+6\bar{r}_+^2<7\bar{r}_+^2,
\end{equation}
from which we conclude that~$\bar{r}_+\rightarrow +\infty$ as $l\rightarrow +\infty$. 
\end{proof}

\begin{remark}\label{rem: sec: delta polynomial, rem 1}
We note that 
\begin{equation}\label{eq: rem: sec: delta polynomial, rem 1, eq 1}
	\sup_{l>0,(a,M)\in\mathcal{B}_l}\left|\frac{a}{M}\right|\in \left(1,\frac{12}{10}\right),
\end{equation}
in contrast to the Kerr case where the extremal parameters are~$\left|\frac{a}{M}\right|=1$. We will not pursue to compute the sharp upper bound of~\eqref{eq: rem: sec: delta polynomial, rem 1, eq 1} as the bound~$\frac{12}{10}$ suffices for the purposes of this paper. Specifically, we only need the bound~$|\frac{a}{M}|\leq \frac{12}{10}$ in the proof of Lemma~\ref{lem: subsec: sec: trapping, subsec 1, lem 1}. 
\end{remark}

The following Lemma will also be important 

\begin{lemma}\label{lem: sec: general properties of Delta, lem 3}
Let~$l>0$ and~$(a,M)\in\mathcal{B}_l$. Then, the function 
\begin{equation}\label{eq: lem: sec: general properties of Delta, lem 3, eq 1}
    \frac{\Delta}{(r^2+a^2)^2}
\end{equation}
attains exactly one critical point in~$(r_+,\bar{r}_+)$, a maximum, which we denote as 
\begin{equation}
    r_{\Delta,\textit{frac}}.
\end{equation}

Moreover, there exists a strictly positive constant~$c(a,M,l)>0$, such that 
\begin{equation}\label{eq: lem: sec: general properties of Delta, lem 3, eq 2}
    \frac{d^2}{dr^2}\left(\frac{\Delta}{(r^2+a^2)^2}\right)(r=r_{\Delta,\textit{frac}})<-c(a,M,l)<0.
\end{equation}

For sufficiently small rotation parameters~$a$, there exists a constant~$C(M,l)>0$, such that 
\begin{equation}
|r_{\Delta,\textit{frac}}-3M|\leq C(M,l)a^2.
\end{equation}
\end{lemma}
\begin{proof}
First, we calculate the derivative
\begin{equation}\label{eq: proof lem: sec: general properties of Delta, lem 3, eq 1}
    \begin{aligned}
        \frac{d}{dr}\frac{\Delta}{(r^2+a^2)^2}&    =\frac{-2}{(r^{2}+a^{2})^{3}}\left(r^{3}(1+\frac{a^2}{l^2})-3Mr^{2}+r(a^{2}+\frac{a^{4}}{l^{2}})+Ma^{2}\right)=-2\frac{s(r)}{(r^2+a^2)^3},
    \end{aligned}
\end{equation} 
where the polynomial 
\begin{equation}\label{eq: proof lem: sec: general properties of Delta, lem 3, eq 2}
    s(r)= r^{3}(1+\frac{a^2}{l^2})-3Mr^{2}+r(a^{2}+\frac{a^{4}}{l^{2}})+Ma^{2}
\end{equation}
enjoys the same roots as the function~$\frac{d}{dr}\frac{\Delta}{(r^2+a^2)^2}$, where we rewrite
\begin{equation}\label{eq: proof lem: sec: general properties of Delta, lem 3, eq 2.1}
    \frac{d}{dr}\frac{\Delta}{(r^2+a^2)^2}=\frac{\frac{d\Delta}{dr}(r^2+a^2)^2 -4r(r^2+a^2)\Delta}{(r^2+a^2)^4}.
\end{equation}
We infer that 
\begin{equation}\label{eq: proof lem: sec: general properties of Delta, lem 3, eq 3}
s(r_+)<0,\qquad s(\bar{r}_+)>0,
\end{equation}
see Lemma~\ref{lem: sec: properties of Delta, lem 1, derivatives of Delta} for the derivatives of~$\Delta$. 

Moreover we note that
\begin{equation}\label{eq: proof lem: sec: general properties of Delta, lem 3, eq 4}
    \frac{d^2 s}{dr^2}(r)=6r(1+\frac{a^2}{l^2})-6M=6(r-M)+6r\frac{a^2}{l^2}>0,\qquad r\in[r_+,\bar{r}_+]
\end{equation}
since~$r_+>M$, see Lemma~\ref{lem: sec: properties of Delta, lem 2, a,M,l}. 

Therefore, from~\eqref{eq: proof lem: sec: general properties of Delta, lem 3, eq 3},~\eqref{eq: proof lem: sec: general properties of Delta, lem 3, eq 4}, we conclude that~$s(r)$ attains exactly one critical point in~$[r_+,\bar{r}_+]$, a maximum, which we denote as
\begin{equation}
r_{\Delta,\textit{frac}}.
\end{equation}

Now, we proceed to prove~\eqref{eq: lem: sec: general properties of Delta, lem 3, eq 2}. A straightforward computation shows that 
\begin{equation}\label{eq: proof lem: sec: general properties of Delta, lem 3, eq 5}
\begin{aligned}
\frac{d}{dr}\frac{\Delta}{(r^2+a^2)^2}	&	= \frac{(r^2+a^2)\frac{d\Delta}{dr}-4r\Delta}{(r^2+a^2)^3}\\
&	=\frac{2}{(r^2+a^2)^3}\Bigg(-\Xi r^3+3M r^2+\left(-a^2\Xi\right)r+\left(-M a^2\right)\Bigg)\\
&	\dot{=} \frac{2}{(r^2+a^2)^3}\: \textit{pol}\:(r).
\end{aligned}
\end{equation}
Note that 
\begin{equation}\label{eq: proof lem: sec: general properties of Delta, lem 3, eq 6}
\frac{d\textit{pol}}{dr}\:(r) =-3\Xi r^2+6M r -a^2\Xi,
\end{equation}
which is a polynomial of degree~$2$ and has discriminant
\begin{equation}
\text{discriminant of }\frac{d\textit{pol}}{dr}=(6M)^2-12 a^2\Xi^2=\left(6M\right)^2\left(1-\frac{12a^2\Xi^2}{(6M)^2}\right).
\end{equation}
By using that~$|\frac{a}{M}|\leq \frac{12}{10},\frac{a^2}{l^2}\leq \frac{1}{4}$, see Lemma~\ref{lem: sec: properties of Delta, lem 2, a,M,l}, we readily see that 
\begin{equation}
\frac{12a^2\Xi^2}{(6M)^2}=\frac{1}{3}\left(\frac{a}{M}\right)^2\left(1+\frac{a^2}{l^2}\right)^2\leq \frac{1}{3}\left(\frac{12}{10}\right)^2\left(1+\frac{1}{4}\right)^2=\frac{3}{4}<1.
\end{equation}
Therefore, the roots of~$\frac{d\textit{pol}}{dr}$ are real and distinct and therefore the zeros of 
\begin{equation}
		\frac{d}{dr}\frac{\Delta}{(r^2+a^2)^2}
\end{equation}
are real and distinct. Therefore, we obtain that  
\begin{equation}
	\frac{d^2}{dr^2}\Big|_{r=r_{\Delta,\textit{frac}}}\frac{\Delta}{(r^2+a^2)^2}<0.
\end{equation}

Finally, we note from the definition of the polynomial~$s(r)$, see Lemma~\ref{lem: sec: general properties of Delta, lem 3}, that 
\begin{equation}
s(r)=r^2(r-3M)+a^2\left(\frac{r^3}{l^2}+r+\frac{a^2}{l^2}r^2+M\right),
\end{equation}
from which we conclude 
\begin{equation}
|r_{\Delta,\textit{frac}}-3M|\leq C(M,l)a^2,
\end{equation}
as desired.
\end{proof}

\begin{remark}\label{rem: sec: delta polynomial, rem 2}
	Note that in the Schwarzschild--de~Sitter case~$(a=0)$, Lemma~\ref{lem: sec: general properties of Delta, lem 3} implies 
	\begin{equation}
		r_{\Delta,\textit{frac}}=3M. 
	\end{equation}
\end{remark}

Finally, we note the following Lemma, which follows immediately from the subextremality conditions of Definition~\ref{def: subextremality, and roots of Delta}. 

\begin{lemma}\label{lem: sec: general properties of Delta, lem 4}
	Let $l>0$ and $(a,M)\in\mathcal{B}_l$. Then,~$\Delta$ attains exactly one critical point in~$(r_+,\bar{r}_+)$, a maximum, which we denote as~$r_{\Delta,\textit{max}}$.
\end{lemma}

\section{Preliminaries for the Klein--Gordon equation}\label{sec: preliminaries}

\subsection{The energy momentum tensor and the divergence theorem}\label{subsec: sec: preliminaries, subsec 1}

\begin{definition}\label{def: sec: preliminaries, def 1}
Let $g$ be a smooth Lorentzian metric. The energy momentum tensor associated with a solution~$\psi$ of the inhomogeneous Klein--Gordon equation
\begin{equation}
	\Box_g\psi-\mu^2_{\textit{KG}}\psi =F,
\end{equation}
where~$F$ is a sufficiently regular function, is defined to be the following:
\begin{equation}
    \mathbb{T}_{\mu\nu}[\psi]\:\dot{=}\:\partial_{\mu}\psi\partial_{\nu}\psi-\frac{1}{2}g_{\mu\nu}\Big(|\nabla\psi|_g^2 +\mu_{\textit{KG}}^2|\psi|^2 \Big),
\end{equation}
where~$\mu_{\textit{KG}}^2\geq 0$ is the Klein--Gordon mass and~$|\nabla\psi|^2_g\:\dot{=}\:g^{\gamma\delta}\partial_{\gamma}\psi\partial_{\delta}\psi$. The energy current, with respect to a vector field~$X$, is defined as
\begin{equation}
    J^{X}_{\mu}[\psi]\:\dot{=}\: \mathbb{T}_{\mu\nu}[\psi]X^{\nu}.
\end{equation}
Note that for~$X,N$ everywhere future directed causal vector fields, the following holds 
\begin{equation}
    J^X_\mu[\psi]N^\mu=\mathbb{T}(X,N)[\psi]\geq 0.
\end{equation}

The divergence of the energy current is 
\begin{equation}
    \nabla^{\mu}J^{X}_{\mu}[\psi] =X(\psi)\cdot F +\mathrm{K}^X[\psi], 
\end{equation}
where 
\begin{equation}
	\begin{aligned}
	 \mathrm{K}^X[\psi]	= \frac{1}{2}\mathbb{T}_{\mu\nu}\:^{(X)}\pi^{\mu\nu}=\frac{1}{2}\left(\partial_\mu\psi\partial_\nu\psi-\frac{1}{2}g_{\mu\nu}|\nabla\psi|^2\right)\pi^{\mu\nu} -\frac{1}{2}\mu^2_{\textit{KG}}|\psi|^2 \Tr \: ^{(X)}\pi,\qquad 
	 ^{(X)}\pi^{\mu\nu}		=\frac{1}{2}\big( \nabla^{\mu}X^{\nu}+\nabla^{\nu}X^{\mu}  \big).
	\end{aligned}
\end{equation}
\end{definition}

The following is the divergence theorem.

\begin{proposition}\label{prop: divergence theorem}
	
Let~$l>0$ and~$(a,M)\in\mathcal{B}_l$  and~$\mu^2_{KG}\geq 0$. Let~$\psi$ satisfy the inhomogeneous Klein--Gordon equation 
\begin{equation}
	\Box_g\psi-\mu_{\textit{KG}}^2\psi=F
\end{equation}
and let~$\{t^\star=\tau\}$,~$\tau\geq 0$, be the hypersurfaces of Section~\ref{subsec: admissible hypersurfaces}. 

Then, the following holds
\begin{equation}\label{eq: lem: divergence theorem, eq 1}
    \begin{aligned}
    	&	\int_{\{t^\star=\tau_2\}}J^{X}_{\mu}[\psi]n^{\mu}+\int_{\mathcal{H}^{+}\cap D(\tau_1,\tau_2)}J^{X}_{\mu}[\psi]n^{\mu}+\int_{\bar{\mathcal{H}}^{+}\cap D(\tau_1,\tau_2)}J^{X}_{\mu}[\psi]n^{\mu}   +\int\int_{D(\tau_1,\tau_2)} \mathrm{K}^X[\psi]\\
    	&	\quad\quad\quad\quad\quad=\int_{\{t^\star=\tau_1\}} J^{X}_{\mu}[\psi]n^{\mu}-\int\int_{D(\tau_1,\tau_2)} X\psi\cdot F,
    \end{aligned}
\end{equation}
for~$\tau_1\leq \tau_2$. The identity~\eqref{eq: lem: divergence theorem, eq 1} is to be understood with respect to the normals of Section~\ref{subsec: admissible hypersurfaces},~\ref{subsec: Hawking-Reall v.f.} and the volume forms of Section~\ref{subsec: volume forms of spacelike hypersurfaces}. 
\end{proposition}

For a further study of currents related to partial differential equations, see the monograph of Christodoulou~\cite{christodoulou}.

\subsection{Hardy inequalities}\label{subsec: Hardy inequalities}
We need the following Hardy inequalities.

\begin{lemma}\label{lem: subsec: Hardy inequalities, lem 1}
	Suppose that $f$ is a differentiable real function of compact support in~$[1,\infty)$, then the following holds
	\begin{equation}
	\begin{aligned}
	   \int_{1}^{\infty}|f|^2(x) dx\leq C\int_{1}^{\infty}x^{2}\left|\frac{df}{dx}\right|^2 (x)dx.
	\end{aligned}
	\end{equation}
\end{lemma}

Moreover, we note the following Hardy inequality
\begin{lemma}\label{lem: subsec: Hardy inequalities, lem 2}
	Suppose that $f$ is a differentiable real function in the interval~$[r_+,\bar{r}_+]$, then the following holds
	\begin{equation}
		\int_{r_+}^{\bar{r}_+} |f|^2 dr \leq C(\delta) \int_{r_+}^{\bar{r}_+} \left|\frac{df}{dr}\right|^2dr + C(\delta) \int_{r_++\delta}^{\bar{r}_+-\delta} |f|^2 dr
	\end{equation}
	for sufficiently small~$\delta>0$. 
\end{lemma}
\begin{proof}
	The proof is straightforward by integration by parts for the function~$\chi \psi$, where the smooth cut-off~$\chi$ satisfies~$\chi(r)=1$ in sufficiently small neighborhoods of~$r_+,\bar{r}_+$ and~$\chi(r)=0$ away. 
\end{proof}

\subsection{Poincare--Wirtinger inequality}\label{subsec: poincare wirtinger}

We present a classical Poincare--Wirtinger type inequality that is taylored to the setting of the present paper.

\begin{lemma}\label{lem: subsec: poincare wirtinger, lem 1}
	Let~$l>0$ and let~$(a,M)\in\mathcal{B}_l$ and~$\mu_{KG}^2\geq 0$. Moreover, assume that~$\psi$ is a smooth function on the manifold~$\mathcal{M}$. Then, for any~$\tau\geq 0 $ we have the following 
	\begin{equation}\label{eq: lem: subsec: poincare wirtinger, lem 1, eq 1}
	\int_{\{t^\star=\tau\}} |\psi-\underline{\psi}(\tau)|^2\leq B \int_{\{t^\star=\tau\}} J^n_\mu[\psi]n^\mu 
	\end{equation}
	where~$\underline{\psi}(\tau)=\frac{1}{|\{t^\star=\tau\}|}\int_{\{t^\star=\tau\}}\psi$ and~$|\{t^\star=\tau\}|=\int_{\{t^\star=\tau\}} 1\cdot dg_{\{t^\star=\tau\}}$ with respect to the volume forms of Section~\ref{subsec: volume forms of spacelike hypersurfaces}. The constant~$B$ in inequality~\eqref{eq: lem: subsec: poincare wirtinger, lem 1, eq 1} depends only on~$|\{t^\star=\tau\}|=c$, where the constant~$c$ is independent of~$\tau$. 
\end{lemma}

\subsection{Twisted derivatives and the twisted energy momentum tensor}\label{subsec: sec: preliminaries, subsec 2}

The twisted covariant derivative that will be defined in this Section has been used in a number of works, see~\cite{warnick.massicAdS,holzegel1,holzegel5} and references therein.

We define the twisting function
\begin{equation}
	f(r)=e^{C_{twi}\cdot r}
\end{equation}
where~$C_{twi}>0$ is an arbitrary positive constant. We define the $f$-twisted covariant derivatives as follows 
\begin{equation}\label{eq: subsec: sec: preliminaries, subsec 2, eq 1}
	\tilde{\nabla}_\mu (\cdot)=f\nabla_\mu\left((\cdot)~\frac{1}{f}\right),\qquad \tilde{\nabla}_\mu^\dagger(\cdot)=-\frac{1}{f}\nabla_\mu\left((\cdot)~f\right).
\end{equation}

We need the following definition
\begin{definition}
	Let~$l>0$,~$(a,M)\in\mathcal{B}_l$ and~$\mu^2_{KG}\geq 0$. We define the $f$-twisted energy momentum tensor for the Klein--Gordon equation~\eqref{eq: kleingordon} as follows 
	\begin{equation}\label{eq: subsec: sec: preliminaries, subsec 2, eq 2}
		\tilde{\mathbb{T}}_{ab}[\psi]= \tilde{\nabla}_a \psi \tilde{\nabla}_b \psi -\frac{1}{2}g_{ab}\left(\tilde{\nabla}^\alpha \psi \tilde{\nabla}_\alpha\psi +V_{twi}|\psi|^2\right),\qquad V_{twi}= -\left(\frac{\Box_g f}{f}-\mu^2_{KG}\right).
	\end{equation}
		
	Moreover, let~$\psi$ be a smooth solution of the inhomogeneous Klein--Gordon equation~$\Box\psi -\mu^2_{KG}\psi=F$. Then, for any smooth vector field~$X$ we define 
	\begin{equation}\label{eq: subsec: sec: preliminaries, subsec 2, eq 3}
		\tilde{J}^X_a [\psi] = \tilde{\mathbb{T}}_{a b} [\psi] X^b,\qquad \tilde{\mathrm{K}}^X [\psi] :=  ~^{(X)}\pi_{a b} \tilde{\mathbb{T}}^{a b} [\psi] + X^b \tilde{S}_b [\psi],
	\end{equation}
	where 
	\begin{equation}
		\tilde{S}_b [\psi] = \frac{\tilde{\nabla}^\dagger_b (fV_{twi})}{2f} \psi^2 + \frac{\tilde{\nabla}^\dagger_b f}{2f} \tilde{\nabla}_\sigma \psi \tilde{\nabla}^\sigma \psi.
	\end{equation}
\end{definition}
These are thus examples of compatible currents in the sense of Christodoulou~\cite{christodoulou}.

For a general smooth~$\psi$ we have
\begin{equation}
	\nabla_a \tilde{\mathbb{T}}^a~_b [\psi] = \left( - \tilde{\nabla}^\dagger_a \tilde{\nabla}^a \psi - V_{twi} \cdot \psi \right) \tilde{\nabla}_b \psi + \tilde{S}_b [\psi].
\end{equation}

Moreover, we compute 
\begin{equation}
	\nabla^a \tilde{J}_a^X[\psi] = F\cdot X(\psi)+ \tilde{\mathrm{K}}^X [\psi].
\end{equation}

\begin{remark}
	Note that if~$X$ is one of the Killing vector fields
	\begin{equation}
		X= \partial_t,\partial_{\varphi}
	\end{equation}
	then we have
	\begin{equation}
		\tilde{\mathrm{K}}^X[\psi]=0.
	\end{equation}
\end{remark}

We note the following lemma

\begin{lemma}\label{lem: subsec: sec: preliminaries, subsec 2, lem 2}
		Let~$l>0$,~$(a,M)\in\mathcal{B}_l$ and~$\mu^2_{KG}\geq 0$. Moreover, let~$X,Y$ be two smooth vector fields which are future directed timelike on~$\{r_+ < r < \bar{r}_+\}$. Furthermore, we make the assumption that on the event horizon~$\mathcal{H}^+$ the vector fields~$X,Y$ are either timelike or equal to~$\partial_{t}+\frac{a\Xi}{r_+^2+a^2}\partial_{\varphi}$. Similarly, we make the assumption that on the cosmological horizon~$\bar{\mathcal{H}}^+$ the vector fields~$X,Y$ are either timelike or equal to~$\partial_t+\frac{a\Xi}{\bar{r}_+^2+a^2}\partial_{\varphi}$.

		Then, there exist strictly positive constants~$c(a,M,l,X,Y)>0,C_{twi}(a,M,l,X,Y)>0$, independent of the Klein--Gordon mass~$\mu_{KG}$, and we have that
	\begin{equation}
		\tilde{\mathbb{T}}[\psi](X,Y)+ C_{twi}\cdot \mathbb{T}[\psi](X,Y)\geq c \cdot \Delta|\psi|^2+ c\Delta \cdot \sum_{i_1+i_2+i_3=1} |\partial_t^{i_1}(Z^\star)^{i_2}\slashed{\nabla}^{i_3}\psi|^2
	\end{equation}
	for all~$r\in[r_+,\bar{r}_+]$.
\end{lemma}
\begin{proof}
	This is immediate after taking~$C_{twi}(a,M,l,\mu_{KG})>0$ sufficiently large, since~$\mathbb{T}(X,Y)\gtrsim \Delta \cdot \sum_{i_1+i_2+i_3=1} |\partial_t^{i_1}(Z^\star)^{i_2}\slashed{\nabla}^{i_3}\psi|^2$.
\end{proof}

\begin{remark}\label{rem: subsec: sec: preliminaries, subsec 2, rem 2}
	We will use the twisted energy momentum tensor in Section~\ref{subsec: sec: continuity argument, subsec 0.2} to control zeroth order future boundary terms of the form~$\int_{\{t^\star=\tau_2\}}  |\psi|^2$, in the case~$\mu_{KG}=0$. 
\end{remark}

\section{Red-shift and Superradiance}\label{sec: redshift and superradiance}

In this section we construct redshift vector field multipliers and then use them to derive redshift estimates. 

\subsection{Red-shift multipliers}
We present the redshift multipliers for the horizons, see Theorem 7.1 of the lecture notes~\cite{DR5}, and~[\cite{Volker}, Proposition~$6.5$].

\begin{proposition}\label{prop: redshift on the event horizon}
Let $l>0$,~$(a,M)\in \mathcal{B}_{l}$ and~$\mu^2_{KG}\geq 0$. Then, there exist positive constants 
\begin{equation}
c(a,M,l,\mu^2_{KG}),\qquad \bar{c}(a,M,l,\mu^2_{KG})>0,
\end{equation}
there exists a parameter~$\epsilon_{red}>0$ sufficiently small satisfying
\begin{equation}
	r_+(a,M,l)<r_+ +2\epsilon_{red}(a,M,l) < \bar{r}_+-2\epsilon_{red}(a,M,l)< \bar{r}_+ ,
\end{equation}
and there exists a vector field 
\begin{equation}
	N
\end{equation}
such that the following holds.

We have
\begin{equation*}
	\begin{aligned}
			&	N\text{ is timelike in } \{r< r_++2\epsilon_{red}\}\cup \{\bar{r}_+-2\epsilon_{red}< r\},
			\qquad [N,\partial_t]=0,\qquad [N,\partial_{\varphi}]=0,
	\end{aligned}
\end{equation*}
\begin{equation*}
	N=0,\quad \textit{for } \{r_++2\epsilon_{red}\leq r\leq \bar{r}_+-2\epsilon_{red}\},
\end{equation*}
and for any sufficiently regular function~$\psi$ we have
\begin{equation}\label{eq: prop: redshift on the event horizon, eq 1}
	\mathrm{K}^{N}[\psi]\geq c J^{N}_{\mu}[\psi]N^{\mu},\quad \textit{for }  \{r\leq r_+ + \epsilon_{red}\},\qquad \mathrm{K}^{N}[\psi]\geq c  J^{N}_{\mu}[\psi]N^{\mu},\quad\textit{for } \{r\geq \bar{r}_+ -\epsilon_{red}\},
\end{equation}
where for~$\mathrm{K}^X$ see Definition~\ref{def: sec: preliminaries, def 1}. The result of the present Proposition also holds for the extended domain~$D_\delta (\tau_1,\tau_2)$, see Definition~\ref{def: causal domain, def 2}. 
\end{proposition}

\begin{proof}
The proof can be inferred from the lecture notes~\cite{DR5} or~\cite{Volker}. We present the basic steps.

First, we note the definition of the surface gravity of the event horizon~$\mathcal{H}^+$ and the cosmological horizon~$\bar{\mathcal{H}}^+$ respectively
\begin{equation}
	\kappa_+=\frac{1}{2}|\frac{d\Delta}{dr}|,\qquad \bar{\kappa}_+=\frac{1}{2}|\frac{d\Delta}{dr}|.
\end{equation}
In view of the subextremality condition~$l>0$ and~$(a,M)\in\mathcal{B}_l$ we note that 
\begin{equation}\label{eq: proof prop: redshift on the event horizon, eq 0}
	\kappa_+,\qquad \bar{\kappa}_+ >0. 
\end{equation}

There exist vector fields 
\begin{equation}
Y^+,\qquad \bar{Y}^+
\end{equation}
that satisfy 
\begin{equation}
	\nabla_{Y^+}Y^+=-\sigma \left(K^++Y^+\right),\qquad \nabla_{\bar{Y}^+}\bar{Y}^+=-\sigma \left(\bar{K}^++\bar{Y}^+\right)
\end{equation}
\begin{equation}
	g(Y^+,K^+)=-2,\qquad g(\bar{Y}^+,\bar{K}^+)=-2
\end{equation}
which are null on the event horizon~$\mathcal{H}^+$ and the cosmological horizon~$\bar{\mathcal{H}}^+$ respectively, where for the Hawking--Reall vector fields~$K^+,\bar{K}^+$ see Section~\ref{subsec: Hawking-Reall v.f.}, and moreover~$Y^+,\bar{Y}^+$ are invariant under pushforwards with respect to~$\partial_t$ and~$\partial_\varphi$. The constant~$\sigma>0$ is to be chosen sufficiently large, see already~\eqref{eq: proof prop: redshift on the event horizon, eq 3}.

We recall from Definition~\ref{def: sec: preliminaries, def 1} that for a vector field~$X$ the following holds
\begin{equation}\label{eq: proof prop: redshift on the event horizon, eq 1}
	\begin{aligned}
		\mathrm{K}^X	&	=\frac{1}{2}\mathbb{T}_{\mu\nu}\pi^{\mu\nu}	= \frac{1}{2}\left(\partial_\mu\psi\partial_\nu\psi-\frac{1}{2}g_{\mu\nu}|\nabla\psi|^2_g\right)\pi^{\mu\nu} -\frac{1}{2}\mu^2_{\textit{KG}}|\psi|^2 \Tr \: ^{(X)}\pi. 
	\end{aligned}
\end{equation}
Therefore, for a sufficiently large~$\sigma>0$ we note that 
\begin{equation}\label{eq: proof prop: redshift on the event horizon, eq 2}
-\Tr \:^{(Y^+)}\pi >0,\qquad -\Tr \:^{(\bar{Y}^+)}\pi >0,
\end{equation}
on the event horizon $\mathcal{H}^+$ and cosmological horizon $\bar{\mathcal{H}}^+$ respectively.

Let~$\epsilon_{red}>0$ be sufficiently small. Again, by following the lecture notes~\cite{DR5} we smoothly extend~$N$ such that it is timelike in~$\{r< r_++2\epsilon_{red}\}\cup \{\bar{r}_+-2\epsilon_{red}< r\}$ and~$\partial_t,\partial_{\varphi}$ invariant with
\begin{equation}
	N=Y^+ + K^+
\end{equation}
in~$\{r\leq r_++\epsilon_{red}\}$
\begin{equation}
	N=\bar{Y}^+ + \bar{K}^+
\end{equation}
in~$\{\bar{r}_+-\epsilon_{red}\leq r\}$ and~$N=0$ at~$ \{r_++2\epsilon_{red}\leq r\leq \bar{r}_+-2\epsilon_{red}\}$. 

Therefore, by following the arguments of the lecture notes~\cite{DR5}, in view of the positivity of the surface gravities~\eqref{eq: proof prop: redshift on the event horizon, eq 0}, we take~$\sigma>0$ sufficiently large and conclude that the following hold 
\begin{equation}\label{eq: proof prop: redshift on the event horizon, eq 3}
	\mathrm{K}^{N}[\psi]\geq b J^{N}_{\mu}[\psi] N^{\mu},\qquad \mathrm{K}^{N}[\psi]\geq b J^{N}_{\mu}[\psi]N^{\mu},
\end{equation}
in~$\{r\leq r_++\epsilon_{red}\},\{\bar{r}_+-\epsilon_{red}\leq r\}$ respectively. This completes the proof.
\end{proof}

The following Lemma will be very useful

\begin{lemma}\label{lem: subsec: redshift multipliers for slow, axi, lem 1}
	
Let~$l>0$,~$(a,M)\in\mathcal{B}_l$ and~$\mu^2_{\textit{KG}}\geq 0$. Let~$\psi$ be a smooth function such that~$\partial_{\varphi}\psi=0$. Then, we have
\begin{equation}
    \mathrm{K}^{W}[\psi]=0,\qquad \tilde{\mathrm{K}}^W[\psi]=0
\end{equation}
where~$W=\partial_t+\frac{a\Xi}{r^2+a^2}\partial_{\varphi}$, see Lemma~\ref{lem: causal vf E,1}. 
\end{lemma}
\begin{proof}
	We denote~$f(r)=\frac{a\Xi}{r^2+a^2}$. Then, we obtain
\begin{equation}
    \begin{aligned}
         \mathrm{K}^{\partial_{t^\star}+f(r)\partial_{\varphi}}[\psi] &=\mathrm{K}^{(f(r)\partial_{\varphi})}[\psi]=\frac{1}{2}\:^{(f\partial_{\varphi})}\pi_{\mu\nu}\mathbb{T}_{\mu_{KG=0}}^{\mu\nu}=\mathbb{T}^{\mu}\:_{\varphi}(\nabla f)_\mu=\mathbb{T}(\nabla f,\partial_{\varphi})[\psi]=0,
    \end{aligned}  
\end{equation}
where we used 
\begin{equation}
	\Tr \: ^{(f\partial_{\varphi})}\pi=g^{\mu\nu} \:^{(f\partial_{\varphi})}\pi_{\mu\nu} =2g^{r\nu}\frac{df(r)}{dr} g(\partial_{\varphi},\partial_\nu) =0,
\end{equation}
to eliminate the Klein--Gordon term of the energy momentum tensor. The proof of~$\tilde{\mathrm{K}}^W=0$ is immediate from~\eqref{eq: subsec: sec: preliminaries, subsec 2, eq 3}. 
\end{proof}

\subsection{The redshift estimates}\label{subsec: red shift estimates}
We will use the redshift multiplier~$N$, of Proposition~\ref{prop: redshift on the event horizon}, to obtain the following.

\begin{proposition}\label{prop: redshift estimate}
Let~$l>0$ and~$(a,M)\in\mathcal{B}_l$ and~$\mu^2_{KG}\geq 0$. Let~$\epsilon_{red}>0$ be as in Proposition~\ref{prop: redshift on the event horizon}. Then, for all~$j\geq 1$ there exists a sufficiently small~$\delta(a,M,l,\mu_{KG},j)>0$ and there exists a positive constant
\begin{equation}
C=C(a,M,l,j)>0,
\end{equation}
such that for all smooth solutions~$\psi$ of the inhomogeneous Klein--Gordon equation
\begin{equation*}
	\Box_{g_{a,M,l}}\psi-\mu^2_{KG}\psi=F
\end{equation*}
where~$F$ is sufficiently regular function, we obtain 
\begin{equation}\label{eq: prop: redshift estimate, eq 1}
\begin{aligned}
&  \sum_{0\leq i\leq j-1}\int\int_{D(\tau_1,\tau_2)\cap \{r\leq r_++\epsilon_{red}\}}   J^N_{\mu}[N^i\psi]N^{\mu}  + \sum_{0 \leq i\leq j-1}\int\int_{D(\tau_1,\tau_2)\cap \{r\geq \bar{r}_+-\epsilon_{red}\}} J^N_{\mu}[N^i\psi]N^{\mu}  \\
& \quad  +\sum_{0\leq i\leq j-1}\Bigg(\int_{\mathcal{H}^{+}
	\cap D(\tau_1,\tau_2)} J^{N}_{\mu}[N^i\psi]n^{\mu} +\int_{\bar{\mathcal{H}}^{+}\cap D(\tau_1,\tau_2)}J^{N}_{\mu}[N^i\psi]n^{\mu}\\
	&	\qquad\qquad\qquad +\int_{\{t^\star=\tau_2\}\cap \{ r\leq r_++\epsilon_{red}\}}J^{N}_{\mu}[N^i\psi]n^{\mu} + \int_{\{t^\star=\tau_2\}\cap \{ r\geq \bar{r}_+-\epsilon_{red}\} }J^{N}_{\mu}[N^i\psi]n^{\mu}\Bigg) \\
&  \quad\quad \leq C\sum_{0\leq i\leq j-1}\int_{\{t^\star=\tau_1\}} J^n_{\mu}[n^i\psi]n^{\mu}\\
&	\qquad\qquad\qquad  + C\sum_{0\leq i\leq j-1}\Big(\int\int_{D(\tau_1,\tau_2)\cap\{ r_++\epsilon_{red}\leq r\leq r_++2\epsilon_{red} \}}   J^n_{\mu}[n^i\psi]n^{\mu} +\int\int_{D(\tau_1,\tau_2)\cap\{ \bar{r}_+-2\epsilon_{red}\leq r\leq \bar{r}_+-\epsilon_{red} \}}   J^n_{\mu}[n^i\psi]n^{\mu} \Big) \\
&	\qquad\qquad\qquad  + C\sum_{1\leq i_1+i_2+i_3\leq j-1}\Big(\int\int_{D(\tau_1,\tau_2)\cap\{ r_+\leq r\leq r_++2\epsilon_{red} \}}  |\slashed{\nabla}^{i_1}(Z^\star)^{i_2}\partial_t^{i_3} F|^2\\
&	\qquad\qquad\qquad\qquad\qquad\qquad\qquad+\int\int_{D(\tau_1,\tau_2)\cap\{ \bar{r}_+-2\epsilon_{red}\leq r\leq \bar{r}_+\}}  |\slashed{\nabla}^{i_1}(Z^\star)^{i_2}\partial_t^{i_3} F|^2 \Big),\\
\end{aligned}
\end{equation}
for all~$\tau_1\leq \tau_2$. The result of the present Proposition also holds replacing~$D(\tau_1,\tau_2)$ with the extended domain~$D_\delta (\tau_1,\tau_2)$ and~$\mathcal{H}^+,\bar{\mathcal{H}}_+$ with~$\mathcal{H}^+_\delta,\bar{\mathcal{H}}^+_\delta$, see Definition~\ref{def: causal domain, def 2}, where~$\{t^\star=c\}$ is now to be understood as a subset of~$D_\delta$. 
\end{proposition}

\begin{proof}
From Proposition~\ref{prop: redshift on the event horizon} we obtain immediately
\begin{equation}\label{eq: proof prop redshift estimate, eq 1}
\begin{aligned}
&  \int_{\mathcal{H}^{+}
	\cap D(\tau_1,\tau_2)} J^{N}_{\mu}[\psi] n^{\mu}_{\mathcal{H}^{+}} + \int_{\{t^\star=\tau_2\}\cap \{ r\leq r_++\epsilon_{red} \} } J^{N}_{\mu}[\psi]n^{\mu} +\int_{\bar{\mathcal{H}}^{+}\cap D(\tau_1,\tau_2)} J^{N}_{\mu} [\psi] n^{\mu}_{\bar{\mathcal{H}}^{+}} + \int_{\{t^\star=\tau_2\}\cap \{ r\geq \bar{r}_+-\epsilon_{red} \} } J^{N}_{\mu}[\psi]n^{\mu} \\
&	\qquad + \int\int_{D(\tau_1,\tau_2)\cap \{r\leq r_++\epsilon_{red}\}}  J^{N}_{\mu}[\psi]N^{\mu} + \int\int_{D(\tau_1,\tau_2)\cap \{r\geq \bar{r}_+-\epsilon_{red}\}}  J^{N}_{\mu}[\psi]N^{\mu}  \\
&  \quad\quad \leq C\int_{\{t^\star=\tau_1\}} J^{n}_{\mu}[\psi]n^{\mu} + C\int\int_{D(\tau_1,\tau_2)\cap\{ r_++\epsilon_{red}\leq r\leq \bar{r}_+-\epsilon_{red} \}}   |\mathrm{K}^N[\psi]|\\
&	\quad\quad + C\Big(\int\int_{D(\tau_1,\tau_2)\cap\{ r_++\epsilon_{red}\leq r\leq r_++2\epsilon_{red} \}}  |F|^2 +\int\int_{D(\tau_1,\tau_2)\cap\{ \bar{r}_+-2\epsilon_{red}\leq r\leq \bar{r}_+-\epsilon_{red} \}}  |F|^2 \Big) \\
\end{aligned}
\end{equation}
which easily implies~\eqref{eq: prop: redshift estimate, eq 1} with~$j=1$.

To obtain~\eqref{eq: prop: redshift estimate, eq 1} for all orders~$j\geq 1$ we proceed as follows.

For~$m=(m_1,m_2,m_3,m_4)$, the following holds on the event horizon~$\mathcal{H}^+$:
\begin{equation}\label{eq: higher order redshift}
\Box(Y^k \psi)-\mu_{\textit{KG}}^2Y^{k}\psi=  \kappa_kY^{k+1}\psi+\sum_{|m|\leq k+1,m_4\leq k}c_mE_1^{m_1}E_2^{m_2}L^{m_3}Y^{m_4}\psi+Y^k F,
\end{equation}
where~$(E_1,E_2,L,Y)$ is a local null frame on~$\mathcal{H}^+$, with~$Y=Y^+$, with~$\kappa_k>0$, and~$c_m$ is a smooth function on the horizon such that~$Y(c_m)=0$. Note that an analogous identity~\eqref{eq: higher order redshift} holds on the cosmological horizon~$\bar{\mathcal{H}}^+$ where~$Y=\bar{Y}^+$, with~$\bar{\kappa}_k>0$ in the place of~$\kappa_k$ and a smooth~$\bar{c}_m$ in the place of~$c_m$, where~$\bar{Y}(\bar{c}_m)=0$.

We apply the analogue of~\eqref{eq: proof prop redshift estimate, eq 1} to~$Y^k \psi$. The positivity~$\kappa_k,~\bar{\kappa}_k>0$ and the structure of the remainder terms in~\eqref{eq: higher order redshift} allows to conclude~\eqref{eq: prop: redshift estimate, eq 1} inductively for any~$j\geq 1$ by standard elliptic estimates. See~\cite{DR5} for more details.
\end{proof}

\subsection{Boundedness in the axisymmetric case}\label{subsec: redshift vs superradiance}

We will use the following proposition for the axisymmetric case of Theorem~\ref{main theorem 1}.

\begin{proposition}\label{prop:redshift vs superradiance estimate, for axisymmetric solutions}
Let~$l>0$ and~$(a,M)\in \mathcal{B}_{l}$ and~$\mu^2_{KG}\geq 0$. There exists a constant
\begin{equation}
C(a,M,l,\mu^2_{\textit{KG}})>0,
\end{equation}
where~$C(a,M,l,\mu_{KG}^2)$ and the following holds.

Let~$\psi$ satisfy the  Klein--Gordon equation~\eqref{eq: kleingordon} and also satisfy the axisymmetric condition
\begin{equation}
\partial_{\varphi^\star}\psi=0.
\end{equation}
Then, 
\begin{equation}\label{eq:redshift vs superradiance estimate, for axisymmetric solutions, 2}
    \begin{aligned}
         \int_{\{t^\star=\tau_2\}} J_\mu^{n}[\psi]n^\mu & \leq C \int_{\{t^\star=\tau_1\}}J_\mu^{n}[\psi]n^\mu \\
         	\int_{\{t^\star=\tau_2\}} J^n_\mu[\psi]n^\mu +|\psi|^2 & \leq C \int_{\{t^\star=\tau_1\}} J^n_\mu[\psi]n^\mu+|\psi|^2,
    \end{aligned}    
\end{equation}
for all~$0\leq \tau_1\leq \tau_2$. 
\end{proposition}

\begin{proof}
	 Recall that the vector field~$W=\partial_t+\frac{a\Xi}{r^2+a^2}\partial_{\varphi}$ is timelike in~$\{r_+<r< \bar{r}_+\}$, see Lemma~\ref{lem: causal vf E,1}, and
	\begin{equation}
		\mathrm{K}^{W}[\psi]=0
	\end{equation}
	for axisymmetric solutions~$\partial_{\varphi}\psi=0$, see Lemma~\ref{lem: subsec: redshift multipliers for slow, axi, lem 1}.  We apply the energy identity for the vector field~$W$, see Proposition~\ref{prop: divergence theorem}, to obtain
	\begin{equation}\label{eq proof: prop:redshift vs superradiance estimate, for axisymmetric solutions, eq 1}
		 \int_{\{t^\star=\tau_2\}} J_\mu^{W}[\psi]n^\mu \leq  \int_{\{t^\star=\tau_1\}}J_\mu^{W}[\psi]n^\mu .  \\
	\end{equation}

	Now, by using the redshift estimate of Proposition~\ref{prop: redshift estimate} for~$j=1$ and~\eqref{eq proof: prop:redshift vs superradiance estimate, for axisymmetric solutions, eq 1} we obtain the first estimate of~\eqref{eq:redshift vs superradiance estimate, for axisymmetric solutions, 2} by following arguments found in~\cite{DR5}.
	
	To obtain the second estimate of~\eqref{eq:redshift vs superradiance estimate, for axisymmetric solutions, 2} we use the now established first estimate, in conjunction with the twisted currents of Section~\ref{subsec: sec: preliminaries, subsec 2}, in view of Lemma~\ref{lem: subsec: redshift multipliers for slow, axi, lem 1}. 
\end{proof}

\begin{remark}
	Note that the constants~$C$ on the RHS of~\eqref{eq:redshift vs superradiance estimate, for axisymmetric solutions, 2} do not blow up in the limit~$\mu^2_{KG}\rightarrow 0$.
\end{remark}

\subsection{Boundedness in the Schwarzschild--de~Sitter case}

We have the following result in Schwarzschild--de~Sitter

\begin{proposition}\label{prop subsec: sec: preliminaries, subsec 2, prop 1}
	Let~$l>0$,~$0<\frac{M^2}{l^2}<\frac{1}{27}$,~$a=0$ and~$\mu^2_{\textit{KG}}\geq 0$. Then, there exists a constant
	\begin{equation}
		C(M,l,\mu^2_{KG})>0
	\end{equation}
	such that the following holds.

	Let~$\psi$ be a solution of the Klein--Gordon equation~\eqref{eq: kleingordon}. Then,
	\begin{equation}\label{eq: lem: subsec: sec: preliminaries, subsec 2, lem 1, eq 1}
		\int_{\{t^\star=\tau_2\}} J^n_\mu[\psi]n^\mu +|\psi|^2\leq C \int_{\{t^\star=\tau_1\}} J^n_\mu[\psi]n^\mu + |\psi|^2,
	\end{equation}
	\begin{equation}\label{eq: lem: subsec: sec: preliminaries, subsec 2, lem 1, eq 2}
		\int_{\mathcal{H}^+\cap \{t^\star\geq \tau_2\}}J^{\partial_t}_\mu[\psi]n^\mu+ \int_{\bar{\mathcal{H}}^+\cap\{t^\star\geq \tau_2\}}J^{\partial_t}_\mu[\psi]n^\mu+  \int_{\{t^\star=\tau_2\}} J^n_\mu[\psi]n^\mu \leq C \int_{\{t^\star=\tau_1\}} J^n_\mu[\psi]n^\mu,
	\end{equation}
	for any~$0\leq \tau_1\leq \tau_2$. We can also obtain higher order versions of the estimates~\eqref{eq: lem: subsec: sec: preliminaries, subsec 2, lem 1, eq 1},~\eqref{eq: lem: subsec: sec: preliminaries, subsec 2, lem 1, eq 2}, for any order of derivatives, by commuting the Klein--Gordon equation with~$\partial_t,N$.
\end{proposition}

\begin{proof}
	The estimate~\eqref{eq: lem: subsec: sec: preliminaries, subsec 2, lem 1, eq 2} is easy to prove, by following the same steps as in the lecture notes~\cite{DR5}~(cf. the proof of Proposition~\ref{prop:redshift vs superradiance estimate, for axisymmetric solutions} above).

To obtain the estimate of~\eqref{eq: lem: subsec: sec: preliminaries, subsec 2, lem 1, eq 1} we use the now established first estimate, in conjunction with the twisted currents of Section~\ref{subsec: sec: preliminaries, subsec 2}, in view of Lemma~\ref{lem: subsec: redshift multipliers for slow, axi, lem 1}. 
\end{proof}

\begin{remark}
	Note that the constants~$C$ on the RHS of~\eqref{eq: lem: subsec: sec: preliminaries, subsec 2, lem 1, eq 1},~\eqref{eq: lem: subsec: sec: preliminaries, subsec 2, lem 1, eq 2} do not blow up in the limit~$\mu^2_{KG}\rightarrow 0$.
\end{remark}

\section{Carter's separation of variables}\label{sec: carter separation}

In this section we discuss Carter's separation of variables, also see~\cite{Carter,holzegel3}. 

\subsection{Sufficiently integrable functions}\label{subsec: sec: carter separation, subsec -1}

We define the class of sufficiently integrable functions, with the help of which we will make sense of Carter's celebrated separation of variables of Section~\ref{subsec: sec: carter separation, radial}.

\begin{definition}\label{def: sec: carter separation, def 1}
	Let~$l>0$,~$(a,M)\in\mathcal{B}_l$ and let~$\mu^2_{\textit{KG}}\geq 0$. Then, we say that a smooth function 
	\begin{equation}
	\Psi:\mathcal{M}\rightarrow \mathbb{C}
	\end{equation}
	is \texttt{sufficiently integrable} if for every~$j\geq 1$ we have 
	\begin{equation}\label{eq: def: sec: carter separation, def 1, eq 1}
	\begin{aligned}
	\sup_{r\in[r_+,\bar{r}_+]}\int_{\mathbb{R}}\int_{\mathbb{S}^2}\left(\mu^2_{\textit{KG}}|\Psi|^2+\sum_{1\leq i_1+i_2+i_3\leq j}\left|\slashed{\nabla}^{i_1}\partial_t^{i_2}(Z^\star)^{i_3}\Psi\right|^2\right)\sin\theta d\theta d\varphi	dt &<	\infty.\\
	\end{aligned}
	\end{equation}
\end{definition}

Note that if~$\psi$ is sufficiently integrable then~$\left(\Box_{g_{a,M,l}}-\mu^2_{KG}\right)\psi$ is also sufficiently integrable.

\begin{remark}\label{rem: sec: carter separation, rem 1}
	Note that Definition~\ref{def: sec: carter separation, def 1} depends on fixing~$\mu_{KG}^2$. Specifically, in the case of the wave equation~$\mu^2_{\textit{KG}}=0$, we do not include the zeroth order term~$\sup_{r\in [r_+,\bar{r}_+]}\int_{\mathbb{R}}\int_{\mathbb{S}^2}|\Psi|^2$ on the right hand side of~\eqref{eq: def: sec: carter separation, def 1, eq 1}, as opposed to~[Definition~5.1.1,~\cite{DR2}] . Therefore, in particular, in the case~$\mu^2_{KG}=0$, constant functions belong to the space of sufficiently integrable functions as defined here. 
\end{remark}

\subsection{The class of outgoing functions}\label{subsec: sec: carter separation, subsec 0}

We will use the following definition in Proposition~\ref{prop: Carters separation, radial part}.

\begin{definition}\label{def: subsec: sec: carter separation, subsec 0, def 1}
	Let~$l>0$ and~$(a,M)\in \mathcal{B}_l$. We say that a smooth function~$\Psi$ is \texttt{outgoing} if there exists an~$\epsilon>0$ such that~$\Psi$ vanishes in~$\{t^\star=\tau\}\cap \{r_+\leq r\leq r_++\epsilon\}$ and in~$\{t^\star=\tau\}\cap \{\bar{r}_+-\epsilon\leq r\leq \bar{r}_+\}$ for all~$\tau\leq -\epsilon^{-1}$. 
\end{definition}

\subsection{Time cutoffs}\label{subsec: sec: carter separation, subsec 1}

Let~$2\leq 1+\tau_1\leq \tau_2$. We choose smooth cut-offs that satisfy the following 
\begin{equation}\label{eq: subsec: sec: carter separation, subsec 1, eq 2}
\chi_{\tau_1,\tau_2}(z) \:\dot{=}\: 
\begin{cases}
\text{1,} &\quad\{1+\tau_1\leq z\leq \tau_2-1\}\\
\text{0,} &\quad\{z\leq \tau_1\}\cup\{z\geq \tau_2\}, \\
\end{cases}
\qquad \chi_{\tau_1}(z) \:\dot{=}\: 
\begin{cases}
\text{1,} &\quad\{z\geq 1+\tau_1\}\\
\text{0,} &\quad\{z\leq \tau_1\}\\
\end{cases}
\end{equation}
where the following hold~$\partial_\varphi \chi_{\tau_1,\tau_2}\equiv 0$,~$\partial_\varphi \chi_{\tau_1}\equiv 0$ and moreover~$\chi_{\tau_1,\tau_2}\big|_{(\tau_1,\tau_1+1)}(s)=\chi(\tau_1+s)$, for some~$\chi(s)$ independent of~$\tau_1$ and moreover~$\chi_{\tau_1,\tau_2}\big|_{(\tau_2-1,\tau_2)}=\bar{\chi}(\tau_2+s)$, for some~$\bar{\chi}(s)$ independent of~$\tau_2$.

Note that for any~$n\geq 1$ there exists a constant~$C(n)>0$ such that for any~$2\leq 1+\tau_1\leq \tau_2$ we have that~$|\frac{d^n}{dz^n}\chi_{\tau_1,\tau_2}|+|\frac{d^n}{dz^n}\chi_{\tau_1}|<C$.

We define the cut-offed functions respectively
\begin{equation}
\psi_{\tau_1,\tau_2}\dot{=}\:\chi_{\tau_1,\tau_2}(t^\star)\psi(t^\star,\cdot), \qquad \psi_{\tau_1}\dot{=}\:\chi_{\tau_1}(t^\star)\psi(t^\star,\cdot).
\end{equation}

\begin{remark}\label{rem: subsec: sec: carter separation, subsec 0, rem 1}
	Let~$\psi:\mathcal{M}\rightarrow \mathbb{R}$ be a smooth function. Then, we note that~$\chi_{\tau_1,\tau_2}\psi,\chi_{\tau_1}\psi$ are outgoing, see Definition~\ref{def: subsec: sec: carter separation, subsec 0, def 1} and moreover we note that~$\chi_{\tau_1,\tau_2}\psi$ is sufficiently integrable, but~$\chi_{\tau_1}\psi$ is not necessarily sufficiently integrable. 
\end{remark}

\subsection{The angular part of Carter's separation, oblate spheroidal harmonics}

We define the following

\begin{definition}\label{def: sec: carter, sturm liouville, def 1}
Let~$f$ be a complex valued function such that~$f\in H^1(\mathbb{S}^2,\mathbb{C})$ and let~$\xi\in\mathbb{R}$. Then, for~$\mu^2_{KG}\geq 0$ we define the~$L^{2}(\sin\theta d\theta d\varphi)$ self-adjoint operator
\begin{equation}
        \begin{aligned}
             P_{\mu_{\textit{KG}}}(\xi)f \:\dot{=}\: & -\frac{1}{\sin \theta}\partial_{\theta}(\Delta_{\theta}\sin\theta\partial_{\theta}f)-\frac{\Xi^{2}}{\Delta_{\theta}}\frac{1}{\sin^{2}\theta}\partial^{2}_{\varphi}f-\Xi \frac{\xi^2}{\Delta_{\theta}}\cos^{2}\theta f-2i\xi\frac{\Xi}{\Delta_{\theta}}\frac{a^{2}}{l^{2}}\cos^{2}\theta \partial_{\varphi}f\\
             &  +\mu_{\textit{KG}}^2a^2\sin^2{\theta}\cdot f,\\          
	\end{aligned}
\end{equation}
where~$\Delta_\theta=1+\frac{a^2}{l^2}\cos^2\theta,~\Xi=1+\frac{a^2}{l^2}$, see Definition~\ref{def: delta+, and other polynomials}.
\end{definition}

By standard arguments for self-adjoint operators, see for example~\cite{holzegel3}, we have the following

\begin{proposition}\label{prop: sec: carter, finster, prop 1}
Let~$l>0$ and~$(a,M)\in\mathcal{B}_l$ and~$\mu^2_{KG}\geq 0$ and let~$\xi\in\mathbb{R}$. Then, there exist countably many real eigenvalues 
\begin{equation}\label{eq: prop: sec: carter, finster, prop 1, eq 1}
    \lambda^{(\xi),\mu_{KG}}_{m\ell}\in\mathbb{R}
\end{equation}
of the operator~$P_{\mu_{\textit{KG}}}(\xi)$, see~Definition~\ref{def: sec: carter, sturm liouville, def 1}, which are indexed by~$\ell\geq |m|$.

Moreover, the eigenvalues~\eqref{eq: prop: sec: carter, finster, prop 1, eq 1} correspond to the oblate spheroidal harmonics 
\begin{equation}
    S^{(\xi),\mu_{KG}}_{m\ell}(\cos\theta)e^{im\varphi},
\end{equation}
which form a complete orthonormal basis of~$L^2(\sin\theta d\theta d\varphi)$. 
\end{proposition}

For brevity we will drop the explicit reference on $\mu^2_{KG}$, i.e. we will denote~$\lambda^{(\xi)}_{m\ell},~S^{(\xi)}_{m\ell}$.

We need the following.
\begin{lemma}\label{lem: inequality for lambda}
The eigenvalues~$\lambda^{(\xi)}_{m\ell}$ of Proposition~\ref{prop: sec: carter, finster, prop 1} satisfy the following inequalities.
If~$m\xi\geq 0$ then we have
\begin{equation}\label{eq: lem: inequality for lambda, eq 1}
	\lambda^{(\xi)}_{m\ell}+\xi^2\geq \Xi^2m^2,\quad 	\lambda^{(\xi)}_{m\ell}+\xi^2 -2m\xi \Xi > 0,\quad \lambda^{(\xi)}_{m\ell}+\xi^2 -2m\xi \frac{a^2}{l^2}> 0.
\end{equation}
If~$m\xi<0$ then we have 
\begin{equation}\label{eq: lem: inequality for lambda, eq 2}
	\lambda^{(\xi)}_{m\ell}+\xi^2 -2m\xi \Xi\geq \Xi^2m^2,\quad \lambda^{(\xi)}_{m\ell}+\xi^2 -2m\xi \frac{a^2}{l^2}\geq \Xi^2m^2.
\end{equation}

Finally, for any~$(m,\xi)\in \mathbb{Z}\times \mathbb{R}$ we have that 
\begin{equation}\label{eq: lem: inequality for lambda, eq 3}
	\lambda^{(\xi)}_{m\ell}+\Xi \xi^2 \geq \Xi m^2
\end{equation}
\end{lemma}
\begin{proof}
For simplicity, we prove the Lemma for~$\mu_{\textit{KG}}^2=0$. Note that the same arguments work for~$\mu_{\textit{KG}}^2>0$. For brevity define 
\begin{equation}
f=S^{(\xi)}_{m\ell}(\xi,\cos \theta)e^{im\varphi},
\end{equation}
and note that~$\|f\|_{L^2(\mathbb{S}^2)}>0.$

First, we have the following
\begin{equation}\label{eq: proof: lem: inequality for lambda, eq 1}
	\begin{aligned}
		P(\xi)f +\xi^2f &	=	-\frac{1}{\sin\theta}\partial_\theta(\Delta_\theta\sin\theta\partial_\theta f)+\xi^2\left(1-\frac{\Xi}{\Delta_\theta}\cos^2\theta\right)f+\frac{\Xi^2}{\Delta_\theta}\frac{1}{\sin^2\theta}m^2f+2\xi m \frac{\Xi}{\Delta_\theta}\frac{a^2}{l^2}\cos^2\theta.
	\end{aligned}
\end{equation}

First, we study the case~$m\xi\geq 0$. We obtain the first inequality of~\eqref{eq: lem: inequality for lambda, eq 1} from~\eqref{eq: proof: lem: inequality for lambda, eq 1} since
\begin{equation}
	\begin{aligned}
		\int_{\mathbb{S}^2}(P(\xi)+\xi^2)|f|^2	&	= \int_{\mathbb{S}^2}\Delta_\theta|\partial_\theta f|^2 \\
		&	\qquad+ \int_{\mathbb{S}^2}\left(\xi^2\left(1-\frac{\Xi}{\Delta_\theta}\cos^2\theta\right)+\frac{\Xi^2}{\Delta_\theta}\frac{1}{\sin^2\theta}m^2+2\xi m \frac{\Xi}{\Delta_\theta}\frac{a^2}{l^2}\cos^2\theta\right)|f|^2\geq \int_{\mathbb{S}^2} m^2\Xi^2 |f|^2
	\end{aligned}
\end{equation}
where we used that~$\frac{1}{\Delta_\theta \sin^2\theta}\geq 1$. To obtain the remaining inequalities of~\eqref{eq: lem: inequality for lambda, eq 1} we proceed as follows. We rewrite the above~\eqref{eq: proof: lem: inequality for lambda, eq 1} as
\begin{equation}\label{eq: proof: lem: inequality for lambda, eq 2}
 	\begin{aligned}
 			P(\xi)f +\xi^2f &	=-\frac{1}{\sin\theta}\partial_\theta(\Delta_\theta\sin\theta\partial_\theta f)+\left(\xi \sqrt{1-\frac{\Xi}{\Delta_\theta}\cos^2\theta}\right)^2f+\left(\frac{\Xi}{\sqrt{\Delta_\theta}}\frac{1}{\sin\theta}m\right)^2f-2\xi m \sqrt{1-\frac{\Xi}{\Delta_\theta}\cos^2\theta}\frac{\Xi}{\sqrt{\Delta_\theta}}\frac{1}{\sin\theta}f\\
 			&	\qquad +2\xi m \frac{\Xi}{\Delta_\theta}\frac{a^2}{l^2}\cos^2\theta f+2\xi m \sqrt{1-\frac{\Xi}{\Delta_\theta}\cos^2\theta}\frac{\Xi}{\sqrt{\Delta_\theta}}\frac{1}{\sin\theta}f
 	\end{aligned}
\end{equation}
and after some calculations we rewrite it again as 
\begin{equation}\label{eq: proof: lem: inequality for lambda, eq 3}
	\begin{aligned}
		P(\xi)f +\xi^2f &	=-\frac{1}{\sin\theta}\partial_\theta(\Delta_\theta\sin\theta\partial_\theta f)+\left(\xi \sqrt{1-\frac{\Xi}{\Delta_\theta}\cos^2\theta}-\frac{\Xi}{\sqrt{\Delta_\theta}}\frac{1}{\sin\theta}m\right)^2f\\
		&	\qquad +2\xi m\Xi\left(1-\frac{1}{\Delta_\theta}+\sqrt{1-\frac{\Xi}{\Delta_\theta}\cos^2\theta}\frac{1}{\sqrt{\Delta_\theta}\sin\theta}\right)f.
	\end{aligned}
\end{equation}
Now, it is straightforward to prove that 
\begin{equation}
	-\frac{1}{\Delta_\theta} +\sqrt{1-\frac{\Xi}{\Delta_\theta}\cos^2\theta}\frac{1}{\sqrt{\Delta_\theta}\sin\theta}\geq 0.
\end{equation}
Therefore, from~\eqref{eq: proof: lem: inequality for lambda, eq 3} and~\eqref{eq: proof: lem: inequality for lambda, eq 1} it is easy to conclude the remaining inequalities of~\eqref{eq: lem: inequality for lambda, eq 1}.

Now, we study the case~$m\xi<0$. We obtain~\eqref{eq: lem: inequality for lambda, eq 2} in a straightforard manner from~\eqref{eq: proof: lem: inequality for lambda, eq 1}, in view of the fact that~$\frac{a^2}{l^2}<\frac{1}{4}$, see Lemma~\ref{lem: sec: properties of Delta, lem 2, a,M,l}. 

Now, we proceed to prove~\eqref{eq: lem: inequality for lambda, eq 3}. We write~\eqref{eq: proof: lem: inequality for lambda, eq 1} as follows 
\begin{equation}
	\begin{aligned}
		(P(\xi)+\frac{\Xi}{\Delta_\theta}\xi^2)f= -\frac{1}{\sin\theta}\partial_\theta(\Delta_\theta \sin\theta \partial_\theta f)+\frac{\Xi}{\Delta_\theta}\left(a\omega \sin\theta +m\frac{a^2}{l^2}\frac{\cos^2\theta}{\sin\theta}\right)^2+ \frac{m^2\Xi^2}{\Delta_\theta}\frac{1-\frac{1}{\Xi}\frac{a^4}{l^4}\cos^4\theta}{\sin^2\theta}f.
	\end{aligned}
\end{equation}
In view of the fact that~$\frac{1-\frac{1}{\Xi}\frac{a^4}{l^4}\cos^4\theta}{\sin^2\theta}\geq 1$ and that~$\frac{\Xi}{\Delta_\theta}\geq 1$ we conclude~\eqref{eq: lem: inequality for lambda, eq 3}. 
\end{proof}

We need the following definition 
\begin{definition}\label{def: subsec: sec: frequencies, subsec 3, def 0}
	Let~$\lambda^{(a\omega)}_{m\ell}$ denote the eigenvalues of the angular operator~$P_{\mu_{\textit{KG}}}$ of Definition~\ref{def: sec: carter, sturm liouville, def 1}, and~$\omega$ denote the time frequency from Carter's separation, see Proposition~\ref{prop: sec: carter, finster, prop 1}.

	Then, we define the following
	\begin{equation}\label{eq: def: subsec: sec: frequencies, subsec 3, def 0, eq 1}
		\tilde{\lambda}^{(a\omega)}_{m\ell}=\lambda^{(a\omega)}_{m\ell}+a^2\omega^2.
	\end{equation}
	We will often drop the dependence of~$\tilde{\lambda}^{(a\omega)}_{m\ell}$ on the frequencies and simply denote them as~$\tilde{\lambda}$. 
\end{definition}

\subsection{The radial part of Carter's separation}\label{subsec: sec: carter separation, radial}

We here discuss the radial part of Carter's celebrated separation of variables. For the first proof of the separation of variables see the works of Carter~\cite{Carter,Carter2,Carter3}. Also see~\cite{holzegel3} for Carter's separation in the case of negative cosmological constant~$\Lambda<0$.

We first define the Fourier transformations that will be used for the rest of the present Section

\begin{definition}\label{def: subsec: sec: carter separation, radial, def 1}
	Let~$l>0$,~$(a,M)\in \mathcal{B}_l$ and~$\mu^2_{KG}\geq 0$. Moreover, let~$\psi$ be a smooth sufficiently integrable function, see Definition~\ref{def: sec: carter separation, def 1}, where moreover the following holds~$\Box_{g_{a,M,l}}\psi-\mu^2_{KG}\psi=F$. Then, for~$\omega\neq 0$ we define the Fourier transforms as follows
	\begin{equation}
		\begin{aligned}
			\psi^{(a\omega)}_{m\ell} &	= \int_{\mathbb{R}} \int_{\mathbb{S}^2} e^{-i\omega t}e^{i m\varphi} S^{(a\omega)}_{m\ell}(\theta) \cdot(-i\omega)^{-1} \partial_t\psi d\sigma dt,\\
			F^{(a\omega)}_{m\ell} &	= \int_{\mathbb{R}} \int_{\mathbb{S}^2} e^{-i\omega t}e^{i m\varphi} S^{(a\omega)}_{m\ell}(\theta) \cdot(-i\omega)^{-1} \partial_tF d\sigma dt,
		\end{aligned}
	\end{equation}
	with~$d\sigma=\sin\theta dtd\theta d\varphi$. For all~$m,\ell$ the functions
	\begin{equation}\label{eq: prop: Carters separation, radial part, eq 1.1}
		\omega^{i_1}(Z^\star)^{i_2}\left(\sqrt{\tilde{\lambda}^{(a\omega)}_{m\ell}}\right)^{i_3}\psi^{(a\omega)}_{m\ell}(r),\qquad \mu^2_{KG}\psi^{(a\omega)}_{m\ell},\qquad \omega^{i_1}(Z^\star)^{i_2}\left(\sqrt{\tilde{\lambda}^{(a\omega)}_{m\ell}}\right)^{i_3} F^{(a\omega)}_{m\ell}
	\end{equation}
	extend to~$L^2(d\omega)$ functions for any~$i_1\geq 1,~ i_2+i_3\geq 0$ and moreover we have
	\begin{equation}\label{eq: prop: Carters separation, radial part, eq 2}
		\begin{aligned}
			\sup_{r\in[r_+,\bar{r}_+]}\sum_{m,\ell}\int_{\mathbb{R}}\sum_{1\leq i_1+i_2+i_3\leq j}\left(\mu^2_{KG}|\psi^{(a\omega)}_{m\ell}|^2+ \left|\omega^{i_1}(Z^\star)^{i_2}\left(\sqrt{\tilde{\lambda}^{(a\omega)}_{m\ell}}\right)^{i_3}\psi^{(a\omega)}_{m\ell}\right|^2  \right) d\omega 	&<	\infty,\\
			\sup_{r\in[r_+,\bar{r}_+]}\sum_{m,\ell}\int_{\mathbb{R}}\sum_{0\leq i_1+i_2+i_3\leq j}\left|\omega^{i_1}(Z^\star)^{i_2}\left(\sqrt{\tilde{\lambda}^{(a\omega)}_{m\ell}}\right)^{i_3}F^{(a\omega)}_{m\ell}\right|^2  d\omega 	&<	\infty,
		\end{aligned}
	\end{equation}
	for any~$j\geq 1$.
\end{definition}

\begin{remark}
	In the case of the wave equation~$\mu^2_{KG}=0$ we note that~$\psi^{(a\omega)}_{m\ell}$ of Definition~\ref{def: subsec: sec: carter separation, radial, def 1} is not necessarily an~$L^2(d\omega)$ function, in view of the requirement~$i_1\geq 1$ in~\eqref{eq: prop: Carters separation, radial part, eq 1.1}. 
\end{remark}

Let~$\Psi$ be a sufficiently integrable function. We denote the Fourier transform with respect to the Boyer--Lindquist coordinate~$\varphi$ as follows
\begin{equation}\label{eq: subsec: sec: carter separation, radial, eq 1}
\mathcal{F}_m(\Psi)(t,r,\theta)= \int_{0}^{2\pi} e^{-i m\varphi} \Psi(t,r,\theta,\varphi) d\varphi. 
\end{equation}

Now we are ready to present Carter's separation for the radial ode. 

\begin{proposition}[Carter's separation,~\cite{Carter,Carter2,Carter3,DR2}]\label{prop: Carters separation, radial part}
	
Let~$l>0$,~$(a,M)\in\mathcal{B}_l$ and~$\mu^2_{\textit{KG}}\geq 0$. Let~$\psi$ be a sufficiently integrable function, see Definition~\ref{def: sec: carter separation, def 1}, that solves 
\begin{equation}
	\Box_{g_{a,M,l}}\psi -\mu_{\textit{KG}}^2 \psi =F
\end{equation}
where for the operator~$\Box_{g_{a,M,l}}$ see~\eqref{eq: kleingordon}. Let~$\psi^{(a\omega)}_{m\ell},F^{(a\omega)}_{m\ell}$ be the Fourier transforms of~$\psi,F$ respectively, see Definition~\ref{def: subsec: sec: carter separation, radial, def 1}.

Then, for almost all~$\omega$ the rescaled function
\begin{equation}\label{eq: ode from carter's separation, eq 0}
u^{(a\omega)}_{m\ell}(r)=\sqrt{r^2+a^2}\psi^{(a\omega)}_{m\ell}(r)
\end{equation}
is smooth for all~$m,\ell$ and all~$r\in (r_+,\bar{r}_+)$ and moreover it satisfies Carter's fixed frequency radial ode
\begin{equation}\label{eq: ode from carter's separation}
[u^{(a\omega)}_{ml}(r)]^{\prime\prime}+(\omega^{2}-V^{(a\omega)}_{ml}(r))u^{(a\omega)}_{ml}(r)=H^{(a\omega)}_{ml}(r), \qquad H^{(a\omega)}_{ml}=\frac{\Delta}{(r^2+a^2)^{3/2}}\left(\rho^2 F\right)^{(a\omega)}_{ml}(r),
\end{equation}
with~$^\prime=\frac{d}{dr^\star}$ see Section~\ref{subsec: tortoise coordinate}.

The potential~$V$ can we written as 
\begin{equation}\label{eq: the potential V}
V\:\dot{=}\:V^{(a\omega)}_{m\ell}(r)=V_{\textit{SL}}(r)+\left(V_0\right)^{(a\omega)}_{m\ell}(r) +V_{\mu_{\textit{KG}}}(r),
\end{equation}
where
\begin{equation}\label{eq: the potentials V0,VSl,Vmu}
    \begin{aligned}
    	V_{\textit{SL}}
    	&   =(r^{2}+a^{2})^{-1/2}\left(\frac{d}{dr^\star}\right)^2(\sqrt{r^{2}+a^{2}})\\ &=-\Delta^{2}\frac{3r^{2}}{(r^{2}+a^{2})^{4}}+\Delta\frac{-5\frac{r^{4}}{l^{2}}+3r^{2}(1-\frac{a^{2}}{l^{2}})-4Mr+a^{2}}{(r^{2}+a^{2})^{3}}, \\
        V_0 &   =\frac{\Delta(\lambda^{(a\omega)}_{m\ell}+\omega^{2}a^{2})-\Xi^{2}a^{2}m^{2}-2m\omega a\Xi(\Delta -(r^{2}+a^{2}))}{(r^{2}+a^{2})^{2}}\\
        &=\frac{\Delta}{(r^2+a^2)^2}(\lambda^{(a\omega)}_{m\ell}+a^2\omega^2-2am\omega\Xi)+\omega^2-\left(\omega-\frac{am\Xi}{r^2+a^2}\right)^2,\\
        V_{\mu_{\textit{KG}}}	&=\mu_{\textit{KG}}^2\frac{\Delta}{r^2+a^2}.
    \end{aligned}
\end{equation}

Moreover, if~$\psi$ is in addition assumed outgoing, see Definition~\ref{def: subsec: sec: carter separation, subsec 0, def 1}, then the following boundary conditions hold
	\begin{equation}\label{eq: lem: sec carters separation, subsec boundary behaviour of u, boundary beh. of u, eq 3}
	\begin{aligned}
		&   \frac{d u}{d r^\star}=-i(\omega-\omega_+m) u,\qquad \frac{d u}{d r^\star}=i(\omega-\bar{\omega}_+m) u,
	\end{aligned}
\end{equation}
at~$r^\star=-\infty,r^\star=+\infty$ respectively, in the sense:
\begin{equation}\label{eq: lem: sec carters separation, subsec boundary behaviour of u, boundary beh. of u, eq 4}
	\lim_{r^\star\rightarrow -\infty}\left(u^\prime (r^\star) +i(\omega-\omega_+ m) u(r^\star)\right)=0,\qquad \lim_{r^\star\rightarrow \infty}\left(u^\prime (r^\star) -i(\omega-\bar{\omega}_+m) u(r^\star)\right)=0,
\end{equation}
where~$\omega_+=\frac{a\Xi}{r_+^2+a^2},~\bar{\omega}_+=\frac{a\Xi}{\bar{r}_+^2+a^2}$. Finally, we note that~$|u|^2(\pm \infty)= \lim_{r^\star\rightarrow\pm \infty} |u|^2(r^\star)$ are well defined. 
\end{proposition}

\subsection{Parseval identities and inequalities}\label{subsec: fourier properties}

We have the following Parseval identities

\begin{proposition}\label{prop: subsec: fourier properties, prop 1}
	Let~$l>0$,~$(a,M)\in\mathcal{B}_l$ and~$\mu^2_{KG}\geq 0$. 
	
	Let~$\Psi$ be a sufficiently integrable function, see Definition~\ref{def: sec: carter separation, def 1}. Then, we have the following identities
	
	\begin{equation}\label{eq: subsec: fourier properties, eq 1}
		\begin{aligned}
			&  \int_{\mathbb{S}^2}\int_\mathbb{R}|\partial_t\Psi|^2 \sin\theta d\theta d\varphi dt=\int_\mathbb{R}\sum_{m,\ell}\omega^2|\Psi^{(a\omega)}_{m,\ell}(r)|^2 d\omega,\\
			& \int_{\mathbb{S}^2}\int_\mathbb{R}\Psi\cdot Y \sin\theta d\theta d\varphi dt =\int_{\mathbb{R}}\sum_{m,\ell}\Psi^{(a\omega)}_{m\ell}\cdot \overline{Y^{(a\omega)}_{m,\ell}}d\omega,\\
			&\int_{\mathbb{R}}\sum_{m,\ell}\lambda^{(a\omega)}_{m\ell}|\Psi^{(a\omega)}_{m,\ell}|^2d\omega = \int_{\mathbb{S}^2}\int_\mathbb{R}|^{d\sigma_{\mathbb{S}^2}}\nabla\Psi|_{d\sigma_{\mathbb{S}^2}}^2\sin\theta d\theta d\varphi dt -a^2\int_{\mathbb{S}^2}\int_\mathbb{R}|\partial_t\Psi|^2\cos^2\theta\sin\theta d\theta d\varphi dt,
		\end{aligned}
	\end{equation}
	where~$^{d\sigma_{\mathbb{S}^2}}\nabla$ is the covariant derivative of~$d\sigma_{\mathbb{S}^2}$ and~$d\sigma_{\mathbb{S}^2}$ is the standard metric of the unit sphere, see Section~\ref{subsec: unit sphere}. For~$\lambda^{(\xi)}_{m\ell}$, with~$\xi=a\omega$, see Proposition~\ref{prop: sec: carter, finster, prop 1}. In the second identity of~\eqref{eq: subsec: fourier properties, eq 1} we also assumed that the integrals are finite.

	Furthermore, there exist constants~$c(a,M,l),C(a,M,l)>0$, such that if~$\Psi$ is in addition an outgoing function, see Definition~\ref{def: subsec: sec: carter separation, subsec 0, def 1}, then for almost all~$\omega$ the radial functions~$|\Psi^{(a\omega)}_{m\ell}|^2(\pm \infty)$ are well defined for all~$m,\ell$, and we have the following identity
	\begin{equation}\label{eq: subsec: fourier properties, eq 2}
		\begin{aligned}
			&	\int_{\mathbb{R}}\int_{\mathbb{S}^2} |\partial_\varphi\Psi|^2(r_+,\theta,\varphi,t) \sin\theta d\theta d\varphi dt= \int_\mathbb{R} \sum_{m,\ell} |m\Psi^{(a\omega)}_{m\ell}(-\infty)|^2d\omega.
		\end{aligned}
	\end{equation}
\end{proposition}

\begin{remark}
	We note that~$\int_{\mathbb{S}^2}\int_\mathbb{R}|^{d\sigma_{\mathbb{S}^2}}\nabla\Psi|_{d\sigma_{\mathbb{S}^2}}^2\sin\theta d\theta d\varphi dt \sim \int_{\mathbb{S}^2}\int_\mathbb{R}|\slashed{\nabla}\Psi|^2\sin\theta d\theta d\varphi dt$.
\end{remark}

\subsection{Energy identity for outgoing solutions~\texorpdfstring{$u$}{u} of~\eqref{eq: ode from carter's separation}}\label{subsec: energy identity}

Note the following

\begin{proposition}\label{prop: subsec: energy identity, prop 1}
Let~$l>0$,~$(a,M)\in\mathcal{B}_l$ and~$\mu^2_{KG}\geq 0$. Let~$(\omega,m,\ell)\in\mathbb{R}\times\bigcup_{m\in\mathbb{Z}}\left(\{m\}\times\mathbb{Z}_{\geq |m|}\right)$. Let~$u$ be a smooth solution of Carter's radial ode~\eqref{eq: ode from carter's separation} which moreover satisfies the boundary conditions~\eqref{eq: lem: sec carters separation, subsec boundary behaviour of u, boundary beh. of u, eq 3}, in the sense of~\eqref{eq: lem: sec carters separation, subsec boundary behaviour of u, boundary beh. of u, eq 4}.

Then, we have the following 
\begin{equation}\label{eq: subsec: energy identity, eq 1}
\left(\omega-\frac{am\Xi}{\bar{r}_+^2+a^2}\right)|u|^2(\infty)+\left(\omega-\frac{am\Xi}{r_+^2+a^2}\right)|u|^2(-\infty)= \Im (\bar{u}H).
\end{equation}
\end{proposition}
\begin{proof}
We multiply the radial ode~\eqref{eq: ode from carter's separation} with~$\bar{u}$ and after integration by parts we obtain the desired result and the fact that we may rewrite the asymptotic formulas~\eqref{eq: lem: sec carters separation, subsec boundary behaviour of u, boundary beh. of u, eq 3} as follows 
\begin{equation}
u^\prime=- i\left(\omega-\frac{am\Xi}{r_+^2+a^2}\right) u,\qquad u^\prime=i \left(\omega-\frac{am\Xi}{\bar{r}_+^2+a^2}\right) u,
\end{equation}
at~$r^\star=-\infty,+\infty$ respectively.
\end{proof}

\subsection{The Wronskian and representation formulas for the inhomogeneous Carter radial ode}\label{subsec: sec: Carter's separation, wronskian}

Let~$l>0$,~$(a,M)\in\mathcal{B}_l$ and~$\mu^2_{KG}\geq 0$. Let~$\omega\neq \omega_+ m,~\omega\neq \bar{\omega}_+ m$, where~$\omega_+=\frac{a\Xi}{r_+^2+a^2},\bar{\omega}_+=\frac{a\Xi}{\bar{r}_+^2+a^2}$. Moreover, let
\begin{equation}\label{eq: subsec: sec: Carter's separation, wronskian, eq 2}
(u_{\mathcal{H}^+})^{(a\omega)}_{m\ell},\qquad (u_{\bar{\mathcal{H}}^+})^{(a\omega)}_{m\ell}
\end{equation}
simply denoted as~$u_{\mathcal{H}^+},~ u_{\bar{\mathcal{H}}^+}$, be the solutions of the homogeneous radial ode
\begin{equation}\label{eq: subsec: sec: Carter's separation, wronskian, eq 1}
	u^{\prime\prime}+(\omega^2-V)u=0,
\end{equation}
that satisfy respectively the outgoing asymptotics
\begin{equation}\label{eq: subsec: sec: Carter's separation, wronskian, eq 1.1}
u_{\mathcal{H}^+}\sim (r-r_+)^{\eta_{\mathcal{H}^+}},\qquad u_{\bar{\mathcal{H}}^+}\sim (\bar{r}_+-r)^{\eta_{\bar{\mathcal{H}}^+}} 
\end{equation}
where 
\begin{equation}
\eta_{\mathcal{H}^+}=-\frac{i\left(\omega-\omega_+m\right)}{2\kappa_{\mathcal{H}^+}},\qquad \eta_{\bar{\mathcal{H}}^+}= \frac{i\left(\omega-\bar{\omega}_+m\right)}{2\kappa_{\bar{\mathcal{H}}^+}},
\end{equation}
where note the surface gravities~$\kappa_{\mathcal{H}^+}=\frac{1}{2}\Big|\frac{d\Delta}{dr}\Big|(r_+),\:\:\kappa_{\bar{\mathcal{H}}^+}=\frac{1}{2}\Big|\frac{d\Delta}{dr}\Big|(\bar{r}_+)$. For the existence of the solutions~\eqref{eq: subsec: sec: Carter's separation, wronskian, eq 2} see the book of Olver~\cite{OlverAsymptotics}.

\begin{definition}\label{def: subsec: sec: Carter's separation, wronskian, def 1}
	Let~$l>0$,~$(a,M)\in \mathcal{B}_l$ and~$\mu^2_{KG}\geq 0$. Moreover, let~$(\omega,m,\ell)\in \mathbb{R}\times \bigcup_{m\in\mathbb{Z}}\{m\}\times \mathbb{Z}_{\geq |m|}$, where moreover~$\omega\neq \omega_+ m,\omega\neq \bar{\omega}_+ m$. We define the Wronskian associated to the homogeneous radial ode~\eqref{eq: subsec: sec: Carter's separation, wronskian, eq 1} to be 
\begin{equation}\label{eq: def: subsec: sec: Carter's separation, wronskian, def 1, eq 1}
\mathcal{W}(a,M,l,\mu_{KG},\omega,m,\ell)=\mathcal{W}[u_{\mathcal{H}^+},u_{\bar{\mathcal{H}}^+}](\omega,m,\ell)= u_{\bar{\mathcal{H}}^+}^\prime u_{\mathcal{H}^+}-u_{\bar{\mathcal{H}}^+}u_{\mathcal{H}^+}^\prime, 
\end{equation}
where we note that the RHS of~\eqref{eq: def: subsec: sec: Carter's separation, wronskian, def 1, eq 1} is independent of~$r^\star$. 
\end{definition}

Let~$\omega\in\mathbb{R},~m\in\mathbb{Z},~\ell\in \mathbb{Z}_{\geq |m|},~\omega\neq \omega_+m,~\omega\neq \bar{\omega}_+m $. We have that a non-trivial real $(\omega,m,\ell)$--frequency mode solution of Carter's radial ode~\eqref{eq: ode from carter's separation} exists if and only if we have~$\mathcal{W}[u_{\mathcal{H}^+},u_{\bar{\mathcal{H}}^+}](\omega,m,\ell)=0$.

\begin{remark}
In our discussion above on real~$(\omega,m,\ell)$-frequency mode solutions we have excluded the case~$\omega=m=0$ where in fact trivial mode solutions do exist in the case~$\mu_{KG}=0$, and they correspond to constant solutions of the wave equation~\eqref{eq: kleingordon}~(with~$\mu_{KG}=0$). These frequencies are treated with the main estimate of Section~\ref{subsubsec: sec: bounded: stationary, 1}, namely Proposition~\ref{prop: energy estimate for the bounded stationary frequencies, 1}. Due to the existence of these constant solutions we have that in Proposition~\ref{prop: energy estimate for the bounded stationary frequencies, 1} we do not control the zeroth order term of the solution of Carter's radial ode, but only control derivative terms of the solution of Carter's radial ode.

Moreover, in the definition of the Wronskain, see Definition~\ref{def: subsec: sec: Carter's separation, wronskian, def 1}, and in our discussion above on real~$(\omega,m,\ell)$ frequency mode solutions we have, more generally, excluded the frequencies~$\omega~=~\omega_+m$, $\omega=\bar{\omega}_+ m$, since we did not uniquely define the solutions~$u_{\mathcal{H}^+},~u_{\bar{\mathcal{H}}^+}$ for those frequency cases. Furthermore, a benefit of excluding those frequencies is that in Proposition~\ref{prop: subsec: sec: Carter's separation, wronskian, prop 1} we can prove that the Wronskian is a continuous function in~$\omega$ in a straightforward manner. Finally, we prove energy estimates for the bounded frequency regimes~$\{0\leq |\omega-\omega_+ m|\leq \epsilon\}\cup \{0\leq |\omega-\bar{\omega}_+m |\leq \epsilon\}$, for a sufficiently small~$\epsilon>0$, see Proposition~\ref{prop: energy estimate in the bounded non stationary frequency regime}. 
\end{remark}

We have the following proposition

\begin{proposition}\label{prop: subsec: sec: Carter's separation, wronskian, prop 1}
	Let~$l>0$,~$(a,M)\in \mathcal{B}_l$ and~$\mu^2_{KG}\geq 0$. Moreover, let~$(\omega,m,\ell)\in \mathbb{R}\times \bigcup_{m\in\mathbb{Z}}\{m\}\times \mathbb{Z}_{\geq |m|}$, where moreover~$\omega\neq \omega_+ m,~\omega\neq \bar{\omega}_+ m$. Then, the Wronskian~$\mathcal{W}(a,M,l,\mu_{KG},\cdot,m,\ell)$, see Definition~\ref{def: subsec: sec: Carter's separation, wronskian, def 1}, is a continuous function in~$\mathbb{R}\setminus\{\omega_+ m,\bar{\omega}_+m\}$. 
\end{proposition}
\begin{proof}
	This follows easily since we have that for~$\omega\neq \omega _+ m$ and~$\omega\neq \bar{\omega}_+ m$ the solutions~$u_{\mathcal{H}^+},u_{\bar{\mathcal{H}}^+}$, see~\eqref{eq: subsec: sec: Carter's separation, wronskian, eq 1.1}, depend continuously on~$\omega$ in the desired domain.
\end{proof}

Note the following Lemma.

\begin{lemma}\label{lem: subsec: sec: Carter's separation, wronskian, lem 1}
Let~$l>0$,~$(a,M)\in\mathcal{B}_l$ and~$\mu^2_{\textit{KG}}\geq 0$. Let~$(\omega,m,\ell)$ with~$\omega\neq \omega_+ m ,\omega\neq \bar{\omega}_+ m $ be such that
\begin{equation}
|\mathcal{W}^{-1}(\omega,m,\ell)|< +\infty
\end{equation}
where for the Wronskian see Definition~\ref{def: subsec: sec: Carter's separation, wronskian, def 1}.

Then, for smooth solutions of Carter's separation of variables, see Proposition~\ref{prop: Carters separation, radial part}, that moreover satisfy the outgoing boundary conditions~\eqref{eq: lem: sec carters separation, subsec boundary behaviour of u, boundary beh. of u, eq 3}, in the sense of~\eqref{eq: lem: sec carters separation, subsec boundary behaviour of u, boundary beh. of u, eq 4}, we obtain the following representation formula
\begin{equation}\label{eq: subsec: sec: Carter's separation, wronskian, eq 3}
u(r^\star)=\mathcal{W}^{-1}\cdot\left(u_{\bar{\mathcal{H}^+}}(r^\star)\int_{-\infty}^{r^\star} u_{\mathcal{H}^+}(x^\star) H(x^\star) dx^\star+u_{\mathcal{H}^+}(r^\star)\int_{r^\star}^{+\infty} u_{\bar{\mathcal{H}^+}}(x^\star)H(x^\star)  \right),
\end{equation}
where the equality~\eqref{eq: subsec: sec: Carter's separation, wronskian, eq 3} holds pointwise; for~$u_{\mathcal{H}^+},u_{\bar{\mathcal{H}}^+}$ see~\eqref{eq: subsec: sec: Carter's separation, wronskian, eq 1.1}.
\end{lemma}

\begin{proof}
Follow verbatim the arguments of Shlapentokh-Rothman~[\cite{Shlapentokh_Rothman_2014_mode_stability}, Section 3], which rely on a classic variation of parameters argument. 
\end{proof}

Moreover, we have the following on Schwarzschild--de~Sitter

\begin{proposition}\label{prop: subsec: energy identity, prop 2}
	Let~$l>0$,~$(0,M)\in \mathcal{B}_l$ and~$\mu^2_{KG}\geq 0$. Then, for any~$\omega\in\mathbb{R}\setminus\{0\},~m\in\mathbb{Z},~\ell\in \mathbb{Z}_{\geq |m|}$ we have that~$\mathcal{W}(\omega,m,\ell)\neq 0$. 
\end{proposition}

\begin{proof}
	From Proposition~\ref{prop: subsec: energy identity, prop 1} we have that 
	\begin{equation}
		\omega |u|^2(\infty)+\omega|u|^2(-\infty)= 0,
	\end{equation}
	where~$u$ is a smooth solution of Carter's homogeneous radial ode~\eqref{eq: ode from carter's separation}, from which it follows easily that~$u\equiv 0$. 
\end{proof}

\subsection{The operator~\texorpdfstring{$\mathcal{P}_{trap}$}{g}}

Let~$\Psi$ be a sufficiently integrable function, see Definition~\ref{def: sec: carter separation, def 1}. We define the following operator
\begin{equation}\label{eq: sec: main theorems, eq 1.1}
	\begin{aligned}
		\mathcal{P}_{\textit{trap}}[\Psi] &	= \frac{1}{\sqrt{2\pi}} \int_{\mathbb{R}}\sum_{m,\ell}\left|1-\frac{r_{\textit{trap}}}{r}\right|e^{i\omega t}\Psi^{(a\omega)}_{m\ell}(r)S^{(a\omega)}_{m\ell}(\theta)e^{- i m\phi}d\omega,
	\end{aligned}
\end{equation}
where~$r_{trap}(\omega,m,\ell)\in (R^-,R^+)\cup\{0\}\subsetneq (r_+,\bar{r}_+)\cup \{0\}$, see already Theorem~\ref{main theorem 1} and Theorem~\ref{thm: sec: proofs of the main theorems, thm 3}.

Furthermore, we have
\begin{equation}\label{eq: sec: main theorems, eq 1.0.1}
	\begin{aligned}
		\partial_t\mathcal{P}_{\textit{trap}}[\Psi]:=&	\frac{1}{\sqrt{2\pi}}\int_{\mathbb{R}}\sum_{m,\ell}\left|1-\frac{r_{\textit{trap}}}{r}\right|e^{i\omega t}i\omega\Psi^{(a\omega)}_{m\ell}(r)S^{(a\omega)}_{m\ell}(\theta)e^{- i m\phi}d\omega, \\		
		|\slashed{\nabla}\mathcal{P}_{\textit{trap}}[\Psi]|:=& \left|\frac{1}{\sqrt{2\pi}}\int_{\mathbb{R}}\sum_{m,\ell}\left|1-\frac{r_{\textit{trap}}}{r}\right|e^{i\omega t}\Psi^{(a\omega)}_{m\ell}(r)\slashed{\nabla}\left(S^{(a\omega)}_{m\ell}(\theta)e^{- i m\phi}\right)d\omega\right|.
	\end{aligned}
\end{equation}

We record a global Parseval identity and an estimate
\begin{equation}\label{eq: sec: main theorems, eq 1.0.2}
	\begin{aligned}
		\int_{\mathbb{R}} \int_{\mathbb{R}}\int_{\mathbb{S}^2}
		|\partial_t\mathcal{P}_{trap}[\Psi]|^2 \sin\theta d\theta d\varphi dr^\star dt &= \int_{\mathbb{R}}\int_{\mathbb{R}}\sum_{m,\ell} \omega^2\left(1-\frac{r_{trap}}{r}\right)^2|\Psi^{(a\omega)}_{m\ell}|^2  dr^\star d\omega,\\
		\int_{\mathbb{R}} \int_{\mathbb{R}}\int_{\mathbb{S}^2}
		|\slashed{\nabla}\mathcal{P}_{trap}[\Psi]|^2 \sin\theta d\theta d\varphi dr^\star dt &\leq  B \int_{\mathbb{R}}\int_{\mathbb{R}}\sum_{m,\ell} \omega^2\left(1-\frac{r_{trap}}{r}\right)^2|\Psi^{(a\omega)}_{m\ell}|^2  dr^\star d\omega \\
		&	\qquad\qquad +B \int_{\mathbb{R}}\int_{\mathbb{R}}\sum_{m,\ell} \lambda^{(a\omega)}_{m\ell}\left(1-\frac{r_{trap}}{r}\right)^2|\Psi^{(a\omega)}_{m\ell}|^2  dr^\star d\omega 
	\end{aligned}
\end{equation}
also see the Section~\ref{subsec: fourier properties} on Parseval identities.

Since~$\Psi$ is sufficiently integrable we have that 
\begin{equation}
	\sup_{r\in [r_+,\bar{r}_+]} \int_{\mathbb{R}}\int_{\mathbb{S}^2} \left(\mu^2_{KG}|\mathcal{P}_{trap}[\Psi]|^2 + \sum_{1\leq i_1+i_2\leq j} |\slashed{\nabla}^{i_1}\partial_t^{i_2}\mathcal{P}_{trap}[\Psi]|^2\right)<\infty.
\end{equation}

\subsection{The superradiant frequencies~\texorpdfstring{$\mathcal{SF}$}{g}}\label{subsec: sec: frequencies, subsec 1}

We need the following

\begin{definition}\label{def: subsec: sec: frequencies, subsec 1, def 1}
	Let~$l>0$ and~$(a,M)\in\mathcal{B}_l$. We define the superradiant frequencies as
	\begin{equation}
		\mathcal{SF}=\{(\omega,m):~am\omega\in\left(\frac{a^2m^2\Xi}{\bar{r}_+^2+a^2},\frac{a^2m^2\Xi}{r_+^2+a^2}\right)\}
	\end{equation}
	or equivalently~$\mathcal{SF}=\{(\omega,m):~\left(\omega-\frac{am\Xi}{\bar{r}_+^2+a^2}\right)\left(\omega-\frac{am\Xi}{r_+^2+a^2}\right) < 0\}$. 
\end{definition}

In the Kerr limit~($\Lambda=0$), in view of the fact that~$\bar{r}_+(\Lambda=0)=\infty$, see Lemma~\ref{lem: sec: properties of Delta, lem 2, a,M,l}, we obtain that the superradiant frequencies are
\begin{equation}
	am\omega\in \left(0,\frac{am^2}{r_+^2+a^2}\right),~~\text{or equivalently}~~ \omega\left(\omega-\frac{am}{r_+^2+a^2}\right) < 0,
\end{equation}
see~\cite{DR2}.

\subsection{The frequency set~\texorpdfstring{$\mathcal{F}_{\mathcal{SF},\mathcal{C}}$}{g}}\label{subsec: sec: carter separation, subsec 2}

Let~$l>0$ and~$(a,M)\in\mathcal{B}_l$. Let~$\mathcal{C}>0$. We define the set
\begin{equation}\label{eq: subsec: sec: carter separation, subsec 2, eq 1}
	\mathcal{F}_{\mathcal{SF},\mathcal{C}} =\Bigl\{(\omega,m,\ell):~|\omega|\in [\mathcal{C}^{-1},\mathcal{C}],~\tilde{\lambda}^{(a\omega)}_{m\ell}\leq \mathcal{C},~(\omega,m)\in \mathcal{SF},~|\omega-\omega_+ m|> \frac{1}{\mathcal{C}},~|\omega-\bar{\omega}_+m|> \frac{1}{\mathcal{C}}  \Bigr\}.
\end{equation}
For the superradiant frequencies~$\mathcal{SF}$, see Definition~\ref{def: subsec: sec: frequencies, subsec 1, def 1}. Moreover, for~$\tilde{\lambda}^{(a\omega)}_{m\ell}$ see Definition~\ref{def: subsec: sec: frequencies, subsec 3, def 0}.

\subsection{The operator~\texorpdfstring{$\mathcal{P}_{\mathcal{SF},\mathcal{C}}$}{g}}

Let~$\Psi$ be a sufficiently integrable function, see Definition~\ref{def: sec: carter separation, def 1}. Given~$\mathcal{C}>0$, we define the Fourier projection~$\mathcal{P}_{\mathcal{SF},\mathcal{C}}$ on bounded non-stationary superradiant frequencies
\begin{equation}\label{eq: sec: main theorems, eq 1}
	\begin{aligned}
		\mathcal{P}_{\mathcal{SF},\mathcal{C}} [\Psi]= \frac{1}{\sqrt{2\pi}} \int_{\mathbb{R}} \sum_{m,\ell} 1_{\mathcal{F}_{\mathcal{SF},\mathcal{C}}}\cdot e^{i\omega t}e^{-im\varphi} S^{(a\omega)}_{m\ell}(\theta) \Psi^{(a\omega)}_{m\ell},
	\end{aligned}
\end{equation}
where for~$\mathcal{F}_{\mathcal{SF},\mathcal{C}}$ see~\eqref{eq: subsec: sec: carter separation, subsec 2, eq 1}.

By using the Parseval identity~\eqref{eq: subsec: fourier properties, eq 2} we also have that 
\begin{equation}
	\begin{aligned}
		\int_{\mathcal{H}^+} |\partial_{\varphi}\mathcal{P}_{\mathcal{SF},\mathcal{C}} [\Psi]|^2 dg  & = \int_{\mathbb{R}}\int_{\mathbb{S}^2} |\partial_{\varphi} \mathcal{P}_{\mathcal{SF},\mathcal{C}}[\Psi]|^2 \sin\theta d\theta d\varphi dt  \\
		& =\int_{\mathbb{R}} \sum_{m,\ell} 1_{\mathcal{F}_{\mathcal{SF},\mathcal{C}}} |m^2\Psi^{(a\omega)}_{m\ell}|^2(r^\star=-\infty)d\omega\\
		&	\sim \int_{\mathbb{R}} \sum_{m,\ell} 1_{\mathcal{F}_{\mathcal{SF},\mathcal{C}}} |\Psi^{(a\omega)}_{m\ell}|^2(r^\star=-\infty)d\omega \\
		&	\sim \int_{\mathcal{H}^+} |\mathcal{P}_{\mathcal{SF},\mathcal{C}}[\Psi]|^2 dg,
	\end{aligned}
\end{equation}
where the similarities above are with respect to constants~$C=C(a,M,l,\mu_{KG})$.

Again, since~$\Psi$ is sufficiently integrable we have 
\begin{equation}
	\sup_{r\in [r_+,\bar{r}_+]} \int_{\mathbb{R}}\int_{\mathbb{S}^2} \left(\mu^2_{KG}|\mathcal{P}_{\mathcal{SF},\mathcal{C}}[\Psi]|^2 + \sum_{1\leq i_1+i_2\leq j} |\slashed{\nabla}^{i_1}\partial_t^{i_2}\mathcal{P}_{\mathcal{SF},\mathcal{C}}[\Psi]|^2\right)<\infty.
\end{equation}

\subsection{The set~\texorpdfstring{$\widetilde{\mathcal{MS}}_{l,\mu_{KG}}$}{g}}

In the following definition we define various sets we need throughout the paper and in particular in Proposition~\ref{prop: subsec: summing in the redshift estimate, prop 1}:

\begin{definition}\label{def: subsec: sec: main theorems, subsec 1, def 1}
	Let~$l>0$ and~$\mu^2_{KG}\geq 0$. We define the set
	\begin{equation}\label{eq: sec: main theorems, eq 1.2}
		\widetilde{\mathcal{MS}}_{l,\mu_{KG}}=\{(a,M)\in\mathcal{B}_l: \forall(\omega,m,\ell)\in\mathbb{R}\times\mathbb{Z}\times\mathbb{Z}_{\geq |m|}\setminus\{\omega= \omega_+m,~\omega= \bar{\omega}_+m\} , ~|\mathcal{W}^{-1}(a,M,l,\mu^2_{KG},\omega,m,\ell)|<+\infty\},
	\end{equation}
	where~$\omega_+=\frac{a\Xi}{r_+^2+a^2},\bar{\omega}_+=\frac{a\Xi}{\bar{r}_+^2+a^2}$.

	We define the set
	\begin{equation}
		\mathcal{B}_{0,l} = \{(0,M)\in\mathcal{B}_l\} = \{(0,M):~0<\frac{M^2}{l^2}<\frac{1}{27}\}.
	\end{equation}
	By Proposition~\ref{prop: subsec: energy identity, prop 2} we have that~$\mathcal{B}_{0,l} \subset \widetilde{\mathcal{MS}}_{l,\mu_{KG}}$. 
\end{definition}

\begin{remark}\label{rem: subsec: sec: main theorems, subsec 1, rem 0}
Note that if~$(a,M)\in \widetilde{\mathcal{MS}}_{l,\mu_{KG}}$ then we have that~$(-a,M)\in \widetilde{\mathcal{MS}}_{l,\mu_{KG}}$, in view of the fact that~$\mathcal{W}(a,M,l,\mu_{KG}^2,\omega,m,\ell)=\mathcal{W}(-a,M,l,\mu_{KG}^2,-\omega,m,\ell)$. 
\end{remark}

\subsection{Quantitative mode stability as a flux bound}

The main result of the present Section is Proposition~\ref{prop: subsec: summing in the redshift estimate, prop 1}, which is a quantitative version of mode stability.

\begin{proposition}\label{prop: subsec: summing in the redshift estimate, prop 1}

	Let~$l>0$,~$\mu^2_{KG}\geq 0$ and~$(a,M)\in\widetilde{\mathcal{MS}}_{l,\mu_{KG}}$. Let~$\psi$ be a sufficiently integrable and outgoing solution of the Klein--Gordon equation~\eqref{eq: kleingordon}.

	Then, for any~$\mathcal{C}>0$ sufficiently large, there exists a constant~$C(a,M,l,\mu_{KG},\mathcal{C})>0$ such that we have the following estimate 
	\begin{equation}\label{eq: prop: subsec: summing in the redshift estimate, prop 1, eq 0}
		\int_{\mathcal{H}^+}\left|\mathcal{P}_{\mathcal{SF},\mathcal{C}} (\chi_{\tau_1}\psi)\right|^2 \leq C\int_{\{t^\star=\tau_1\}} J^n_\mu[\psi]n^\mu,
	\end{equation}
	where~$1\leq \tau_1$. 
\end{proposition}

\begin{proof}
	Given that~$l>0$,~$\mu^2_{KG}\geq 0$,~$(a,M)\in \widetilde{\mathcal{MS}}_{l,\mu_{KG}}$, we repeat the estimates of~[\cite{Shlapentokh_Rothman_2014_mode_stability}, Section 3], which we here only sketch. First, if~$\psi$ is a smooth solution of the Klein Gordon equation~\eqref{eq: kleingordon} then we have 
	\begin{equation}
		\Box(\chi_{\tau_1}\psi)=F=2\nabla \chi_{\tau_1}\cdot \nabla\psi +\psi\Box\chi_{\tau_1}.
	\end{equation}
	Furthermore, we define~$u=\sqrt{r^2+a^2}(\chi_{\tau_1}\psi)^{(a\omega)}_{m\ell}$, and~$H=\frac{\Delta}{(r^2+a^2)^{3/2}}(\rho^2F)^{(a\omega)}_{m\ell}$. We have the representation formula
	\begin{equation}\label{eq: prop: subsec: summing in the redshift estimate, prop 1, eq 1}
		u(r^\star)=\mathcal{W}^{-1}\cdot\left(u_{\bar{\mathcal{H}^+}}(r^\star)\int_{-\infty}^{r^\star} u_{\mathcal{H}^+}(x^\star) H(x^\star) dx^\star+u_{\mathcal{H}^+}(r^\star)\int_{r^\star}^{+\infty} u_{\bar{\mathcal{H}^+}}(x^\star)H(x^\star)  \right),
	\end{equation}
	from Lemma~\ref{lem: subsec: sec: Carter's separation, wronskian, lem 1} which recall holds for~$\omega\in\mathbb{R},m\in \mathbb{Z},\ell\in \mathbb{Z}_{\geq |m|}$ such that~$|\mathcal{W}^{-1}(\omega,m,\ell)|<\infty$ and~$\omega\neq \omega_+ m,\omega\neq\bar{\omega}_+ m$. Now, in view of the continuity of the Wronskian, see Proposition~\ref{prop: subsec: sec: Carter's separation, wronskian, prop 1}, we have that for any~$\mathcal{C}>0$ there exists a~$C(a,M,l,\mu_{KG},\mathcal{C})$ such that for~$(\omega,m,\ell)\in \mathcal{F}_{\mathcal{SF},\mathcal{C}}$~(which is a compact set) we have that 
	\begin{equation}\label{eq: prop: subsec: summing in the redshift estimate, prop 1, eq 2}
		|\mathcal{W}^{-1}| \leq C(a,M,l,\mu_{KG},\mathcal{C}). 
	\end{equation}
	For~$\mathcal{F}_{\mathcal{SF},\mathcal{C}}$ see~\eqref{eq: subsec: sec: carter separation, subsec 2, eq 1}. By using~\eqref{eq: prop: subsec: summing in the redshift estimate, prop 1, eq 2} then the representation formula~\eqref{eq: prop: subsec: summing in the redshift estimate, prop 1, eq 1}, in view of Parseval identities, see Section~\ref{subsec: fourier properties}, implies the desired result. 
\end{proof}

\section{The detailed version of Theorem~\ref{main theorem 1}}\label{sec: main theorems}

In this Section we present our main theorem.

\subsection{Boundedness and Morawetz estimate: Theorem~\ref{main theorem 1}}\label{subsec: sec: main theorems, subsec 1}

First, we need the following definition

\begin{definition}\label{def: sec: main theorems, def 1}
	Let~$l>0$ and~$\mu^2_{KG}\geq 0$. We define the following set
	\begin{equation}\label{eq: sec: main theorems, eq 2}
		\begin{aligned}
			\mathcal{MS}_{l,\mu_{KG}}
		\end{aligned}
	\end{equation}
	to be the connected component of~$\widetilde{\mathcal{MS}}_{l,\mu_{KG}}$ containing~$\mathcal{B}_{0,l}$ with the standard Euclidean topology, where for the sets~$\widetilde{\mathcal{MS}}_{l,\mu_{KG}},\mathcal{B}_{0,l}$ see Definition~\ref{def: subsec: sec: main theorems, subsec 1, def 1}. 
\end{definition}

The following Proposition follows immediately from Remark~\ref{rem: subsec: sec: main theorems, subsec 1, rem 0}

\begin{proposition}\label{prop: sec: main theorems, prop 1}
	Let~$l>0$ and~$\mu^2_{KG}\geq 0$. Moreover, let~$(a,M)\in \mathcal{MS}_{l,\mu_{KG}}$. Then, we have that~$(-a,M)\in \mathcal{MS}_{l,\mu_{KG}}$. 
\end{proposition}

The boundedness and Morawetz estimates read as follows

\begin{customTheorem}{1}[detailed version]\label{main theorem 1}
Let~$l>0$,~$\mu^2_{\textit{KG}}\geq 0$ and 
\begin{equation*}
	\begin{aligned}
		(a,M)\in \mathcal{MS}_{l,\mu_{KG}},
	\end{aligned}
\end{equation*}
where for~$\mathcal{MS}_{l,\mu_{KG}}$ see~\eqref{eq: sec: main theorems, eq 2}.

Then, there exist constants
\begin{equation}
	C=C(a,M,l,\mu^2_{\textit{KG}})>0,\qquad R^-(a,M,l,\mu_{KG})>0,\qquad R^+(a,M,l,\mu_{KG})>0
\end{equation}
where~$r_+<R^-<R^+<\bar{r}_+$, such that if~$\psi$ satisfies the Klein--Gordon equation~\eqref{eq: kleingordon} in~$D(\tau_1,\infty)$ then we have the following estimates
\begin{equation}\label{eq: main theorem 1, eq 1}
    \begin{aligned}
        &   \int\int_{D(\tau_1,\tau_2)} \mu^2_{\textit{KG}}|\psi|^2+ |Z^\star\psi|^2+|\slashed{\nabla}\mathcal{P}_{trap}(\chi_{\tau_1}\psi)|^2+|\partial_{t^\star}\mathcal{P}_{trap}(\chi_{\tau_1}\psi)|^2+\sum_m \left( 1_{|m|>0}|\mathcal{F}_{m}\psi|^2\right)\\
        &	\qquad\qquad \leq C \int_{\{t^\star=\tau_1\}} J^n_{\mu}[\psi]n^{\mu} , 
    \end{aligned}    
\end{equation}
\begin{equation}\label{eq: main theorem 1, eq 2}
	\int_{\mathcal{H}^+\cap D(\tau_1,\tau_2)}J^n_{\mu}[\psi] n^\mu+\int_{\bar{\mathcal{H}}^+\cap D(\tau_1,\tau_2)}J^n_{\mu}[\psi] n^\mu+\int_{\{t^\star=\tau_2\}}J^n_{\mu}[\psi] n^\mu  \leq C\int_{\{t^\star=\tau_1\}}J^n_{\mu}[\psi]n^{\mu}, 
\end{equation}
\begin{equation}\label{eq: main theorem 1, eq 3}
 \int_{\{t^\star=\tau_2\}} J^n_\mu[\psi]+|\psi|^2\leq C\int_{\{t^\star=\tau_1\}}J^n_\mu[\psi]n^\mu +|\psi|^2
\end{equation}
for all~$0\leq \tau_1\leq \tau_2$, where for~$\mathcal{P}_{trap}$ see~\eqref{eq: sec: main theorems, eq 1.1}, for~$r_{trap}\in (R^-,R^+)\cup \{0\}$ see Theorem~\ref{thm: sec: proofs of the main theorems, thm 3}, and for~$\mathcal{F}_{m}$ see~\eqref{eq: subsec: sec: carter separation, radial, eq 1}.  
\end{customTheorem}

\begin{proof}
See Section~\ref{sec: proof of Theorem 1}. 
\end{proof}

\begin{remark}\label{rem: main theorem 1, rem 1}
	
It is easy to see that in Theorem~\ref{main theorem 1} we can replace the cut-off~$\chi_{\tau_1}$ with a general cut-off ~$\chi:\mathcal{M}\rightarrow \mathbb{R}$ with the property
	\begin{equation}
		\supp\chi \cap \{t^\star< \tau_1\}=\emptyset. 
	\end{equation}
	Specifically, in our companion~\cite{mavrogiannis2} we will use the cut-off~$\chi=\eta^{(T)}_{\tau_1}\chi^2_{\tau_1,\tau_1+T^2}$ where
	\begin{equation}
		\chi^{(T)}_{\tau_1,\tau_2}(t^\star) \:=\: 
		\begin{cases}
			\text{1,} &\quad\{\tau_1+T\leq t^\star\leq \tau_2-T\}\\
			\text{0,} &\quad\{t^\star\leq \tau_1\}\cup\{t^\star\geq \tau_2\}, \\
		\end{cases}
		\qquad \eta^{(T)}_{\tau_1}(t^\star) \:=\: 
		\begin{cases}
			\text{1,} &\quad\{t^\star\geq T+\tau_1\}\\
			\text{0,} &\quad\{t^\star\leq \tau_1\},\\
		\end{cases}
	\end{equation}
	for some~$T>0$. For notational ease Theorem~\ref{main theorem 1} is proved with the cut-off~$\chi_{\tau_1}$.
\end{remark}

\begin{remark}
Given~$\mu_{KG}^2\geq 0$, suppose that mode stability on the real axis holds for all subextremal black hole parameters, i.e.~$\widetilde{\mathcal{MS}}_{l,\mu_{KG}}=\mathcal{B}_l$. Then, since~$\mathcal{B}_l$ is connected, see Lemma~\ref{lem: subsec: delta polynomial, lem 1}, we have~$\mathcal{MS}_{l,\mu_{KG}}=\mathcal{B}_l$ and thus Theorem~\ref{main theorem 1} would hold for all subextremal parameters.
\end{remark}

 \begin{remark}
	For fixed~$(a,M,l)$, the constants in the RHS of the estimates~\eqref{eq: main theorem 1, eq 1},~\eqref{eq: main theorem 1, eq 2}, blow up in the limit~$\mu_{KG}\rightarrow 0$. Note, however, that $\mu_{KG}=0$ is indeed allowed in the theorem, i.e. the constant for $\mu_{KG}=0$ is finite. In contrast, the constants on the right hand side of the boundedness estimate~\eqref{eq: main theorem 1, eq 3} can be chosen uniformly for sufficiently small~$\mu_{KG}^2>0$. Finally, it is clear from the steps of the proof of our Theorem~\ref{main theorem 1} that if we allow zeroth order terms~$|\psi|^2$ ιn the hypersurface terms of the right hand side of estimates~\eqref{eq: main theorem 1, eq 1},~\eqref{eq: main theorem 1, eq 2}, then the constants of these estimates can also be chosen uniformly for sufficiently small Klein--Gordon masses~$\mu_{KG}^2>0$. 
\end{remark}

\begin{remark}
As discussed in the introduction, the set~$\mathcal{MS}_{l,\mu_{KG}}$ includes the very slowly rotating black hole parameters~$|a|\ll M,l$, where the smallness depends on~$\mu^2_{KG}$. Moreover, the set~$\mathcal{MS}_{l,\mu_{KG}}$ includes arbitrary black hole parameters~$|a|<M\ll l$, and~$\mu^2_{KG}=0$, where the constant implicit in~$\ll$ depends on the difference to extremality~$M-|a|$, in view of the recent work of Hintz~\cite{hintz2021mode}. 
\end{remark}

\begin{remark}\label{rem: subsec: sec: main theorems, subsec 1, rem 1}
	The LHS of~\eqref{eq: main theorem 1, eq 1}, in conjunction with the local estimates of Lemma~\ref{lem: subsec: sec: the main proposition, subsec 0, lem 1}, also controls the spacetime integral
	\begin{equation}\label{eq: rem: subsec: sec: main theorems, subsec 1, rem 1, eq 1}
		\int\int_{D(\tau_1,\tau_2)} \mu^2_{\textit{KG}}|\psi|^2+ |Z^\star\psi|^2+\zeta_{\textit{trap}}(r)\big( |\slashed{\nabla}\psi|^{2} + (\partial_{t^\star}\psi)^{2} \big)   \leq C \int_{\{t^\star=\tau_1\}} J^n_{\mu}[\psi]n^{\mu} , 
	\end{equation}
	where
	\begin{equation}	\zeta_{\textit{trap}}(r)=\left(1-1_{[R^-,R^+]}\right)\left(1-\frac{\frac{R^-+R^+}{2}}{r}\right)^2,
	\end{equation}
	and for~$R^\pm$ see Corollary~\ref{cor: subsec: summing in the redshift estimate, cor 1}.

	Note that on a Schwarzschild--de~Sitter background~$a=0$ we have
	\begin{equation}
		R^-=R^+=r_{\Delta,\textit{frac}}=3M,\qquad \zeta_{\textit{trap}}(r)=  \left(1-\frac{3M}{r}\right)^2.
	\end{equation}
	
	Moreover, if the solution of the Klein--Gordon equation~\eqref{eq: kleingordon} is axisymmetric~$\partial_{\varphi}\psi=0$, then~\eqref{eq: rem: subsec: sec: main theorems, subsec 1, rem 1, eq 1} holds for~$(a,M)\in\mathcal{B}_l$, where we have~$R^-=R^+$.
\end{remark}

By also utilizing the higher order redshift estimate of Proposition~\ref{prop: redshift estimate}, applied in the extended domain~$D_\delta$, we obtain the following higher order Corollary, that also treats the inhomogeneous case

\begin{customCorollary}{1.1}\label{cor: main theorem 1, cor 1}
Let the assumptions of Theorem~\ref{main theorem 1} be satisfied, where~$\psi$ is now a solution of the inhomogeneous Klein--Gordon equation~$\Box_{g_{a,M,l}}\psi-\mu^2_{KG}\psi=F$, where~$F$ is sufficiently regular. Then, for any~$j\geq 1$ there exists a sufficiently small~$\delta(a,M,l,\mu_{KG},j)$ such that the following hold. Let~$\{t^\star=\tau^\prime\}$ be the extended spacelike hypersurfaces, defined in the extended domain~$D_\delta(\tau_1,\infty)$, see Definition~\ref{def: causal domain, def 2}. Then, there exists a constant
\begin{equation}
C=C(j,a,M,l,\mu^2_{\textit{KG}})>0,
\end{equation}
such that we obtain
\begin{equation}\label{eq: cor: main theorem 1, cor 1, eq 1}
    \begin{aligned}
        &	\int\int_{D(\tau_1,\infty)}  \sum_{1\le i_1+i_2\le j}|\nabb^{i_1}(\partial_{t^\star})^{i_2}\mathcal{P}_{trap}(\chi\psi)|^2+\int\int_{D_\delta(\tau_1,\infty)\setminus D(\tau_1,\infty)}\sum_{1\leq i_1+i_2+i_3\leq j} |\slashed{\nabla}^{i_1}\partial_t^{i_2}(Z^\star)^{i_3}\psi|^2 \\
        & \qquad\qquad + \sum_{1\le i_1+i_2+i_3\le j-1}
        \left(|\nabb^{i_1}(\partial_{t^\star})^{i_2}(Z^\star)^{i_3+1}\psi|^2+|\nabb^{i_1}(\partial_{t^\star})^{i_2}(Z^\star)^{i_3}\psi|^2\right)  \\
        &	\qquad\qquad\qquad \leq C\int_{\{t^\star=\tau_1\}}  \sum_{i_1+i_2+i_3}|\slashed{\nabla}^{i_1}\partial_t^{i_2}(Z^\star)^{i_3}\psi|^2+\sum_{1\leq i_1\leq j}\int\int_{D_\delta(\tau_1,\infty)} |\partial_t^{i_1}\psi \cdot \partial_t^{i_1-1}F|+ |\partial_{\varphi}^{i_1}\psi \cdot \partial_{\varphi}^{i_1-1}F| +|N^{i_1-1}F|^2,
    \end{aligned}
\end{equation}
\begin{equation}\label{eq: cor: main theorem 1, cor 1, eq 2}
	\begin{aligned}
		&	\int_{\mathcal{H}^+\cap D_\delta(\tau_1,\infty)} \sum_{0\le i \le j-1}{J}^N_\mu[N^{i}\psi]n^\mu_{\mathcal{H}^+}+\int_{\bar{\mathcal{H}}^+\cap D_\delta(\tau_1,\infty)}\sum_{0\le i \le j-1}{J}^{N}_\mu[N^{i}\psi]n^\mu_{\bar{\mathcal{H}}^+}\\
		&	\qquad\qquad\leq  C\int_{\{t^\star=\tau_1\}}  \sum_{1\leq i_1+i_2+i_3}|\slashed{\nabla}^{i_1}\partial^{i_2}(Z^\star)^{i_3}\psi|^2+\sum_{1\leq i_1\leq j}\int\int_{D_\delta(\tau_1,\infty)} |\partial_t^{i_1}\psi \cdot \partial_t^{i_1-1}F|+ |\partial_{\varphi}^{i_1}\psi \cdot \partial_{\varphi}^{i_1-1}F| +|N^{i_1-1}F|^2,
	\end{aligned}
\end{equation}
\begin{equation}\label{eq: cor: main theorem 1, cor 1, eq 3}
	\begin{aligned}
		&	\int_{\{t^\star=\tau_2\}}  \sum_{1\leq i_1+i_2+i_3}|\slashed{\nabla}^{i_1}\partial_t^{i_2}(Z^\star)^{i_3}\psi|^2\\
		&	\qquad\qquad\leq C\int_{\{t^\star=\tau_1\}}  \sum_{1\leq i_1+i_2+i_3}|\slashed{\nabla}^{i_1}\partial_t^{i_2}(Z^\star)^{i_3}\psi|^2\\
		&	\qquad\qquad\qquad+\sum_{1\leq i_1\leq j}\int\int_{D_\delta(\tau_1,\infty)} \epsilon^\prime|\partial_t^{i_1}\psi|^2 +(\epsilon^\prime)^{-1} |\partial_t^{i_1-1}F|^2+ \epsilon^\prime|\partial_{\varphi}^{i_1}\psi|^2 + (\epsilon^\prime)^{-1}|\partial_{\varphi}^{i_1-1}F| +|N^{i_1-1}F|^2, 
	\end{aligned}
\end{equation}
for any~$1\leq1+\tau_1\leq \tau_2$ and for any~$0\leq \epsilon^\prime<1$, for some sufficiently small~$\delta(j,a,M,l)$. Moreover, for the redshift vector field~$N$ see Proposition~\ref{prop: redshift on the event horizon}. 
\end{customCorollary}
\begin{proof}	
	We first discuss the estimates~\eqref{eq: cor: main theorem 1, cor 1, eq 1},~\eqref{eq: cor: main theorem 1, cor 1, eq 2}. The case~$j=1$ follows immediately from the consideration of the proof of Theorem~\ref{main theorem 1}, where we now keep track of the inhomogeneity. Specifically, note that Proposition~\ref{prop: subsec: summing in the redshift estimate, prop 1} also holds for the inhomogeneous Klein--Gordon equation, with the additional term~$\int\int_{D_\delta(\tau_1,\infty)}|F|^2$ on the RHS of~\eqref{eq: prop: subsec: summing in the redshift estimate, prop 1, eq 0}. To prove the cases~$j>1$ we proceed as follows. We commute the inhomogeneous Klein--Gordon equation~$\Box_{g_{a,M,l}}\psi-\mu^2_{KG}\psi=F$ with~$\partial_t,\partial_{\varphi},N$ appropriately many times and use elliptic estimates~(see for example the lecture notes~\cite{DR5}), where we also use the redshift Proposition~\ref{prop: redshift estimate} to control derivatives near the horizons.

Now, we discuss the boundedness estimate~\eqref{eq: cor: main theorem 1, cor 1, eq 3}. Again, the~$j=1$ case follows from the considerations of the proof of Theorem~\ref{main theorem 1}, where we keep track of the inhomogeneities. Specifically, we use Young's inequalities on the error terms on the RHS of estimate~\eqref{eq: cor: main theorem 1, cor 1, eq 1} and folllow the steps of Section~\ref{sec: proof of boundedness}, where we will now have appropriate inhomogeneity on the RHS of the estimates of Section~\ref{sec: proof of boundedness}. We can obtain the higher order estimates by commuting with~$\partial_t,\partial_{\varphi},N$. 
\end{proof}

\subsection{Logic of the proof of Theorem~\ref{main theorem 1}}

We here present the steps of the proof of Theorem~\ref{main theorem 1}, in an order that further clarifies the logic of the proof.

The first step is to prove the desired Morawetz estimate for sufficiently integrable solutions of the inhomogeneous Klein--Gordon equation, valid for all subextremal Kerr--de~Sitter black hole parameters~$(a,M)\in \mathcal{B}_l$, but with an additional horizon term on the RHS, see Theorem~\ref{thm: subsec: summing in the redshift estimate, thm -1} in Section~\ref{sec: the main theorem, no cases}. To prove Corollary~\ref{cor: subsec: summing in the redshift estimate, cor 1} we use the fixed frequency estimate of Theorem~\ref{thm: sec: proofs of the main theorems, thm 3}, which is proved in Section~\ref{sec: proof: prop: sec: proofs of the main theorems}. 

In the case where moreover~$(a,M)\in\widetilde{\mathcal{MS}}_{l,\mu_{KG}}$, we use the quantitative mode stability in the form of the flux bound of Proposition~\ref{prop: subsec: summing in the redshift estimate, prop 1}, to remove the additional horizon term on the RHS, and we obtain Corollary~\ref{cor: subsec: summing in the redshift estimate, cor 1} for `future integrable' solutions of the homogeneous Klein--Gordon equation~\eqref{eq: kleingordon}.

In the~$(a,M)\in\mathcal{MS}_{l,\mu_{KG}}$ we use a continuity argument, see Section~\ref{sec: continuity argument}, to prove that all solutions of the Klein--Gordon equation~\eqref{eq: kleingordon} are in fact `future integrable', see Theorem~\ref{prop: sec: continuity argument, prop 1} in Section~\ref{sec: continuity argument}.

In Section~\ref{sec: proof of Theorem 1} we combine Corollary~\ref{cor: subsec: summing in the redshift estimate, cor 1} and Theorem~\ref{prop: sec: continuity argument, prop 1} to conclude the proof of the Morawetz estimate~\eqref{eq: main theorem 1, eq 1} of Theorem~\ref{main theorem 1}, for~$(a,M)\in\mathcal{MS}_{l,\mu_{KG}}$.

The remaining boundedness estimates~\eqref{eq: main theorem 1, eq 2},~\eqref{eq: main theorem 1, eq 3} are proved in Section~\ref{sec: proof of boundedness}, where note we follow a similar approach to~\cite{DR2}.

For the axisymmetric case of Theorem~\ref{main theorem 1} we do not need a continuity argument, and in fact the proofs are significantly simpler, see the proof in Section~\ref{sec: proof of main theorem in axisymmetry}. Specifically, by using the boundedness estimates for axisymmetric solutions of Proposition~\ref{prop:redshift vs superradiance estimate, for axisymmetric solutions}, we conclude that axisymmetric solutions of the Klein--Gordon equation are in fact future integrable for all~$(a,M)\in\mathcal{B}_l$ and therefore we can use directly the Morawetz estimate of Theorem~\ref{thm: subsec: summing in the redshift estimate, thm -1}, which concludes the proof.

\section{A Morawetz estimate for sufficiently integrable functions}\label{sec: the main theorem, no cases}

The main result of this Section is Theorem~\ref{thm: subsec: summing in the redshift estimate, thm -1}. To obtain Theorem~\ref{thm: subsec: summing in the redshift estimate, thm -1} we will sum the fixed frequency ode estimate of Theorem~\ref{thm: sec: proofs of the main theorems, thm 3}, by using Proposition~\ref{prop: Carters separation, radial part} and Proposition~\ref{prop: subsec: fourier properties, prop 1}, and we will also use the redshift estimates of Section~\ref{sec: redshift and superradiance}.~(The proof of Theorem~\ref{thm: sec: proofs of the main theorems, thm 3} is deferred to Section~\ref{sec: proof: prop: sec: proofs of the main theorems}.)

\subsection{Notation on constants}\label{subsec: sec: proofs of the main theorems, subsec 1}

From this point onwards we use the constants
\begin{equation}
	b(a,M,l,\mu_{\textit{KG}})>0,\qquad B(a,M,l,\mu_{\textit{KG}})>0,
\end{equation}
to denote respectively potentially small and potentially large positive constants both depending only on the black hole parameters~$a,M,l$, the Klein--Gordon mass~$\mu^2_{KG}$. Specifically, we note that for certain inequalities the constants~$b,B$ will depend on the sufficiently large and sufficiently small parameters~$r^\star_{\pm \infty}$, which will be determined later. Furthermore, for certain inequalities the constants~$b,B$ will depend on a sufficiently large~$\mathcal{C}>0$ constant. This dependence will be clearly denoted when used.

Informally, we record the following algebra of constants 
\begin{equation}
	B+B=B,\qquad B\cdot B=B, \qquad b+b=b,\qquad b\cdot b=b.
\end{equation}

Moreover, let~$f,g\geq 0$. When we use~$f\sim g$ we mean~$b g\leq  f\leq B g$.

\subsection{The main Theorem~\ref{thm: subsec: summing in the redshift estimate, thm -1}}\label{subsec: sec: proofs of the main theorems, subsec 3}

We have the following Theorem

\begin{theorem}\label{thm: subsec: summing in the redshift estimate, thm -1}
	Let~$l>0$ and~$(a,M)\in\mathcal{B}_l$ and~$\mu^2_{\textit{KG}}\geq 0$. Let~$\mathcal{C}>0$ be sufficiently large. Let~$\psi$ be a sufficiently integrable and outgoing function, see Definitions~\ref{def: sec: carter separation, def 1},~\ref{def: subsec: sec: carter separation, subsec 0, def 1}, satisfying the inhomogeneous Klein--Gordon equation~$\Box_{g_{a,M,l}}\psi-\mu^2_{KG}\psi=F$.
	
	Then, we have the following
	\begin{equation}\label{eq: thm: subsec: summing in the redshift estimate, thm -1, eq 1}
		\begin{aligned}
			\int\int_{D(-\infty,+\infty)} & \mu^2_{\textit{KG}}
			|\psi|^2+ |\partial_{r^\star}\psi|^2+ |Z^\star\psi|^2 +|\partial_t\mathcal{P}_{\textit{trap}}[\psi]|^2+|\slashed{\nabla}\mathcal{P}_{\textit{trap}}[\psi]|^2 +\sum_m 1_{|m|>0} |\mathcal{F}_{m}(\psi)|^2\\
			&	\quad\leq   B(\mathcal{C}) \int_{\mathcal{H}^+}\left|\mathcal{P}_{\mathcal{SF},\mathcal{C}} (\psi)\right|^2+	B(\mathcal{C})\int\int_{D(-\infty,+\infty)} |\partial_t\psi \cdot F| +|\partial_{\varphi}\psi \cdot F| +|F|^2,
		\end{aligned}
	\end{equation}
	where for~$\mathcal{P}_{\textit{trap}},\mathcal{P}_{\mathcal{SF},\mathcal{C}}$ see~\eqref{eq: sec: main theorems, eq 1.1},~\eqref{eq: sec: main theorems, eq 1} respectively, and for~$\mathcal{F}_{m}$ see~\eqref{eq: subsec: sec: carter separation, radial, eq 1}.
\end{theorem}

\begin{proof}
See Section~\ref{subsec: summing in the redshift estimate}. We will need first the fixed frequency ode estimate of Theorem~\ref{thm: sec: proofs of the main theorems, thm 3}, to be proven in Section~\ref{subsec: sec: proofs of the main theorems, subsec 4} below. 
\end{proof}

\subsection{The main fixed frequency ode estimate}\label{subsec: sec: proofs of the main theorems, subsec 4}

In the proof of our main Theorem~\ref{thm: subsec: summing in the redshift estimate, thm -1} we will need to appeal to the following fixed frequency integrated estimate for smooth solutions of Carter's radial ode that satisfy outgoing boundary conditions.

\begin{customTheorem}{8.3}\label{thm: sec: proofs of the main theorems, thm 3}
	Let~$l>0$~$(a,M)\in\mathcal{B}_l$ and~$\mu^2_{\textit{KG}}\geq 0$. Then, there exist constants
	\begin{equation}
		\begin{aligned}
			E=E(a,M,l,\mu_{\textit{KG}})>0,\qquad r_+<R^-<R^+<\bar{r}_+,
		\end{aligned} 
	\end{equation} 
	and a frequency dependent~$r$-value
	\begin{equation}
		r_{trap}(\omega,m,\ell)\in (R^-,R^+)\cup \{0\},
	\end{equation}
	 such that for all~$r_{-\infty}$ sufficiently close to~$r_+$ and for all~$r_{+\infty}$ sufficiently close to~$\bar{r}_+$, and for any sufficiently large~$\mathcal{C}>0$ there exists a constant~$B(\mathcal{C},r_{\pm\infty})$, as in Section~\ref{subsec: sec: proofs of the main theorems, subsec 1}, such that for smooth solutions
	\begin{equation}
		u(\omega,m,\ell,r),
	\end{equation}
	of Carter's inhomogeneous radial ode, see~\eqref{eq: ode from carter's separation}, that moreover satisfy the outgoing boundary conditions~\eqref{eq: lem: sec carters separation, subsec boundary behaviour of u, boundary beh. of u, eq 3} we have the following integrated estimate
	\begin{equation}\label{eq: cor: sec: proofs of the main theorems, cor 1, eq 1}
	\begin{aligned}
		&	\int_{r^\star_{-\infty}}^{r^\star_{\infty}} \left( |\Psi^\prime|^2+1_{\{m>0\}}\Delta |\Psi|^2+1_{\{|m|>0\}}|u^\prime|^2 + (\omega^2+\tilde{\lambda})\left(1-\frac{r_{trap}}{r}\right)^2|u|^2 +\left(\mu^2_{KG}+1_{\{|m|>0\}}\right)|u|^2 \right)  \\
		&	\qquad \leq  B(\mathcal{C},r_{\pm \infty})\cdot 1_{\{m>0\}}\int_{-\infty}^{r^\star_{-\infty}}  |u^\prime H|+|\omega\cdot uH|+|m\cdot u H|+B(\mathcal{C},r_{\pm \infty})\cdot 1_{\{m>0\}} \int_{r^\star_\infty}^\infty  |u^\prime H|+|\omega\cdot uH|+|m\cdot u H|\\
		&	\qquad\qquad + B(\mathcal{C},r_{\pm \infty})\int_{\mathbb{R}}dr^\star \frac{1}{\Delta} |H|^2\\
		&	\qquad\qquad +E \int_{\mathbb{R}} \left(\left(\omega-\frac{am\Xi}{r_+^2+a^2}\right)\Im (\bar{u}H)+\left(\omega-\frac{am\Xi}{r_+^2+a^2}\right)\Im(\bar{u}H)\right)dr^\star\\
		&	\qquad\qquad +B(\mathcal{C},r_{\pm \infty}) 1_{\mathcal{F}_{\mathcal{SF},\mathcal{C}}} |\omega-\omega_+m||\omega-\bar{\omega}_+ m||u|^2(-\infty),
	\end{aligned}
\end{equation}
where for~$\mathcal{F}_{\mathcal{SF},\mathcal{C}}$ see~\eqref{eq: subsec: sec: carter separation, subsec 2, eq 1}. 
\end{customTheorem}
\begin{proof}
	Immediate by inspecting Theorem~\ref{thm: sec: proofs of the main theorems}. 
\end{proof}

\begin{remark}\label{rem: thm: sec: proofs of the main theorems, rem 1}
	Note that the term in~\eqref{eq: cor: sec: proofs of the main theorems, cor 1, eq 1} that multiplies~$E$ in general does not have a sign. Therefore, in what follows we cannot replace~$E$ by a larger constant. 
\end{remark}

\begin{remark}\label{rem: thm: sec: proofs of the main theorems, rem 2}
	We only control the terms~$\int_{-\infty}^{\infty} 1_{\{m>0\}}\Delta |\Psi|^2$ on the left hand side of~\eqref{eq: cor: sec: proofs of the main theorems, cor 1, eq 1}. The lack of control of this term comes from the lack of control of the respective term in Proposition~\ref{prop: energy estimate for the bounded stationary frequencies, 1}, for the case~$\mu^2_{KG}=0$. 
\end{remark}

\begin{remark}
	We may replace the superradiant set~$\mathcal{SF}$ in the definition of~$\mathcal{F}_{\mathcal{SF},\mathcal{C}}$ on the RHS of~\eqref{eq: cor: sec: proofs of the main theorems, cor 1, eq 1} with the smaller set found in~\cite{casals2021hidden}, for which real frequency mode stability is still not known. 
\end{remark}

We obtain immediately the following Corollary 
\begin{customCorollary}{8.1}\label{lem: sec: proofs of the main theorems}
 	Let~$l>0$ and~$(a,M)\in\mathcal{B}_l$, ~$\mu^2_{\textit{KG}}\geq 0$. Then, for any~$	(\omega,m)\in (\mathcal{SF})^c$ and~$\omega\neq \omega_+ m,\omega\neq \bar{\omega}_+ m$ and any~$\ell\geq |m|$ we have that 
 	\begin{equation}\label{eq: lem: sec: proofs of the main theorems, eq 1}
 		|\mathcal{W}^{-1}(a,M,l,\mu_{KG},\omega,m,\ell)|<+\infty.
 	\end{equation}
 	Moreover, for 
 	\begin{equation}\label{eq: lem: sec: proofs of the main theorems, eq 2}
 		(\omega,m,\ell)\in\left(\{(\omega,m,\ell):~\mathcal{C}^{-1}\leq |\omega|\leq \mathcal{C}\}\cap \{(\omega,m,\ell):~\tilde{\lambda}\leq\mathcal{C}\}\right)^c\setminus \{\omega=\omega_+ m,\omega=\bar{\omega}_+ m\},
 	\end{equation}
 	where~$\mathcal{C}>0$ is sufficiently large, we have that~\eqref{eq: lem: sec: proofs of the main theorems, eq 1} again holds.

 	Furthermore, for~$\mathcal{C}>0$ sufficiently large then for
 	\begin{equation}\label{eq: lem: sec: proofs of the main theorems, eq 3}
 		(\omega,m,\ell)\in 	 \{0<|\omega-\omega_+m|\leq \mathcal{C}^{-1}\}\cup \{0<|\omega-\bar{\omega}_+ m|\leq \mathcal{C}^{-1}\}
 	\end{equation}
 	we again have that~\eqref{eq: lem: sec: proofs of the main theorems, eq 1} holds. 
\end{customCorollary}
\begin{proof}
Let~$(\omega,m)\in (\mathcal{SF})^c$ and~$\omega\neq \omega_+m,\omega\neq \bar{\omega}_+ m$. Then, in view of the energy estimate of Proposition~\ref{prop: subsec: energy identity, prop 1} we conclude the desired result~\eqref{eq: lem: sec: proofs of the main theorems, eq 1}.

Now, let~$(\omega,m,\ell)\in\left(\{(\omega,m,\ell):~\mathcal{C}^{-1}\leq |\omega|\leq \mathcal{C}\}\cap \{(\omega,m,\ell):~\tilde{\lambda}\leq\mathcal{C}\}\right)^c\setminus \{\omega=\omega_+ m,\omega=\bar{\omega}_+ m\}$, for any sufficiently large~$\mathcal{C}>0$. Then, in view of Theorem~\ref{thm: sec: proofs of the main theorems, thm 3} we note that 
	\begin{equation}
			\begin{aligned}
				 \int_{r^\star_{-\infty}}^{r^\star_\infty} \left( |u^\prime|^2+|u|^2+\left(1-\frac{r_{\textit{trap}}}{r}\right)^2\left(\tilde{\lambda}+\omega^2\right)|u|^2\right)  =0,
			\end{aligned}
	\end{equation}
	for~$H=0$, which immediately concludes that~$u\equiv 0$ for the desired frequencies, which in turn concludes~\eqref{eq: lem: sec: proofs of the main theorems, eq 1}.

	Now we study the frequencies
	\begin{equation}\label{eq: proof prop: subsec: sec: continuity argument, subsec 1, prop 1, eq 2.1}
	  \{0<|\omega-\omega_+m|\leq \mathcal{C}^{-1}\}\cup \{0<|\omega-\bar{\omega}_+ m|\leq \mathcal{C}^{-1}\}.
	\end{equation}	
For the frequencies~\eqref{eq: proof prop: subsec: sec: continuity argument, subsec 1, prop 1, eq 2.1} we can use Theorem~\ref{thm: sec: proofs of the main theorems, thm 3} for solutions~$u$ of the homogeneous Carter's radial ode that satisfy the boundary conditions~\eqref{eq: lem: sec carters separation, subsec boundary behaviour of u, boundary beh. of u, eq 3} and obtain the result~\eqref{eq: cor: sec: proofs of the main theorems, cor 1, eq 1} without the last boundary term on the RHS~(since~$\mathcal{F}_{\mathcal{SF},\mathcal{C}}\cap$\eqref{eq: proof prop: subsec: sec: continuity argument, subsec 1, prop 1, eq 2.1}~$=\emptyset$) and without the inhomogeneities. Therefore, we conclude that for the frequencies~\eqref{eq: proof prop: subsec: sec: continuity argument, subsec 1, prop 1, eq 2.1} we have~$|\mathcal{W}^{-1}|<\infty$.
\end{proof}

\subsection{Applying Parseval and summing in the redshift estimates: Proof of Theorem~\ref{thm: subsec: summing in the redshift estimate, thm -1}}\label{subsec: summing in the redshift estimate}

We are now ready to prove Theorem~\ref{thm: subsec: summing in the redshift estimate, thm -1}.

\begin{proof}[\textbf{Proof of Theorem~\ref{thm: subsec: summing in the redshift estimate, thm -1}}]

We apply Carter's separation of variables to the inhomogeneous pde
\begin{equation}
	\Box_{g_{a,M,l}}\psi -\mu^2_{KG} \psi =F.
\end{equation}

In view of the assumption of the present Theorem, specifically sufficient integrability and outgoingness of~$\psi$, then Proposition~\ref{prop: Carters separation, radial part} guarantees that for almost all~$\omega$ the function~$u^{(a\omega)}_{m\ell}(r)$ is smooth for all~$m,\ell$ and~$r\in (r_+,\bar{r}_+)$, and moreover the outgoing boundary conditions~\eqref{eq: lem: sec carters separation, subsec boundary behaviour of u, boundary beh. of u, eq 3} are satisfied. Therefore, we may apply the fixed frequency integrated result of Theorem~\ref{thm: sec: proofs of the main theorems, thm 3}, and then integrate with~$\int_{-\infty}^\infty d\omega \sum_{m\ell}$, to obtain 
\begin{equation}\label{eq: proof prop: subsec: summing in the redshift estimate, prop 0, eq -1}
	\begin{aligned}
		&	\int_{-\infty}^\infty d\omega \sum_{m\ell}\int_{r^\star_{-\infty}}^{r^\star_{\infty}} \left( |\Psi^\prime|^2 + (\omega^2+\tilde{\lambda})\left(1-\frac{r_{trap}}{r}\right)^2|u|^2 +\left(\mu^2_{KG}+1_{\{|m|>0\}}\right)|u|^2 \right) dr^\star \\
		&	\qquad \leq  B(\mathcal{C},r_{\pm \infty}) \int_{-\infty}^\infty d\omega \sum_{m\ell}\Big(\int_{-\infty}^{r^\star_{-\infty}} 1_{\{|m|>0\}} (|u^\prime H|+|\omega\cdot uH|+|m\cdot u H|)dr^\star\\
		&	\qquad\qquad\qquad\qquad\qquad\qquad+ \int_{r^\star_\infty}^\infty 1_{\{|m|>0\}}( |u^\prime H|+|\omega\cdot uH|+|m\cdot u H|)\Big)dr^\star\\
		&	\qquad\qquad + B(\mathcal{C},r_{\pm \infty}) \int_{\mathbb{R}} d\omega \sum_{m,\ell}\int_{\mathbb{R}} \frac{1}{\Delta} |H|^2dr^\star\\
		&	\qquad\qquad +E \int_{-\infty}^\infty d\omega \sum_{m\ell} \int_{\mathbb{R}} \left(\left(\omega-\frac{am\Xi}{r_+^2+a^2}\right)\Im (\bar{u}H)+\left(\omega-\frac{am\Xi}{r_+^2+a^2}\right)\Im(\bar{u}H)\right)dr^\star\\
		&	\qquad\qquad +B(\mathcal{C},r_{\pm \infty}) \int_{-\infty}^\infty d\omega \sum_{m\ell} 1_{\mathcal{F}_{\mathcal{SF},\mathcal{C}}}|u|^2(-\infty),
	\end{aligned}
\end{equation}
where in the above we used
\begin{equation}
	H=(H)^{(a\omega)}_{ml}=\frac{\Delta}{(r^2+a^2)^{3/2}}\left(\rho^2 F\right)^{(a\omega)}_{ml}(r).
\end{equation}

We begin by bounding the inhomogeneous terms on the RHS of~\eqref{eq: proof prop: subsec: summing in the redshift estimate, prop 0, eq -1}. Specifically, let~$C^{(a\omega)}_{m\ell}$ denote the inhomogeneous terms (the terms related to~$H$) on the right hand side of inequality~\eqref{eq: proof prop: subsec: summing in the redshift estimate, prop 0, eq -1}.	It is a direct consequence of Young's inequalities and Parseval's identities that we have the following bound
	\begin{equation}\label{eq: proof prop: subsec: summing in the redshift estimate, prop 0, eq 0}
	\begin{aligned}
		&	\int_{\mathbb{R}}\int_{\mathbb{R}}d\omega   \sum_{m\ell}C^{(a\omega)}_{m\ell}dr^\star \\
		&	\leq   B(\mathcal{C},r_{\pm \infty}) \epsilon_{\textit{cut}}\int_{\mathbb{R}}d\omega \sum_{|m|>0,\ell}\int_{-\infty}^{r^\star_{-\infty}} \Delta\left(\omega^2|u|^2 +m^2|u|^2\right)dr^\star \\
		&	\quad +B(\mathcal{C},r_{\pm \infty}) \epsilon_{\textit{cut}}\int_{\mathbb{R}}d\omega \sum_{|m|>0,\ell}\int_{r^\star_\infty}^{\infty} \Delta\left(\omega^2|u|^2 +m^2|u|^2\right)dr^\star \\
		&	\quad + \frac{B(\mathcal{C},r_{\pm \infty})}{\epsilon_{\textit{cut}}}\int_{\mathbb{R}}\sum_{m,\ell}\int_{\mathbb{R}} \frac{1}{\Delta} |H^{(a\omega)}_{m\ell}|^2 dr^\star\\
		&	\quad +E\left|\frac{1}{i}\int_{\mathbb{R}}dt\int_{\mathbb{R}}dr^\star \int_{\mathbb{S}^2}\sin\theta d\theta d\varphi \left(\partial_t\psi\cdot H -\frac{a\Xi}{\bar{r}_+^2+a^2}\partial_{\varphi}\psi\cdot H+\partial_t\psi\cdot H -\frac{a\Xi}{r_+^2+a^2}\partial_{\varphi}\psi\cdot H\right)\right|,
	\end{aligned}
\end{equation}
for a sufficiently small~$\epsilon_{cut}>0$, where in the last line of~\eqref{eq: proof prop: subsec: summing in the redshift estimate, prop 0, eq 0} we used that~$H=\frac{\Delta}{(r^2+a^2)^{3/2}}\left(\rho^2 F\right)$. Note that it is significant that in~\eqref{eq: proof prop: subsec: summing in the redshift estimate, prop 0, eq -1} that term multiplying~$E$ appears in that non-signed form. In~\eqref{eq: proof prop: subsec: summing in the redshift estimate, prop 0, eq 0}, in order to bound the term multiplying~$E$, we used the second Parseval's identity of~\eqref{eq: subsec: fourier properties, eq 1}.

Now, in view of the inequality
\begin{equation}
	\int_{-\infty}^\infty d\omega \sum_{m\ell} 1_{\mathcal{F}_{\mathcal{SF},\mathcal{C}}}|u|^2(-\infty)\leq B\int_{\mathcal{H}^+}\left|\mathcal{P}_{\mathcal{SF},\mathcal{C}} (\psi)\right|^2,
\end{equation}
where for the Fourier projection on bounded non-stationary superradiant frequencies~$\mathcal{P}_{\mathcal{SF},\mathcal{C}}$ see~\eqref{eq: sec: main theorems, eq 1}, we combine~\eqref{eq: proof prop: subsec: summing in the redshift estimate, prop 0, eq -1} and the bound of the cut-off terms~\eqref{eq: proof prop: subsec: summing in the redshift estimate, prop 0, eq 0} to obtain that   
\begin{equation}\label{eq: proof prop: subsec: summing in the redshift estimate, prop 0, eq 1}
\begin{aligned}
&\int_{r^\star_{-\infty}}^{r^\star_{-\infty}}\int_{\mathbb{R}}d\omega\sum_{m,\ell} \left( |\Psi^\prime|^2+ 1_{|m|>0}|u^\prime|^2 + (\omega^2+\tilde{\lambda})\left(1-\frac{r_{trap}}{r}\right)^2|u|^2 +\left(\mu^2_{KG}+1_{\{|m|>0\}}\right)|u|^2 \right) dr^\star \\
&	\leq   B(\mathcal{C},r_{\pm \infty}) \epsilon_{\textit{cut}}\int_{\mathbb{R}}d\omega \sum_{|m|>0,\ell}\int_{-\infty}^{r^\star_{-\infty}} \Delta\left(\omega^2|u|^2 +m^2|u|^2\right)dr^\star \\
&	\quad +B(\mathcal{C},r_{\pm \infty}) \epsilon_{\textit{cut}}\int_{\mathbb{R}}d\omega \sum_{|m|>0,\ell}\int_{r^\star_\infty}^{\infty} \Delta\left(\omega^2|u|^2 +m^2|u|^2\right)dr^\star \\
&	\quad + \frac{B(\mathcal{C},r_{\pm \infty})}{\epsilon_{\textit{cut}}}\int_{\mathbb{R}}\sum_{m,\ell}\int_{\mathbb{R}} \frac{1}{\Delta} |H^{(a\omega)}_{m\ell}|^2dr^\star \\
&	\quad +E\left|\frac{1}{i}\int_{\mathbb{R}}dt\int_{\mathbb{R}}dr^\star \int_{\mathbb{S}^2}\sin\theta d\theta d\varphi \left(\partial_t\psi \cdot H -\frac{a\Xi}{\bar{r}_+^2+a^2}\partial_{\varphi}\psi_{\tau_1}\cdot H+\partial_t\psi \cdot H -\frac{a\Xi}{r_+^2+a^2}\partial_{\varphi}\psi\cdot H\right)\right|\\
&	\quad +B(\mathcal{C},r_{\pm \infty})\cdot\int_{\mathcal{H}^+}\left|\mathcal{P}_{\mathcal{SF},\mathcal{C}} (\psi)\right|^2,
\end{aligned}
\end{equation}
where for the sufficiently small constant~$\epsilon_{\textit{cut}}$ see the estimate~\eqref{eq: proof prop: subsec: summing in the redshift estimate, prop 0, eq 0}. In the second to last line of~\eqref{eq: proof prop: subsec: summing in the redshift estimate, prop 0, eq 1} we used that~$H=\frac{\Delta}{(r^2+a^2)^{3/2}}\left(\rho^2 F\right)$.

Now, since~$\Psi$ is sufficiently integrable and outgoing, we can use a pigeonhole argument and we can find a sequence~$\tau_n\rightarrow-\infty$ such that 
 \begin{equation}
 	\int_{\{t^\star=\tau_n\}} J^n_\mu[\Psi]n^\mu \leq \frac{C}{\tau_n}
 \end{equation}
for some constant~$C(a,M,l,\Psi)$~(Also, see already Lemma~\ref{lem: sec: proof of Theorem 1, lem 1} where we present a similar result, for functions that vanish on the horizons.) We use the redshift estimate of Proposition~\ref{prop: redshift estimate} in the spacetime domain~$D(\tau_n,\tau_2)$ to obtain 
\begin{equation}\label{eq: proof prop: subsec: summing in the redshift estimate, prop 0, eq 1.3}
	\begin{aligned}
		& \int\int_{D(\tau_n,\tau_2)\cap \{r\leq r_++\epsilon_{red}\}}   J^N_{\mu}[\Psi]N^{\mu}  + \int\int_{D(\tau_n,\tau_2)\cap \{r\geq \bar{r}_+-\epsilon_{red}\}} J^N_{\mu}[\Psi]N^{\mu}  \\
		&  \quad\quad \leq B \int_{\{t^\star=\tau_n\}} J^n_{\mu}[\psi]n^{\mu}\\
		&	\qquad\qquad\qquad  +B\Big(\int\int_{D(\tau_n,\tau_2)\cap\{ r_++\epsilon_{red}\leq r\leq r_++2\epsilon_{red} \}}   J^n_{\mu}[\psi]n^{\mu} +\int\int_{D(\tau_n,\tau_2)\cap\{ \bar{r}_+-2\epsilon_{red}\leq r\leq \bar{r}_+-\epsilon_{red} \}}   J^n_{\mu}[\psi]n^{\mu} \Big) \\
		&	\qquad\qquad\qquad  + B\Big(\int\int_{D(\tau_n,\tau_2)\cap\{ r_+\leq r\leq r_++2\epsilon_{red} \}}  | F|^2 +\int\int_{D(\tau_n,\tau_2)\cap\{ \bar{r}_+-2\epsilon_{red}\leq r\leq \bar{r}_+\}}  |F|^2 \Big) \\
	\end{aligned}
\end{equation}

Now, we sum~\eqref{eq: proof prop: subsec: summing in the redshift estimate, prop 0, eq 1.3} and~\eqref{eq: proof prop: subsec: summing in the redshift estimate, prop 0, eq 1}, for~$r_{-\infty}<r_++\epsilon_{red}<\bar{r}_+-\epsilon_{red}<r_{+\infty}$, we take~$\tau_n\rightarrow -\infty$, and we conclude that the following holds
\begin{equation}\label{eq: proof prop: subsec: summing in the redshift estimate, prop 0, eq 3}
	\begin{aligned}
		\int\int_{D(-\infty,+\infty)} & \mu^2_{\textit{KG}}
		|\psi|^2+ 1_{|m|>0}|\psi|^2+|\partial_{r^\star}\psi|^2+ |Z^\star\psi|^2+|\partial_t\mathcal{P}_{\textit{trap}}[\chi_{\tau_1}\psi]|^2+|\slashed{\nabla}\mathcal{P}_{\textit{trap}}[\chi_{\tau_1}\psi]|^2  \\
		& \qquad  \leq   B(\mathcal{C},r_{\pm \infty})\int_{\mathcal{H}^+}\left|\mathcal{P}_{\mathcal{SF},\mathcal{C}} (\psi)\right|^2\\
		&	\qquad\qquad +B(\mathcal{C},r_{\pm \infty}) \int\int_{D(-\infty,\infty)} |\partial_t\psi \cdot F| +|\partial_{\varphi}\psi \cdot F| +|F|^2,
	\end{aligned}
\end{equation}
after taking~$\epsilon_{cut}>0$ sufficiently small. For the operator~$\mathcal{P}_{\textit{trap}}$ see~\eqref{eq: sec: main theorems, eq 1.1}.

We have concluded the result of Theorem~\ref{thm: subsec: summing in the redshift estimate, thm -1}. 
\end{proof}

\subsection{Finite in time energy estimates}\label{subsec: sec: the main proposition, subsec 0}

Note the following Lemma 
\begin{lemma}\label{lem: subsec: sec: the main proposition, subsec 0, lem 1}
	Let $l>0$, $(a,M)\in\mathcal{B}_l$ and $\mu^2_{\textit{KG}}\geq 0$. Given~$0\leq \tau_1\leq\tau_2<+\infty$, there exists a constant 
	\begin{equation}
		C(\tau_2-\tau_1,a,M,l)>0
	\end{equation}
	such that for any solution~$\psi$ of the inhomogeneous Klein--Gordon equation~$\Box_{g_{a,M,l}}\psi-\mu^2_{KG}\psi=F$, where~$F$ is a smooth function, then the following hold
	\begin{equation}\label{eq: lem: subsec: sec: the main proposition, subsec 0, lem 1, eq 1}
		\int_{\{t^\star=\tau_2\}} J^n_\mu[\psi] n^\mu  \leq C(\tau_2-\tau_1,a,M,l) \left(\int_{\{t^\star=\tau_1\}} J^n_\mu[\psi] n^\mu+ \int\int_{D(\tau_1,\tau_2)} |F|^2\right),
	\end{equation}
	\begin{equation}\label{eq: lem: subsec: sec: the main proposition, subsec 0, lem 1, eq 2}
		\int_{\{t^\star=\tau_2\}}|\psi|^2 \leq  C(\tau_2-\tau_1,a,M,l)\left(\int_{\{t^\star=\tau_1\}} |\psi|^2 +J^n_\mu[\psi]n^\mu  + \int\int_{D(\tau_1,\tau_2)} |F|^2\right).
	\end{equation}
\end{lemma}
\begin{proof}
	This is a standard proof. 
	
	To prove~\eqref{eq: lem: subsec: sec: the main proposition, subsec 0, lem 1, eq 1} we use divergence theorem, see Proposition~\ref{prop: divergence theorem}, with a~$W=\partial_t+\frac{a\Xi}{r^2+a^2}\partial_\varphi$ multiplier in conjunction with an~$N$-redshift multiplier, see Proposition~\ref{prop: redshift estimate}. Then, we use Gr\"onwall and conclude. 
	
	To prove~\eqref{eq: lem: subsec: sec: the main proposition, subsec 0, lem 1, eq 2} we use Gr\"onwall's inequality in view of the direct application of fundamental theorem of calculus
	\begin{equation}
		|\psi|^2(\tau_2,r,\theta,\varphi) \leq C(\tau_2-\tau_1) \int_{\tau_1}^{\tau_2} |\partial_t\psi|^2(\tau,r,\theta,\varphi) d\tau +|\psi|^2(\tau_1,r,\theta,\varphi)
	\end{equation}
	and conclude by also using~\eqref{eq: lem: subsec: sec: the main proposition, subsec 0, lem 1, eq 1}. 
\end{proof}

\subsection{Corollary for future integrable homogeneous solutions for parameters in the set~\texorpdfstring{$\widetilde{\mathcal{MS}}_{l,\mu_{KG}}$}{PDFstring}}

In the following Corollary~\ref{cor: subsec: summing in the redshift estimate, cor 1} we restrict the assumptions of Theorem~\ref{thm: subsec: summing in the redshift estimate, thm -1}, and study `future integrable' solutions of the homogeneous Klein--Gordon equation~\eqref{eq: kleingordon}. 

First, we define the `future integrable' class of solutions:
\begin{definition}\label{def: subsec: sec: proofs of the main theorems, subsec 2, def 1}
	Let $l>0$ and~$(a,M)\in\mathcal{B}_l$. Moreover, assume that~$\psi$ is a smooth solution of the Klein--Gordon equation~\eqref{eq: kleingordon} arising from smooth initial data on the hypersurface~$\{t^\star=0\}$. We say that~$\psi$ is \underline{future integrable} if for every~$\tau_1>1$ the function~$\chi_{\tau_1}\psi$ is sufficiently integrable, see Definition~\ref{def: sec: carter separation, def 1}. For the cut-off~$\chi_{\tau_1}$ see Section~\ref{subsec: sec: carter separation, subsec 1}. 
  \end{definition}

We have the following 

\begin{customCorollary}{8.2}\label{cor: subsec: summing in the redshift estimate, cor 1}
	Let~$l>0$,~$\mu^2_{\textit{KG}}\geq 0$ and~$(a,M)\in\widetilde{\mathcal{MS}}_{l,\mu_{KG}}$. Let~$\psi$ be a future integrable solution, see Definition~\ref{def: subsec: sec: proofs of the main theorems, subsec 2, def 1}, of the Klein--Gordon equation~\eqref{eq: kleingordon} arising from smooth initial data on~$\{t^\star=0\}$. Then, we have
	\begin{equation}\label{eq: cor: subsec: summing in the redshift estimate, cor 1, eq 1}
		\begin{aligned}
			\int\int_{D(\tau_1,+\infty)} & \mu^2_{\textit{KG}}
			|\psi|^2+ |\partial_{r^\star}\psi|^2+ |Z^\star\psi|^2 +|\partial_t\mathcal{P}_{\textit{trap}}[\chi_{\tau_1}\psi]|^2+|\slashed{\nabla}\mathcal{P}_{\textit{trap}}[\chi_{\tau_1}\psi]|^2 \leq  B\int_{\{t^\star=\tau_1\}} J_{\mu}^{n}[\psi]n^{\mu}+ |\psi|^2,
		\end{aligned}
	\end{equation}
	for any~$0\leq \tau_1$, where for~$\mathcal{P}_{\textit{trap}}$ see~\eqref{eq: sec: main theorems, eq 1.1}, and for~$\chi_{\tau_1}$ see Section~\ref{sec: carter separation}. For~$\widetilde{\mathcal{MS}}_{l,\mu_{KG}}$ see Definition~\ref{def: subsec: sec: main theorems, subsec 1, def 1}. 
\end{customCorollary}

\begin{proof}[\textbf{Proof of Corollary~\ref{cor: subsec: summing in the redshift estimate, cor 1}}]

	We cut off the solutions of the Klein--Gordon equation~\eqref{eq: kleingordon} as follows 
	\begin{equation}
		\Box_{g_{a,M,l}}(\chi_{\tau_1}\psi)-\mu^2_{KG}\cdot\chi_{\tau_1}\psi =F
	\end{equation}
where~$F= 2\nabla\chi_{\tau_1}\cdot \nabla \psi+ \psi \Box(\chi_{\tau_1})$. 

The function~$\psi$ is future integrable and therefore~$\chi_{\tau_1}\psi$ sufficiently integrable and outgoing. We use the result of Theorem~\ref{thm: subsec: summing in the redshift estimate, thm -1} for~$\chi_{\tau_1}\psi$ in the place of~$\psi$ to obtain the following 
	\begin{equation}\label{eq proof: cor: subsec: summing in the redshift estimate, cor 1, eq 1}
	\begin{aligned}
		\int\int_{D(-\infty,+\infty)} & \mu^2_{\textit{KG}}
		|\chi_{\tau_1}\psi|^2+ |\partial_{r^\star}(\chi_{\tau_1}\psi)|^2+ |Z^\star(\chi_{\tau_1}\psi)|^2 +|\partial_t\mathcal{P}_{\textit{trap}}[\chi_{\tau_1}\psi]|^2+|\slashed{\nabla}\mathcal{P}_{\textit{trap}}[\chi_{\tau_1}\psi]|^2 \\
		&	\quad\leq   B \int_{\mathcal{H}^+}\left|\mathcal{P}_{\mathcal{SF},\mathcal{C}} (\chi_{\tau_1}\psi)\right|^2+	B\int\int_{D(-\infty,+\infty)} |\partial_t\chi_{\tau_1}\psi \cdot F| +|\partial_{\varphi}\chi_{\tau_1}\psi \cdot F| +|F|^2.
	\end{aligned}
\end{equation}

	Now, by using the bound of Proposition~\ref{prop: subsec: summing in the redshift estimate, prop 1}, for the horizon term of~\eqref{eq proof: cor: subsec: summing in the redshift estimate, cor 1, eq 1}, and the finite in time energy estimates of Lemma~\ref{lem: subsec: sec: the main proposition, subsec 0, lem 1},
	we obtain the desired result~\eqref{eq: cor: subsec: summing in the redshift estimate, cor 1, eq 1}. 
\end{proof}

\section{The continuity argument}\label{sec: continuity argument}

The main result of this Section is the following

\begin{theorem}\label{prop: sec: continuity argument, prop 1}
	Let~$l>0$ and~$\mu^2_{KG}\geq 0$. Then, if~$(a,M)\in\mathcal{MS}_{l,\mu_{KG}}$, see~\eqref{eq: sec: main theorems, eq 2}, then the solution~$\psi$ of the Klein--Gordon equation~\eqref{eq: kleingordon} that arises from smooth initial data~$(\psi_0,\psi_1)$ is future integrable, see Definition~\ref{subsec: sec: proofs of the main theorems, subsec 1}. 
\end{theorem}

\begin{proof}
	See Section~\ref{subsec: sec: continuity argument, subsec 2}. 
\end{proof}

\begin{remark}
	
We will use Theorem~\ref{prop: sec: continuity argument, prop 1} in the next Section~\ref{sec: proof of Theorem 1}. If the solution of the Klein--Gordon equation~\eqref{eq: kleingordon} is axisymmetric~$\partial_{\varphi^\star}\psi=0$ then we do not need the continuity argument of the present Section.
\end{remark}

We use the generic constants of Section~\ref{subsec: sec: proofs of the main theorems, subsec 1}, with the algebra of constants discussed there, but we will also use the notation~$B(m)$ to denote that the constant~$B$ additionally depends on the azimuthal frequency~$m$.

Before presenting and proving our main results we need several preparatory Lemmata and Propositions regarding alternative trapping parameters for fixed azimuthal frequency~$m$.

\subsection{The causal vector field~\texorpdfstring{$W_{mod}$}{ng}}

The Lemma of this Section will be used later in the fixed azimuthal frequency Proposition~\ref{prop: sec: continuity argument, prop 2}.

Let~$\epsilon_{trap}>0$ be sufficiently small, as chosen in Lemma~\ref{lem: sec: continuity argument, lem 1}. We now modify the timelike vector field~$W$ of Lemma~\ref{lem: causal vf E,1} to make it Killing in the sufficiently small region 
\begin{equation}
	\{r_{\Delta,\textit{frac}}-2\epsilon_{\textit{trap}}\leq r\leq r_{\Delta,\textit{frac}}+2\epsilon_{\textit{trap}}\}.
\end{equation}
We state this result in the following Lemma 

\begin{lemma}\label{lem: sec: continuity argument, lem 2} 
Let~$l>0$ and~$(a,M)\in\mathcal{B}_l$. Moreover, let~$\epsilon_{\textit{trap}}>0$ of Lemma~\ref{lem: sec: continuity argument, lem 1} be sufficiently small.

Then, there exists a causal vector field 
\begin{equation}
	W_{\textit{mod}}
\end{equation}
such that 
\begin{equation}\label{eq: lem: sec: continuity argument, lem 2, eq 1} 
	\begin{aligned}
		W_{\textit{mod}}	&	=W,\qquad \{r_+\leq r\leq r_{\Delta,\textit{frac}}-3\epsilon_{\textit{trap}}\}\cup\{r_{\Delta,\textit{frac}}+3\epsilon_{\textit{trap}}\leq r\leq\bar{r}_+\},\\
		W_{\textit{mod}}	&	\text{ is timelike in } \{r_{\Delta,\textit{frac}}-3\epsilon_{\textit{trap}}\leq r \leq r_{\Delta,\textit{frac}}-2\epsilon_{\textit{trap}}\}\cup \{r_{\Delta,\textit{frac}}+2\epsilon_{\textit{trap}}\leq r\leq r_{\Delta,\textit{frac}}+3\epsilon_{\textit{trap}}\},\\
		W_{\textit{mod}}	&	\text{ is Killing and timelike in } \{r_{\Delta,\textit{frac}}-2\epsilon_{\textit{trap}}\leq r\leq r_{\Delta,\textit{frac}}+2\epsilon_{\textit{trap}}\},
	\end{aligned}
\end{equation}
where~$W=\partial_t+\frac{a\Xi}{r^2+a^2}\partial_{\varphi}$, see Lemma~\ref{lem: causal vf E,1}. Moreover, we have the following pointwise estimate
\begin{equation}\label{eq: lem: sec: continuity argument, lem 2, eq 2} 
	|\mathrm{K}^{W_{mod}}[\psi]| \lesssim  \Delta \cdot 1_{\{r_+\leq r\leq \bar{r}_+\}\setminus\{r_{\Delta,\textit{frac}}-2\epsilon_{\textit{trap}}\leq r\leq r_{\Delta,\textit{frac}}+2\epsilon_{\textit{trap}}\}} \cdot\left( \epsilon|Z^\star\psi|^2+ \epsilon(\partial_{r^\star}\psi)^2 +\epsilon^{-1} |\partial_\phi \psi|^2 \right),
\end{equation}
for a sufficiently small~$\epsilon>0$, where~$r_{\Delta,frac}$ is the unique critical point of~$\frac{\Delta}{(r^2+a^2)^2}$, see Lemma~\ref{lem: sec: general properties of Delta, lem 3}. 
\end{lemma}

\begin{proof}
	The proof of~\eqref{eq: lem: sec: continuity argument, lem 2, eq 1}  is straightforward by considering a vector field of the form
	\begin{equation}
		W_{mod}= \partial_t + f_{mod} (r) \partial_\varphi,
	\end{equation} 
	where~$f_{mod}(r)$ is constant in~$\{r_{\Delta,\textit{frac}}-2\epsilon_{\textit{trap}}\leq r\leq r_{\Delta,\textit{frac}}+2\epsilon_{\textit{trap}}\}$.

	To prove~\eqref{eq: lem: sec: continuity argument, lem 2, eq 2} we note that the following holds pointwise
	\begin{equation}
		\mathrm{K}^{W_{mod}} = 2\frac{\Delta}{\rho^2} \mathbb{T}(\nabla f,\partial_\phi)[\psi]
	\end{equation}
	and conclude by using Young's inequality.
\end{proof}

\subsection{Morawetz estimate for fixed azimuthal frequency~\texorpdfstring{$m$}{m}}\label{subsec: sec: continuity argument, subsec 0.2}

Note the following Proposition, which \underline{does not} assume future integrability.

\begin{proposition}\label{prop: sec: continuity argument, prop 2}
Let~$l>0$,~$(a,M)\in\mathcal{MS}_{l,\mu_{KG}}$ and $\mu^2_{KG}\geq 0$. For any~$j\geq 1$ and~$|m|\geq 0$ there exists a constant 
\begin{equation}
	B=B(m,j,a,M,l,\mu^2_{\textit{KG}})>0,
\end{equation}
and there exists a smooth function 
\begin{equation}
\tilde{\zeta}_{\textit{trap}}(r)
\end{equation}
with
\begin{equation}\label{eq: prop: sec: continuity argument, prop 2, eq 0}
\tilde{\zeta}_{\textit{trap}}(r)=
\begin{cases}
1,\qquad \{r_+\leq r\leq r_{\Delta,\textit{frac}}-2\epsilon_{\textit{trap}}\}\cup\{r_{\Delta,\textit{frac}}+2\epsilon_{\textit{trap}},\bar{r}_+\}\\
0,\qquad \{r_{\Delta,\textit{frac}}-\epsilon_{\textit{trap}}\leq r\leq r_{\Delta,\textit{frac}}+\epsilon_{\textit{trap}}\},
\end{cases}
\end{equation}
where for the sufficiently small~$\epsilon_{\textit{trap}}>0$ see Lemma~\ref{lem: sec: continuity argument, lem 2}, such that for any smooth solution~$\psi$ of Klein--Gordon equation~\eqref{eq: kleingordon} with fixed azimuthal frequency
\begin{equation}
|m|> 0
\end{equation}
we obtain 
\begin{equation}\label{eq: prop: sec: continuity argument, prop 2, eq 1}
\begin{aligned}
&\int\int_{D(\tau_1,\tau_2)}  \tilde{\zeta}_{\textit{trap}}(r) \sum_{1\le i_1+i_2+i_3\le j}
|\nabb^{i_1}(\partial_{t^\star})^{i_2}(Z^\star)^{i_3}\psi|^2 \\
&\qquad\qquad\qquad\qquad+ \sum_{1\le i_1+i_2+i_3\le j-1}
|\nabb^{i_1}(\partial_{t^\star})^{i_2}(Z^\star)^{i_3+1}\psi|^2+|\nabb^{i_1}(\partial_{t^\star})^{i_2}(Z^\star)^{i_3}\psi|^2 \\
& \qquad   \leq B(m,j)\int_{\{t^\star=\tau_1\}} \sum_{0 \leq i_1+i_2+i_3\leq j}
|\nabb^{i_1}(\partial_{t^\star})^{i_2}(Z^\star)^{i_3}\psi|^2 + B(j)\int_{\{t^\star=\tau_2\}} \sum_{1\leq i_1+i_2+i_3\leq j}
|\nabb^{i_1}(\partial_{t^\star})^{i_2}(Z^\star)^{i_3}\psi|^2\\
\end{aligned}
\end{equation}
and 
\begin{equation}\label{eq: prop: sec: continuity argument, prop 2, eq 2}
	\begin{aligned}
	&	\int_{\{t^\star=\tau_2\}} \sum_{1\leq i_1+i_2+i_3\leq j}
		|\nabb^{i_1}(\partial_{t^\star})^{i_2}(Z^\star)^{i_3}\psi|^2 \\
		&\quad+\int\int_{D(\tau_1,\tau_2)}  \tilde{\zeta}_{\textit{trap}}(r) \sum_{1\le i_1+i_2+i_3\le j}
		|\nabb^{i_1}(\partial_{t^\star})^{i_2}(Z^\star)^{i_3}\psi|^2 \\
		&\qquad\qquad\qquad\qquad+ \sum_{1\le i_1+i_2+i_3\le j-1}
		|\nabb^{i_1}(\partial_{t^\star})^{i_2}(Z^\star)^{i_3+1}\psi|^2+|\nabb^{i_1}(\partial_{t^\star})^{i_2}(Z^\star)^{i_3}\psi|^2 \\
		&	\qquad \leq B(m,j)\int_{\{t^\star=\tau_1\}} \sum_{0 \leq i_1+i_2+i_3\leq j}
		|\nabb^{i_1}(\partial_{t^\star})^{i_2}(Z^\star)^{i_3}\psi|^2 \\
		&	\qquad\qquad   + B(m,j)\int_{\tau_1}^{\tau_2}d\tau\left(\int_{\{t^\star=\tau\}\cap [r_+,r_{\Delta,\textit{frac}}-2\epsilon_{\textit{trap}})}+\int_{\{t^\star=\tau\}\cap (r_{\Delta,\textit{frac}}+2\epsilon_{\textit{trap}},\bar{r}_+]}\right) |\partial_{\varphi}\psi|^2 \\
	\end{aligned}
\end{equation}
for all~$2\leq 1+\tau_1\leq\tau_2$. The value~$r_{\Delta,\textit{frac}}$ is the unique critical point of~$\frac{\Delta}{(r^2+a^2)^2}$, see Lemma~\ref{lem: sec: general properties of Delta, lem 3}.

Moreover, in the case of future integrable solutions of the Klein--Gordon equation we can drop the two last bulk terms from the RHS of~\eqref{eq: prop: sec: continuity argument, prop 2, eq 1}.

Furthermore, for \underline{the~axisymmetric~case~$m=0$} we replace~$(a,M)\in\mathcal{MS}_{l,\mu_{KG}}$ with the entire subextremal family~$(a,M)\in\mathcal{B}_l$ and moreover inequality~\eqref{eq: prop: sec: continuity argument, prop 2, eq 1} holds without the two last bulk terms from the RHS of~\eqref{eq: prop: sec: continuity argument, prop 2, eq 1}.
\end{proposition}
\begin{proof}	

\boxed{First,~ we~ prove~\eqref{eq: prop: sec: continuity argument, prop 2, eq 1}}~ for~ a~ fixed~ azimuthal~ frequency~$|m|>0$. For any~$\epsilon_{trap}>0$ sufficiently small, we take~$\omega_{high}(m,\epsilon_{trap}),~\lambda_{low}^{-1}(\epsilon_{trap})$ of Lemma~\ref{lem: sec: continuity argument, lem 1} sufficiently large and in view of Lemma~\ref{lem: sec: continuity argument, lem 2} we conclude 
\begin{equation}
	\supp \tilde{\zeta}_{trap} \subset (r_{\Delta,frac}-\epsilon_{trap},r_{\Delta,frac}+\epsilon_{trap}). 
\end{equation}
where~$\tilde{\zeta}_{trap}$ is smooth and
\begin{equation}
	\tilde{\zeta}_{\textit{trap}}(r)=
	\begin{cases}
		1,\qquad \{r_+\leq r\leq r_{\Delta,\textit{frac}}-2\epsilon_{\textit{trap}}\}\cup\{r_{\Delta,\textit{frac}}+2\epsilon_{\textit{trap}},\bar{r}_+\}\\
		0,\qquad \{r_{\Delta,\textit{frac}}-\epsilon_{\textit{trap}}\leq r\leq r_{\Delta,\textit{frac}}+\epsilon_{\textit{trap}}\}.
	\end{cases}
\end{equation}

Given~$\psi$ we now apply the cut-off 
\begin{equation}
\chi_{\tau_1,\tau_2},
\end{equation}
for~$2\leq 1+\tau_1\leq\tau_2$, see Section~\ref{subsec: sec: carter separation, subsec 1}, in the place of~$\chi_{\tau_1}$ and repeat the arguments from the proof of Theorem~\ref{thm: subsec: summing in the redshift estimate, thm -1}, which note holds for the inhomogeneous Klein--Gordon equation~$\Box_{g_{a,M,l}}\psi-\mu^2_{KG}\psi=F$.

Specifically, we fix the constant
\begin{equation*}
	\mathcal{C}(a,M,l,\mu_{KG})>0, 
\end{equation*}
see Theorem~\ref{thm: subsec: summing in the redshift estimate, thm -1}, and we obtain the following
\begin{equation}\label{eq: proof prop: sec: continuity argument, prop 2, eq 0}
\begin{aligned}
&\int\int_{D(\tau_1,\tau_2)}  \tilde{\zeta}_{\textit{trap}}(r) \sum_{ i_1+i_2+i_3 = 1}
|\nabb^{i_1}(\partial_{t^\star})^{i_2}(Z^\star)^{i_3}\psi|^2+|Z^\star\psi|^2 \\
& \qquad   \leq B(m,j)\int_{\{t^\star=\tau_1\}} \sum_{0\leq i_1+i_2+i_3\leq 1}
|\nabb^{i_1}(\partial_{t^\star})^{i_2}(Z^\star)^{i_3}\psi|^2 +B(m,j)\int_{\{t^\star=\tau_2\}} \sum_{0\leq i_1+i_2+i_3 \leq 1}
|\nabb^{i_1}(\partial_{t^\star})^{i_2}(Z^\star)^{i_3}\psi|^2\\
&	\qquad\qquad +B(m,j) \cdot \int_{\mathcal{H}^+}\left|\mathcal{P}_{\mathcal{SF},\mathcal{C}} (\chi_{\tau_1,\tau_2}\psi)\right|^2\\
& \qquad   \leq B(m,j)\int_{\{t^\star=\tau_1\}} \sum_{0\leq i_1+i_2+i_3\leq 1}
|\nabb^{i_1}(\partial_{t^\star})^{i_2}(Z^\star)^{i_3}\psi|^2 +B(m,j)\int_{\{t^\star=\tau_2\}} \sum_{i_1+i_2+i_3 = 1}
|\nabb^{i_1}(\partial_{t^\star})^{i_2}(Z^\star)^{i_3}\psi|^2,\\
\end{aligned}
\end{equation}
for any~$2\leq 1+\tau_1\leq \tau_2$, where~$\mathcal{P}_{\mathcal{SF},\mathcal{C}}$ is a Fourier projection to bounded non-stationary superradiant frequencies, see~\eqref{eq: sec: main theorems, eq 1}. Note that to obtain the last estimate of~\eqref{eq: proof prop: sec: continuity argument, prop 2, eq 0} we used that
\begin{equation}
	m^2|\psi|^2=(\partial_\varphi\psi)^2
\end{equation}
for our fixed azimuthal frequency~$|m|>0$ in order to bound the future flux appropriately so that no zeroth order terms appear.

By commuting appropriately many times with~$\partial_t,N$, and by using elliptic estimates, we obtain the higher order statement 
\begin{equation}\label{eq: proof prop: sec: continuity argument, prop 2, eq 2}
	\begin{aligned}
		&\int\int_{D(\tau_1,\tau_2)}  \tilde{\zeta}_{\textit{trap}}(r) \sum_{1\le i_1+i_2+i_3\le j}
		|\nabb^{i_1}(\partial_{t^\star})^{i_2}(Z^\star)^{i_3}\psi|^2 \\
		&\qquad\qquad\qquad\qquad+ \sum_{1\le i_1+i_2+i_3\le j-1}
		|\nabb^{i_1}(\partial_{t^\star})^{i_2}(Z^\star)^{i_3+1}\psi|^2+|\nabb^{i_1}(\partial_{t^\star})^{i_2}(Z^\star)^{i_3}\psi|^2 \\
		& \qquad   \leq B(m,j)\int_{\{t^\star=\tau_1\}} \sum_{0 \leq i_1+i_2+i_3\leq j}
		|\nabb^{i_1}(\partial_{t^\star})^{i_2}(Z^\star)^{i_3}\psi|^2 + B(m,j)\int_{\{t^\star=\tau_2\}} \sum_{1 \leq i_1+i_2+i_3\leq j}
		|\nabb^{i_1}(\partial_{t^\star})^{i_2}(Z^\star)^{i_3}\psi|^2,\\
	\end{aligned}
\end{equation}
which concludes~\eqref{eq: prop: sec: continuity argument, prop 2, eq 1} for solutions supported in the fixed azimuthal frequency~$|m|>0$.

\boxed{Now,~ we~ prove~\eqref{eq: prop: sec: continuity argument, prop 2, eq 1}} for fixed azimuthal frequency~$|m|=0$, i.e.~for axisymmetric solutions~$\psi$ of the Klein--Gordon equation. By using Theorem~\ref{thm: subsec: summing in the redshift estimate, thm -1} we have that for~$l>0$ and~$(a,M)\in\mathcal{B}_l$~(namely the entire subextremal parameter family without assuming mode stability) and~$\mu^2_{KG}\geq 0$ then for axisymmetric solutions~$\psi$ of the Klein--Gordon equation the following holds 
\begin{equation}\label{eq: proof prop: sec: continuity argument, prop 2, eq 3}
	\begin{aligned}
		&\int\int_{D(\tau_1,\tau_2)}  \left(1-\frac{r_{\Delta,frac}}{r}\right)^2 \sum_{ i_1+i_2+i_3 = 1}
		|\nabb^{i_1}(\partial_{t^\star})^{i_2}(Z^\star)^{i_3}\psi|^2+|Z^\star\psi|^2 \\
		& \qquad   \leq B(j)\int_{\{t^\star=\tau_1\}} \sum_{0\leq i_1+i_2+i_3\leq 1}
		|\nabb^{i_1}(\partial_{t^\star})^{i_2}(Z^\star)^{i_3}\psi|^2 +B(j)\int_{\{t^\star=\tau_2\}} \sum_{0\leq i_1+i_2+i_3 \leq 1}
		|\nabb^{i_1}(\partial_{t^\star})^{i_2}(Z^\star)^{i_3}\psi|^2,\\
	\end{aligned}
\end{equation}
where note we included the zero order terms on the future fluxes. Since~$\psi$ is axisymmetric we have 
\begin{equation}\label{eq: proof prop: sec: continuity argument, prop 2, eq 3.9}
	\mathcal{P}_{\mathcal{SF},\mathcal{C}}\psi=0,
\end{equation}
see~\eqref{eq: sec: main theorems, eq 1} for the definition of the operator~$\mathcal{P}_{\mathcal{SF},\mathcal{C}}$. Moreover, we note from the redshift Proposition for axisymmetric solutions, see Proposition~\ref{prop:redshift vs superradiance estimate, for axisymmetric solutions}, that the following holds 
\begin{equation}\label{eq: proof prop: sec: continuity argument, prop 2, eq 5}
	\int_{\{t^\star=\tau_2\}} J^n_\mu[\psi]n^\mu +|\psi|^2 \leq B \int_{\{t^\star=\tau_1\}} J^n_\mu[\psi]n^\mu+|\psi|^2.
\end{equation}

Therefore, by using the integrated estimate~\eqref{eq: proof prop: sec: continuity argument, prop 2, eq 3} and the bounds~\eqref{eq: proof prop: sec: continuity argument, prop 2, eq 5} we conclude that the following holds for axisymmetric solutions
\begin{equation}\label{eq: proof prop: sec: continuity argument, prop 2, eq 3.1}
	\begin{aligned}
		&\int\int_{D(\tau_1,\tau_2)}  \left(1-\frac{r_{\Delta,frac}}{r}\right)^2 \sum_{ i_1+i_2+i_3 = 1}
		|\nabb^{i_1}(\partial_{t^\star})^{i_2}(Z^\star)^{i_3}\psi|^2+|Z^\star\psi|^2 \\
		& \qquad   \leq B(j)\int_{\{t^\star=\tau_1\}} \sum_{0\leq i_1+i_2+i_3\leq 1}
		|\nabb^{i_1}(\partial_{t^\star})^{i_2}(Z^\star)^{i_3}\psi|^2.
	\end{aligned}
\end{equation}
We have concluded~\eqref{eq: prop: sec: continuity argument, prop 2, eq 1} for~$m=0$.

\boxed{Now, ~we ~proceed ~to~ prove~\eqref{eq: prop: sec: continuity argument, prop 2, eq 2}}. We use the~$W_{mod}$ multiplier of Lemma~\ref{lem: sec: continuity argument, lem 2}, the redshift vector field~$N$, in conjunction with the integrated estimate~\eqref{eq: prop: sec: continuity argument, prop 2, eq 1} and the standard redshift argument found in the lecture notes~\cite{DR5} to appropriately bound the~$\dot{H}^1$ future fluxes on the RHS of~\eqref{eq: prop: sec: continuity argument, prop 2, eq 1} as follows 
\begin{equation}\label{eq: proof prop: sec: continuity argument, prop 2, eq 2.2}
			\begin{aligned}
				&	\int_{\{t^\star=\tau_2\}} \sum_{1\leq i_1+i_2+i_3\leq j}
				|\nabb^{i_1}(\partial_{t^\star})^{i_2}(Z^\star)^{i_3}\psi|^2  \\
				&	\qquad \leq B(j)\int_{\{t^\star=\tau_1\}} \sum_{1\leq i_1+i_2+i_3\leq j}
				|\nabb^{i_1}(\partial_{t^\star})^{i_2}(Z^\star)^{i_3}\psi|^2 \\
				&	\qquad\qquad   + B(m,j)\int_{\tau_1}^{\tau_2}d\tau\left(\int_{\{t^\star=\tau\}\cap [r_+,r_{\Delta,\textit{frac}}-2\epsilon_{\textit{trap}})}+\int_{\{t^\star=\tau\}\cap (r_{\Delta,\textit{frac}}+2\epsilon_{\textit{trap}},\bar{r}_+]}\right) |\psi|^2\\
			\end{aligned}
		\end{equation}
		for all~$2\leq 1+\tau_1\leq\tau_2$

Therefore, we bound the RHS of the integrated estimate~\eqref{eq: prop: sec: continuity argument, prop 2, eq 1} by using the bound~\eqref{eq: proof prop: sec: continuity argument, prop 2, eq 2.2} on the future fluxes to obtain the desired inequality~\eqref{eq: prop: sec: continuity argument, prop 2, eq 2}.

Finally, we note that if~the~solution~$\psi$~is~future~integrable, see Definition~\ref{subsec: sec: proofs of the main theorems, subsec 1}, then we can use the cut-offed wave~$\chi_{\tau_1}\psi$, see Section~\ref{subsec: sec: carter separation, subsec 1}, and repeat the arguments of the present proof to obtain the estimate~\eqref{eq: prop: sec: continuity argument, prop 2, eq 1}, without the future boundary term on the RHS. 
\end{proof}

\subsection{Reduction to fixed azimuthal frequency~\texorpdfstring{$m$}{g}}\label{subsec: sec: continuity argument, subsec 0}

The main result Theorem~\ref{prop: sec: continuity argument, prop 1} of the present Section will follow as a Corollary from the following fixed azimuthal frequency Proposition, which will be proved with a continuity argument.

\begin{proposition}\label{prop: sec: continuity argument, prop 1.3}
	Let the assumptions of Theorem~\ref{prop: sec: continuity argument, prop 1} hold, but moreover assume that~$\psi=\psi_m$ where~$\psi_m$ is the~$m$-th azimuthal mode. Then,~$\psi$ is future integrable, see Definition~\ref{def: subsec: sec: proofs of the main theorems, subsec 2, def 1}. 
\end{proposition} 

\begin{proof}
	See Section~\ref{subsec: sec: continuity argument, subsec 1}.
\end{proof}

Note the following Lemma

\begin{lemma}\label{lem: subsec: sec: continuity argument, subsec 0, lem 1}
	Let~$l>0$,~$\mu^2_{KG}\geq 0$. Moreover, let~$a,M,\psi,m$ be as in Proposition~\ref{prop: sec: continuity argument, prop 1.3}. Then,~$\psi$ is future integrable if 
	\begin{equation}\label{eq: cor: subsec: sec: continuity argument, subsec 0, cor 1, eq 1}
	 \sup_{\tau\geq 0}\int_{\{t^\star=\tau\}} \sum_{1\le i_1+i_2+i_3\le j}
		|\nabb^{i_1}(\partial_{t^\star})^{i_2}(Z^\star)^{i_3}\psi|^2<+\infty,\qquad \forall j\geq 1.
	\end{equation}
\end{lemma}
\begin{proof}
	This is immediate in view of Proposition~\ref{prop: sec: continuity argument, prop 2}.
\end{proof}

In the following definition, we define a set we need for the remainder of the present Section 
\begin{definition}\label{def: subsec: sec: continuity argument, subsec 0, def 1}
	Let~$l>0$ and~$\mu^2_{KG}\geq 0$.  We define 
	\begin{equation}\label{eq: sec: continuity argument, eq 1}
		\begin{aligned}
			\mathcal{MS}_{l,\mu_{KG},m} &	= \{(a,M)\in \mathcal{MS}_{l,\mu_{KG}}:~\eqref{eq: cor: subsec: sec: continuity argument, subsec 0, cor 1, eq 1}~\text{ holds for all solutions}~\psi~ \text{of}~\eqref{eq: kleingordon}~ \text{supported}\\
			&	\qquad\qquad \text{on the fixed azimuthal frequency}~m~\text{with respect to}~g_{a,M}\}.
		\end{aligned}
	\end{equation}
	For the set~$\mathcal{MS}_{l,\mu_{KG}}$ see~\eqref{eq: sec: main theorems, eq 2}.
\end{definition}

\subsection{Proof of Proposition~\ref{prop: sec: continuity argument, prop 1.3}}\label{subsec: sec: continuity argument, subsec 1}

The proof of Proposition~\ref{prop: sec: continuity argument, prop 1.3} will use a continuity argument. Note that the original continuity argument of~\cite{DR2} was a continuity argument only in the rotation parameter~$a$ for a fixed black hole mass~$M$. In contrast to~\cite{DR2} the continuity argument of the present Section needs to take place in the space of parameters
\begin{equation}
	\mathcal{MS}_{l,\mu_{KG}}
\end{equation}
which is by definition connected, see Section~\ref{subsec: sec: main theorems, subsec 1}.

First, we prove the following Proposition
\begin{proposition}\label{prop: subsec: sec: continuity argument, subsec 1, prop 1}
	Let~$l>0$ and~$\mu^2_{KG}\geq 0$. Then,~$\mathcal{MS}_{l,\mu_{KG}}$ is an open subset of~$\mathcal{B}_l$. 
\end{proposition}

\begin{proof}
Now, let~$(a,M)\in \mathcal{MS}_{l,\mu_{KG}}$. We recall from Definition~\ref{def: sec: main theorems, def 1} that the set~$\mathcal{MS}_{l,\mu_{KG}}$ is the connected component of the set~$\mathcal{B}_{0,l}$ in the set 
\begin{equation}\label{eq: proof prop: subsec: sec: continuity argument, subsec 1, prop 1, eq 1}
		\widetilde{\mathcal{MS}}_{l,\mu_{KG}}=\{(a,M)\in\mathcal{B}_l: \forall(\omega,m,\ell)\in\mathbb{R}\times\mathbb{Z}\times \mathbb{Z}_{\geq |m|},~\omega\neq \omega_+m,~\omega\neq \bar{\omega}_+m 
		, ~~|\mathcal{W}^{-1}(a,M,l,\omega,m,\ell)|<+\infty\}.
\end{equation}
For~$\epsilon>0$ sufficiently small we obtain that if
\begin{equation}
	|\mathring{a}-a|+|\mathring{M}-M|\leq \epsilon
\end{equation}
then~$(\mathring{a},\mathring{M})\in \mathcal{B}_l$.

For both the sufficiently large frequencies
\begin{equation}
	\{\omega^2+m^2+\tilde{\lambda}\geq \mathcal{C}\},
\end{equation}
and the frequencies
\begin{equation}
	 \{0<|\omega-\omega_+m|\leq \mathcal{C}^{-1}\}\cup \{0<|\omega-\bar{\omega}_+ m|\leq \mathcal{C}^{-1}\}
\end{equation}
we use Corollary~\ref{lem: sec: proofs of the main theorems} and we obtain
\begin{equation}
	|\mathcal{W}^{-1}(\mathring{a},\mathring{M},l,\omega,m,\ell)|<+\infty.
\end{equation}

For the remaining bounded frequencies
\begin{equation}
	\{\omega^2+m^2+\tilde{\lambda}\leq\mathcal{C}\}\cap \left(\{|\omega-\omega_+m| \geq \mathcal{C}^{-1}\}\cup \{|\omega-\bar{\omega}_+ m| \geq \mathcal{C}^{-1}\}\right).
\end{equation}
we use that the Wronskian is a continuous function, on a compact domain, and conclude that indeed~$|\mathcal{W}^{-1}(\mathring{a},\mathring{M},l,\omega,m,\ell)|<+\infty$ holds, for a sufficiently small~$\epsilon>0$.

We conclude~$(\mathring{a},\mathring{M})\in\mathcal{MS}_{l,\mu_{KG}}$.
\end{proof}

Now, we are ready to prove Proposition~\ref{prop: sec: continuity argument, prop 1.3}

\begin{proof}[\textbf{Proof of Proposition~\ref{prop: sec: continuity argument, prop 1.3}}]

Let~$l>0$. In this Proposition we assume that 
\begin{equation}\label{eq: proof prop: sec: continuity argument, prop 1.3, eq 1}
\psi\:\:\text{has fixed azimuthal frequency}\:\: m .
\end{equation}
Specifically, we assume that the solution~$\psi$ of the Klein--Gordon equation~\eqref{eq: kleingordon} arises from initial data
\begin{equation}\label{eq: proof prop: sec: continuity argument, prop 1.3, eq 1.1}
	(\psi_0,\psi_1)
\end{equation}
of fixed azimuthal frequency~$m$.

First, we observe that for~$m=0$ that indeed the solution~$\psi$ of Klein--Gordon arising from~$(\psi_0,\psi_1)$ is future integrable, see Corollary~\ref{cor: subsec: summing in the redshift estimate, cor 1} with~$\mathcal{P}_{\mathcal{SF},\mathcal{C}}=0$ and Lemma~\ref{lem: subsec: sec: continuity argument, subsec 0, lem 1}.

Therefore, \underline{in the rest of the proof we assume that~$|m|>0$}. We follow the analogous arguments of~\cite{DR2}.

First, we recall Lemma~\ref{lem: subsec: sec: continuity argument, subsec 0, lem 1} where we proved that solutions~$\psi$ of the Klein--Gordon equation~\eqref{eq: kleingordon} that satisfy 
\begin{equation}\label{eq: cor: sec: continuity argument, cor 1, eq 1}
 \sup_{\tau\geq 0}\int_{\{t^\star=\tau\}} \sum_{1\le i_1+i_2+i_3\le j}
|\nabb^{i_1}(\partial_{t^\star})^{i_2}(Z^\star)^{i_3}\psi|^2<+\infty,\qquad \forall j\geq 1,
\end{equation}
are future integrable. We will use this to show that the set~$\mathcal{MS}_{l,\mu_{KG},m}$ is open and closed as a subset of~$\mathcal{MS}_{l,\mu_{KG}}$.

Now, we proceed with the clopen argument that will prove that
\begin{equation}\label{eq: cor: sec: continuity argument, cor 1, eq 1.1}
	\mathcal{MS}_{l,\mu_{KG},m}=\mathcal{MS}_{l,\mu_{KG}},
\end{equation}
for the former see Definition~\ref{def: subsec: sec: continuity argument, subsec 0, def 1} and for the latter see Definition~\ref{def: sec: main theorems, def 1}. 

\boxed{\texttt{To prove the non-emptyness}} of the set~$\mathcal{MS}_{l,\mu_{KG},m}$ in the~$\mathcal{MS}_{l,\mu_{KG}}$ topology we note that for all~$\mu^2_{KG}\geq 0$ we have
\begin{equation}
	\mathcal{B}_{0,l}\subset \mathcal{MS}_{l,\mu_{KG},m},
\end{equation}
see Proposition~\ref{prop subsec: sec: preliminaries, subsec 2, prop 1} where we prove that for solutions of the Klein--Gordon equation in a Schwarzschild--de~Sitter background the following boundedness estimate holds 
\begin{equation}
	\left(\int_{\{t^\star=\tau_2\}}+\int_{\mathcal{H}_+\cap D(\tau_1,\tau_2)}+\int_{\bar{\mathcal{H}}^+\cap D(\tau_1,\tau_2)} \right) J^n_\mu[\psi]n^\mu \leq B \int_{\{t^\star=\tau_1\}} J^n_\mu[\psi]n^\mu. 
\end{equation}
(Note that we do not need to use fixed azimuthal frequency.) For the set~$\mathcal{B}_{0,l}$ see Definition~\ref{def: subsec: sec: main theorems, subsec 1, def 1}.

\boxed{\texttt{To prove the closedness}} of the set~$\mathcal{MS}_{l,\mu_{KG},m}$ in the~$\mathcal{MS}_{l,\mu_{KG}}$ topology we follow the same arguments that proved an analogue result in~[\cite{DR2}, Section~11, Proposition~11.3.1]. However, in the present case we do not use a sequence of black hole rotations~$a_k$ for a fixed black hole mass~$M$, but rather we consider a sequence
\begin{equation}
(a_k,M_k)\in \mathcal{MS}_{l,\mu_{KG},m}.
\end{equation}
such that~$(a_k,M_k)\rightarrow (a,M)\in\mathcal{MS}_{l,\mu_{KG}}$.~(This is because the subextremal black hole parameter space of Kerr--de~Sitter pocesses one degree of freedom more than the subextremal black hole parameter space of Kerr.)

Let~$\psi$ be a fixed~$m$-solution of~$\Box_{g_{a,M,l}}\psi -\mu^2_{KG}\psi =0$. We define a sequence of functions
\begin{equation}
	\psi_{(k)}
\end{equation}
that solve the Klein--Gordon equation~\eqref{eq: kleingordon} for the black hole parameters~$(a_k,M_k,l)$ with the \underline{same initial data} as~$\psi$, see~\eqref{eq: proof prop: sec: continuity argument, prop 1.3, eq 1.1}. Note that thus~$\psi_{(k)}$ are future integrable by Lemma~\ref{lem: subsec: sec: continuity argument, subsec 0, lem 1}, and moreover~$\psi_{(k)}$ are fixed $m$-solutions for any~$k$~(which follows from the fact that the Killing vector field~$\partial_{\varphi}$ does not depend on the black hole parameters~$(a_k,M_k,l)$.)

By using Proposition~\ref{prop: sec: continuity argument, prop 2} for future integrable solutions we obtain 
	\begin{equation}
	\begin{aligned}
	\begin{aligned}
		&\int\int_{D(\tau_1,\tau_2)}  \tilde{\zeta}_{\textit{trap}}(r) \sum_{1\le i_1+i_2+i_3\le j}
		|\nabb^{i_1}(\partial_{t^\star})^{i_2}(Z^\star)^{i_3}\psi_{(k)}|^2 \\
		&\qquad\qquad\qquad\qquad+ \sum_{1\le i_1+i_2+i_3\le j-1}
		|\nabb^{i_1}(\partial_{t^\star})^{i_2}(Z^\star)^{i_3+1}\psi_{(k)}|^2+|\nabb^{i_1}(\partial_{t^\star})^{i_2}(Z^\star)^{i_3}\psi_{(k)}|^2 \\
		& \qquad   \leq B(j)\int_{\{t^\star= 0\}} \sum_{0 \leq i_1+i_2+i_3\leq j}
		|\nabb^{i_1}(\partial_{t^\star})^{i_2}(Z^\star)^{i_3}\psi|^2.
	\end{aligned}
	\end{aligned}
\end{equation}

Now, by using an~$N$ energy based energy estimate and the result of Proposition~\ref{prop: sec: continuity argument, prop 2} for a sufficiently small~$\epsilon_{\textit{trap}}>0$ we follow the exact steps of~\cite{DR3} and obtain that 
\begin{equation}\label{eq: cor: sec: continuity argument, cor 1, eq 4}
	\sup_{\tau\geq 0} \int_{\{t^\star=\tau\}} \sum_{0\leq i_1+i_2+i_3\leq j}|\slashed{\nabla}\partial_t^{i_t}(Z^\star)^{i_3}\psi_{(k)} |^2 \leq Β	\int_{\{t^\star=0\}} \sum_{0\leq i_1+i_2+i_3\leq j}|\slashed{\nabla}\partial_t^{i_t}(Z^\star)^{i_3}\psi_{(k)} |^2.
\end{equation}
Finally, we obtain
\begin{equation}
		\begin{aligned}
				\int_{\{t^\star=\tau_2 \}} \sum_{0\leq i_1+i_2+i_3\leq j}|\slashed{\nabla}\partial_t^{i_t}(Z^\star)^{i_3}\psi|^2 &	= \lim_{k\rightarrow \infty} 	\int_{\{t^\star=\tau_2 \}} \sum_{0\leq i_1+i_2+i_3\leq j}|\slashed{\nabla}\partial_t^{i_t}(Z^\star)^{i_3}\psi_{(k)}|^2\\
				&	\leq B \int_{\{t^\star= 0 \}} \sum_{0\leq i_1+i_2+i_3\leq j}
			|\nabb^{i_1}(\partial_{t^\star})^{i_2}(Z^\star)^{i_3}\psi|^2, 
		\end{aligned}
\end{equation}
where the first equality follows from well-posedness of the Klein--Gordon equation and the smooth dependence of~$g_{a,M,l}$ on~$(a,M)$ and the fast that~$\psi_{(k)}$ and~$\psi$ have the same initial data. The last inequality follows from Fatou's lemma. It follows from~\eqref{eq: cor: sec: continuity argument, cor 1, eq 4} that~$\psi$ satisfies~\eqref{eq: cor: subsec: sec: continuity argument, subsec 0, cor 1, eq 1}. Since~$\psi$ was arbitrary it follows that~$(a,M)\in \mathcal{MS}_{l,\mu_{KG},m}$. Thus~$\mathcal{MS}_{l,\mu_{KG},m}$ is closed in the topology of~$\mathcal{MS}_{l,\mu_{KG}}$, where recall from Proposition~\ref{prop: subsec: sec: continuity argument, subsec 1, prop 1} that~$\mathcal{MS}_{l,\mu_{KG}}$ is an open subset of~$\mathcal{B}_l$.

\boxed{\texttt{To prove the openess}} of~$\mathcal{MS}_{l,\mu_{KG},m}$ in the~$\mathcal{MS}_{l,\mu_{KG}}$ topology  we proceed as follows. We first note that~$\mathcal{MS}_{l,\mu_{KG}}$ is an open subset of~$\mathcal{B}_l$, see Proposition~\ref{prop: subsec: sec: continuity argument, subsec 1, prop 1}. We follow similar arguments that proved an analogous result in~[\cite{DR2}, Section~11, Proposition~11.2.1]. Namely, we want to prove that for
\begin{equation}
	(\mathring{a},\mathring{M})\in\mathcal{MS}_{l,\mu_{KG},m}
\end{equation}
there exists an~$\epsilon>0$ sufficiently small and for any rotation parameter~$a$ such that
\begin{equation}\label{eq: proof prop: sec: continuity argument, prop 1.3, eq 2}
|a-\mathring{a}|+|M-\mathring{M}| \leq \epsilon
\end{equation}
we also obtain~$(a,M)\in \mathcal{MS}_{l,\mu_{KG},m}$. One of the `key' ingredients to be able to repeat the arguments of~\cite{DR2} in the context of the present paper is the existence of the timelike vector field
\begin{equation}
	W_{\textit{mod}}
\end{equation}
see Lemma~\ref{lem: sec: continuity argument, lem 2}, which is Killing in the region where trapping happens for the fixed azimuthal frequency~$m$, see Proposition~\ref{prop: sec: continuity argument, prop 2} and that that the hypersurface leaves~$\{t^\star=c\}$ are independent of the black hole parameters~$a,M,l$, see Remarks,~\ref{rem: sec: the kerr de sitter spacetime, rem 1},~\ref{rem: sec: the kerr de sitter spacetime, rem 2}.

First, by using the divergence theorem with multiplier~$W_{\textit{mod}}$ and the fixed frequency Morawetz estimate of Proposition~\ref{prop: sec: continuity argument, prop 2}~(for the inhomogeneous Klein--Gordon with inhomogeneity~$F$) we obtain the following 
\begin{equation}\label{eq: proof prop: sec: continuity argument, prop 1.3, eq 3}
\begin{aligned}
&\int\int_{D(0,\tau_2)}  \tilde{\zeta}_{\textit{trap}}(r) \sum_{ i_1+i_2+i_3=1}
|\nabb^{i_1}(\partial_{t^\star})^{i_2}(Z^\star)^{i_3}\psi|^2 +(Z^\star\psi)^2\\
& \qquad   \leq B(j)\int_{\{t^\star=0\}} \sum_{0\leq i_1+i_2+i_3 \leq 1}
|\nabb^{i_1}(\partial_{t^\star})^{i_2}(Z^\star)^{i_3}\psi|^2 \\
& \qquad\qquad 	 + B(j,m)\int_{0}^{\tau_2}d\tau\left(\int_{\{t^\star=\tau\}\cap [r_+,r_{\Delta,\textit{frac}}-2\epsilon_{\textit{trap}})}+\int_{\{t^\star=\tau\}\cap (r_{\Delta,\textit{frac}}+2\epsilon_{\textit{trap}},\bar{r}_+]}\right) |\partial_{\varphi}\psi|^2 \\
&	\qquad\qquad +\int\int_{D(0,\tau_2)} |\partial_t\psi \cdot F|+|\partial_{\varphi}\psi \cdot F|  +|F|^2
\end{aligned}
\end{equation}
where 
\begin{equation}
\tilde{\zeta}_{\textit{trap}}(r)=
\begin{cases}
1,\qquad \{r_+\leq r\leq r_{\Delta,\textit{frac}}-2\epsilon_{\textit{trap}}\}\cup\{r_{\Delta,\textit{frac}}+2\epsilon_{\textit{trap}},\bar{r}_+\}\\
0,\qquad \{r_{\Delta,\textit{frac}}-\epsilon_{\textit{trap}}\leq r\leq r_{\Delta,\textit{frac}}+\epsilon_{\textit{trap}}\}.
\end{cases}
\end{equation}

Now, suppose that~$\delta_0>0$ is sufficiently small and define an interpolating metric 
\begin{equation}
	\tilde{g}_{\tau_2} =\chi_{\tau_2} g_{\mathring{a},\mathring{M},l} +(1-\chi_{\tau_2})g_{a,M,l}
\end{equation}
where~$\chi_{\tau_2}=0$ in the past of~$\{t^\star=\tau_2-\delta_0\}$ and~$\chi=1$ in the future of~$\{t^\star=\tau_2\}$, where moreover~$\partial_{\varphi^\star}\chi_{\tau_2} \equiv 0$. We define~$\tilde{\psi}_{\tau_2}$ to be the intepolating solution of the wave equation
\begin{equation}
	\Box_{\tilde{g}_{\tau_2}}\tilde{\psi}_{\tau_2}=0
\end{equation}
with \underline{the same initial data} as~$\psi$. We note that~$\partial_{\varphi^\star}$ is a Killing vector field for~$\tilde{g}_{\tau_2}$ and therefore the intepolating solution will be supported on the same azimuthal frequency~$m$ as the original solution~$\psi$. Hence, since~$(\mathring{a},\mathring{M},l)\in \mathcal{MS}_{l,\mu_{KG},m}$ and by~\eqref{eq: cor: sec: continuity argument, cor 1, eq 1} we obtain that~$\tilde{\psi}_{\tau_2}$ is future integrable with respect to~$(\mathring{a},\mathring{M},l)$. We write
\begin{equation}\label{eq: proof prop: sec: continuity argument, prop 1.3, eq 4}
	\Box_{g_{\mathring{a},\mathring{M},l}}\tilde{\psi}_{\tau_2} = \left( \Box_{g_{\mathring{a},\mathring{M},l}}-\Box_{\tilde{g}_{\tau_2}}\right) \tilde{\psi}_{\tau_2} ~ \dot{=}~ F
\end{equation}
and approximate as follows 
\begin{equation}\label{eq: proof prop: sec: continuity argument, prop 1.3, eq 4.1}
	|\left(\Box_{g_{\mathring{a},\mathring{M},\mathring{l}}}-\Box_{\tilde{g}_{\tau_2}} \right)\tilde{\psi}_{\tau_2}|^2\leq \epsilon^2 \sum_{1\leq i_1+i_2+i_3\leq 2} \left|\slashed{\nabla}^{i_1}\partial_t^{i_2}(Z^\star)^{i_3}\tilde{\psi}_{\tau_2} \right|^2
\end{equation}
by using the smooth dependence of~$g_{a,M,l}$ on~$a,M$, where for~$\epsilon>0$ see~\eqref{eq: proof prop: sec: continuity argument, prop 1.3, eq 2}. The  angular derivative~$\slashed{\nabla}$ is with respect to the metric~$g_{\mathring{a},\mathring{M}}$.

Keeping in mind that~\eqref{eq: proof prop: sec: continuity argument, prop 1.3, eq 4} is supported in the past of~$\{t^\star=\tau_2\}$, then we use the integrated estimate~\eqref{eq: proof prop: sec: continuity argument, prop 1.3, eq 3} to obtain
\begin{equation}\label{eq: proof prop: sec: continuity argument, prop 1.3, eq 5}
	\begin{aligned}
		& \int_{0}^{\tau_2}d\tau\left(\int_{\{t^\star=\tau\}\cap [r_+,r_{\Delta,\textit{frac}}-2\epsilon_{\textit{trap}})}+\int_{\{t^\star=\tau\}\cap (r_{\Delta,\textit{frac}}+2\epsilon_{\textit{trap}},\bar{r}_+]}\right)\left(J^n_\mu[\psi]n^\mu+ |\partial_{\varphi}\psi|^2 \right) \\
		&	\qquad  \leq 	B(m)\epsilon^2 \Big( \int_0^{\tau_2}  \int_{\{t^\star=\tau^\prime\}} \sum_{1\leq i_1+i_2+i_3\leq 2} \left|\slashed{\nabla}^{i_1}\partial_t^{i_2}(Z^\star)^{i_3}\tilde{\psi}_{\tau_2} \right|^2d\tau^\prime +\int_{\{t^\star=\tau_2-\delta_0\}}\sum_{1\leq i_1+i_2+i_3\leq 2} \left|\slashed{\nabla}^{i_1}\partial_t^{i_2}(Z^\star)^{i_3}\tilde{\psi}_{\tau_2} \right|^2 \Big) \\
		&	\qquad\qquad 	+	B(m)\int_{\{t^\star=0\}} J^n_\mu[\psi]n^\mu,
	\end{aligned}
\end{equation}
where we also used finite in time energy estimates, see Lemma~\ref{lem: subsec: sec: the main proposition, subsec 0, lem 1}, and an easy domain of dependence argument.

Now, we note the integrated estimate~\eqref{eq: prop: sec: continuity argument, prop 2, eq 2} of Proposition~\ref{prop: sec: continuity argument, prop 2}~(the fixed azimuthal frequency Morawetz and boundedness estimate) for~$j\geq 3$ 
\begin{equation}\label{eq: proof prop: sec: continuity argument, prop 1.3, eq 6}
	\begin{aligned}
		&	\int_{\{t^\star=\tau_2\}} \sum_{1\leq i_1+i_2+i_3\leq j}
		|\nabb^{i_1}(\partial_{t^\star})^{i_2}(Z^\star)^{i_3}\psi|^2 \\
		&\quad+\int\int_{D(0,\tau_2)}  \tilde{\zeta}_{\textit{trap}}(r) \sum_{1\le i_1+i_2+i_3\le j}
		|\nabb^{i_1}(\partial_{t^\star})^{i_2}(Z^\star)^{i_3}\psi|^2 \\
		&\qquad\qquad\qquad\qquad+ \sum_{1\le i_1+i_2+i_3\le j-1}
		|\nabb^{i_1}(\partial_{t^\star})^{i_2}(Z^\star)^{i_3+1}\psi|^2+|\nabb^{i_1}(\partial_{t^\star})^{i_2}(Z^\star)^{i_3}\psi|^2 \\
		&	\qquad \leq B(m,j)\int_{\{t^\star=0\}} \sum_{0 \leq i_1+i_2+i_3\leq j}
		|\nabb^{i_1}(\partial_{t^\star})^{i_2}(Z^\star)^{i_3}\psi|^2 \\
		&	\qquad\qquad   + B(m,j)\int_{0}^{\tau_2}d\tau\left(\int_{\{t^\star=\tau\}\cap [r_+,r_{\Delta,\textit{frac}}-2\epsilon_{\textit{trap}})}+\int_{\{t^\star=\tau\}\cap (r_{\Delta,\textit{frac}}+2\epsilon_{\textit{trap}},\bar{r}_+]}\right) |\partial_{\varphi}\psi|^2. \\
	\end{aligned}
\end{equation}

We use the bound~\eqref{eq: proof prop: sec: continuity argument, prop 1.3, eq 5}, to absorb the bulk terms on the RHS of~\eqref{eq: proof prop: sec: continuity argument, prop 1.3, eq 6} in the LHS of~\eqref{eq: proof prop: sec: continuity argument, prop 1.3, eq 6}, for a sufficiently small~$\epsilon>0$. We obtain that for solutions supported in the fixed azimuthal frequency~$|m|>0$ the following holds
\begin{equation}
	\begin{aligned}
		& \sup_{\tau^\prime\in[0,\tau_2]} \int_{\{t^\star=\tau^\prime\}}  \sum_{1\leq i_1+i_2+i_3\leq j}  \left| \slashed{\nabla}^{i_1}\partial_t^{i_2}(Z^\star)^{i_3} \psi\right|^2 \\
		&	\qquad \leq B(j,m) \int_{\{t^\star=0\}}  \sum_{0 \leq i_1+i_2+i_3\leq j}  \left| \slashed{\nabla}^{i_1}\partial_t^{i_2}(Z^\star)^{i_3} \psi\right|^2 <\infty, \qquad \forall j\geq 3
	\end{aligned}
\end{equation}
for all~$1\leq \tau_2$. In view of~\eqref{eq: cor: sec: continuity argument, cor 1, eq 1} we conclude that~$\mathcal{MS}_{l,\mu_{KG},m}$ is open.

We have concluded that the set~$\mathcal{MS}_{l,\mu_{KG},m}$ is non empty and clopen in the~$\mathcal{MS}_{l,\mu_{KG}}$ topology and therefore~\eqref{eq: cor: sec: continuity argument, cor 1, eq 1.1} holds.

Therefore, we conclude that the solution~$\psi=\psi_m$ is future integrable. We conclude the proof of the Proposition.
\end{proof}

\subsection{Proof of Theorem~\ref{prop: sec: continuity argument, prop 1}}\label{subsec: sec: continuity argument, subsec 2}

We are ready to prove our main result of the present Section.

\begin{proof}[\textbf{Proof of Theorem~\ref{prop: sec: continuity argument, prop 1}}]

Note the azimuthal mode expansion
\begin{equation}
	\psi=\sum_m\psi_m,
\end{equation}
where two distinct azimuthal modes~$\psi_m$ are orthogonal in the~$l^2$ sense.

Then, for~$\psi$ as in the assumptions of Theorem~\ref{prop: sec: continuity argument, prop 1} we note from Proposition~\ref{prop: sec: continuity argument, prop 1.3} that each~$\psi_m$ is future integrable. Therefore it follows from Corollary~\ref{cor: subsec: summing in the redshift estimate, cor 1} that the energy estimates~\eqref{eq: main theorem 1, eq 1},~\eqref{eq: cor: main theorem 1, cor 1, eq 1} hold for~$\psi_m$. Orthogonality of two distinct~$\psi_m$ immediately implies that~\eqref{eq: main theorem 1, eq 1},~\eqref{eq: cor: main theorem 1, cor 1, eq 1} hold for~$\psi$.

Thus,~$\psi$ is sufficiently integrable, as desired. 
\end{proof}

\section{Proof of the Morawetz estimate of Theorem~\ref{main theorem 1}}\label{sec: proof of Theorem 1}

Now we are ready to prove the Morawetz estimate~\eqref{eq: main theorem 1, eq 1} of Theorem~\ref{main theorem 1}. We note that we have already proved Corollary~\ref{cor: subsec: summing in the redshift estimate, cor 1} and Theorem~\ref{prop: sec: continuity argument, prop 1} which will be used in the rest of the present Section.

We have the following

\begin{proof}[\textbf{Proof of the Morawetz estimate of Theorem~\ref{main theorem 1}}]

We fix~$\tau_1> 0$. Let~$\psi$ be a solution of the Klein--Gordon equation~\eqref{eq: kleingordon} that arises from smooth initial data on~$\{t^\star=0\}$. Then, in view of Theorem~\ref{prop: sec: continuity argument, prop 1}, the result of Corollary~\ref{cor: subsec: summing in the redshift estimate, cor 1} holds for~$\psi$.

 To conclude the Morawetz estimate~\eqref{eq: main theorem 1, eq 1} of Theorem~\ref{main theorem 1} we need to prove that the following holds in addition
\begin{equation}\label{eq: sec: proof of Theorem 1, eq 1}
	\begin{aligned}
			\int\int_{D(\tau_1,\tau_2)} \sum_m  \left( 1_{|m|>0}|\mathcal{F}_{m}\psi|^2\right)+\int\int_{D(\tau_1,\tau_2)} &(Z^\star\psi)^2 +(\partial_t\mathcal{P}_{\textit{trap}}[\chi\psi])^2+|\slashed{\nabla}\mathcal{P}_{\textit{trap}}[\chi\psi]|^2   \leq  B\int_{\{t^\star=\tau_1\}} J_{\mu}^{n}[\psi]n^{\mu},
	\end{aligned}    
\end{equation}
where for the operator~$\mathcal{P}_{trap}$ see~\eqref{eq: sec: main theorems, eq 1.1}.

It suffices to prove the above after assuming the following:
\begin{enumerate}
	\item  First, we make the assumption that the solution~$\psi$ is generated from smooth initial data on the hypersurface~$\{t^\star=0\}$. This assumption is needed to appeal to Theorem~\ref{prop: sec: continuity argument, prop 1}. It suffices to have this assumption in view of general density arguments.\\
	\item Second, for the case~$\mu^2_{KG}=0$ we make the assumption that the mean associated with the hypersurface~$\{t^\star=\tau_1\}$, satisfies the following
	\begin{equation}\label{eq: sec: proof of Theorem 1, eq 0}
		\underline{\psi}(\tau_1)~\dot{=}~\frac{1}{|\{t^\star=\tau_1\}|}\int_{\{t^\star=\tau_1\}}\psi=0. 
	\end{equation}
	It suffices to make this assumption since if~$\psi$ is a solution of the wave equation, see~\eqref{eq: kleingordon} with~$\mu^2_{KG}=0$, then~$\psi+c$ is also a solution, where~$c$ is any constant. 
\end{enumerate}

We use the main result of the continuity argument Section~\ref{sec: continuity argument}, see Theorem~\ref{prop: sec: continuity argument, prop 1}, and the result of Corollary~\ref{cor: subsec: summing in the redshift estimate, cor 1} to obtain that the following energy estimate holds
	\begin{equation}\label{eq: sec: proof of Theorem 1, eq 0.1}
	\begin{aligned}
			\int\int_{D(\tau_1,\tau_2)} \sum_m  \left( 1_{|m|>0}|\mathcal{F}_{m}\psi|^2\right)+\int\int_{D(\tau_1,\tau_2)} &|Z^\star\psi|^2 +|\partial_t\mathcal{P}_{\textit{trap}}[\chi_{\tau_1}\psi]|^2+|\slashed{\nabla}\mathcal{P}_{\textit{trap}}[\chi_{\tau_1}\psi]|^2  \\
			&\qquad\qquad \leq  B\int_{\{t^\star=\tau_1\}} J_{\mu}^{n}[\psi]n^{\mu}+ |\psi|^2,
	\end{aligned}
\end{equation} 
for all~$2\leq 1+\tau_1\leq \tau_2\leq \infty$, where for the operator~$\mathcal{P}_{\textit{trap}}$ see~\eqref{eq: sec: main theorems, eq 1.1}. When~$\mu^2_{KG}>0$ we note that the RHS of~\eqref{eq: sec: proof of Theorem 1, eq 0.1} is equivalent to~$\int_{\{t^\star=\tau_1\}}J^n_\mu[\psi]n^\mu$ with a constant
 that blows up in the limit~$\mu^2_{KG}\rightarrow 0$. For the wave equation case~$\mu^2_{KG}=0$ we use the assumption~\eqref{eq: sec: proof of Theorem 1, eq 0} and use the Poincare--Wirtinger inequality of Lemma~\ref{lem: subsec: poincare wirtinger, lem 1} to obtain~\eqref{eq: sec: proof of Theorem 1, eq 1}.

We have concluded~\eqref{eq: sec: proof of Theorem 1, eq 1.1}. 
\end{proof}

\section{Proof of the boundedness estimates of Theorem~\ref{main theorem 1}}\label{sec: proof of boundedness}

In this Section we prove the boundedness estimates~\eqref{eq: main theorem 1, eq 2},~\eqref{eq: main theorem 1, eq 3} of Theorem~\ref{main theorem 1}.

\begin{proof}[\textbf{The proof of the boundedness estimates of Theorem~\ref{main theorem 1}}]
	
	We fix~$\tau_1> 0$. Let~$\psi$ be a solution of the Klein--Gordon equation~\eqref{eq: kleingordon} that arises from smooth initial data on~$\{t^\star=0\}$. It suffices to have this assumption in view of general density arguments. Then, the results of Corollary~\ref{cor: subsec: summing in the redshift estimate, cor 1} and of Theorem~\ref{prop: sec: continuity argument, prop 1} hold for~$\psi$.

	To conclude the boundedness estimates~\eqref{eq: main theorem 1, eq 2},~\eqref{eq: main theorem 1, eq 3} of Theorem~\ref{main theorem 1} we need to prove that the following hold in addition
\begin{equation}\label{eq: sec: proof of Theorem 1, eq 1.1}
	\int_{\{t^\star=\tau_2\}} J^n_\mu[\psi]n^\mu  \leq C\int_{\{t^\star=\tau_1\}}  J^n_{\mu}[\psi]n^{\mu},
\end{equation}
\begin{equation}\label{eq: sec: proof of Theorem 1, eq 1.2}
	\int_{\{t^\star=\tau_2\}} J^n_\mu[\psi]n^\mu+|\psi|^2\leq C\int_{\{t^\star=\tau_1\}}J^n_\mu[\psi]n^\mu +|\psi|^2,
\end{equation}
\begin{equation}\label{eq: sec: proof of Theorem 1, eq 1.3}
	\left(\int_{\mathcal{H}^+\cap D(\tau_1,\tau_2)}+\int_{\bar{\mathcal{H}}^+\cap D(\tau_1,\tau_2)}\right)  J^n_{\mu}[\psi] n^\mu   \leq C\int_{\{t^\star=\tau_1\}}  J^n_{\mu}[\psi]n^{\mu}.
\end{equation}

Again, it suffices to prove the above assuming the following:
\begin{enumerate}
	\item First, in the case~$\mu^2_{KG}=0$ we make the assumption that the mean associated with the hypersurface~$\{t^\star=\tau_1\}$, satisfies the following
	\begin{equation}\label{eq: sec: proof of Theorem 1, eq 1.4}
		\underline{\psi}(\tau_1)~\dot{=}~\frac{1}{|\{t^\star=\tau_1\}|}\int_{\{t^\star=\tau_1\}}\psi=0. 
	\end{equation}
	It suffices to make this assumption since if~$\psi$ is a solution of the wave equation, see~\eqref{eq: kleingordon} with~$\mu^2_{KG}=0$, then~$\psi+c$ is also a solution, where~$c$ is any constant. \\
	\item Second, we make the assumption that~$\psi_0\equiv 0$, where~$\psi = \psi_0 +\sum_{|m|\neq 0}\psi_m$, and~$\psi_m$ is the $m$-th azimuthal mode. It suffices to make this assumption since in view of Proposition~\ref{prop:redshift vs superradiance estimate, for axisymmetric solutions} it is easy to obtain the desired boundedness estimates~\eqref{eq: sec: proof of Theorem 1, eq 1.1},~\eqref{eq: sec: proof of Theorem 1, eq 1.2},~\eqref{eq: sec: proof of Theorem 1, eq 1.3} for axisymmetric solutions. Moreover, this assumption will allow us to bound zeroth order bulk terms from data, see already~\eqref{eq: sec: proof of Theorem 1, eq 3}. 
\end{enumerate}

\begin{center}
	\textbf{Extending the solution}
\end{center}

For reasons that will become apparent in what follows we extend the solution~$\psi$ from~$\{t^\star\geq 0\}$ to a solution of~\eqref{eq: kleingordon}, that we will again denote by~$\psi$, on~$\mathcal{M}\cup \mathcal{H}^-\cup \bar{\mathcal{H}}^-\cup \mathcal{B}$, where~$\mathcal{H}^-,\bar{\mathcal{H}}^-,\mathcal{B}$ are respectively the past event horizon, the past cosmological horizon and the bifurcation spheres. We note that by a domain of dependence argument and by finite in time energy estimates we can extend the solution in such a way so that it satisfies the following two properties:

First, the extended solution satisfies 
\begin{equation}\label{eq: sec: proof of Theorem 1, eq 2}
	\int_{\hat{\Sigma}_0} J^n_\mu [\psi] n^\mu \leq B \int_{\{t^\star =0\}} J^n_\mu [\psi] n^\mu
\end{equation}
where~$\hat{\Sigma}_0$ is the image of~$\{t^\star=0\}$ under the Boyer--Lindquist coordinate defined map~$t \mapsto -t$.

Second, our extension of~$\psi$ satisfies inequality~\eqref{eq: sec: proof of Theorem 1, eq 2} for any spacelike hypersurface in the set~$J^+(\hat{\Sigma}_0)\cup J^-(\{t^\star=0\})$, in the place of~$\hat{\Sigma}_0$.

\begin{center}
	\textbf{Integrated local energy decay for the extended solution}
\end{center}

We note that the Boyer--Lindquist defined map~$t\mapsto -t$ and~$a\mapsto -a$ is an isometry. Therefore, the result of Corollary~\ref{cor: subsec: summing in the redshift estimate, cor 1} holds if one goes to the past instead of the future, also see Proposition~\ref{prop: sec: main theorems, prop 1} for the reflection symmetry of the set~$\mathcal{MS}_{l,\mu_{KG}}$ with respect to the map~$a \mapsto -a$. Namely we replace~$\{t^\star =\tau\}$ with~$\hat{\Sigma}_{\tau}$ and by keeping in mind~\eqref{eq: sec: proof of Theorem 1, eq 2} we obtain 
\begin{equation}\label{eq: sec: proof of Theorem 1, eq 3}
	\begin{aligned}
			\int_{-\infty}^\infty d\tau \int_{\{t^\star=\tau\}} & 		|\psi|^2+ |\partial_{r^\star}\psi|^2 +\zeta_{trap}(r)\left( |\partial_t\psi|^2+|\slashed{\nabla}\psi|^2 \right) \leq  B\int_{\{t^\star=\tau_1\}} J_{\mu}^{n}[\psi]n^{\mu}+ |\psi|^2,
	\end{aligned}
\end{equation}
where for~$\zeta_{trap}$ see Remark~\ref{rem: subsec: sec: main theorems, subsec 1, rem 1}. Note that the derivatives~$\partial_{r^\star},\partial_t,\slashed{\nabla}$ are regular on~$\mathcal{M}\cup \mathcal{H}^-\cup \bar{\mathcal{H}}^-\cup \mathcal{B}$. Although inequality~\eqref{eq: sec: proof of Theorem 1, eq 3} is useful, we need a slightly different integrated energy decay estimate with a more detailed degeneration at trapping, see already~\eqref{eq: sec: proof of Theorem 1, eq 4.2}. Again, note that we include the zeroth order term in the bulk on the RHS of~\eqref{eq: sec: proof of Theorem 1, eq 3} in view of the third assumption above and of Theorem~\ref{thm: sec: proofs of the main theorems, thm 3}.

We choose~$A_0,A_1>0$ and $\delta>0$ such that 
\begin{equation}
	r_+<A_0<A_0+\delta<A_0+2\delta<r_++\epsilon_{red}\qquad \bar{r}_+-\epsilon_{red}< A_1-2 \delta <A_1-\delta<A_1<\bar{r}_+,
\end{equation}
where for~$\epsilon_{red}>0$ see the redshift Proposition~\ref{prop: redshift on the event horizon}. Now, we smoothly cut-off the solution~$\psi$ of the Klein--Gordon equation as follows
\begin{equation}\label{eq: sec: proof of Theorem 1, eq 2.1}
	\tilde{\psi} = \psi\cdot \chi_{[A_0,A_1]}
\end{equation}
where
\begin{equation}
	\chi_{[A_0,A_1]}=
	\begin{cases}
		1,~r\in [A_0+\delta,A_1-\delta]\\
		0,~r\not\in [A_0,A_1]. 
	\end{cases}
\end{equation}
where~$\delta>0$ is sufficiently small. We note that 
\begin{equation}\label{eq: sec: proof of Theorem 1, eq 4}
	\Box \tilde{\psi} -\mu^2_{KG} \tilde{\psi} = \tilde{F} := \Box \chi_{[A_0,A_1]} \psi +2 \nabla \chi \cdot \nabla \psi,
\end{equation}
where it is easy to see that 
\begin{equation}\label{eq: sec: proof of Theorem 1, eq 4.1}
	\begin{aligned}
		\int_{-\infty}^\infty \int_{\{t^\star =\tau\}} |\tilde{F}|^2 \leq B	\int_{-\infty}^\infty \int_{\{t^\star =\tau\}}  \left( |\psi|^2 + |\partial_{r^\star}\psi|^2 \right) \leq B \int_{\{t^\star=\tau_1\}} J^n_\mu [\psi] n^\mu
	\end{aligned}
\end{equation}
where in the last inequality we used the integrated estimate~\eqref{eq: sec: proof of Theorem 1, eq 3}.

We now apply Carter's separation of variables, see Proposition~\ref{prop: Carters separation, radial part}, to equation~\eqref{eq: sec: proof of Theorem 1, eq 4} to obtain 
\begin{equation}
	\tilde{u}^{\prime\prime} +(\omega^2 -V)\tilde{u} = \tilde{H}
\end{equation} 
where~$\tilde{u}$ is the frequency localization of~$\sqrt{r^2+a^2}\tilde{\psi}$, and the support of~$\tilde{\psi}$ is inherited by~$\tilde{u}$. Specifically, we note that~$|\tilde{u}|^2 (-\infty) = 0$. Moreover, the RHS of all frequency localized multipliers, see Definition~\ref{def: currents} are~$\mathcal{O}(\tilde{H})$, and therefore are supported in~$[A_0,A_0+\delta]\cup [A_1-\delta,A_1]$. Therefore, we apply ode estimates of Section~\ref{sec: proof: prop: sec: proofs of the main theorems} to conclude the following 
\begin{equation}
	\begin{aligned}
		b \int_{-\infty}^\infty \int_{\{t^\star=\tau\}} \left( |\tilde{\psi}|^2 + |\partial_{r^\star}\tilde{\psi}|^2 +|\partial_t \mathcal{P}_{trap}\tilde{\psi}|^2 +|\slashed{\nabla}\mathcal{P}_{trap}\tilde{\psi}|^2\right) & \leq 	\int_{-\infty}^\infty \int_{\{t^\star =\tau\}\cap \left([A_0,A_0+\delta]\cup [A_1-\delta,A_1]\right)}  \left( |\psi|^2 + J^n_\mu [\psi] n^\mu\right), \\
	\end{aligned}
\end{equation}
where for the operator~$\mathcal{P}_{trap}$ see~\eqref{eq: sec: main theorems, eq 1.1}.

Now, for the wave equation case~$\mu^2_{KG}=0$ we use~\eqref{eq: sec: proof of Theorem 1, eq 3} and the assumption~\eqref{eq: sec: proof of Theorem 1, eq 1.4} on the vanishing of the mean to obtain
\begin{equation}\label{eq: sec: proof of Theorem 1, eq 4.2}
		b \int_{-\infty}^\infty \int_{\{t^\star=\tau\}} \left( |\tilde{\psi}|^2 + |\partial_{r^\star}\tilde{\psi}|^2 +|\partial_t \mathcal{P}_{trap}\tilde{\psi}|^2 +|\slashed{\nabla}\mathcal{P}_{trap}\tilde{\psi}|^2\right)	\leq B \int_{\{t^\star=\tau_1\}} J^n_\mu [\psi]n^\mu.
\end{equation}

\begin{center}
	\textbf{A decomposition}
\end{center}

In what follows we will use the integrated estimate~\eqref{eq: sec: proof of Theorem 1, eq 4.2} to prove the desired boundedness estimates. In short, we will decompose~$\tilde{\psi}$ into pieces~$\tilde{\psi}_i$ that experience trapping in sufficiently small regions of~$r$. 

First, we need the following definition

\begin{definition}\label{def: sec: proof of Theorem 1, def 1}
	Let~$\epsilon_{part}>0$ be sufficiently small and
	\begin{equation}
		k=k(\epsilon_{part})=[\epsilon_{part}^{-1} (R^+-R^-)]
	\end{equation}
	where for~$R^\pm$ see Theorem~\ref{thm: sec: proofs of the main theorems, thm 3}.

	We denote
	\begin{equation}
		a_1=R^-,\qquad b_{k}=R^+,
	\end{equation}
	and we define the intervals~$P_i=[a_i,b_i)$ for~$i=1,\dots, k$ where~$P_i\cap P_j=\emptyset$ if~$i \neq j$ to be such that 
	\begin{equation}
		\bigcup_{i=1}^{k(\epsilon_{part})} P_i = [R^-,R^+)
	\end{equation}
	where
	\begin{equation}
		|P_i|=\epsilon_{part}	
	\end{equation}
	for all~$i=1,\dots, k$. 
	
	Finally, we define 
	\begin{equation}
		\begin{aligned}
			\mathcal{C}_0 &	=	\{(\omega,m,\tilde{\lambda}):r_{trap}(\omega,m,\tilde{\lambda})=0\}\\
			\mathcal{C}_i &	=	\{(\omega,m,\ell):r_{trap}(\omega,m,\ell) \in P_i\}
		\end{aligned}
	\end{equation}
	for all~$i=1,\dots k$, where for~$r_{trap}$ see Theorem~\ref{thm: sec: proofs of the main theorems}. 
\end{definition}

\begin{remark}
	In view of the definition of~$r_{trap}$, see Theorem~\ref{thm: sec: proofs of the main theorems}, we note that each~$(\omega,m,\ell)$ lies in exactly one~$\mathcal{C}_i$. 
\end{remark}

We take~$\epsilon_{part}>0$ of Definition~\ref{def: sec: proof of Theorem 1, def 1} sufficiently small. We can easily construct~$\varphi_\tau$ invariant timelike vector fields
\begin{equation}
	V_i
\end{equation}
with $i=1,\dots k$, which are Killing in the regions~$[a_i-\frac{\epsilon_{part}}{10^4},b_i+\frac{\epsilon_{part}}{10^4}]$,~$i=1,\dots,k,$ for a sufficiently small~$\epsilon_{part}>0$, where recall that~$P_i=[a_i,b_i)$,~$b_i-a_i=\epsilon_{part}$ and~$r_+<a_i<b_i<\bar{r}_+$ for all~$i$.

We note that 
\begin{equation}\label{eq: sec: proof of Theorem 1, eq 2.1.0}
	\tilde{\psi}=\sum_{i=0}^k \tilde{\psi}_i,
\end{equation}
for
\begin{equation}\label{eq: sec: proof of Theorem 1, eq 2.1.1}
	\tilde{\psi}_i ~ \dot{=}~ \frac{1}{\sqrt{2\pi}} \int_{-\infty}^{+\infty} \sum_{m\ell} 1_{\mathcal{C}_i} e^{-i\omega t} e^{im\varphi} S^{(a\omega)}_{m\ell}   \tilde{\psi}^{(a\omega)}_{m\ell} d\omega
\end{equation}
where for the number~$k=k(\epsilon_{part})$ and for~$\mathcal{C}_i$ see Definition~\ref{def: sec: proof of Theorem 1, def 1}. Note that the equality~\eqref{eq: sec: proof of Theorem 1, eq 2.1.1} is to be understood in an~$Η^1$ sense if~$\mu^2_{KG}>0$ and in an~$\dot{H}^1$ sense if~$\mu^2_{KG}=0$. Furthermore, we have that
\begin{equation}
	\Box \tilde{\psi}_i -\mu^2_{KG} \tilde{\psi}_i = \tilde{F}_i,
\end{equation}
where~$\tilde{F}_i = \Box \chi_{[A_0,A_1]} \cdot \tilde{\psi}_i +2 \nabla \chi_{[A_0,A_1]}\cdot \nabla\tilde{\psi}_i$.

The following lemma is straightforward to prove with a pigeonhole principle argument

\begin{lemma}\label{lem: sec: proof of Theorem 1, lem 1}
	
	Let~$l>0$,~$(a,M)\in \mathcal{B}_l$ and~$\mu^2_{KG}\geq 0$. Let~$\epsilon>0$ be sufficiently small. Let~$\Psi$ be sufficiently integrable, see Definition~\ref{def: sec: carter separation, def 1}, and let~$\Psi\equiv 0$ in~$\{r\le r_+ +\epsilon\} \cup \{r\ge \bar{r}_+ -\epsilon \}$.

	Then, there exists a constant~$C=C(\Psi)>0$ and a dyadic sequence~$\{\tau_n\}_{n=1}^\infty$ such that~$\tau_n \rightarrow -\infty$ as~$n\rightarrow\infty$ and
	\begin{equation}
		\int_{\{t^\star=\tau_n\}} J^n_\mu [\Psi]n^\mu\leq \frac{C}{\tau_n}. 
	\end{equation}
\end{lemma}

Now, we note that~$\tilde{\psi}$ and~$\tilde{\psi}_i$ are sufficiently integrable, see Definition~\ref{def: sec: carter separation, def 1}. Since~$\tilde{\psi}_i$ is sufficiently integrable and supported in~$\{A_0\leq r\leq A_1 \}$ for any~$i=0,\dots ,k$ we note from Lemma~\ref{lem: sec: proof of Theorem 1, lem 1} that for any~$i=0,\dots ,k$ there exists a constant~$C_i$ and a dyadic sequence~$\tau_n^{(i)}$ such that for~$\tilde{\psi}_i$ as in~\eqref{eq: sec: proof of Theorem 1, eq 2.1.0} we obtain 
\begin{equation}\label{eq: sec: proof of Theorem 1, eq 2.1.2}
	\int_{\{t^\star=\tau_n^{(i)}\}} J^n_\mu [\tilde{\psi}_i]n^\mu\leq \frac{C_i}{|\tau_n^{(i)}|}, 
\end{equation}
where~$\tau_n^{(i)}\rightarrow -\infty$ as~$n\rightarrow \infty$.

We apply the energy identity asssociated with the vector field~$V_i$ and~$\tilde{\psi}_i$ between the hypersurfaces~$\{t^\star = \tau_n^{(i)}\}$,~$\{t^\star = \tau_2\}$. By recalling that the vector field~$V_i$ is Killing in the region~$[a_i-\frac{\epsilon_{part}}{10^4},b_i+\frac{\epsilon_{part}}{10^4}]$, and that in the region~$[a_i-\frac{\epsilon_{part}}{10^4},b_i+\frac{\epsilon_{part}}{10^4}]^c$ and for the frequencies~$(\omega,m,\ell)\in\mathcal{C}_i$ we have~$|r-r_{trap}|\geq b(\epsilon_{part})$, see also the definition~\eqref{eq: sec: main theorems, eq 1.1} of~$\mathcal{P}_{trap}$ and the decomposition~\eqref{eq: sec: proof of Theorem 1, eq 2.1.1}, we obtain the following
\begin{equation}\label{eq: sec: proof of Theorem 1, eq 2.2}
	\begin{aligned}
		\int_{\{t^\star=\tau_2\}} J^{V_i}_\mu [\tilde{\psi_i}]n^\mu &	\leq B \int_{\tau_n^{(i)}}^{\tau_2} \int_{\{t^\star=s\}\cap [a_i-\frac{\epsilon_{part}}{10^4},b_i+\frac{\epsilon_{part}}{10^4}]^c} J^{V_i}_\mu [\tilde{\psi}_i]n^\mu + \int_{\{t^\star= \tau_n^{(i)}\}} J^{V_i}_\mu [\tilde{\psi}_i]n^\mu \\
		&	\leq B \int_{-\infty}^\infty \int_{\{t^\star=\tau\}\cap[a_i-\frac{\epsilon_{part}}{10^4},b_i+\frac{\epsilon_{part}}{10^4}]^c} J^n_\mu [\mathcal{P}_{trap}\tilde{\psi}_i]n^\mu  + \frac{BC_i}{\tau_n^{(i)}}\\
		&	\leq B \int_{\{t^\star=\tau_1\}} J^n_\mu [\psi]n^\mu  + \frac{BC_i}{|\tau_n^{(i)}|}
	\end{aligned}
\end{equation}
for~$\tau_1 <\tau_2$, where we also used~\eqref{eq: sec: proof of Theorem 1, eq 2.1.2} and the already proven integrated decay estimate~\eqref{eq: sec: proof of Theorem 1, eq 4.2}. Now, by taking~$n\rightarrow \infty$ we conclude that 
\begin{equation}\label{eq: sec: proof of Theorem 1, eq 2.3}
	\int_{\{t^\star=\tau_2\}} J^n [\tilde{\psi}]n^\mu \leq B \int_{\{t^\star=\tau_1\}} J^n _\mu [\psi] n^\mu 
\end{equation}
after summing over~$i=0,\dots, k$. Therefore, we have proved the boundedness for~$\psi$ away from the horizons~$\mathcal{H}^+,\bar{\mathcal{H}}^+$. 

Now, it is an easy application of the redshift estimate, see Proposition~\ref{prop: redshift estimate} and of the already proven Morawetz estimate~\eqref{eq: sec: proof of Theorem 1, eq 1} to prove that 
\begin{equation}\label{eq: sec: proof of Theorem 1, eq 2.4}
	\begin{aligned}
		\int_{\{t^\star=\tau_2\}\cap [r_+,A_0+\delta]} J^n_\mu[\psi] n^\mu &\leq \int_{\{t^\star=\tau_1\}} J^n_\mu [\psi] n^\mu +B\int_{\tau_1}^{\tau_2} ds \int_{\{t^\star=s\} \cap [A_0+\delta,A_0+2\delta]} J^n_\mu [\psi] n^\mu\\
		&	 \leq B \int_{\{t^\star=\tau_1\}} J^n_\mu [\psi] n^\mu,
	\end{aligned}
\end{equation}
in view of the fact that~$J^N_\mu [\psi] n^\mu \sim J^n_\mu [\psi] n^\mu$ in the region~$\{r_+\leq r\leq A_0+\delta\}$. We proceed similarly to prove
\begin{equation}\label{eq: sec: proof of Theorem 1, eq 2.5}
	\int_{\{t^\star=\tau_2\}\cap [A_1-\delta,\bar{r}_+]} J^n_\mu[\psi] n^\mu  \leq B \int_{\{t^\star=\tau_1\}} J^n_\mu [\psi] n^\mu. 
\end{equation}

Therefore, in view of~\eqref{eq: sec: proof of Theorem 1, eq 2.3} and~\eqref{eq: sec: proof of Theorem 1, eq 2.4},~\eqref{eq: sec: proof of Theorem 1, eq 2.5} we conclude~\eqref{eq: sec: proof of Theorem 1, eq 1.1}.

For the case~$\mu^2_{KG}>0$ we conclude~\eqref{eq: sec: proof of Theorem 1, eq 1.2} from~\eqref{eq: sec: proof of Theorem 1, eq 1.1}. For the case~$\mu^2_{KG}=0$, in view of the assumption~\eqref{eq: sec: proof of Theorem 1, eq 1.4} and the Poincare--Wirtinger inequality of Lemma~\ref{lem: subsec: poincare wirtinger, lem 1}, we conclude~\eqref{eq: sec: proof of Theorem 1, eq 1.2} from~\eqref{eq: sec: proof of Theorem 1, eq 1.1}.

Finally, in order to obtain the boundedness of the horizon fluxes~\eqref{eq: sec: proof of Theorem 1, eq 1.3} we appeal to the redshift estimate of Proposition~\ref{prop: redshift estimate} and the integrated estimate~\eqref{eq: sec: proof of Theorem 1, eq 4.2}. 

\end{proof}

\section{The axisymmetric case}\label{sec: proof of main theorem in axisymmetry}

We prove the Morawetz and boundedness estimates of Theorem~\ref{main theorem 1}, for axisymmetric solutions of the Klein--Gordon equation~\eqref{eq: kleingordon}. Note that we do not appeal to a continuity argument.

\begin{proof}

First, the boundedness estimates
\begin{equation}\label{eq: proof: sec: proof of main theorem in axisymmetry, eq 0}
	\begin{aligned}
		\int_{\{t^\star=\tau_2\}} J_\mu^{n}[\psi]n^\mu & \leq C \int_{\{t^\star=\tau_1\}}J_\mu^{n}[\psi]n^\mu \\
		\int_{\{t^\star=\tau_2\}} J^n_\mu[\psi]n^\mu +|\psi|^2 & \leq C \int_{\{t^\star=\tau_1\}} J^n_\mu[\psi]n^\mu+|\psi|^2
	\end{aligned}    
\end{equation}
 follow immediately from Proposition~\ref{prop:redshift vs superradiance estimate, for axisymmetric solutions}.

Second, we prove the Morawetz estimate~\eqref{eq: sec: proof of Theorem 1, eq 1}. In the case~$\mu^2_{KG}=0$ it suffices to make the assumption
\begin{equation}\label{eq: proof: sec: proof of main theorem in axisymmetry, eq 1}
	\underline{\psi}(\tau_1)~\dot{=}~\frac{1}{|\{t^\star=\tau_1\}|}\int_{\{t^\star=\tau_1\}}\psi=0. 
\end{equation}
It suffices to make this assumption since if~$\psi$ is a solution of the wave equation, see~\eqref{eq: kleingordon} with~$\mu^2_{KG} = 0$, then~$\psi+c$ is also a solution, where~$c$ is any constant.

Now, in view of Lemma~\ref{lem: subsec: sec: continuity argument, subsec 0, lem 1} and the now established boundedness estimates~\eqref{eq: main theorem 1, eq 2},~\eqref{eq: main theorem 1, eq 3} we have that~$\psi$ is future  integrable. Therefore, we can now use the result of Theorem~\ref{thm: subsec: summing in the redshift estimate, thm -1} and obtain
\begin{equation}\label{eq: proof: sec: proof of main theorem in axisymmetry, eq 2}
	\begin{aligned}
			\int\int_{D(\tau_1,\tau_2)} &\mu^2_{\textit{KG}}
		|\psi|^2+  |\partial_{r^\star}\psi|^2+ |Z^\star\psi|^2+\zeta_{\textit{trap}}(r)\left(|\partial_t\psi|^2+|\slashed{\nabla}\psi|^2\right)\\
		&	\qquad  \leq  B\int_{\{t^\star=\tau_1\}} J_{\mu}^{n}[\psi]n^{\mu}+|\psi|^2,
	\end{aligned}
\end{equation}
where~$\zeta_{\textit{trap}}(r)=\left(1-\frac{r_{\Delta,\textit{trap}}}{r}\right)^2.$

In the case~$\mu^2_{KG}>0$ we conclude the desired Morawetz estimate from~\eqref{eq: proof: sec: proof of main theorem in axisymmetry, eq 2}. In the case~$\mu^2_{KG}=0$, in view of the assumption~\eqref{eq: proof: sec: proof of main theorem in axisymmetry, eq 1} and the Poincare--Wirtinger inequality of Lemma~\ref{lem: subsec: poincare wirtinger, lem 1} we obtain 
\begin{equation}\label{eq: proof: sec: proof of main theorem in axisymmetry, eq 3}
	\begin{aligned}
		\int\int_{D(\tau_1,\tau_2)} &\mu^2_{\textit{KG}}
		|\psi|^2+  |\partial_{r^\star}\psi|^2+ |Z^\star\psi|^2+\zeta_{\textit{trap}}(r)\left(|\partial_t\psi|^2+|\slashed{\nabla}\psi|^2\right)\\
		&	\qquad  \leq  B\int_{\{t^\star=\tau_1\}} J_{\mu}^{n}[\psi]n^{\mu},
	\end{aligned}
\end{equation}
and we conclude the desired Morawetz estimate as well. 

Now, in view of the Morawetz estimate~\eqref{eq: proof: sec: proof of main theorem in axisymmetry, eq 3}, the redshift estimate of Proposition~\ref{prop: redshift estimate} and the estimates~\eqref{eq: proof: sec: proof of main theorem in axisymmetry, eq 0} we obtain the boundedness estimates~\eqref{eq: main theorem 1, eq 2},~\eqref{eq: main theorem 1, eq 3}. 
\end{proof}

\section{The critical points of the potential~\texorpdfstring{$V_0$}{V}}\label{sec: trapping}

The rest of this paper is dedicated to the formulation and proof of Theorem~\ref{thm: sec: proofs of the main theorems, thm 3}, of Section~\ref{sec: the main theorem, no cases}. We will need the present preliminary Section as well as Sections~\ref{sec: frequency localized multiplier estimates},~\ref{sec: frequencies} as well. Theorem~\ref{thm: sec: proofs of the main theorems, thm 3} will then be proven in Section~\ref{sec: proof: prop: sec: proofs of the main theorems} as a Corollary of Theorem~\ref{thm: sec: proofs of the main theorems}.

The main result of the present Section is Lemma~\ref{lem: subsec: sec: trapping, subsec 1, lem 1}. Specifically, we study the properties of the potential~$V_0$, see~\eqref{eq: the potentials V0,VSl,Vmu}, for fixed frequencies~$(\omega,m,\ell)$, and prove that~$(r^2+a^2)^3\frac{dV_0}{dr}$ attains at most one critical point in~$[r_+,\bar{r}_+]$.

\begin{lemma}\label{lem: subsec: sec: trapping, subsec 1, lem 1}
	Let $l>0$ and $(a,M)\in\mathcal{B}_l$. Then, the function
	\begin{equation}\label{eq: lem: /sec: Carters separation/ the critical points of V0, eq 0}
		(r^2+a^2)^3\frac{dV_0}{dr},
	\end{equation}
	where~$V_0$ is the potential from equation~\eqref{eq: the potentials V0,VSl,Vmu}, attains at most one critical point in the interval~$[r_+,\bar{r}_+]$, which is moreover a local maximum. Furthermore, if the local maximum exists, then we also have
	\begin{equation}\label{eq: lem: subsec: sec: trapping, subsec 1, lem 1, eq 1}
		\frac{d}{dr}\left( (r^2+a^2)^3 \frac{dV_0}{dr}\right)(r_{V_0,max}) <0. 
	\end{equation}

	Finally, the potential 
	\begin{equation}
		V_0
	\end{equation}
	attains at most two critical points in~$[r_+,\bar{r}_+]$ that satisfy
	\begin{equation}\label{eq: lem: /sec: Carters separation/ the critical points of V0, eq 1}
		r_+\leq r_{V_0,\textit{min}}< r_{V_0,\textit{max}}\leq\bar{r}_+,
	\end{equation}
	which are respectively a local minimum and a local maximum.
\end{lemma}

\begin{proof}
	First, note the following 
	\begin{equation}\label{eq: /lem proof /sec: Carters separation/ the critical points of V0/ eq 1}
		\begin{aligned}
			\frac{d}{dr}\left((r^2+a^2)^3\frac{dV_0}{dr}\right)= -6\Xi \left(\tilde{\lambda}-2m\omega a\frac{a^2}{l^2}\right)r^2+12M(\tilde{\lambda}-2m\omega a\Xi)r-2a^2\tilde{\lambda}+(2am\Xi)^2+\frac{4 a^5 m \Xi  \omega }{l^2}.
		\end{aligned}
	\end{equation}
	The polynomial of equation \eqref{eq: /lem proof /sec: Carters separation/ the critical points of V0/ eq 1} enjoys the same roots as the polynomial 
	\begin{equation}\label{eq: /lem proof /sec: Carters separation/ the critical points of V0/ eq 2}
		r^2-\frac{2M(\tilde{\lambda}-2m\omega a \Xi)}{\Xi\left(\tilde{\lambda}-2m\omega a \frac{a^2}{l^2}\right)}r+\frac{2a^2\tilde{\lambda}\Xi-(2am\Xi)^2 - \frac{4 a^5 m \Xi  \omega }{l^2}}{6\Xi(\tilde{\lambda}-2m\omega a \frac{a^2}{l^2})}.
	\end{equation}
	The roots of the polynomial~\eqref{eq: /lem proof /sec: Carters separation/ the critical points of V0/ eq 2} are 
	\begin{equation}\label{eq: /lem proof /sec: Carters separation/ the critical points of V0/ eq 3}
		r_{0,\pm}=\frac{M(\tilde{\lambda}-2m\omega a \Xi)}{\Xi (\tilde{\lambda}-2m\omega a \frac{a^2}{l^2})}\pm \sqrt{\left(\frac{M(\tilde{\lambda}-2m\omega a \Xi)}{\Xi (\tilde{\lambda}-2m\omega a \frac{a^2}{l^2})}\right)^2-\frac{2a^2\Xi\tilde{\lambda}-(2am\Xi)^2-\frac{4 a^5 m \Xi  \omega }{l^2}}{6\Xi \left(\tilde{\lambda}-2m\omega a \frac{a^2}{l^2}\right)}}.
	\end{equation}
	
Recall from Lemma~\ref{lem: inequality for lambda} that the following hold
	\begin{equation}\label{eq: /lem proof /sec: Carters separation/ the critical points of V0/ eq 3.9}
		\tilde{\lambda}-2m\omega a \Xi>0,\qquad\tilde{\lambda}-2m\omega a \frac{a^2}{l^2}>0.
	\end{equation}

	For the case~$a m\omega \geq 0$ we obtain that since~$	\frac{a^2}{l^2}<\Xi$ then we have
	\begin{equation}
	\tilde{\lambda}-2m\omega a \frac{a^2}{l^2}>\tilde{\lambda}-2m\omega a \Xi\implies \frac{\tilde{\lambda}-2m\omega a\Xi}{\tilde{\lambda}-2m\omega a\frac{a^2}{l^2}}<1,
	\end{equation}
	which implies
	\begin{equation}
		\Re (r_{0,-})<\frac{M}{\Xi}<M<r_+,
	\end{equation}
	which concludes that~\eqref{eq: lem: /sec: Carters separation/ the critical points of V0, eq 0} attains at most one critical point for the case~$am\omega \geq 0$ since then the only potential root in~$[r_+,\bar{r}_+]$ is $r_{0,+}$.

	For the case~$a m\omega <0$ we proceed as follows. First, we write~$r_{0,-}$ in the following manner
	\begin{equation}\label{eq: /lem proof /sec: Carters separation/ the critical points of V0/ eq 4}
		r_{0,-}=\frac{M(\tilde{\lambda}-2m\omega a \Xi)}{\Xi \left(\tilde{\lambda}-2m\omega a \frac{a^2}{l^2}\right)}\left(1-\sqrt{1-\frac{\left(2a^2\Xi\tilde{\lambda}-(2am\Xi)^2-\frac{4a^5 m\Xi \omega}{l^2}\right)\Xi\left(\tilde{\lambda}-2m\omega a \frac{a^2}{l^2}\right)}{6M^2\left(\tilde{\lambda}-2m\omega a\Xi\right)^2}}\right). 
	\end{equation}

	We have
	\begin{equation}\label{eq: /lem proof /sec: Carters separation/ the critical points of V0/ eq 5}
		\begin{aligned}
			&   \frac{\left(2a^2\Xi\tilde{\lambda}-(2am\Xi)^2-\frac{4a^5m\Xi\omega}{l^2}\right)\Xi\left(\tilde{\lambda}-2m\omega a \frac{a^2}{l^2}\right)}{6M^2\left(\tilde{\lambda}-2m\omega a\Xi\right)^2}=\frac{a^2\Xi^2\left(\tilde{\lambda}-2m^2\Xi-\frac{2a^3m\omega}{l^2}\right)\left(\tilde{\lambda}-2m\omega a \frac{a^2}{l^2}\right)}{3M^2\left(\tilde{\lambda}-2m\omega a\Xi\right)^2}\\
			&	\quad =\frac{a^2\Xi^2}{3M^2}\frac{\left(\tilde{\lambda}-2m^2\Xi-\frac{2a^3m\omega}{l^2}\right)}{\left(\tilde{\lambda}-2m\omega a\Xi\right)}\frac{\left(\tilde{\lambda}-2m\omega a \frac{a^2}{l^2}\right)}{\left(\tilde{\lambda}-2m\omega a\Xi\right)}.\\
		\end{aligned}
	\end{equation}
By recalling the bounds
\begin{equation}
	\frac{a^2}{l^2}<\frac{1}{4},\quad \left|\frac{a}{M}\right|\leq \frac{12}{10}
\end{equation}
from Lemma~\ref{lem: sec: properties of Delta, lem 2, a,M,l},
and that~\eqref{eq: /lem proof /sec: Carters separation/ the critical points of V0/ eq 3.9} hold we obtain 
	\begin{equation}
		\frac{a^2\Xi^2}{3M^2}\leq \frac{3}{4}
	\end{equation}
	and 
	\begin{equation}
		\frac{\left(\tilde{\lambda}-2m^2\Xi-\frac{2a^3m\omega}{l^2}\right)}{\left(\tilde{\lambda}-2m\omega a\Xi\right)}\frac{\left(\tilde{\lambda}-2m\omega a \frac{a^2}{l^2}\right)}{\left(\tilde{\lambda}-2m\omega a\Xi\right)}\leq 1.
	\end{equation}
	Therefore, we bound the right hand side of \eqref{eq: /lem proof /sec: Carters separation/ the critical points of V0/ eq 5} as follows
	\begin{equation}\label{eq: /lem proof /sec: Carters separation/ the critical points of V0/ eq 6}
		\textit{RHS of }\eqref{eq: /lem proof /sec: Carters separation/ the critical points of V0/ eq 5}\leq \frac{3}{4}.
	\end{equation}
	
	Now, we note the inequality
	\begin{equation}\label{eq: /lem proof /sec: Carters separation/ the critical points of V0/ eq 7}
		\sqrt{1-x}\geq 1-\frac{2}{3}x,\qquad\textit{for}\:\: 0\leq x\leq \frac{3}{4},
	\end{equation}
	where the equality holds for 
	\begin{equation}
		x=0,\qquad x=\frac{3}{4}.
	\end{equation}

	Νow that we have established the bound \eqref{eq: /lem proof /sec: Carters separation/ the critical points of V0/ eq 6}, we use inequality \eqref{eq: /lem proof /sec: Carters separation/ the critical points of V0/ eq 7} to obtain the following bound 
	\begin{equation}
		\begin{aligned}
			r_{0,-}&    \leq \frac{M(\tilde{\lambda}-2m\omega a \Xi)}{\Xi \left(\tilde{\lambda}-2m\omega a \frac{a^2}{l^2}\right)}\frac{2}{3}\frac{\frac{2a^2\Xi\tilde{\lambda}-(2am\Xi)^2-\frac{4a^5 m\omega\Xi}{l^2}}{6\Xi\left(\tilde{\lambda}-2m\omega a \frac{a^2}{l^2}\right)}}{\left(\frac{M(\tilde{\lambda}-2m\omega a \Xi)}{\Xi \left(\tilde{\lambda}-2m\omega a \frac{a^2}{l^2}\right)}\right)^2}=\frac{2}{3}\frac{\frac{2a^2\tilde{\lambda}\Xi-(2am\Xi)^2-\frac{4a^5m\omega\Xi}{l^2}}{6\Xi\left(\tilde{\lambda}-2m\omega a \frac{a^2}{l^2}\right)}}{\left|\frac{M(\tilde{\lambda}-2m\omega a \Xi)}{\Xi \left(\tilde{\lambda}-2m\omega a \frac{a^2}{l^2}\right)}\right|}\\
			&  =\frac{2}{9}\frac{a^2}{M}\frac{\tilde{\lambda}\Xi-2m^2\Xi^2-2\frac{a^3}{l^2}m\omega\Xi}{\tilde{\lambda}-2m\omega a \Xi}\\
			&	<\frac{2}{9}\frac{a^2}{M}\frac{2\tilde{\lambda}-2\frac{a^2}{l^2}am\omega\Xi}{\tilde{\lambda}-2m\omega a \Xi}\\
			&<\frac{4}{9}\frac{a^2}{M^2}M<\frac{4}{9}\left(\frac{12}{10}\right)^2M< M<r_+	\\
		\end{aligned}
	\end{equation}
	where we used that~$\frac{2\tilde{\lambda}-2\frac{a^2}{l^2}am\omega\Xi}{\tilde{\lambda}-2m\omega a \Xi}<2$, which concludes that~\eqref{eq: lem: /sec: Carters separation/ the critical points of V0, eq 0} attains at most one critical point for the case~$am\omega <0$ since the only potential root in~$[r_+,\bar{r}_+]$ is~$r_{0,+}$.

	Now, in view of the already proven
	\begin{equation}
		r_{0,-}<r_+<r_{0,+}
	\end{equation}
	we conclude that 
	\begin{equation}\label{eq: /lem proof /sec: Carters separation/ the critical points of V0/ eq 8}
		\frac{d}{dr}\left((r^2+a^2)^3\frac{dV_0}{dr}\right)(r=r_{0,+})<0,
	\end{equation} 
	and therefore~$r_{0,+}$ is a local maximum for the function 
	\begin{equation}
		(r^2+a^2)^3\frac{dV_0}{dr}.
	\end{equation}

	Finally, to conclude that indeed $V_0$ attains at most two critical points in~$[r_+,\bar{r}_+]$ which satisfy 
	\begin{equation}
		r_+\leq r_{V_0,\textit{min}}<r_{V_0,\textit{max}}\leq\bar{r}_+
	\end{equation}
	we inspect the form of the function 
	\begin{equation}
		\frac{d}{dr}\left((r^2+a^2)^3\frac{dV_0}{dr}\right),
	\end{equation}
	see~\eqref{eq: /lem proof /sec: Carters separation/ the critical points of V0/ eq 1}, and also use that 
	\begin{equation}
		(r^2+a^2)^3\frac{dV_0}{dr}
	\end{equation}
	attains at most one critical point in $[r_+,\bar{r}_+]$. We conclude~\eqref{eq: lem: subsec: sec: trapping, subsec 1, lem 1, eq 1} from~\eqref{eq: /lem proof /sec: Carters separation/ the critical points of V0/ eq 8}.

	We conclude the Lemma. 
\end{proof}

\section{The currents and the frequency localized multipliers}\label{sec: frequency localized multiplier estimates}

We use the notation
\begin{equation}
^\prime=\frac{d}{dr^\star},
\end{equation}
see Section~\ref{subsec: tortoise coordinate} for the tortoise coordinate.

\subsection{The currents for~\texorpdfstring{$u$}{u}}\label{subsec: currents}

We define the following

\begin{definition}\label{def: currents}
	Let $V$ be the potential of Carter's separation of variables, see Proposition~\ref{prop: Carters separation, radial part}. We define the currents
	\begin{equation}\label{eq: def: currents, eq 1}
	\begin{aligned}
	Q^{h}[u] &\:\dot{=}\:h \Re(u^\prime\bar{u})-\frac{1}{2}h^\prime|u|^{2}, \\
	Q^{y}[u] &\:\dot{=}\:y(|u^\prime|^{2}+(\omega^{2}-V)|u|^{2}),\\
	Q^{f}[u]	&	\:\dot{=}\: f(|u^\prime|^2 +(\omega^2-V)|u|^{2}) + f^\prime \Re(u^\prime\bar{u})-\frac{1}{2}f^{\prime\prime}|u|^2. 
	\end{aligned}
	\end{equation}
	Note that for the last current of~\eqref{eq: def: currents, eq 1} we chose~$h=\frac{d f}{d r^\star}$ and~$y=f$ and summed the first two currents of~\eqref{eq: def: currents, eq 1}. 
\end{definition}

\begin{lemma}\label{lem: subsec: currents, lem 1}
		Let~$u$ be a smooth solution of Carter's radial ode~\eqref{eq: ode from carter's separation}. The derivatives of the currents of Definition~\ref{def: currents} are
	\begin{equation}
	\begin{aligned}
		(Q^{h}[u])^\prime &= h|u^\prime|^2  + \Big( h(V-\omega ^{2}) -\frac{1}{2} h^{\prime\prime}\Big)|u|^{2}+h \Re(u\bar{H}), \\
		 (Q^{y}[u])^\prime &= y^\prime|u^\prime|^2 +\Big(y^\prime(\omega^{2}-V)-yV^\prime \Big) |u|^{2} +2y \Re (u^\prime\bar{H})\\
	&= y^\prime |u^\prime|^{2}+\Big(\omega^{2}y^\prime-(yV)^\prime \Big)|u|^{2}  +2y \Re (u^\prime\bar{H}), \\
	 (Q^{f}[u])^\prime	&=2f^\prime|u^\prime|^2 +\left( -fV^\prime -\frac{1}{2}f^{\prime\prime\prime} \right)|u|^{2}+ \Re\left(2f\bar{H}u^\prime +f^\prime\bar{H}u \right).
	\end{aligned}
	\end{equation}
\end{lemma}

\subsection{The currents for~\texorpdfstring{$\Psi$}{u}}

We study the ode for $\Psi$ instead of the ode for $u=\sqrt{r^{2}+a^{2}}\Psi$. The currents of the present Section will be used in Section~\ref{subsec: bounded frequencies}.

\begin{lemma}\label{lem: sec: frequency localized multiplier estimates, lem 2}
	Let~$u$ be a smooth solution of Carter's radial ode~\eqref{eq: ode from carter's separation}. Then, the function~$\Psi$, where~$u=\sqrt{r^2+a^2}\Psi$, satisfies the following ode
	\begin{equation}\label{eq: lem: sec: frequency localized multiplier estimates, lem 2, eq 1}
	\begin{aligned}
	\Psi^{\prime\prime} +\frac{2r\Delta}{(r^{2}+a^{2})^{2}}\Psi^\prime+\Psi \Big( \omega^{2} -\tilde{V} \Big) =\frac{H^{(a\omega)}_{m\ell}}{\sqrt{r^{2}+a^{2}}},
	\end{aligned}
	\end{equation}
	where 
	\begin{equation}
	\tilde{V}\:\dot{=}\:V_0+V_{\mu_{\textit{KG}}}=\frac{(\lambda^{(a\omega)}_{m\ell}+a^2\omega^2-2m\omega a \Xi)\Delta}{(r^{2}+a^{2})^{2}} +\Delta\mu_{\textit{KG}}^2\frac{r^{2}+a^{2}}{(r^{2}+a^{2})^{2}}+\omega^2-\left(\omega-\frac{am\Xi}{r^2+a^2}\right)^2,
	\end{equation}
	where for~$V_0,V_{\mu_{\textit{KG}}}$ see~\eqref{eq: the potentials V0,VSl,Vmu}. 
\end{lemma}
\begin{proof}
	This is straightforward from equation~\eqref{eq: ode from carter's separation}.
\end{proof}

We define the following

\begin{definition}\label{def: currents for the ell=0 case}
	The modified currents that correspond to the new form of the equation are 
	\begin{equation}
	\begin{aligned}
	Q^h_{\textit{stat}}[\Psi] &\:\dot{=}\: h\Re(\Psi\bar{\Psi}^\prime)+\frac{r\Delta}{(r^{2}+a^{2})^{2}}h|\Psi|^{2}-\frac{1}{2}h^\prime|\Psi|^2,\\
	Q^y_{\textit{stat}}[\Psi] &\:\dot{=}\:y|\Psi^\prime|^{2}+y(\omega^{2}-\tilde{V})|\Psi|^{2},
	\end{aligned}
	\end{equation}
\end{definition}	
	
\begin{lemma}\label{lem: subsec: currents, lem 2}
		Let~$\Psi$ be a smooth solution of~\eqref{eq: lem: sec: frequency localized multiplier estimates, lem 2, eq 1}. The derivatives of the currents of Definition~\ref{def: currents for the ell=0 case} are 	
	\begin{equation}\label{eq: derivative of current1 for the ell=0 case}
	\begin{aligned}
		  (Q^h_{\textit{stat}}[\Psi])^\prime	&	= h|\Psi^\prime|^2 + \left(\left(\frac{r\Delta}{(r^{2}+a^{2})^{2}}h\right)^\prime +h(\tilde{V}-\omega^{2})-\frac{1}{2} h^{\prime\prime} \right) |\Psi|^{2} +\Re\left(\frac{2\Psi h\overline{H^{(a\omega)}_{m\ell}}}{\sqrt{r^{2}+a^{2}}}\right),\\
		 (Q^y_{\textit{stat}}[\Psi])^\prime	&	=\Big( y^\prime -y\frac{4r\Delta}{(r^{2}+a^{2})^{2}}  \Big)|\Psi^\prime|^2 +\Big( y^\prime(\omega^2-\tilde{V})-y\tilde{V}^\prime \Big)|\Psi|^{2} +\Re\left(\frac{y\Psi^\prime\overline{H^{(a\omega)}_{m\ell}}}{\sqrt{r^{2}+a^{2}}} \right).
	\end{aligned}
	\end{equation}
\end{lemma}

By adding the two currents, after taking~$f=y$ and~$h=\frac{d f}{d r^\star}$, we obtain a third current
\begin{equation}
\begin{aligned}
(Q^f_{\textit{stat}}[\Psi])^\prime  &= \left( 2f^\prime-\frac{4r\Delta(r)f}{(r^2+a^2)^2} \right)|\Psi^\prime|^2 +\left(\left( \frac{r\Delta(r) }{(r^2+a^2)^2}f^\prime\right)^\prime-\frac{1}{2}f^{\prime\prime\prime} -f\tilde{V}^\prime \right)|\Psi|^2   \\
& \quad +  \Re\left( \frac{2\Psi f^\prime\overline{H^{(a\omega)}_{m\ell}}}{\sqrt{r^2+a^2}} \right) 
+\Re\left( \frac{f\Psi^\prime \overline{H^{(a\omega)}_{m\ell}}}{\sqrt{r^2+a^2}} \right).
\end{aligned}
\end{equation}

\subsection{The currents $Q^{\partial_t}[u],Q^{K^+}[u],Q^{\bar{K}^+}[u]$}\label{subsec: sec: frequency localized multiplier estimates, subsec 2}

We define the following currents
\begin{definition}\label{def: sec: currents: def 1, QT, QK currents}
	We define the currents 
	\begin{equation}\label{eq: def: QT, QK currents, eq 1}
	\begin{aligned}
	&   Q^{\partial_t}[u]=\omega \Im (u^\prime\bar{u}),\qquad   Q^{K^+}[u]=\left(\omega-\frac{am\Xi}{r_+^2+a^2}\right)\Im (u^\prime\bar{u}),\qquad  Q^{\bar{K}^+}[u]=\left(\omega-\frac{am\Xi}{\bar{r}_+^2+a^2}\right)\Im (u^\prime\bar{u}).
	\end{aligned}
	\end{equation}
\end{definition}

\begin{lemma}\label{lem: subsec: sec: frequency localized multiplier estimates, subsec 2, lem 1}
	Let~$u$ be a smooth solutions of Carter's radial ode~\eqref{eq: ode from carter's separation}. The derivatives of the currents of Definition~\ref{def: sec: currents: def 1, QT, QK currents} are respectively
	\begin{equation}
	\begin{aligned}
	&   \left(Q^{\partial_t}[u]\right)^\prime=\omega \Im (H\bar{u}),\qquad  \left(Q^{K^+}[u]\right)^\prime=\left(\omega-\frac{am\Xi}{r_+^2+a^2}\right)\Im (H\bar{u}),\qquad  \left(Q^{\bar{K}^+}[u]\right)^\prime=\left(\omega-\frac{am\Xi}{\bar{r}_+^2+a^2}\right)\Im (H\bar{u}).
	\end{aligned}
	\end{equation}
\end{lemma}

\section{The frequency regimes}\label{sec: frequencies}

\subsection{The de~Sitter frequencies}\label{subsec: sec: frequencies, subsec 2}

We define the `de~Sitter frequencies' as follows
\begin{equation}\label{eq: subsec: sec: frequencies, subsec 2, eq 1}
	\mathcal{DSF}=\{(\omega,m):~	|\omega|\in \left[0,\frac{|am|\Xi}{\bar{r}_+^2+a^2}\right]\}
\end{equation}
and note that they are disjoint from the set of superradiant frequencies, see Definition~\ref{def: subsec: sec: frequencies, subsec 1, def 1}. 

For high angular frequencies, we define the `high de~Sitter frequency regime'~$\mathcal{F}_{dS}$, see the next Section~\ref{subsec: sec: frequencies, subsec 3}, which note that, for certain black hole parameters will correspond to a trapped frequency regime, also see Proposition~\ref{prop: energy estimate in the de Sitter freqs, l larger that omega}. 

Note that the interior of the set~\eqref{eq: subsec: sec: frequencies, subsec 2, eq 1} is not empty, in contrast with the Kerr case where~$\bar{r}_+=\infty$, see Lemma~\ref{lem: sec: properties of Delta, lem 2, a,M,l}.

\subsection{The frequency regimes}\label{subsec: sec: frequencies, subsec 3}

To ease the comparison with~\cite{DR2}, we retain the notation for the frequency regimes introduced there, while note that we here introduce the frequency regime~$\mathcal{F}_{\textit{dS}}$.

\begin{definition}[Definition of the frequency regimes]\label{def: subsec: sec: frequencies, subsec 3, def 1}
	Let
	\begin{equation}
		\omega_{\textit{low}},\qquad \omega_{\textit{high}},\qquad \lambda_{\textit{low}},\qquad \alpha>0
	\end{equation}
	be sufficiently small, sufficiently large, sufficiently small and sufficiently small positive real parameters respectively. We choose these parameters later, see already Section~\ref{subsec: choice of omega1, lambda2}.

	Then, we cover the frequency space~$\mathbb{R}\times\bigcup_{m\in\mathbb{Z}}\left(\{m\}\times\mathbb{Z}_{\geq |m|}\right)$ by the following~(potentially overlapping) frequency regimes:
	\begin{equation}\label{eq: subsec: sec: frequencies, subsec 3, eq 1}
		\begin{aligned}
			&\text{The high de Sitter frequency regime:}\\
			&\qquad\mathcal{F}_{\textit{dS}}    =\{(\omega,m,\ell)\: : \: \tilde{\lambda}\geq \omega_{\textit{high}}\}\cap\{ (\omega,m)\in\mathcal{DSF},~am\omega\geq 0\}, \\
			&	\text{which will be further split into a trapped and non-trapped case.}\\
			&	\text{The enlarged high superradiant frequency regime:}\\
			&\qquad \mathcal{F}^{\sharp}   =\{(\omega,m,\ell)\: :\: \tilde{\lambda}\geq\left(\frac{|a|\Xi}{r_+^2+a^2}+\alpha\right)^{-1}\omega_{\textit{high}}\}\cap\{am\omega\in\Big(\frac{a^2m^2\Xi}{\bar{r}_+^2+a^2},\frac{a^2m^2\Xi}{r_+^2+a^2}+|a|\alpha\tilde{\lambda}\Big)\} \\
			&	\text{The bounded frequency regime:}\\
			&	\qquad \mathcal{F}_{\flat}=  \{(\omega,m,\ell)\: : \:|\omega|\leq \omega_{\textit{high}},|\tilde{\lambda}|<\lambda_{\textit{low}}^{-1}\omega^2_{\textit{high}}\}\\
			&	\text{The $\lambda$--dominated frequency regime:}\\
			&	\qquad \mathcal{F}_{\lessflat} = \{(\omega,m,\ell)\: : \: \tilde{\lambda}\geq \lambda_{\textit{low}}^{-1}\omega^2_{\textit{high}},\: \tilde{\lambda}>\lambda_{\textit{low}}^{-1} (\omega^2+a^2m^2)\}\cap\{am\omega\slashed{\in}\Big[0,\frac{a^2m^2\Xi}{r_+^2+a^2}+|a|\alpha\tilde{\lambda}\Big)\}\\
			&	\text{The high~$\omega\sim \lambda$ frequency regime:}\\
			&	\qquad\mathcal{F}_\natural =  \{(\omega,m,\ell)\: : \:|\omega|\geq \omega_{\textit{high}},\:\lambda_{\textit{low}}\tilde{\lambda}\leq \omega^2+a^2m^2\leq \lambda_{\textit{low}}^{-1}\tilde{\lambda}\}\cap\{a m\omega\slashed{\in}\Big[0,\frac{a^2m^2\Xi}{r_+^2+a^2}+|a|\alpha\tilde{\lambda}\Big)\}\\
			&	\text{which will be further split into a trapped and non-trapped case.}\\
			&	\text{The \texorpdfstring{$\omega$}{g}~dominated frequency regime:}\\
			&	\qquad \mathcal{F}_{\sharp} =\{(\omega,m,\ell)\: : \: |\omega|\geq \omega_{\textit{high}},\:|\tilde{\lambda}| <\lambda_{\textit{low}}(\omega^2+a^2m^2)\}\cap\{a m\omega\slashed{\in}\Big[0,\frac{a^2m^2\Xi}{r_+^2+a^2}+|a|\alpha|\tilde{\lambda|}\Big)\}.
		\end{aligned}
	\end{equation}
\end{definition}

By inspecting~\eqref{eq: subsec: sec: frequencies, subsec 3, eq 1} we have that 
\begin{equation}\label{eq: subsec: sec: frequencies, subsec 3, eq 2}
	\mathcal{F}_{dS}\cup	\mathcal{F}^\natural\cup \mathcal{F}_\flat\cup 	\mathcal{F}_{\lessflat}\cup 	\mathcal{F}_\natural \cup 	\mathcal{F}_\natural =\mathbb{R}\times\bigcup_{m\in\mathbb{Z}}\left(\{m\}\times\mathbb{Z}_{\geq |m|}\right).
\end{equation}

\begin{remark}
	Note that in the~$\Lambda=0$ Kerr case we obtain~$\mathcal{F}_{\textit{dS}}=\emptyset$, see~\cite{DR2}, since~$\bar{r}_+(\Lambda=0)=\infty$. Moreover, note that in the Schwarzschild--de~Sitter case we also obtain~$\mathcal{F}_{\textit{dS}}=\emptyset$.
	
	For certain black hole parameters the frequency regimes~$\mathcal{F}_{dS}$ admits trapping, see already Section~\ref{subsec: de Sitter frequency regime}, even for frequencies~$\omega$ sufficiently close to~$0$. This corresponds to the existence of~$\partial_t$ orthogonal trapped null geodesics in the background manifold~$\mathcal{M}$ of Kerr--de~Sitter, see Section~\ref{sec: geodesics}. 
\end{remark}

Note the following lemma

\begin{remark}\label{rem: subsec: sec: frequencies, subsec 3, rem 1}
	Let~$l>0$ and~$(a,M)\in\mathcal{B}_l$. Let the parameters~$\lambda_{\textit{low}}(a,M,l),~\omega_{\textit{high}}(a,M,l),~\alpha(a,M,l)>0$ be arbitrary positive real numbers. Then, the only frequency regime intersections that are not empty are 
	\begin{equation}\label{eq: lem: subsec: sec: frequencies, subsec 3, lem 1, eq 1}
		\mathcal{F}_\flat\cap \mathcal{F}^{\sharp},\qquad \mathcal{F}_\flat\cap \mathcal{F}_{\textit{dS}}.
	\end{equation}
To prove this we note that for~$|\omega|\geq \omega_{\textit{high}}$ and~$am\omega\in\left(0,\frac{a^2m^2\Xi}{r_+^2+a^2}+|a|\alpha\tilde{\lambda}\right)$, we obtain 
	\begin{equation}
		\begin{aligned}
			&   0<a m\omega<\frac{a^2m^2\Xi}{r_+^2+a^2}+|a|\alpha\tilde{\lambda}\leq \left(\frac{a^2\Xi}{r_+^2+a^2}+|a|\alpha\right)\tilde{\lambda}\implies |a\omega_{\textit{high}}|\leq \left(\frac{a^2\Xi}{r_+^2+a^2}+|a|\alpha\right)\tilde{\lambda}
		\end{aligned}
	\end{equation}
	where we used that for~$am\omega\geq 0$ we have~$\lambda^{(a\omega)}_{m\ell}\geq \Xi^2 m^2$, from Lemma~\ref{lem: inequality for lambda}, and that~$|m|\geq 1$. We readily obtain 
	\begin{equation}\label{eq: lem: subsec: sec: frequencies, subsec 3, lem 1, eq 2}
		\tilde{\lambda}\geq \left(\frac{|a|\Xi}{r_+^2+a^2}+\alpha\right)^{-1}\omega_{\textit{high}}.
	\end{equation}
	We conclude that the only frequency regimes interesections that are not empty are~$    \mathcal{F}_\flat\cap \mathcal{F}^{\sharp},~\mathcal{F}_\flat\cap \mathcal{F}_{\textit{dS}}$. Now, we have that~\eqref{eq: lem: subsec: sec: frequencies, subsec 3, lem 1, eq 2} implies that~$\tilde{\lambda}\geq \omega_{high}$ in view of the fact that~$\frac{a}{r_+^2+a^2}<1,\alpha<1$.  
\end{remark}

\section{Fixed frequency analysis: Proof of Theorem~\ref{thm: sec: proofs of the main theorems, thm 3}}\label{sec: proof: prop: sec: proofs of the main theorems}

The main result of this Section is Theorem~\ref{thm: sec: proofs of the main theorems}. Theorem~\ref{thm: sec: proofs of the main theorems, thm 3} is a direct Corollary.

\subsection{Notation}\label{subsec: sec: proof: prop: sec: proofs of the main theorems, subsec 1}

For convenience we use the notation on the primed derivatives~$^{\prime}=\frac{d}{dr^\star}$ and on the constants~$b,B$ already discussed in the beginning of Section~\ref{sec: the main theorem, no cases}.

Throughout this Section we use the notation introduced in Definition~\ref{def: subsec: sec: frequencies, subsec 3, def 0}, namely that 
\begin{equation}
	\tilde{\lambda}=\lambda^{(a\omega)}_{m\ell}+(a\omega)^2.
\end{equation}

\subsection{The fixed frequency Theorem~\ref{thm: sec: proofs of the main theorems}}

We have the following fixed frequency ode result

\begin{customTheorem}{16.1}\label{thm: sec: proofs of the main theorems}
	Let~$l>0$~$(a,M)\in\mathcal{B}_l$ and~$\mu^2_{\textit{KG}}\geq 0$. Then, there exist parameters
	\begin{equation}
	 E>0, \qquad	\alpha^{-1}>0,\qquad (\lambda_{\textit{low}})^{-1}>0,\qquad (\omega_{\textit{low}})^{-1}>0,\qquad \omega_{\textit{high}}>0,
	\end{equation}
	\begin{equation*}
		\begin{aligned}
				r_+<R^-<R^+<\bar{r}_+,
		\end{aligned}
	\end{equation*} 
	and an
	\begin{equation}
		r_{trap}(\omega,m,\ell) \in (R^-,R^+)\cup\{0\}
	\end{equation}
	such that for all~$r_{-\infty}$ sufficiently close to~$r_+$ and all~$r_{+\infty}$ sufficiently close to~$\bar{r}_+$ and for any sufficiently large~$\mathcal{C}>0$ there exists a constant~$B(\mathcal{C},r^\star_{\pm\infty})$, as in Section~\ref{subsec: sec: proof: prop: sec: proofs of the main theorems, subsec 1}, such that for smooth solutions
	\begin{equation}
		u(\omega,m,\ell,r)
	\end{equation}
	of Carter's inhomogeneous radial ode, see~\eqref{eq: ode from carter's separation}, where recall the rescaling~$u=\sqrt{r^2+a^2}\Psi$, that moreover satisfy the outgoing boundary conditions~\eqref{eq: lem: sec carters separation, subsec boundary behaviour of u, boundary beh. of u, eq 3}, we have the following integrated estimate
	\begin{equation}\label{eq: prop: sec: proofs of the main theorems, eq 1}
		\begin{aligned}
			&	1_{\{\mathcal{F}_\flat \cap \{m=0\}\}}\int_{\mathbb{R}}\Delta \left( |\Psi^\prime|^2 + (\omega^2+\tilde{\lambda}+\mu_{KG}^2)|\Psi|^2 \right)dr^\star +1_{\{|m|>0\}}\cdot \int_{r^\star_{-\infty}}^{r^\star_\infty}  \left(|u|^2 +\mu^2_{\textit{KG}} |u|^2\right)dr^\star \\
			&	\qquad\qquad\qquad 	+ 1_{(\mathcal{F}_{\flat})^c} \int_{\mathbb{R}}\Delta \left( |u^\prime|^2+|u|^2+\left(1-\frac{r_{\textit{trap}}}{r}\right)^2\left(\tilde{\lambda}+\omega^2\right)|u|^2\right)dr^\star \\
			&	\qquad\leq B \cdot  1_{\mathcal{F}_\flat \cap \{m=0\}}\int_{\mathbb{R}}\left( \big|\omega\Im (\bar{\Psi}H)\Big|+\Big|\omega\Im(\bar{\Psi}H)\big|+ |H|^2 + |\Psi^\prime H|\right)dr^\star\\
			&	\qquad\qquad + B  1_{|m|>0} \int_{\mathbb{R}}\left(\left|\Re (u\bar{H})\right|+\left| \Re (u\bar{H})\right|+\left|\Re (u^\prime \bar{H})\right|\right) dr^\star \\
			&	\qquad\qquad +B\cdot  1_{ \mathcal{F}^\sharp\cup \mathcal{F}_\flat  \cup\mathcal{F}_{\lessflat}\cup \mathcal{F}_\sharp}\int_{\mathbb{R}}\left(\left|\omega-\frac{am\Xi}{r_+^2+a^2}\right||\Im (\bar{u}H)|+\left|\omega-\frac{am\Xi}{r_+^2+a^2}\right||\Im(\bar{u}H)|\right)dr^\star\\
			&	\qquad\qquad +E \int_{\mathbb{R}} \left(\left(\omega-\frac{am\Xi}{r_+^2+a^2}\right)\Im (\bar{u}H)+\left(\omega-\frac{am\Xi}{r_+^2+a^2}\right)\Im(\bar{u}H)\right)dr^\star\\
			&	\qquad\qquad +1_{\mathcal{F}_{\mathcal{SF},\mathcal{C}}}|\omega-\omega_+ m||\omega-\bar{\omega}_+ m| |u|^2(-\infty),
		\end{aligned}
	\end{equation}
	where for the definition of the frequency regimes~$\mathcal{F}^\sharp,~\mathcal{F}_\flat,~\mathcal{F}_{dS},~\mathcal{F}_\sharp,~\mathcal{F}_\lessflat$ see Section~\ref{sec: frequencies}, and for~$\mathcal{F}_{\mathcal{SF},\mathcal{C}}$ see~\eqref{eq: subsec: sec: carter separation, subsec 2, eq 1}. 
\end{customTheorem}

We recall from~\eqref{eq: subsec: sec: frequencies, subsec 3, eq 2} that we cover the entire frequency space
\begin{equation}
		\mathcal{F}_{dS}\cup	\mathcal{F}^\natural\cup \mathcal{F}_\flat\cup 	\mathcal{F}_{\lessflat}\cup 	\mathcal{F}_\natural \cup 	\mathcal{F}_\natural =\mathbb{R}\times\bigcup_{m\in\mathbb{Z}}\{m\}\times\mathbb{Z}_{\geq |m|}
\end{equation}

The parameters~$\omega_{high},\omega_{low}^{-1},\lambda_{low}^{-1},E,\alpha$ will be constrained to some largeness, where the largeness restriction only depends on the black hole parameters~$a,M,l$. Until the parameters are fixed then the constants~$b,B$, discussed previously, will depend on these parameters. We will fix the parameters at the end of the present Section, see already Section~\ref{subsec: choice of omega1, lambda2}.

\subsection{Multipliers for the high de~Sitter frequency regime~$\mathcal{F}_{\textit{dS}}$}\label{subsec: de Sitter frequency regime}

This frequency regime is non existent in the~$\Lambda=0$ Kerr case, since~$\bar{r}_+(\Lambda=0)=\infty$ see Lemma~\ref{lem: sec: properties of Delta, lem 2, a,M,l}.

This frequency regime is non-superradiant and when 
\begin{equation}
	\max_{r\in[r_+,\bar{r}_+]}\frac{\Delta}{a^2}\leq 1,\quad |a|\not{=}0
\end{equation}
holds, then it is subject to trapping. This is related to the presence of trapped null geodesics orthogonal to~$\partial_t$, see Section~\ref{sec: geodesics}.

The main Proposition of this Section is the following 
\begin{proposition}\label{prop: energy estimate in the de Sitter freqs, l larger that omega}
	
	Let~$l>0$,~$(a,M)\in \mathcal{B}_l$,~$\mu^2_{KG}\geq 0$. For any 
	\begin{equation}
	\omega_{\textit{high}}>0,\qquad E>0
	\end{equation} 
	both sufficiently large, then for~$(\omega,m,\ell)\in\mathcal{F}_{\textit{dS}}(\omega_{high})$ there exist smooth multipliers~$f,~h,~y$, satisfying the uniform bounds 
	\begin{equation}
	|f|+|f^\prime|+|f^{\prime\prime}|+|f^{\prime\prime\prime}| +|h|+|h'|+|h''|+|h'''|+|y|+|y^\prime|\leq B,
	\end{equation}
	and there exists an
	\begin{equation}
		r_{trap}(\omega,m,\ell)\in (r_++\epsilon_{away},\bar{r}_+-\epsilon_{away})\cup \{0\},
	\end{equation}
	where~$\epsilon_{away}(a,M,l)>0$, such that for all smooth solutions~$u$ of Carter's radial ode~\eqref{eq: ode from carter's separation} satisfying the outgoing boundary conditions~\eqref{eq: lem: sec carters separation, subsec boundary behaviour of u, boundary beh. of u, eq 3}, we have 
	\begin{equation}\label{eq: prop: energy estimate in the de Sitter freqs, l larger that omega, eq 0}
	\begin{aligned}
	b \int_{\mathbb{R}}\Delta \left(|u^\prime|^2+|u|^2+\left(1-\frac{r_{\textit{trap}}}{r}\right)^2(1+\omega^2+\tilde{\lambda})|u|^2\right)dr^\star\leq  B & \int_{\mathbb{R}} \left(|f \Re (u\bar{H})|+|2f\Re (u^\prime \bar{H})|\right)dr^\star\\
	&	+B\int_{\mathbb{R}} \left(|h\Re (u\bar{H})|+|f \Re (u\bar{H})|+|2f\Re (u^\prime \bar{H})|\right)dr^\star\\
	&	+B \int_{\mathbb{R}} |2y\Re (u^\prime H)|dr^\star \\
	&  + E\int_{\mathbb{R}}\left(\omega-\frac{am\Xi}{r_+^2+a^2}\right)\Im (\bar{u}H)dr^\star\\
	&	+ E\int_{\mathbb{R}}\left(\omega-\frac{am\Xi}{r_+^2+a^2}\right)\Im(\bar{u}H)dr^\star. 
	\end{aligned}
	\end{equation}
\end{proposition}

\begin{proof}[\textbf{Proof of Proposition~\ref{prop: energy estimate in the de Sitter freqs, l larger that omega}}]
		We have that
		\begin{equation}
			V_0-\omega^2 =\frac{\Delta}{(r^2+a^2)^2}(\tilde{\lambda}-2m\omega a\Xi)-\left(\omega-\frac{am\Xi}{r^2+a^2}\right)^2
		\end{equation}
and we compute
		\begin{equation}\label{eq: proof: prop: energy estimate in the de Sitter freqs, l larger that omega, eq 0}
		\frac{dV_0}{dr}= \frac{d}{dr}\frac{\Delta}{(r^2+a^2)^2} \left( \tilde{\lambda}-2m\omega a \Xi\right)-\frac{4 a m \Xi  r \left(\omega  \left(a^2+r^2\right)-a m \Xi \right)}{\left(a^2+r^2\right)^3},
	\end{equation}
	which we also write as 
	\begin{equation}\label{eq: proof: prop: energy estimate in the de Sitter freqs, l larger that omega, eq 0.1}
		\begin{aligned}
			\frac{dV_0}{dr}	&	= \frac{1}{l^2(r^2+a^2)^3}\Big(r^3 \left(4 a^3 m \Xi  \omega -2 a^2 \tilde{\lambda}  -2 \tilde{\lambda}  l^2\right)\\
			&	\qquad\qquad\qquad +r^2 \left(6 \tilde{\lambda}  l^2 M-12 a l^2 m M \Xi  \omega \right)\\
			&	\qquad\qquad\qquad+r \left(4 a^5 m \Xi  \omega -2 a^4 \tilde{\lambda}  -2 a^2 \tilde{\lambda}   l^2+4 a^2 l^2 m^2 \Xi ^2\right)+4 a^3 l^2 m M \Xi  \omega -2 a^2 \tilde{\lambda}  l^2 M\Big).
		\end{aligned}
	\end{equation}

		First, in view of Lemma~\ref{lem: inequality for lambda} we have that 
	\begin{equation}\label{eq: proof: prop: energy estimate in the de Sitter freqs, l larger that omega, eq 1.0}
		\tilde{\lambda}-2m\omega a\Xi \geq \Xi^2 m^2-2\frac{a^2\Xi^2}{\bar{r}_+^2+a^2}m^2\geq \left(1-2\frac{a^2}{\bar{r}_+^2+a^2}\right)\Xi^2m^2\geq bm^2. 
	\end{equation}
	for any~$(\omega,m,\ell)\in \mathcal{F}_{dS}$. A similar argument also shows that
	\begin{equation}\label{eq: proof: prop: energy estimate in the de Sitter freqs, l larger that omega, eq 1.0.1}
		\tilde{\lambda} -2m\omega a\Xi = \epsilon_2\tilde{\lambda}+(1-\epsilon_2)\tilde{\lambda} -2m\omega a \Xi \geq \epsilon_2\tilde{\lambda}+\left(1-\epsilon_2-2\frac{a^2}{\bar{r}_+^2+a^2}\right)\Xi^2m^2\geq b(\epsilon_2)\tilde{\lambda},
	\end{equation}
	for some sufficiently small~$\epsilon_2(a,M,l)>0$ which we now fix.

	Since	
		\begin{equation}
			am\omega\leq \frac{a^2m^2\Xi}{\bar{r}_+^2+a^2}< \frac{a^2m^2\Xi}{r_+^2+a^2}
		\end{equation}
		we have that for~$(\omega,m,\ell)\in\mathcal{F}_{dS}(\omega_{high})$ the following holds
			\begin{equation}\label{eq: prop: energy estimate in the de Sitter freqs, l larger that omega, eq 1}
			\frac{dV_0}{dr}(r_+)>b\tilde{\lambda}. 
		\end{equation}	
		for any sufficiently large~$\omega_{high}>0$.

		Now, let~$\epsilon_{trap}>0$ be sufficiently small and let~$\omega_{high}>0$ be sufficiently large. We study the two cases 
		\begin{equation}\label{eq: proof: prop: energy estimate in the de Sitter freqs, l larger that omega, eq 1}
		\begin{aligned}
			\text{(Trap)}	&=		\{(\omega,m,\ell)\in\mathcal{F}_{dS}:~|\max_{r\in [r_+,\bar{r}_+] }V_0-\omega^2|\leq \epsilon_{trap} \tilde{\lambda}\}, \\
			\text{(Not Trap)} &	=	\{(\omega,m,\ell)\in\mathcal{F}_{dS}:~|\max_{r\in [r_+,\bar{r}_+]} V_0-\omega^2|> \epsilon_{trap} \tilde{\lambda}\}. 
		\end{aligned}
	\end{equation}

	\underline{We study the first case} of~\eqref{eq: proof: prop: energy estimate in the de Sitter freqs, l larger that omega, eq 1}. For the sake of contradiction, suppose that the maximum of~$V_0$ is attained at the cosmological horizon~$\bar{r}_+$. Then, we would have that 
	\begin{equation}\label{eq: proof: prop: energy estimate in the de Sitter freqs, l larger that omega, eq 1.1}
		\left(\omega-\frac{am\Xi}{\bar{r}_+^2+a^2}\right)^2 \leq \epsilon_{trap}\tilde{\lambda}. 
	\end{equation} 
	Then, we use~\eqref{eq: proof: prop: energy estimate in the de Sitter freqs, l larger that omega, eq 1.1} for~$\epsilon_{trap}>0$ sufficiently small and we compute
	\begin{equation}\label{eq: proof: prop: energy estimate in the de Sitter freqs, l larger that omega, eq 1.2}
		\frac{dV_0}{dr}(\bar{r}_+)< -b\tilde{\lambda}.
	\end{equation}
	Therefore, in view of Lemma~\ref{lem: subsec: sec: trapping, subsec 1, lem 1} we obtain a contradiction. We have also concluded that~\eqref{eq: proof: prop: energy estimate in the de Sitter freqs, l larger that omega, eq 1.2} holds for the first case of~\eqref{eq: proof: prop: energy estimate in the de Sitter freqs, l larger that omega, eq 1}.

	Now, in view of Lemma~\ref{lem: subsec: sec: trapping, subsec 1, lem 1} we have that for any~$\epsilon_{away}>0$ sufficiently small the critical point~$r_{V_0,max}$ is attaιned at~$(\bar{r}_+-\epsilon_{away},\bar{r}_+]$~(the critical point cannot be attained at~$[r_+,r_++\epsilon_{away}]$ since~\eqref{eq: prop: energy estimate in the de Sitter freqs, l larger that omega, eq 1} holds).

	In view of~\eqref{eq: prop: energy estimate in the de Sitter freqs, l larger that omega, eq 1},~\eqref{eq: proof: prop: energy estimate in the de Sitter freqs, l larger that omega, eq 1.2} and Lemma~\ref{lem: subsec: sec: trapping, subsec 1, lem 1}  we can now rewrite~$\frac{dV_0}{dr}$ as follows 
	\begin{equation}\label{eq: proof: prop: energy estimate in the de Sitter freqs, l larger that omega, eq 1.3}
		\frac{dV_0}{dr}= \left(2\tilde{\lambda}l^2+2a^2(\tilde{\lambda}-2am\omega \Xi)\right) (r-r_{V_0,max}) s(r),
	\end{equation}
	see~\eqref{eq: proof: prop: energy estimate in the de Sitter freqs, l larger that omega, eq 0.1} for the calculation of~$\frac{dV_0}{dr}$, where the function~$s(r)$ satisfies~$s(r)>b$ for all~$r\in [r_+,\bar{r}_+]$.

	Now, recall from Proposition~\ref{prop: Carters separation, radial part} that~$V=V_0+V_{SL}+V_{\mu_{KG}}$, where~$V_{SL},V_{\mu_{KG}}$ are frequency independent.
	Therefore, it is an immediate consequence of~\eqref{eq: proof: prop: energy estimate in the de Sitter freqs, l larger that omega, eq 1.3} that for~$\omega_{high}>0$ sufficiently large and for~$\epsilon_{away}(a,M,l,\mu_{KG})>0$ sufficiently small we have that 
	\begin{equation}\label{eq: proof: prop: energy estimate in the de Sitter freqs, l larger that omega, eq 1.4}
		\begin{aligned}
			&	\frac{dV}{dr}>b\tilde{\lambda},\qquad r\in [r_+,r_{V_0,max}-2\epsilon_{away}]\\
			&	\frac{dV}{dr}<-b\tilde{\lambda},\qquad r\in [r_{V_0,max}+2\epsilon_{away},\bar{r}_+],
		\end{aligned}
	\end{equation}
	in~$\mathcal{F}_{dS}$. In view of~\eqref{eq: proof: prop: energy estimate in the de Sitter freqs, l larger that omega, eq 1.3} we have that for a sufficiently large~$\omega_{high}>0$ and for a sufficiently small~$\epsilon_{away}>0$
	\begin{equation}\label{eq: proof: prop: energy estimate in the de Sitter freqs, l larger that omega, eq 1.5}
		\frac{d^2V_0}{dr^2} < -b\tilde{\lambda},\qquad r\in [r_{V_0,max}-2\epsilon_{away},r_{V_0,max}+2\epsilon_{away}].
	\end{equation}
	Therefore, for a sufficiently large~$\omega_{high}>0$ and a sufficiently small~$\epsilon_{away}$ we have that
	\begin{equation}\label{eq: proof: prop: energy estimate in the de Sitter freqs, l larger that omega, eq 1.6}
		\frac{d^2V}{dr^2} < -b\tilde{\lambda},\qquad r\in [r_{V_0,max}-2\epsilon_{away},r_{V_0,max}+2\epsilon_{away}],
	\end{equation}
	where the constant~$b$ of~\eqref{eq: proof: prop: energy estimate in the de Sitter freqs, l larger that omega, eq 1.6} differs from the relevant constant of~\eqref{eq: proof: prop: energy estimate in the de Sitter freqs, l larger that omega, eq 1.5}. Moreover, it is evident from~\eqref{eq: proof: prop: energy estimate in the de Sitter freqs, l larger that omega, eq 1.4},~\eqref{eq: proof: prop: energy estimate in the de Sitter freqs, l larger that omega, eq 1.6}, that the potential~$V$ attains a unique critical point~$r_{V,max}\in (r_++\epsilon_{away},\bar{r}_+-\epsilon_{away})$, a maximum, where we also have
	\begin{equation}
		-(r-r_{V,max})V^\prime>b\tilde{\lambda}(r-r_{V,max})^2\Delta,\qquad r\in [r_+,\bar{r}_+].
	\end{equation}

	We define 
	\begin{equation}
		r_{trap}=r_{V,max}.
	\end{equation}

For a sufficiently small~$\epsilon_{away}(a,M,l,\mu_{KG})>0$ it is straightforward to solve the following ode
	\begin{equation}
		-\frac{1}{2}f^{\prime\prime\prime} =\Delta,
	\end{equation} 
	in the neighborhood~$(r_{trap}-\epsilon_{away},r_{trap}+\epsilon_{away})$ of ~$r_{\textit{trap}}$, with initial conditions
	\begin{equation}
		f(r_{\textit{trap}})=0,\qquad f^\prime(r_{\textit{trap}})=1,\qquad f^{\prime\prime}(r_{\textit{trap}})=0.
	\end{equation}
	Then, we smoothly extend~$f$ in the~$[r_+,\bar{r}_+]$ such that
	\begin{equation}\label{eq: proof: prop: energy estimate in the de Sitter freqs, l larger that omega, eq 3}
		f(r^\star=\infty)=1,\quad f(r^\star=-\infty)=-1,\qquad \frac{d f}{dr}(r) >b,\:r\in [r_+,\bar{r}_+]
	\end{equation}
	\begin{equation}
		|f(r)|,\quad |\frac{d f}{dr}|,\quad |\frac{d^2 f}{dr^2}|,\quad |\frac{d^3 f }{dr^3}| < B,\qquad r\in [r_+,\bar{r}_+].
	\end{equation}
	
	Therefore, by integrating the energy identity of Lemma~\ref{lem: subsec: currents, lem 1}, associated to~$Q^f$, and by taking~$\omega_{\textit{high}}>0$ sufficiently large such that 
	\begin{equation}
		-fV^\prime -\frac{1}{2}f^{\prime\prime\prime}\geq b(r-r_{trap})^2 \tilde{\lambda}\Delta +b\Delta
	\end{equation}
	we obtain for~$(\omega,m,\ell)\in \mathcal{F}_{dS}(\omega_{high})$ the following holds
	\begin{equation}\label{eq: proof/prop: energy estimate in the de Sitter freqs, l larger that omega, eq 1}
	\begin{aligned}
	&   \int_{\mathbb{R}}\Delta \left(|u^\prime|^2+|u|^2+\left(1-\frac{r_{\textit{trap}}}{r}\right)^2(1+\omega^2+\tilde{\lambda})|u|^2\right)dr^\star\\
	&   \quad \leq  Q^f(\infty)-  Q^f(-\infty)+B \int_{\mathbb{R}}\left(|2f \Re (u^\prime \bar{H})|+|f\Re (u\bar{H})|\right)dr^\star\\
	&   \quad \leq  B f(\infty)\left(\omega-\frac{am\Xi}{\bar{r}_+^2+a^2}\right)^2|u|^2(r^\star=\infty)+B(-f(-\infty))\left(\omega-\frac{am\Xi}{r_+^2+a^2}\right)^2|u|^2(r^\star=-\infty)\\
	&   \quad\quad +B\int_{\mathbb{R}}\left(|2f \Re (u^\prime \bar{H})|+|f\Re (u\bar{H})|\right)dr^\star,
	\end{aligned}
\end{equation}
where we also used~\eqref{eq: proof: prop: energy estimate in the de Sitter freqs, l larger that omega, eq 1.0} to generate all of the derivatives on the LHS.

In order to bound the boundary terms at~$\pm\infty$ on the right hand side of~\eqref{eq: proof/prop: energy estimate in the de Sitter freqs, l larger that omega, eq 1} we use that the high de~Sitter frequency regime~$\mathcal{F}_{\textit{dS}}$ is non-superradiant. Specifically, for any sufficiently large and independent of the frequencies
	\begin{equation}
		E>0,
\end{equation}
we integrate the formula of Lemma~\ref{lem: subsec: sec: frequency localized multiplier estimates, subsec 2, lem 1} associated with the current
\begin{equation}
	-E Q^{K^+}-E Q^{\bar{K}^+}, 
\end{equation}
to obtain
\begin{equation}
	-E Q^{K^+}(r^\star)-E Q^{\bar{K}^+}(r^\star) \Big|_{-\infty}^{+\infty}= -E \int_{\mathbb{R}}\left(\omega-\omega_+ m\right)\Im (H\bar{u})dr^\star-E \int_{\mathbb{R}} (\omega-\bar{\omega}_+ m)\Im (H \bar{u})dr^\star,
\end{equation}
which we rewrite as follows 
\begin{equation}\label{eq: proof/prop: energy estimate in the de Sitter freqs, l larger that omega, eq 1.1}
	E\left(\omega-\frac{am\Xi}{r_+^2+a^2}\right)^2|u|^2(-\infty)+E\left(\omega-\frac{am\Xi}{\bar{r}_+^2+a^2}\right)^2|u|^2(+\infty)= -E \int_{\mathbb{R}}\left(\omega-\omega_+ m\right)\Im (H\bar{u})dr^\star-E \int_{\mathbb{R}} (\omega-\bar{\omega}_+ m)\Im (H \bar{u})dr^\star,
\end{equation}
We add~\eqref{eq: proof/prop: energy estimate in the de Sitter freqs, l larger that omega, eq 1.1} equation \eqref{eq: proof/prop: energy estimate in the de Sitter freqs, l larger that omega, eq 1} and conclude the desired result~\eqref{eq: prop: energy estimate in the de Sitter freqs, l larger that omega, eq 0}~(with~$h=0$), in view of Proposition~\ref{prop: subsec: energy identity, prop 1}.

We now fix~$\epsilon_{trap}(a,M,l,\mu_{KG})>0$.

\underline{We study the second case} of~\eqref{eq: proof: prop: energy estimate in the de Sitter freqs, l larger that omega, eq 1}. For frequencies~$(\omega,m,\ell)\in (Not~Trap)$ we define
\begin{equation}
	r_{trap}(\omega,m,\ell)=0.
\end{equation}
Let~$\epsilon'>0$ be sufficiently small. We split the set~$\text{(Not Trap)}=\{|\max_{r\in [r_+,\bar{r}_+]}V_0-\omega^2|\geq \epsilon_{trap}\tilde{\lambda}\}$ as follows 
\begin{equation}\label{eq: proof: prop: energy estimate in the de Sitter freqs, l larger that omega, eq 5}
	\begin{aligned}
		&	\{(\omega,m,\ell)\in \mathcal{F}_{dS}: \left(\omega-\frac{am\Xi}{\bar{r}_+^2+a^2}\right)^2\leq \epsilon^\prime \tilde{\lambda}\}\cap \{|\max_{r\in [r_+,\bar{r}_+]}V_0-\omega^2|\geq \epsilon_{trap}\tilde{\lambda}\},\\
		&	\{(\omega,m,\ell)\in \mathcal{F}_{dS}:\left(\omega-\frac{am\Xi}{\bar{r}_+^2+a^2}\right)^2> \epsilon^\prime \tilde{\lambda}\}\cap \{\max_{r\in [r_+,\bar{r}_+]}V_0 -\omega^2\geq \epsilon_{trap}\tilde{\lambda}\},\\
		&	\{(\omega,m,\ell)\in \mathcal{F}_{dS}:~\left(\omega-\frac{am\Xi}{\bar{r}_+^2+a^2}\right)^2> \epsilon^\prime \tilde{\lambda}\}\cap \{\omega^2-\max_{r\in [r_+,\bar{r}_+]}V_0\geq \epsilon_{trap}\tilde{\lambda}\}.\\
	\end{aligned}
\end{equation}

\underline{We discuss the first case of~\eqref{eq: proof: prop: energy estimate in the de Sitter freqs, l larger that omega, eq 5}}. In view of~\eqref{eq: proof: prop: energy estimate in the de Sitter freqs, l larger that omega, eq 0} and of the condition~$\left(\omega-\frac{am\Xi}{\bar{r}_+^2+a^2}\right)^2\leq \epsilon^\prime \tilde{\lambda}$ we compute
\begin{equation}
	\begin{aligned}
		\frac{dV_0}{dr}(\bar{r}_+)	&	= \frac{d}{dr}\left(\frac{\Delta}{r^2+a^2}\right)(\tilde{\lambda}-2m\omega a \Xi) -\frac{4 a m \Xi  \bar{r}_+ \left(\omega  \left(a^2+\bar{r}_+^2\right)-a m \Xi \right)}{\left(a^2+\bar{r}_+^2\right)^3}\\
		&	< \frac{d}{dr}\left(\frac{\Delta}{r^2+a^2}\right)(\tilde{\lambda}-2m\omega a \Xi) +\sqrt{\epsilon^\prime}\tilde{\lambda}.
	\end{aligned}
\end{equation}
Therefore, in view of~\eqref{eq: proof: prop: energy estimate in the de Sitter freqs, l larger that omega, eq 1.0.1} and of the inequality~$\Xi^2m^2\leq \tilde{\lambda}$, see Lemma~\ref{lem: inequality for lambda}, we have that 
\begin{equation}\label{eq: proof: prop: energy estimate in the de Sitter freqs, l larger that omega, eq 6}
	\frac{dV_0}{dr}(\bar{r}_+)< -b\tilde{\lambda}
\end{equation}
for any~$\epsilon^\prime>0$ sufficiently small. Therefore, in view of~\eqref{eq: proof: prop: energy estimate in the de Sitter freqs, l larger that omega, eq 6},~\eqref{eq: prop: energy estimate in the de Sitter freqs, l larger that omega, eq 1}, Lemma~\ref{lem: subsec: sec: trapping, subsec 1, lem 1}, we have that the potential~$V_0$ attains a unique critical point~$r_{V_0,max}\in(r_++\epsilon_{away},\bar{r}_+-\epsilon_{away})$, for any~$\epsilon_{away}>0$ sufficiently small. Therefore, by following similar arguments presented earlier we have that for any~$\omega_{high}>0$ sufficiently large we obtain that the potential~$V$ attains a unique critical point~$r_{V,max}\in(r_++\epsilon_{away},\bar{r}_+-\epsilon_{away})$, for any~$\epsilon_{away}>0$ sufficiently small.

We construct a smooth~$f$ multiplier by solving the ode
\begin{equation}\label{eq: proof/prop: energy estimate in the de Sitter freqs, l larger that omega, eq 1.4}
	-\frac{1}{2}f^{\prime\prime\prime}=\Delta
\end{equation}
in a neighborhood of~$r_{V,max}$ with initial conditions 
\begin{equation}\label{eq: proof/prop: energy estimate in the de Sitter freqs, l larger that omega, eq 1.4.1}
	f(r_{V_0,max})=0,\qquad f^\prime(r_{V_0,max})=1,\qquad f^{\prime\prime}(r_{V_0,max})=0,
\end{equation}
and then we smoothly extend it such that
\begin{equation}\label{eq: proof/prop: energy estimate in the de Sitter freqs, l larger that omega, eq 1.4.2}
	f(r^\star=\infty)=1,\qquad f(r^\star=-\infty)=-1,\qquad \frac{df}{dr}>b,~r\in [r_+,\bar{r}_+],
\end{equation}
\begin{equation}
	|f(r)|,\quad |\frac{d f}{dr}|,\quad |\frac{d^2 f}{dr^2}|,\quad |\frac{d^3 f }{dr^3}| < B,\qquad r\in [r_+,\bar{r}_+].
\end{equation}
We integrate the energy identity of Lemma~\ref{lem: subsec: currents, lem 1} associated with the current~$Q^f$ and obtain the following energy estimate
	\begin{equation}\label{eq: proof/prop: energy estimate in the de Sitter freqs, l larger that omega, eq 3}
	\begin{aligned}
		&   \int_{\mathbb{R}}\Delta \left(|u^\prime|^2+|u|^2+\left(1-\frac{r_{V,max}}{r}\right)^2(1+\omega^2+\tilde{\lambda})|u|^2\right)dr^\star\\
		&   \quad \leq  B f(\infty)\left(\omega-\frac{am\Xi}{\bar{r}_+^2+a^2}\right)^2|u|^2(r^\star=\infty)+B(-f(-\infty))\left(\omega-\frac{am\Xi}{r_+^2+a^2}\right)^2|u|^2(r^\star=-\infty)\\
		&   \quad\quad +B\int_{\mathbb{R}} \left( |2f \Re (u^\prime \bar{H})|+|f\Re (u\bar{H})|\right) dr^\star. 
	\end{aligned}
\end{equation}
Now, for any~$\epsilon^\prime>0$ sufficiently small we construct a smooth~$h$ multiplier with the following properties 
\begin{equation}
	h= \begin{cases}
		1,\qquad r\in (r_{V,max}-\frac{1}{10}\epsilon^\prime,r_{V,max}+\frac{1}{10}\epsilon^\prime)\\
		0,\qquad r\in[r_+,\bar{r}_+]\setminus (r_{V,max}-\frac{1}{5}\epsilon^\prime,r_{V,max}+\frac{1}{5}\epsilon^\prime).
	\end{cases}
\end{equation}
In view of the non trapping property~(Not Trap) we have that
\begin{equation}\label{eq: proof/prop: energy estimate in the de Sitter freqs, l larger that omega, eq 1.3}
	V_0-\omega^2 \geq \epsilon_{trap} \tilde{\lambda},\qquad r\in (r_{V,max}-\frac{1}{10}\epsilon^\prime,r_{V,max}+\frac{1}{10}\epsilon^\prime)
\end{equation}
for any~$\epsilon^\prime>0$ both sufficiently small. Now, we integrate the energy identity of Lemma~\ref{lem: subsec: currents, lem 1} associated with the current~$Q^h$ and obtain
\begin{equation}
	\begin{aligned}
			&	\int_{\mathbb{R}} h |u|^2dr^\star+ \int_{(r_{V,max}-\frac{\epsilon'}{10},r_{V,max}+\frac{\epsilon'}{10})^\star} \left( h(V-\omega^2)-\frac{1}{2}h^{\prime\prime} \right)|u|^2dr^\star\\
			&	\qquad +  \int_{\mathbb{R}\setminus(r_{V,max}-\frac{\epsilon'}{10},r_{V,max}+\frac{\epsilon'}{10})^\star} \left( h(V-\omega^2)-\frac{1}{2}h^{\prime\prime} \right)|u|^2 dr^\star= -\int_{\mathbb{R}}h\Re (u\bar{H})dr^\star. 
	\end{aligned}
\end{equation}
Therefore, in view of the non trapping condition~\eqref{eq: proof/prop: energy estimate in the de Sitter freqs, l larger that omega, eq 1.3} we take~$\omega_{high}>0$ sufficiently large and conclude
\begin{equation}\label{eq: proof/prop: energy estimate in the de Sitter freqs, l larger that omega, eq 1.2}
	\begin{aligned}
		&	\int_{\mathbb{R}} h |u|^2dr^\star+ \int_{(r_{V,max}-\frac{\epsilon'}{10},r_{V,max}+\frac{\epsilon'}{10})^\star}  \tilde{\lambda}\Delta|u|^2dr^\star\\
		&	\qquad +  \int_{\mathbb{R}\setminus(r_{V,max}-\frac{\epsilon'}{10},r_{V,max}+\frac{\epsilon'}{10})^\star} \left( h(V-\omega^2)-\frac{1}{2}h^{\prime\prime} \right)|u|^2dr^\star \leq B |\int_{\mathbb{R}}h\Re (u\bar{H})|dr^\star. 
	\end{aligned}
\end{equation}

Then, we sum the above~\eqref{eq: proof/prop: energy estimate in the de Sitter freqs, l larger that omega, eq 1.2} to the energy estimate~\eqref{eq: proof/prop: energy estimate in the de Sitter freqs, l larger that omega, eq 3} to obtain the following
	\begin{equation}\label{eq: proof/prop: energy estimate in the de Sitter freqs, l larger that omega, eq 2}
	\begin{aligned}
		&   \int_{\mathbb{R}}\Delta \left(|u^\prime|^2+|u|^2+(1+\omega^2+\tilde{\lambda})|u|^2\right)dr^\star\\
		&   \quad \leq  B f(\infty)\left(\omega-\frac{am\Xi}{\bar{r}_+^2+a^2}\right)^2|u|^2(r^\star=\infty)-B f(-\infty)\left(\omega-\frac{am\Xi}{r_+^2+a^2}\right)^2|u|^2(r^\star=-\infty)\\
		&   \quad\quad +B\int_{\mathbb{R}}\left(| 2f \Re (u^\prime \bar{H})|+|f\Re (u\bar{H})|\right)dr^\star+B\int_{\mathbb{R}}\left(|h\Re (u\bar{H})|+|f \Re (u\bar{H})|+2|f\Re (u^\prime \bar{H})|\right)dr^\star.
	\end{aligned}
\end{equation}

Similarly to the argument presented before, see~\eqref{eq: proof/prop: energy estimate in the de Sitter freqs, l larger that omega, eq 1.1}, we integrate the energy identity of Lemma~\ref{lem: subsec: currents, lem 1} associated with the current~$-E Q^{K^+}- E Q^{\bar{K}^+}$, for a sufficiently large~$E>0$, we sum it to the above~\eqref{eq: proof/prop: energy estimate in the de Sitter freqs, l larger that omega, eq 2} and we conclude the desired result~\eqref{eq: prop: energy estimate in the de Sitter freqs, l larger that omega, eq 0}.

We now fix~$\epsilon^\prime(a,M,l,\mu_{KG})>0$.

\underline{We discuss the second case of~\eqref{eq: proof: prop: energy estimate in the de Sitter freqs, l larger that omega, eq 5}}. We note that the maximum of~$V_0$ cannot be attained at~$\bar{r}_+$ since then we would have~$-\left(\omega-\frac{am\Xi}{\bar{r}_+^2+a^2}\right)^2\geq \epsilon_{trap}\tilde{\lambda}>0$. Therefore, for any~$\omega_{high}>0$ sufficiently large and for any~$\epsilon_{away}>0$ sufficiently small we have that~$r_{V,max}\in [r_++\epsilon_{away},\bar{r}_+-\epsilon_{away}]$.

We fix~$\epsilon_{away}(a,M,l,\mu_{KG})>0$. We construct an~$f$ multiplier that satisfies exactly the properties~\eqref{eq: proof/prop: energy estimate in the de Sitter freqs, l larger that omega, eq 1.4},~\eqref{eq: proof/prop: energy estimate in the de Sitter freqs, l larger that omega, eq 1.4.1},~\eqref{eq: proof/prop: energy estimate in the de Sitter freqs, l larger that omega, eq 1.4.2}. We construct a smooth~$h$ multiplier that satisfies
\begin{equation}
	h= \begin{cases}
		1,\qquad r\in (r_{V,max}-\frac{\epsilon_1}{10},r_{V,max}+\frac{\epsilon_1}{10})\\
		0,\qquad r\in[r_+,\bar{r}_+]\setminus (r_{V,max}-\frac{\epsilon_1}{5},r_{V,max}+\frac{\epsilon_1}{5}),
	\end{cases}
\end{equation}
for some sufficiently small~$\epsilon_1(a,M,l,\mu_{KG})>0$. Now, we integrate the energy identities of Lemmata~\ref{lem: subsec: currents, lem 1},~\ref{lem: subsec: sec: frequency localized multiplier estimates, subsec 2, lem 1}, associated with the current 
\begin{equation}
	Q^f+Q^h-E Q^{K^+}-E Q^{\bar{K}^+},
\end{equation}
for a sufficiently large~$E>0$, and we conclude the desired energy estimate~\eqref{eq: prop: energy estimate in the de Sitter freqs, l larger that omega, eq 0}.

\underline{We discuss the third case of~\eqref{eq: proof: prop: energy estimate in the de Sitter freqs, l larger that omega, eq 5}}. This last case is the easiest since the non trapping condition~$\omega^2-\max_{r\in [r_+,\bar{r}_+]}V_0\geq \epsilon_{trap}\tilde{\lambda}$ implies immediately that
\begin{equation}
	\omega^2-V\geq \epsilon_{trap}\tilde{\lambda}
\end{equation}
for any~$r\in [r_+,\bar{r}_+]$. We construct a~$y$ multiplier of the form
\begin{equation}
	y(r)=e^{C_{large}\cdot r}
\end{equation}
where~$C_{large}>0$ will be chosen sufficiently large. We now integrate the energy identity of Lemma~\ref{lem: subsec: currents, lem 1} associated with the current~$Q^y$ to obtain 
\begin{equation}
	\int_{\mathbb{R}} \left(y^\prime |u^\prime|^2 +\left(y^\prime (\omega^2-V)-yV^\prime\right)|u|^2\right)dr^\star = Q^y[r^\star]\Big|_{r^\star=-\infty}^{r^\star=+\infty}+ \int_{\mathbb{R}}2y\Re (u^\prime H)dr^\star. 
\end{equation}
Now, we take~$C_{large}>0$ sufficiently large and we conclude 
\begin{equation}\label{eq: proof/prop: energy estimate in the de Sitter freqs, l larger that omega, eq 6}
	\int_{\mathbb{R}} \left(\Delta |u^\prime|^2 +\Delta\tilde{\lambda}|u|^2\right)dr^\star \leq B \Big|Q^y[r^\star]\Big|_{r^\star=-\infty}^{r^\star=+\infty}\big|+ |\int_{\mathbb{R}}2y\Re (u^\prime H)|dr^\star. 
\end{equation}
We fix~$C_{large}>0$. Now, we integrate the energy identities of Lemma~\ref{lem: subsec: sec: frequency localized multiplier estimates, subsec 2, lem 1} associated with the current~$-E Q^{K^+}-E  Q^{\bar{K}^+}$, for a sufficiently large~$E>0$ we sum it to~\eqref{eq: proof/prop: energy estimate in the de Sitter freqs, l larger that omega, eq 6}, in order to absorb the boundary terms, and obtain the desired energy estimate~\eqref{eq: prop: energy estimate in the de Sitter freqs, l larger that omega, eq 0}. 

\end{proof}

We have the following remarks

\begin{remark}
		We claim that for a sufficiently small~$0\leq a\ll M,l$ the set~$(\text{Trap})$, see~\eqref{eq: proof: prop: energy estimate in the de Sitter freqs, l larger that omega, eq 1}, would be empty, for any~$\epsilon_{trap}>0$ sufficiently small. To prove our claim we note that for~$(\omega,m,\ell)\in \mathcal{F}_{dS}$ we have~$0\leq |\omega|\leq \frac{|am|\Xi}{\bar{r}_+^2+a^2}\leq \frac{|am|\Xi}{r^2+a^2}$ for any~$r\in [r_+,\bar{r}_+]$. Therefore, for any~$0\leq|a|\ll M,l$ sufficiently small we have that~$\frac{dV_0}{dr}(r_+)\geq b\tilde{\lambda}$ and~$\frac{dV_0}{dr}(\bar{r}_+)< -b\tilde{\lambda}$ and therefore, in view of Lemma~\ref{lem: subsec: sec: trapping, subsec 1, lem 1}, we have that $\max_{r\in [r_+,\bar{r}_+]}V_0-\omega^2\geq b \tilde{\lambda}$, where the constant~$b$ is independent of~$\epsilon_{trap}$. Therefore, for any~$\epsilon_{trap}>0$ sufficiently small the set~$(Trap)$ is empty in the slowly rotating case~$0\leq|a|\ll M,l$.  
\end{remark}

\begin{remark}
When the condition~$\max_{r\in [r_+,\bar{r}_+]}\frac{\Delta}{a^2}\leq 1$ holds, see Section~\ref{sec: geodesics}, then we have that~$(Trap)\neq \emptyset$. Specifically, there exists a sequence of frequencies~$(\omega_n,m_n,\ell_n)\in \mathcal{F}_{dS}$ such that
\begin{equation*}
	\omega^2_n-\max_{r\in[r_+,\bar{r}_+]}V_0(r,a,M,l,\omega_n,m_n,\ell_n)\rightarrow 0,\qquad \omega_n\rightarrow 0. 
\end{equation*}
Namely, there exists a trapped null geodesic which is orthogonal to~$\partial_t$. 
\end{remark}

\subsection{Multipliers for the enlarged high superradiant frequency regime~$\mathcal{F}^{\sharp}$}\label{subsec: superradiant frequency regime}

This frequency regime is superradiant and `quantitatively' non-trapped, where for the later see already Lemma~\ref{lem: superradiant frequency energy estimate, lem 1}.

The main Proposition of this Section is the following
\begin{proposition}\label{prop: energy estimate in superradiant frequencies}
	
	Let~$l>0$,~$(a,M)\in \mathcal{B}_l$,~$\mu^2_{KG}\geq 0$. For any 
	\begin{equation}
	\alpha^{-1}>0,\qquad \omega_{high}>0
	\end{equation}
	both sufficiently large, then for~$(\omega,m,\ell)\in\mathcal{F}^{\sharp}(\omega_{high},\alpha)$, there exist smooth functions~$f,h,\chi_1,\chi_2$, satisfying the uniform bound
	\begin{equation}
	|f|+|f^{\prime}|+|f^{\prime\prime}|+|f^{\prime\prime\prime}|+|h|+|h^\prime|+|h^{\prime\prime}|+|\chi_1|+|\chi_2|+|\chi_1^\prime|+|\chi_2^\prime|\leq B,
	\end{equation}
	such that for all smooth solutions~$u$ of Carter's radial ode~\eqref{eq: ode from carter's separation} satisfying the outgoing boundary conditions~\eqref{eq: lem: sec carters separation, subsec boundary behaviour of u, boundary beh. of u, eq 3}, we have
	\begin{equation}\label{eq: prop: energy estimate in superradiant frequencies, eq 1}
	\begin{aligned}
	&   \int_{\mathbb{R}}\Delta\left(|u^\prime|^2+(1+\omega^2+\lambda^{(a\omega)}_{m\ell})|u|^2\right)dr^\star\\
	&   \quad\leq B\int_{\mathbb{R}}\left(\left|2f \Re (u^\prime\bar{H})\right|+\left|(f^\prime+h)\Re (u\bar{H})\right|+\left|\chi_2 (\omega-\frac{am\Xi}{\bar{r}_+^2+a^2})\Im (\bar{u}H)\right|+\left|\chi_1 (\omega-\frac{am\Xi}{r_+^2+a^2})\Im (\bar{u}H)\right|\right)dr^\star.
	\end{aligned}
	\end{equation}
\end{proposition}

First, note the following lemma 
\begin{lemma}\label{lem: energy estimate in superradiant frequencies, lem 1}
	Let~$l>0$,~$(a,M)\in\mathcal{B}_l$ and~$\mu^2_{KG}\geq 0$. Let~$\omega_{high}>0$. For any sufficiently large
	\begin{equation}
		\alpha^{-1}>0
	\end{equation}
	then for~$(\omega,m,\ell)\in\mathcal{F}^{\sharp}(\alpha,\omega_{high})$ we have
	\begin{equation}\label{eq: sec: energy estimate for superradiant regime/ eq 1}
	\tilde{\lambda}-2a\omega m\Xi\geq b\tilde{\lambda}>0.
	\end{equation}
\end{lemma}
\begin{proof}
	We use the following bound 
	\begin{equation}
	-2am\omega\Xi\geq -2\frac{(am\Xi)^2}{r_+^2+a^2}-2a\Xi\alpha\tilde{\lambda},    
	\end{equation}
	which follows immediately from the definition of the enlarged high superradiant frequency regime~$\mathcal{F}^\sharp$, see Section~\ref{sec: frequencies}, to obtain 
	\begin{equation}\label{eq: sec: energy estimate for superradiant regime/ eq 2}
	\tilde{\lambda}-2am\omega\Xi\geq \tilde{\lambda}-2\frac{(am\Xi)^2}{r_+^2+a^2}-2a\Xi\alpha\tilde{\lambda}.
	\end{equation}
	In view of the inequality~$r_+^2>a^2$, see Lemma~\ref{lem: sec: properties of Delta, lem 2, a,M,l}, we conclude that for a sufficiently small~$0<\epsilon(a,M,l)<1$ we obtain
	\begin{equation}\label{eq: sec: energy estimate for superradiant regime/ eq 2.5}
		\frac{2a^2}{r_+^2+a^2}\leq 1-\epsilon(a,M,l).
	\end{equation}
	
	Therefore, from~\eqref{eq: sec: energy estimate for superradiant regime/ eq 2.5},~\eqref{eq: sec: energy estimate for superradiant regime/ eq 2} in conjunction with the inequality~$\lambda^{(a\omega)}_{m\ell}+(a\omega)^2\geq \Xi^2 m^2$ for~$am\omega>0$, see Lemma~\ref{lem: inequality for lambda}, we obtain 
	\begin{equation}\label{eq: sec: energy estimate for superradiant regime/ eq 3}
		\begin{aligned}
			\tilde{\lambda}-2\frac{(am\Xi)^2}{r_+^2+a^2}	&	\geq \tilde{\lambda}+\Xi^2m^2\left(\epsilon(a,M,l)-1\right) \geq \left(1-\frac{\epsilon(a,M,l)}{2}\right)\tilde{\lambda} +\frac{\epsilon(a,M,l)}{2}\tilde{\lambda}-  \Xi^2m^2\left(1-\epsilon(a,M,l)\right)\\
			&	\geq \left(1-\frac{\epsilon(a,M,l)}{2}\right)m^2\Xi^2 -(1-\epsilon(a,M,l))\Xi^2m^2+\frac{\epsilon(a,M,l)}{2}\tilde{\lambda} \\
			&	\geq b\tilde{\lambda},
		\end{aligned} 
	\end{equation}
	for a sufficiently small~$\epsilon(a,M,l)>0$. 
	
	Finally, by using  equations \eqref{eq: sec: energy estimate for superradiant regime/ eq 2}, \eqref{eq: sec: energy estimate for superradiant regime/ eq 3}, we take~$\alpha>0$ sufficiently small and conclude. 
\end{proof}

The following lemma is a quantitative manifestation of the fact that the high (enlarged) superradiant frequencies are not trapped.
 
\begin{lemma}\label{lem: superradiant frequency energy estimate, lem 1}
	Let~$l>0$,~$(a,M)\in \mathcal{B}_l$ and~$\mu^2_{KG}\geq 0$. Then, for any
	\begin{equation}
	\alpha^{-1}>0,\qquad \omega_{\textit{high}}>0
	\end{equation}
	both sufficiently large, then for~$(\omega,m,\ell)\in\mathcal{F}^{\sharp}(\alpha,\omega_{high})$ there exists a unique global maximum~$r_{V,\textit{max}}\in (r_+,\bar{r}_+)$ of~$V$, where we have 
	\begin{equation}
	V(r_{V,\textit{max}})-\omega^2\geq b\tilde{\lambda}.
	\end{equation}
\end{lemma}
\begin{proof}
	We study the following cases
	\begin{equation}\label{eq: /subsec: energy estimate for superrad freq/ lem 1/ split the frequencies}
	\begin{aligned}
	  	 am\omega \in \left[\frac{a^2m^2\Xi}{r_+^2+a^2},\frac{a^2m^2\Xi}{r_+^2+a^2}+|a|\alpha\tilde{\lambda}\right),\qquad a m\omega \in\left(\frac{a^2m^2\Xi}{\bar{r}_+^2+a^2}\,\frac{a^2m^2\Xi}{r_+^2+a^2}\right).
	\end{aligned}
	\end{equation}
	
	In this proof, we use
	\begin{equation}
		r_{V_0,\textit{max}},\qquad r_{V,\textit{max}}
	\end{equation}
	to denote respectively the locations of the unique global maximums of the potentials~$V_0,V$.

		\boxed{For ~the ~first ~case~ of~\eqref{eq: /subsec: energy estimate for superrad freq/ lem 1/ split the frequencies}} we use Lemma~\ref{lem: energy estimate in superradiant frequencies, lem 1}, in view of the that~$\frac{dV_0}{dr}(r_+)\geq 0,\frac{dV_0}{dr}(\bar{r}_+)<-b\tilde{\lambda}$ and conclude that if we take~$\alpha>0$ sufficiently small, then
	\begin{equation}
	V_0(r_{V_0,\textit{max}})-\omega^2\geq b\tilde{\lambda},
	\end{equation}
	for~$r_{V_0,max}\in (r_+,\bar{r}_+)$. Therefore, by taking~$\omega_{high}>0$ sufficiently large we obtain the desired result for~$V$.

	\boxed{For ~the ~second ~case~ of~\eqref{eq: /subsec: energy estimate for superrad freq/ lem 1/ split the frequencies}}, namely 
	\begin{equation}
	a m\omega\in\left(\frac{a^2m^2\Xi}{\bar{r}_+^2+a^2},\frac{a^2m^2\Xi}{r_+^2+a^2}\right)    
	\end{equation}
	we proceed as follows. There exists aν~$r_s (\omega,m) \in(r_+,\bar{r}_+)$ such that
	\begin{equation}
	a m\omega=\frac{a^2m^2\Xi}{r_s^2+a^2}.
	\end{equation}
	We use Lemma~\ref{lem: energy estimate in superradiant frequencies, lem 1} and obtain~$\omega^2-V_0(r_s)<-b\tilde{\lambda}$ which implies
	\begin{equation}
	V_0(r_s)-\omega^2 > b\tilde{\lambda}. 
	\end{equation}
	and therefore we trivially obtain 
	\begin{equation}
	V_{0}(r_{V_0,\textit{max}})-\omega^2\geq b\tilde{\lambda}. 
	\end{equation}

	Therefore, in all the cases of~\eqref{eq: /subsec: energy estimate for superrad freq/ lem 1/ split the frequencies} we have concluded that for any~$\alpha>0$ sufficiently small there exist a global maximum~$r_{V_0,\textit{max}}$, of~$V_0$, such that 
	\begin{equation}
	V_0(r_{V_0,\textit{max}})-\omega^2\geq b\tilde{\lambda}.
	\end{equation} 
	Then, for a sufficiently large~$\omega_{\textit{high}}>0$ we obtain 
	\begin{equation}
	V(r_{V,\textit{max}})-\omega^2\geq b\tilde{\lambda},   
	\end{equation}
	where~$r_{V,\textit{max}}\in (r_+,\bar{r}_+)$ is the value of the global maximum of~$V$.

	We conclude the proof. 
\end{proof}

Moreover, note the following lemma. 
\begin{lemma}\label{lem: /superradiant frequency energy estimate/ lem 2/ critical points of V}
	
	Let~$l>0$ and~$(a,M)\in\mathcal{B}_l$ and~$\mu^2_{KG}\geq 0$. Then, for any
	\begin{equation}
		\alpha^{-1}>0,\qquad 
	\omega_{\textit{high}}>0,\qquad \delta^{-1}>0,
	\end{equation}
	all sufficiently large, then for~$(\omega,m,\ell)\in \mathcal{F}^\sharp(\omega_{high},\alpha)$ the following hold  \begin{equation}\label{eq: lem: /superradiant frequency energy estimate/ lem 2/ critical points of V/ eq 1}
	\begin{aligned}
	&   V(r)-\omega^2\geq b\tilde{\lambda},\qquad r\in (r_{V,\textit{max}}-\delta,r_{V,\textit{max}}+\delta)\\ 
	&   -(r-r_{V,\textit{max}})\frac{dV}{dr}\geq b\tilde{\lambda} (r-r_{V,\textit{max}})^2,\qquad r\in[r_+,\bar{r}_+].
	\end{aligned}
	\end{equation} 
\end{lemma}
\begin{proof}
	The first inequality of \eqref{eq: lem: /superradiant frequency energy estimate/ lem 2/ critical points of V/ eq 1}, follows readily from Lemma \ref{lem: superradiant frequency energy estimate, lem 1} by the continuity of~$V$, after taking~$\delta>0$ sufficiently small. 
	
	For the second inequality of equation \eqref{eq: lem: /superradiant frequency energy estimate/ lem 2/ critical points of V/ eq 1} we proceed as follows. We compute the derivative 
	\begin{equation}\label{eq:lem: /superradiant frequency energy estimate/ lem 2/ critical points of V, eq 1}
	\frac{dV_0}{dr}(r)=\frac{d}{dr}\left(\frac{\Delta}{(r^2+a^2)^2}\right)(\tilde{\lambda}-2a\omega m\Xi)-2\left(\omega-\frac{am\Xi}{r^2+a^2}\right)\left(2r\frac{am\Xi}{(r^2+a^2)^2}\right),
	\end{equation}
	and we note that, in view of Lemma~\ref{lem: energy estimate in superradiant frequencies, lem 1}, the following hold
	\begin{equation}\label{eq:lem: /superradiant frequency energy estimate/ lem 2/ critical points of V, eq 2}
	\frac{dV_0}{dr}(r_+)=\frac{d}{dr}\left(\frac{\Delta}{(r^2+a^2)^2}\right)(r_+)(\tilde{\lambda}-2a\omega m\Xi)-2\left(\omega-\frac{am\Xi}{r^2+a^2}\right)\left(2r\frac{am\Xi}{(r^2+a^2)^2}\right)(r_+)>b\tilde{\lambda}>0,
	\end{equation}
	\begin{equation}\label{eq:lem: /superradiant frequency energy estimate/ lem 2/ critical points of V, eq 3}
	\frac{dV_0}{dr}(\bar{r}_+)=\frac{d}{dr}\left(\frac{\Delta}{(r^2+a^2)^2}\right)(\bar{r}_+)(\tilde{\lambda}-2a\omega m\Xi)-2\left(\omega-\frac{am\Xi}{r^2+a^2}\right)\left(2r\frac{am\Xi}{(r^2+a^2)^2}\right)(\bar{r}_+)<-b\tilde{\lambda}<0,
	\end{equation}
	where we used inequality~\eqref{eq: sec: energy estimate for superradiant regime/ eq 1} and that~$a m\omega\in\left(\frac{a^2m^2\Xi}{\bar{r}_+^2+a^2},\frac{a^2m^2\Xi}{r_+^2+a^2}+|a|\alpha\tilde{\lambda}\right)$ for~$\alpha>0$ sufficiently small.

	We claim that the potential~$V_0$ attains exactly one critical point~$r_{V_0,\textit{max}}\in (r_+,\bar{r}_+)$. To prove this note that if the potential~$V_0$ attained more than one critical points, then it would have to attain at least three critical points, because of the behaviour of the derivatives~\eqref{eq:lem: /superradiant frequency energy estimate/ lem 2/ critical points of V, eq 2},~\eqref{eq:lem: /superradiant frequency energy estimate/ lem 2/ critical points of V, eq 3}. But then the functions~$\frac{dV_0}{dr}$ and~$(r^2+a^2)^3\frac{dV_0}{dr}$ would have to change sign at least three times. This implies that 
	\begin{equation}
	(r^2+a^2)^3\frac{dV_0}{dr}
	\end{equation}
	would attain at least 2 local extrema in~$[r_+,\bar{r}_+]$. This is a contradiction to Lemma~\ref{lem: subsec: sec: trapping, subsec 1, lem 1}. Therefore, the potential~$V_0$ attains exactly one critical point~$r_{V_0}\in [r_+,\bar{r}_+]$.

	Again, in view of Lemma~\ref{lem: subsec: sec: trapping, subsec 1, lem 1}, and specifically in view of the fact that the function~$(r^2+a^2)^3\frac{dV_0}{dr}$ is strictly negative, and therefore for~$\omega_{high}>0$ sufficiently large the function
	\begin{equation}
		(r^2+a^2)^3\frac{dV}{dr}
	\end{equation}
	is strictly negative. Therefore, we obtain the second inequality of~\eqref{eq: lem: /superradiant frequency energy estimate/ lem 2/ critical points of V/ eq 1}.

	Now, we prove the first inequality of~\eqref{eq: lem: /superradiant frequency energy estimate/ lem 2/ critical points of V/ eq 1} which is a quantitative manifestation of that the enlarged high superradiant frequencies are not trapped. For the superradiant frequencies~$\mathcal{SF}$ we use that there exists a~$r_s\in (r_+,\bar{r}_+)$ such that~$am\omega = \frac{a^2m^2\Xi}{r_s^2+a^2}$. We easily conclude that
	\begin{equation}
		V(r_s)-\omega^2 \geq b \tilde{\lambda}
	\end{equation}
	and therefore
	\begin{equation}
		\max_{r\in (r_+,\bar{r}_+)}V(r)-\omega^2 \geq b\tilde{\lambda}.
	\end{equation}
	Therefore, for~$\delta>0$ sufficiently small and~$\omega_{high}>0$ sufficiently large the inequalities~\eqref{eq: lem: /superradiant frequency energy estimate/ lem 2/ critical points of V/ eq 1}  hold.

	Now, we prove the first inequality of~\eqref{eq: lem: /superradiant frequency energy estimate/ lem 2/ critical points of V/ eq 1} for the enlarged superradiant frequencies
	\begin{equation}
		am\omega \in \big(\frac{a^2m^2\Xi}{\bar{r}_+^2+a^2},\frac{a^2m^2\Xi}{r_+^2+a^2}+|a|\alpha\tilde{\lambda} \big). 
	\end{equation}
	Specifically, for~$am\omega \in \big(\frac{a^2m^2\Xi}{r_+^2+a^2},\frac{a^2m^2\Xi}{r_+^2+a^2}+|a|\alpha\tilde{\lambda} \big)$ we proceed as follows. 
	
	We have that
	\begin{equation}
		V_0(r_++\sqrt{\alpha})-\omega^2=\frac{\Delta}{(r^2+a^2)^2}\Big|_{r=r_++\sqrt{\alpha}}(\tilde{\lambda}-2m\omega a\Xi)-\left(\omega-\frac{am\Xi}{(r_++\sqrt{\alpha})^2+a^2}\right)^2
	\end{equation}
	Now, we note that as~$\alpha\rightarrow 0$ we have that
	\begin{equation}
		\frac{\Delta}{(r^2+a^2)^2}\Big|_{r=r_++\sqrt{\alpha}}(\tilde{\lambda}-2m\omega a \Xi)\sim \sqrt{\alpha}(\tilde{\lambda}-2m\omega a \Xi),\qquad \left(\omega-\frac{am\Xi}{(r_++\sqrt{\alpha})^2+a^2}\right)^2\lesssim \alpha\left(|a|\tilde{\lambda} +(am)^2\right)
	\end{equation}
	Therefore, we immediately obtain that
	\begin{equation}
		V_0(r_++\sqrt{\alpha})-\omega^2\geq b(\alpha)\tilde{\lambda}.
	\end{equation}
	 We may now easily conclude the desired result.
\end{proof}

Now, we are ready to prove Proposition \ref{prop: energy estimate in superradiant frequencies}. 

\begin{proof}[\textbf{Proof of Proposition~\ref{prop: energy estimate in superradiant frequencies}}]
	Let~$\delta>0$ be sufficiently small~(also see Lemma~\ref{lem: /superradiant frequency energy estimate/ lem 2/ critical points of V}). We define a smooth multiplier function~$h$, such that 
	\begin{equation}
	h=
	\begin{cases}
	1,\: r\in (r_{V,\textit{max}}-\frac{\delta}{2},r_{V,\textit{max}}+\frac{\delta}{2})\\
	0,\: r\in (r_{V,\textit{max}}-\delta,r_{V,\textit{max}}+\delta)^c.
	\end{cases}
	\end{equation}
	
	We define a smooth multiplier function~$f$, such that 
	\begin{equation}
	f=(r-r_{V,\textit{max}})\tilde{f}    
	\end{equation}
	where~$\tilde{f},\frac{d f}{dr}>c_1>0$, for some strictly positive constant~$c_1=c_1(a,M,l,\mu_{KG})>0$.

	Therefore, we integrate the energy identity of Lemma~\ref{lem: subsec: currents, lem 1} associated to~$Q^f+Q^h$.  by summing the two currents~$Q^h,Q^f$ we obtain the identity
	\begin{equation}
	\begin{aligned}
	\int_\mathbb{R}\left(\left(h+2f^\prime\right)|u^\prime|^2+\left(h(V-\omega^2)-\frac{1}{2}h^{\prime\prime}-fV^\prime-\frac{1}{2}f^{\prime\prime\prime}\right)|u|^2\right)dr^\star&=Q^{f}(\infty)-Q^f(-\infty)\\
	&   \quad +\int_{\mathbb{R}}\left(-h\Re(u\bar{H})-\Re(2f\bar{H}u^\prime+f^\prime \bar{H}u)\right)dr^\star,
	\end{aligned}
	\end{equation}
	which, by using Lemmata~\ref{lem: superradiant frequency energy estimate, lem 1},~\ref{lem: /superradiant frequency energy estimate/ lem 2/ critical points of V} and by taking~$\omega_{\textit{high}}>0$ sufficiently large, we obtain  
	\begin{equation}\label{eq: energy estimate in superradiant frequencies, eq 1}
	\begin{aligned}
	&  b \int_\mathbb{R}  \Delta\left(|u^\prime|^2+\left(1+\omega^2+\lambda^{(a\omega)}_{m\ell}\right)|u|^2\right)dr^\star\\
	& \quad \leq  Q^{f}(\infty)-Q^f(-\infty) + \int_{\mathbb{R}}\left(-h\Re(u\bar{H})-\Re(2f\bar{H}u^\prime+f^\prime \bar{H}u)\right)dr^\star\\
	&   \quad \leq  f(\infty)\left(\omega-\frac{am\Xi}{\bar{r}_+^2+a^2}\right)^2|u|^2(\infty)-f(-\infty)\left(\omega-\frac{am\Xi}{r_+^2+a^2}\right)^2|u|^2(-\infty)\\
	&   \quad\quad+ \int_{\mathbb{R}}\left(-h\Re(u\bar{H})-\Re(2f\bar{H}u^\prime+f^\prime \bar{H}u)\right)dr^\star.
	\end{aligned}
	\end{equation}

	Now, we will absorb the boundary terms on the right hand side of equation~\eqref{eq: energy estimate in superradiant frequencies, eq 1}. We choose smooth cut-off functions~$\chi_1,~\chi_2$ such that
	\begin{equation}\label{eq: energy estimate in superradiant frequencies, eq 2}
	\chi_1 =
	\begin{cases}
	1,\: r\in[r_+,r_{V,\textit{max}}-\delta]\\
	0,\: r\in[r_{V,\textit{max}}+\delta,\bar{r}_+]
	\end{cases}
	,\qquad 	\textit{supp}(\chi^\prime_1)\subset (r_{V,\textit{max}}-\frac{\delta}{2},r_{V,\textit{max}}+\frac{\delta}{2}),
	\end{equation}
	\begin{equation}
	\chi_2=
	\begin{cases}
	0,\: r\in[r_+,r_{V,\textit{max}}-\delta]\\
	1,\: r\in[r_{V,\textit{max}}+\delta,\bar{r}_+]  
	\end{cases}
	,\qquad \textit{supp}(\chi^\prime_2)\subset (r_{V,\textit{max}}-\frac{\delta}{2},r_{V,\textit{max}}+\frac{\delta}{2}).
	\end{equation}
	By using the fundamental theorem of calculus the following inequality holds
	\begin{equation}\label{eq: energy estimate in superradiant frequencies, eq 3}
		\begin{aligned}
			\left(|u^\prime|^2 +(\omega-\frac{am\Xi}{r_+^2+a^2})|u|^2 \right)(r_+)\leq E_0 \Big| \int_{r_{V,max}-\delta/2}^{r_{V,max}+\delta/2} \chi_1^\prime (\omega-\frac{am\Xi}{r_+^2+a^2})\Im (u^\prime\bar{u})  dr\Big| +E_0\int_\mathbb{R} \chi_1 (\omega-\frac{am\Xi}{r_+^2+a^2})\Im (H\bar{u})dr
		\end{aligned}
	\end{equation}
	where we also have the bound
	\begin{equation}
		\begin{aligned}
			 E_0 \Big| \int_{r_{V,max}-\delta/2}^{r_{V,max}+\delta/2} \chi_1^\prime (\omega-\frac{am\Xi}{r_+^2+a^2})\Im (u^\prime\bar{u})  dr\Big| &	\leq E_0\delta^{-1} \int_{r_{V,max}-\delta/2}^{r_{V,max}+\delta/2} \left( |u^\prime|^2 + (\omega^2+m^2)|u|^2 \right)dr\\
			 &	\ll A \int_{r_{V,max}-\delta/2}^{r_{V,max}+\delta/2} \left( |u^\prime|^2 + (\omega^2+m^2)|u|^2 \right)dr,
		\end{aligned}
	\end{equation}
	where we choose~$E_0\delta^{-1}\ll \frac{1}{2}\tilde{\epsilon}\omega^2_{high}$ and~$A=\frac{1}{2}\tilde{\epsilon}\omega^2_{high}$. Of course, a similar estimate holds for~$\chi_2$.

	Finally, we integrate the energy identity of Lemma~\ref{lem: subsec: sec: frequency localized multiplier estimates, subsec 2, lem 1} associated with the current
	\begin{equation}
	-E_0\chi_1Q^{K^+}-E_0\chi_2 Q^{\bar{K}^+},
	\end{equation}
	for a sufficiently large~$E_0>0$, such that  in view of Proposition~\ref{prop: subsec: energy identity, prop 1} we bound the boundary terms of~\eqref{eq: energy estimate in superradiant frequencies, eq 1} in view of~\eqref{eq: energy estimate in superradiant frequencies, eq 3} and by using the quantitative non trapping inequality that holds for the enlarged high superradiant frequencies, see the first inequality of~\eqref{eq: lem: /superradiant frequency energy estimate/ lem 2/ critical points of V/ eq 1}.

	We conclude the Proposition~\ref{prop: energy estimate in superradiant frequencies}.  
\end{proof}

\subsection{Multipliers for the bounded frequency regime~$F_{\flat}$}\label{subsec: bounded frequencies}
We split the bounded frequency regime~$\mathcal{F}_\flat$, see Section~\ref{sec: frequencies}, into the following near stationary cases 
\begin{equation}
	\begin{aligned}
		&	\bullet \{|\omega|\leq \omega_{\textit{low}}\}\cap\{|m|>0\}\quad\text{and}\quad |a|\leq a_0\\ 
		&	\bullet 	\{|\omega|\leq \omega_{\textit{low}}\}\cap\{|m|=0\}\\
		&	\bullet 	\{|\omega|\leq \omega_{\textit{low}}\}\cap\{|m|>0\}\quad\text{and}\quad |a|\geq a_0
	\end{aligned}
\end{equation}
and the non-stationary case
\begin{equation}
	\{|\omega|\geq \omega_{\textit{low}}\}
\end{equation}
where
\begin{equation}
a_0,\qquad \omega_{\textit{low}}>0,
\end{equation}
are both sufficiently small, and are chosen in Propositions~\ref{prop: energy estimate for the bounded stationary frequencies},~\ref{prop: subsubsec: bounded stationary freqs, large a} respectively.

\subsubsection{\textbf{Near stationary case} \texorpdfstring{$\{|\omega|\leq \omega_{\textit{low}}\}\cap\{|m|\geq 1\}$ and \texorpdfstring{$\{|a|\leq a_0\}$}{g}}{g}}\label{subsubsec: sec: bounded: stationary}

The main Proposition of this Section is the following 
\begin{proposition}\label{prop: energy estimate for the bounded stationary frequencies}
	
	Let~$l>0$,~$(a,M)\in \mathcal{B}_l$,~$\mu^2_{KG}\geq 0$. Let~$\lambda_{\textit{low}},\omega_{\textit{high}}>0$ be arbitrary. Then, for any
	\begin{equation}
	a_0(M,l,\mu_{KG})>0,\qquad \omega_{\textit{low}}>0,
	\end{equation}
	both sufficiently small we have that if
	\begin{equation}
	|a|\leq a_0
	\end{equation}
then for~$(\omega,m,\ell)\in\mathcal{F}_\flat(\omega_{high},\lambda_{low}) \cap\{|\omega|\leq \omega_{\textit{low}}\}\cap\{|m|>0\}$, there exist sufficiently regular functions~$h,y$, where~$y$ is piecewise~$C^1$, satisfying the uniform bound
	\begin{equation}
	|h|+|h^\prime|+|h^{\prime\prime}|+|y|+|y^\prime|\leq B,
	\end{equation}
	such that for all smooth solutions~$u$ of Carter's radial ode~\eqref{eq: ode from carter's separation} satisfying the outgoing boundary conditions~\eqref{eq: lem: sec carters separation, subsec boundary behaviour of u, boundary beh. of u, eq 3}, we have 
	\begin{equation}\label{eq: prop: energy estimate for the bounded stationary frequencies, eq 1}
	\begin{aligned}
	\int_{\mathbb{R}}  & \Delta\big( |\Psi^\prime|^2+(1+\omega^2+\lambda^{(a\omega)}_{m\ell})|\Psi|^2  \big)dr^\star\\
	&   \leq B \int_{\mathbb{R}}\Bigg( \left|\left(\omega-\frac{am\Xi}{\bar{r}_+^2+a^2}\right)\Im (\bar{\Psi}H)\right|+\left|\left(\omega-\frac{am\Xi}{r_+^2+a^2}\right)\Im(\bar{\Psi}H)\right|\\
	&   \qquad\qquad\qquad + \left|\Re\left(\frac{2\Psi h\bar{H}}{\sqrt{r^{2}+a^{2}}}\right)\right|+ \left|\Re\left(\frac{y\Psi^{\prime}\bar{H}}{\sqrt{r^{2}+a^{2}}} \right)\right|\Bigg)dr^\star.
	\end{aligned}
	\end{equation}
\end{proposition}

We recall from Lemma~\ref{lem: sec: frequency localized multiplier estimates, lem 2} that 
\begin{equation}\label{eq: subsubsec: sec: bounded: stationary, eq 1}
\tilde{V}\:\dot{=}\:\frac{(\lambda^{(a\omega)}_{m\ell}+a^2\omega^2-2m\omega a \Xi)\Delta}{(r^{2}+a^{2})^{2}} +\Delta\mu_{\textit{KG}}^2\frac{r^{2}+a^{2}}{(r^{2}+a^{2})^{2}}+\omega^2-\left(\omega-\frac{am\Xi}{r^2+a^2}\right)^2.
\end{equation}

Note the following Lemma on the behaviour of~$\tilde{V}$ in a neighborhood of the infinities.

\begin{lemma}\label{lem: stationary frequencies, lem 0}
	Let~$l>0$,~$(a,M)\in \mathcal{B}_l$ and~$\mu^2_{KG}\geq 0$. Then, for any
	\begin{equation}
		 a_0^{-1}>0,\qquad \omega_{\textit{low}}^{-1}>0,\qquad R^\star>0
	\end{equation}
	all sufficiently large, then if~$|a|\leq a_0$ then for~$(\omega,m,\ell)\in\mathcal{F}_\flat \cap\{|\omega|\leq \omega_{\textit{low}}\}\cap\{|m|>0\}$ we have
	\begin{equation}\label{eq: lem: stationary frequencies, lem 0, eq 1}
	\begin{aligned}
	\tilde{V}^\prime 	 <-b\Delta\tilde{\lambda},\:\:r^\star\in(R^\star,\infty),\qquad 
	\tilde{V}^\prime >b \Delta\tilde{\lambda},\:\: r^\star\in (-\infty,-R^\star).
	\end{aligned}
	\end{equation}
	\begin{equation}\label{eq: lem: stationary frequencies, lem 0, eq 2}
		|\tilde{V}|\leq \Delta \tilde{\lambda}+(a_0m)^2+ \omega_{low}^2,\:\: r^\star\in (R^\star,\infty), 
	\end{equation}	
	where for the potential~$\tilde{V}$ see~\eqref{eq: subsubsec: sec: bounded: stationary, eq 1}. 
	
	Moreover, we have that 
	\begin{equation}
	\lambda^{(a\omega)}_{m\ell} \geq \frac{1}{2}. 
	\end{equation}
\end{lemma}
\begin{proof}
From Lemma~\ref{lem: inequality for lambda} we note that~$|m|>0$ implies that~$\lambda^{(a\omega)}_{m\ell}\geq \frac{1}{2}$, for~$a_0>0$ sufficiently small

	For~$\omega_{\textit{low}}$ sufficiently small Lemma~\ref{lem: inequality for lambda} implies the following 
	\begin{equation}
	\tilde{\lambda}-2m\omega a \Xi >b>0. 
	\end{equation}
	Therefore, again it is immediate that for any~$R^\star$ sufficiently large, we differentiate the potential~$\tilde{V}$, see~\eqref{eq: subsubsec: sec: bounded: stationary, eq 1}, and obtain the desired result~\eqref{eq: lem: stationary frequencies, lem 0, eq 1}, for~$a_0,~\omega_{\textit{low}}>0$ sufficiently small. 
	
	Then, it is easy to see that for~$\omega_{low}>0$ sufficiently small we obtain inequality~\eqref{eq: lem: stationary frequencies, lem 0, eq 2} by inspecting~\eqref{eq: subsubsec: sec: bounded: stationary, eq 1}.

	We conclude the proof of the Lemma. 
\end{proof}

Now, we proceed to construct the multipliers~$h,y$ of Proposition~\ref{prop: energy estimate for the bounded stationary frequencies}. For a graphic representation see Figure~\ref{fig: current}, where note that the (red) graph that vanishes at the horizons corresponds to the function~$h$ and the green graph corresponds to~$y$. 
	\begin{figure}[htbp]
	\centering
	\includegraphics[scale=1.5]{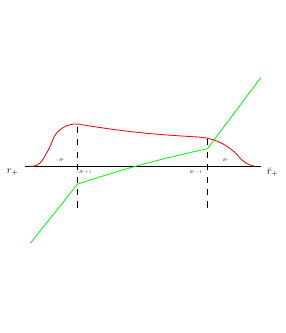}
	\caption{The red graph is~$h$, the green graph is~$y$}
	\label{fig: current}
\end{figure}

First, we construct a piecewise~$C^1$ multiplier~$y$ in~$(-R^\star+1,R^\star-1)$ and employ a~$Q^y_{\textit{stat}}$ current, see Definition~\ref{def: currents for the ell=0 case}.

\begin{lemma}\label{lem: stationary frequencies for y, 1}
	Let~$R^\star>0$ be sufficiently large and let~$A_{\textit{large}}>0$. Then, there exists a piecewise~$C^1$ multiplier~$y$ such that the following hold. 
	
	For~$r^\star\in (-R^\star+1,R^\star-1)$ we have 
	\begin{equation}\label{eq: lem: stationary frequencies for y, 1, eq 1}
	\begin{aligned}
	(Q^y_{\textit{stat}})^\prime &  =\left(\left(\frac{r_++\bar{r}_+}{2}\right)^2+a^2\right)\frac{4r\Delta(r)}{(r^2+a^2)^2}|\Psi^\prime|^2 \\
	&   \quad + \left(  4r(r^2+a^2)\frac{\Delta}{r^2+a^2}(\omega^2-\tilde{V}) -\left( (r^2+a^2)^2 -\left(\left(\frac{r_++\bar{r}_+}{2}\right)^2+a^2\right)^2   \right)\tilde{V}^\prime \right)|\Psi|^2 \\
	&  \quad +\Re\left(\frac{y\Psi^\prime\bar{H}}{\sqrt{r^2+a^2}}\right),
	\end{aligned}
	\end{equation}
	where for the current~$Q^y_{\textit{stat}}$ see Definition~\ref{def: currents for the ell=0 case}.
	
	For~$r^\star\in (-\infty,-R^\star+1)\cup  (R^\star-1,\infty)$ we have 
	\begin{equation}\label{eq: lem: stationary frequencies for y, 1, eq 2}
		\frac{d y}{dr} -\frac{4r}{r^2+a^2}y=A_{\textit{large}}>0.	
	\end{equation}
\end{lemma}
\begin{proof}
	Let~$A_{\textit{large}}>0$ be any positive number and let~$R^\star>0$ be sufficiently large. In the region~$r^\star\in (-R^\star+1,R^\star-1)$ we define the multiplier
\begin{equation}\label{eq: proof: lem: stationary frequencies for y, 1, eq 1}
y=(r^2+a^2)^2-\left(\left(\frac{r_++\bar{r}_+}{2}\right)^2+a^2\right)^2,    
\end{equation}
which satisfies~\eqref{eq: lem: stationary frequencies for y, 1, eq 1}.

In the regions~$r^\star\in (-\infty,-R^\star+1)$ and~$r^\star\in (R^\star-1,\infty)$ we define respectively 
\begin{equation}\label{eq: proof lem: stationary frequencies for y, 1, eq 2}
y(r^\star(r))=\left( \frac{y}{(r^2+a^2)^2}\right)(r_2)(r^2+a^2)^2+A_{\textit{large}}(r^2+a^2)^2\int_{r_2}^{r} \frac{d r}{(r^2+a^2)^2},
\end{equation}
\begin{equation}\label{eq: proof lem: stationary frequencies for y, 1, eq 3}
y(r^\star(r))=\left( \frac{y}{(r^2+a^2)^2}\right)(r_1)(r^2+a^2)^2+A_{\textit{large}}(r^2+a^2)^2\int_{r_1}^r \frac{d r}{(r^2+a^2)^2},
\end{equation}
where~$r_1,r_2$ are such that~$r^\star(r_2)=-R^\star-1,~r^\star(r_1)=R^\star-1$ where note that for both~\eqref{eq: proof lem: stationary frequencies for y, 1, eq 2},~\eqref{eq: proof lem: stationary frequencies for y, 1, eq 3} the equation~\eqref{eq: lem: stationary frequencies for y, 1, eq 2} holds. 
\end{proof}

We construct a multiplier~$h$ in order to employ a~$Q^h_{\textit{stat}}$ current, see Definition~\ref{def: currents for the ell=0 case}.

\begin{lemma}\label{lem: h multiplier in the middle region in stationary frequencies, 2} 
Let~$l>0$ and~$(a,M)\in\mathcal{B}_l$ and~$\mu^2_{KG}\geq 0$. Let~$R^\star>0$ be any sufficiently large positive number, and let the multiplier~$y$ be as in Lemma~\ref{lem: stationary frequencies for y, 1}. Then, for any
\begin{equation}
	a_0(M,l,\mu_{KG})>0,\qquad \omega_{\textit{low}}>0,
\end{equation}
both sufficiently small such that if~$|a|\leq a_0$ then for~$(\omega,m,\ell)\in\mathcal{F}_\flat \cap\{|\omega|\leq \omega_{\textit{low}}\}\cap\{|m|>0\}$ and for any sufficiently large constant 
\begin{equation}
	C_{\textit{large}}(a,a_0,M,l,\mu_{KG},y|_{(-R^\star+1,R^\star-1)},\omega_{low})>0
\end{equation}
there exists a smooth~$h$ multiplier, satisfying the uniform bound 
\begin{equation}
|h|+|h^\prime|+|h^{\prime\prime}|\leq B(\omega_{low},a_0),
\end{equation}
for all~$r^\star\in \mathbb{R}$ and~$h\equiv 0$ in neighborhoods of~$\pm \infty$, and we have
	\begin{equation}\label{eq: lem: h multiplier in the middle region in stationary frequencies, 2, eq 1}
	h=C_{\textit{large}}\frac{\tilde{\lambda}}{r^2+a^2},\qquad y^\prime \omega^2-(y\tilde{V})^\prime-\frac{1}{2}h^{\prime\prime}+ \left( \frac{r\Delta h}{(r^2+a^2)^2}\right)^\prime +h(\tilde{V}-\omega^2)\geq b\tilde{\lambda},\qquad y^\prime -\frac{4r\Delta}{(r^2+a^2)^2}y+h\geq b\Delta,
	\end{equation} 
	for~$r^\star\in(-R^\star+1,R^\star-1)$. 
\end{lemma}

\begin{proof}
	Consider a
	\begin{equation}
	h=C_{\textit{large}}\tilde{\lambda} \frac{1}{r^2+a^2}   
	\end{equation}
	multiplier, where~$C_{\textit{large}}>0$, and note
	\begin{equation}\label{eq: proof lem: h multiplier in the middle region in stationary frequencies, 2, eq 1}
	g(r)\equiv \frac{1}{r^2+a^2}\implies -\frac{1}{2}g^{\prime\prime}+\left(\frac{r\Delta}{(r^2+a^2)^2} g\right)^\prime\equiv 0.
	\end{equation}
	Therefore, for any~$a_0(M,l,\mu_{KG})>0$,~$\omega_{\textit{low}}(a_0,M,l,\mu_{KG})>0$ both sufficiently small, we obtain the following
	\begin{equation}\label{eq: proof lem: h multiplier in the middle region in stationary frequencies, 2, eq 2}
	\begin{aligned}
	&   -\frac{1}{2}h^{\prime\prime}+ \left( \frac{r\Delta h}{(r^2+a^2)^2}\right)^\prime +h(\tilde{V}-\omega^2) \\
	&	= C_{\textit{large}}\frac{1}{r^2+a^2}\tilde{\lambda}(\tilde{V}-\omega^2) \\
	&   =C_{\textit{large}} \frac{1}{r^2+a^2} \tilde{\lambda}\left(\frac{\Delta}{(r^2+a^2)^2}(\tilde{\lambda}-2m\omega a\Xi)-\left(\omega-\frac{am\Xi}{r^2+a^2}\right)^2\right)\\
	&	\geq C_{\textit{large}} b\tilde{\lambda},
	\end{aligned}
	\end{equation}
	for~$r^\star\in(-R^\star+1,R^\star-1)$, where~$C_{\textit{large}}>0$ is a frequency independent constant. In the last inequality of~\eqref{eq: proof lem: h multiplier in the middle region in stationary frequencies, 2, eq 2} we used Lemma~\ref{lem: stationary frequencies, lem 0} and also took~$\omega_{\textit{low}},~a_0>0$ sufficiently small.

	Therefore if we add the current~$Q^y$ and take~$C_{\textit{large}}>0$ sufficiently large we conclude the first inequality of~\eqref{eq: lem: h multiplier in the middle region in stationary frequencies, 2, eq 1}. We also readily obtain the last inequality of~\eqref{eq: lem: h multiplier in the middle region in stationary frequencies, 2, eq 1} by recalling the construction of~$y$ in Lemma~\ref{lem: stationary frequencies for y, 1}.

	We extend~$h$ in~$(-\infty,-R^\star+1)\cup(R^\star-1,\infty)$ such that~$h=0$ in open neighborhoods of~$\pm \infty$, with 
	\begin{equation}
	|h|+|h^\prime|+|h^{\prime\prime}|\leq B,
	\end{equation}
	for all~$r^\star$ in a straightforward manner. 
\end{proof}

We need the following Lemma

\begin{lemma}\label{lem: stationary frequencies for y, 4}
	
	Let the constants~$C_{large},~A_{large}>0$ be both sufficiently large, where moreover~$A_{large}\geq C_{large}^2$.

	Then, the following hold
	\begin{equation}\label{eq: lem: stationary frequencies for y, 4, eq 1}
	y^\prime \omega^2-(y\tilde{V})^\prime+h(\tilde{V}-\omega^2)-\frac{1}{2}h^{\prime\prime}+ \left( \frac{r\Delta h}{(r^2+a^2)^2}\right)^\prime \geq b \Delta \tilde{\lambda},\qquad y^\prime -\frac{4r\Delta}{r^2+a^2}y+h\geq b\Delta,
	\end{equation}
	for~$r^\star\in (-\infty,-R^\star-1)\cup (R^\star-1,\infty)$.  
\end{lemma}

\begin{proof}
	First, we note from Lemma~\ref{lem: h multiplier in the middle region in stationary frequencies, 2}, that the multiplier~$h$ satisfies 
	\begin{equation}
	h(\tilde{V}-\omega^2)-\frac{1}{2}h^{\prime\prime}+ \left( \frac{r\Delta h}{(r^2+a^2)^2}\right)^\prime \geq 0,\qquad r^\star\in (-R^\star+1,-R^\star)\cup(R^\star-1,R^\star),
	\end{equation}
	\begin{equation}
	|h|+|h^\prime|+|h^{\prime\prime}|\leq B,\quad r^\star\in \mathbb{R}.
	\end{equation}
	
	We prove the first inequality of~\eqref{eq: lem: stationary frequencies for y, 4, eq 1} in~$(R^\star,\infty)$. Similar estimates also conclude the first inequality of~\eqref{eq: lem: stationary frequencies for y, 4, eq 1} in~$(-\infty,-R^\star)$.

	We compute
	\begin{equation}
	\begin{aligned}
	&   y^\prime\omega^2-y^\prime \tilde{V} -y             \tilde{V}^\prime+h(\tilde{V}-\omega^2) -\frac{1}{2}h^{\prime\prime} +\Big( \frac{r\Delta}{(r^2+a^2)^2}h\Big)^\prime \\
	& =\frac{\Delta}{r^2+a^2}  \left( \frac{d y}{dr}\omega^2-\frac{dy}{dr}\tilde{V}-y\frac{d\tilde{V}}{dr}-\frac{1}{2}\frac{d}{dr}\big( \frac{\Delta}{r^2+a^2}\frac{d h}{dr} \big)+\frac{\Delta}{r^2+a^2}h+r\frac{d}{dr}\big( \frac{\Delta}{(r^2+a^2)^2}\big)h+\frac{r\Delta}{(r^2+a^2)^2}\frac{d h}{dr} \right)\\
	& \quad  +h(\tilde{V}-\omega^2).
	\end{aligned}
	\end{equation}
	In view of the construction of the multiplier~$y$, see Lemma~\ref{lem: stationary frequencies for y, 1} and of the properties of the potential~$\tilde{V}$, see Lemma~\ref{lem: stationary frequencies, lem 0}, we note that 
	\begin{equation}\label{eq: proof: lem: stationary frequencies for y, 4, eq 1}
		\begin{aligned}
			&	\frac{d y}{dr}\omega^2-y\frac{d\tilde{V}}{dr} - \frac{dy}{dr}\tilde{V}\\
			&	\qquad =\left(\frac{4r}{r^2+a^2}y\right)\omega^2-\left(\left( \frac{y}{(r^2+a^2)^2}\right)(r_1)(r^2+a^2)^2+A_{\textit{large}}(r^2+a^2)^2\int_{r_1}^r \frac{d r}{(r^2+a^2)^2}\right)\frac{d\tilde{V}}{dr}-\frac{dy}{dr}\tilde{V}\\
			&	\qquad \geq b A_{\textit{large}}\tilde{\lambda}, 
		\end{aligned}
	\end{equation}
	for~$r^\star\in(R^\star,+\infty]$. In the last inequality of~\eqref{eq: proof: lem: stationary frequencies for y, 4, eq 1} we took~$a_0,\omega_{\textit{low}}>0$ both sufficiently small. Therefore, we obtain
	\begin{equation}
	y^\prime \omega^2-(yV)^\prime+h(\tilde{V}-\omega^2)-\frac{1}{2}h^{\prime\prime}+ \left( \frac{r\Delta h}{(r^2+a^2)^2}\right)^\prime \geq b(\epsilon) \Delta \tilde{\lambda},
	\end{equation}
	in~$(R^\star,\infty]$, where we take $A_{\textit{large}}(C_{\textit{large}})>0$ larger if needed.

	We conclude that after taking~$a_0,\omega_{\textit{low}}$ both sufficiently small, and~$A_{\textit{large}}(C_{\textit{large}})$ sufficiently larger if needed, we obtain that 
	\begin{equation}
		y^\prime \omega^2-(y\tilde{V})^\prime+h(\tilde{V}-\omega^2)-\frac{1}{2}h^{\prime\prime}+ \left( \frac{r\Delta h}{(r^2+a^2)^2}\right)^\prime \geq b\Delta \tilde{\lambda},
	\end{equation}
	for~$r^\star\in(R^\star,\infty)$, and conclude the first inequality of~\eqref{eq: lem: stationary frequencies for y, 4, eq 1}. We also readily conclude the second inequality of~\eqref{eq: lem: stationary frequencies for y, 4, eq 1}. 
\end{proof}

The currents~$y,h$ have now been constructed.

\begin{proof}[\textbf{Proof of Proposition~\ref{prop: energy estimate for the bounded stationary frequencies}}]

	By combining the Lemmata~\ref{lem: stationary frequencies for y, 1},~\ref{lem: h multiplier in the middle region in stationary frequencies, 2},~\ref{lem: stationary frequencies for y, 4} we integrate the energy identity of Lemma~\ref{lem: subsec: currents, lem 2} associated with the current 
	\begin{equation}
		Q^y_{stat}+Q^h_{stat}
	\end{equation}
	and obtain
	\begin{equation}\label{eq: proof subsec: stationary frequencies, prop: energy estimate for stationary frequencies, eq 1}
	\begin{aligned}
	\int_{\mathbb{R}} & \big( \Delta |\Psi^\prime|^2+\Delta(\omega^2+\lambda^{(a\omega)}_{m\ell}+1)|\Psi|^2  \big)dr^\star+\int_{-R^\star}^{R^\star}C_{\textit{large}}\big( \Delta |\Psi^\prime|^2+\Delta(\omega^2+\lambda^{(a\omega)}_{m\ell})|\Psi|^2  \big)dr^\star\\
	&   \leq Q^y[\Psi](+\infty)-Q^y[\Psi](-\infty) +  B\int_{\mathbb{R}}\left(\Re\left(\frac{2\Psi h\bar{H}}{\sqrt{r^{2}+a^{2}}}\right)+ \Re\left(\frac{y\Psi^{\prime}\bar{H}}{\sqrt{r^{2}+a^{2}}} \right)\right)dr^\star\\
	&   \leq  B y(\infty)\left(\omega-\frac{am\Xi}{\bar{r}_+^2+a^2}\right)^2|\Psi|^2(\infty)-By(-\infty)\left(\omega-\frac{am\Xi}{r_+^2+a^2}\right)^2|\Psi|^2(-\infty) \\
	&   \quad\quad +  B\int_{\mathbb{R}}\left(\left|\Re\left(\frac{2\Psi h\bar{H}}{\sqrt{r^{2}+a^{2}}}\right)\right|+ \left|\Re\left(\frac{y\Psi^{\prime}\bar{H}}{\sqrt{r^{2}+a^{2}}} \right)\right|\right)dr^\star,\\
	\end{aligned}
	\end{equation}
	where we utilized that~$h(r^\star)\equiv 0$ in open neighborhoods of~$-\infty,\infty$ respectively, see Lemma~\ref{lem: h multiplier in the middle region in stationary frequencies, 2}. Moreover, we used that~$\lambda^{(a\omega)}_{m\ell}+(a\omega)^2\geq \frac{1}{2}$ from Lemma~\ref{lem: stationary frequencies, lem 0}.

	Now, we need to absorb the boundary terms on the right hand side of inequality~\eqref{eq: proof subsec: stationary frequencies, prop: energy estimate for stationary frequencies, eq 1}. Let~$\chi(r^\star)$ be a smooth function such that~$\chi=1$ in~$(-\infty,-R^\star)$ and~$\chi=0$ in~$(R^\star,\infty)$. Moreover, let~$\bar{\chi}(r^\star)$ be a smooth function such that~$\bar{\chi}(r^\star)=0$ in~$(-\infty,-R^\star)$ and~$\bar{\chi}(r^\star)=1$ in~$(R^\star,\infty)$. Now, in view of Proposition~\ref{prop: subsec: energy identity, prop 1}, we integrate the energy identity of Lemma~\ref{lem: subsec: sec: frequency localized multiplier estimates, subsec 2, lem 1} associated with the current 
	\begin{equation}
	-E_0\chi(r^\star)Q^{K^+}[u]-E_0 \bar{\chi}(r^\star)Q^{\bar{K}^+}[u]
	\end{equation}
	for a sufficiently large~$E_0$, and the we sum it in inequality~\eqref{eq: proof subsec: stationary frequencies, prop: energy estimate for stationary frequencies, eq 1}. By taking~$a_0>0$ sufficiently small and~$\omega_{\textit{low}}>0$ sufficiently small~(if necessary) we obtain 
	\begin{equation}\label{eq: proof subsec: stationary frequencies, prop: energy estimate for stationary frequencies, eq 2}
		\begin{aligned}
			\int_{\mathbb{R}} & \big( \Delta |\Psi^\prime|^2+\Delta(\omega^2+\lambda^{(a\omega)}_{m\ell})|\Psi|^2  \big)dr^\star+\int_{-R^\star}^{R^\star}C_{\textit{large}}\big( \Delta |\Psi^\prime|^2+\Delta(\omega^2+\lambda^{(a\omega)}_{m\ell})|\Psi|^2  \big)dr^\star\\
			&   \leq  B\cdot E_0 \int_{-R^\star}^{R^\star} \left(|\chi^\prime Q^{K^+}[u]|+|\bar{\chi}^\prime Q^{\bar{K}^+}[u]| \right)dr^\star\\
			&   \quad\quad +  B\int_{\mathbb{R}}\left(\left|\Re\left(\frac{2\Psi h\bar{H}}{\sqrt{r^{2}+a^{2}}}\right)\right|+ \left|\Re\left(\frac{y\Psi^{\prime}\bar{H}}{\sqrt{r^{2}+a^{2}}} \right)\right|\right)dr^\star.\\
		\end{aligned}
 \end{equation}
	Then, to absorb the terms generated at~$(-R^\star,R^\star)$, generated by the cut-offs~$\chi,\bar{\chi}$, namely 
	\begin{equation}
	E_0\int_{-R^\star}^{R^\star} \left(|\chi^\prime Q^{K^+}[u]|+|\bar{\chi}^\prime Q^{\bar{K}^+}[u]| \right)dr^\star\leq B\cdot  E_0 \int_{-R^\star}^{R^\star} \left(\left(1+\frac{1}{\tilde{\epsilon}}\right)|\Psi|^2 +\tilde{\epsilon}|\Psi^\prime|^2\right)dr^\star,
	\end{equation}
	for a sufficiently small~$\tilde{\epsilon}>0$, we take the constant~$C_{\textit{large}}>0$ sufficiently larger if needed, see Lemma~\ref{lem: h multiplier in the middle region in stationary frequencies, 2}, and absorb them.

	Thus the proof of Proposition~\ref{prop: energy estimate for the bounded stationary frequencies} is complete.
\end{proof}

\subsubsection{\textbf{Near stationary case}~$\{|\omega|\leq \omega_{\textit{low}}\}\cap\{|m|=0\}$}\label{subsubsec: sec: bounded: stationary, 1}

In view of the definition of the frequency regimes, see Definition~\ref{def: subsec: sec: frequencies, subsec 3, def 1}, this frequency regime is manifestly non-superradiant. 

\begin{remark}
	Note that Proposition~\ref{prop: energy estimate for the bounded stationary frequencies, 1} is the only main Proposition of Section~\ref{sec: proof: prop: sec: proofs of the main theorems} where in fact the LHS of~\eqref{eq: prop: energy estimate for the bounded stationary frequencies, 1, eq 1} does not control the lower order terms~$|\Psi|^2$ in the bulk of the left hand side in the case~$\mu^2_{KG}=0$.
\end{remark}

The main Proposition of this Section is the following

\begin{proposition}\label{prop: energy estimate for the bounded stationary frequencies, 1}

		Let~$l>0$,~$(a,M)\in \mathcal{B}_l$,~$\mu^2_{KG}\geq 0$. Let~$	\lambda_{\textit{low}},\omega_{\textit{high}}>0$ be arbitrary. For any
	\begin{equation}
	\omega_{\textit{low}}>0
	\end{equation}
	sufficiently small, then for~$(\omega,m,\ell)\in\mathcal{F}_\flat(\lambda_{low},\omega_{high}) \cap\{|\omega|\leq \omega_{\textit{low}}\}\cap\{|m|=0\}$, there exist sufficiently regular functions~$h,y$ satisfying the uniform bound
	\begin{equation}
	|h|+|h^\prime|+|h^{\prime\prime}|+|y|+|y^\prime|\leq B,
	\end{equation}
	where~$y$ is piecewise~$C^1$, such that for all smooth solutions~$u$ of Carter's radial ode~\eqref{eq: ode from carter's separation} satisfying the outgoing boundary conditions~\eqref{eq: lem: sec carters separation, subsec boundary behaviour of u, boundary beh. of u, eq 3}, we have 
	\begin{equation}\label{eq: prop: energy estimate for the bounded stationary frequencies, 1, eq 1}
	\begin{aligned}
	\int_{\mathbb{R}}  & \Delta\left(|\Psi^\prime|^2+(\omega^2+\lambda^{(a\omega)}_{0\ell})|\Psi|^2 +\mu^2_{\textit{KG}}|\Psi|^2\right)dr^\star\\
	&   \leq B\int_{\mathbb{R}}\Bigg( \left|\omega\Im (\bar{\Psi}H)\right|+\left|\omega\Im(\bar{\Psi}H)\right|+ \mu^2_{KG}\left|\Re\left(\frac{2\Psi h\bar{H}}{\sqrt{r^{2}+a^{2}}}\right)\right|+ \left|\Re\left(\frac{y\Psi^{\prime}\bar{H}}{\sqrt{r^{2}+a^{2}}} \right)\right|\Bigg)dr^\star.\\
	\end{aligned}
	\end{equation}
\end{proposition}

First, we note that following lemma that for the case of axisymmetry~$m=0$ the potential~$\tilde{V}$, see Lemma~\ref{lem: sec: frequency localized multiplier estimates, lem 2}, can be written as 
\begin{equation}\label{eq: proof: lem: subsubsec: sec: bounded: stationary, 1, lem 0, eq 1}
	\tilde{V}\:\dot{=}\:\frac{\tilde{\lambda}\Delta}{(r^{2}+a^{2})^{2}} +\Delta\mu_{\textit{KG}}^2\frac{r^{2}+a^{2}}{(r^{2}+a^{2})^{2}}.
\end{equation}

We are ready to prove the main Proposition

\begin{proof}[\textbf{Proof of Proposition~\ref{prop: energy estimate for the bounded stationary frequencies, 1}}]
	
First, we discuss the case~$\mu^2_{KG}>0$. We note that we may apply the same multipliers that we used in the previous Section~\ref{subsubsec: sec: bounded: stationary}, where recall that we treated the near stationary regime~$\{|\omega|\leq \omega_{low}\}\cap\{|m|>0\}$ and~$|a|\leq a_0$, and conclude the desired result~\eqref{eq: prop: energy estimate for the bounded stationary frequencies, 1, eq 1} for any~$\omega_{low} >0$ sufficiently small.

Now, we continue with the more elaborate case~$\mu^2_{KG}=0$. Let~$\epsilon'>0$ be a sufficiently small real number. First, it is immediate from Lemma~\ref{lem: inequality for lambda} that~$\tilde{\lambda}\geq 0$ for~$m=0$. We consider the three cases
 \begin{enumerate}
 	\item $ \tilde{\lambda}\leq \epsilon^\prime \omega^2$
 	\item $\tilde{\lambda}\geq \frac{1}{\epsilon^\prime}\omega^2$
 	\item $\epsilon^\prime\omega^2\leq \tilde{\lambda}\leq \frac{1}{\epsilon^\prime}\omega^2$. 
 \end{enumerate}

 \underline{In the first case} we proceed as follows. We define the multiplier
 \begin{equation}
 	y=(r^2+a^2)^3
 \end{equation}
and we calculate
\begin{equation}
	y^\prime -y \frac{4r\Delta}{(r^2+a^2)^2}=2r(r^2+a^2)\Delta\geq b\Delta.
\end{equation}
We integrate the energy identity of Lemma~\ref{lem: subsec: currents, lem 2}, associated to~$Q^y_{stat}$, and obtain that for any~$\epsilon^\prime>0$ sufficiently small the following holds
\begin{equation}\label{eq: proof prop: energy estimate for the bounded stationary frequencies, 1, eq 0}
	\begin{aligned}
		\int_{\mathbb{R}}\left(\Delta|\Psi^\prime|^2 +\Delta\omega^2|\Psi|^2 \right)dr^\star\leq  B |\int_{\mathbb{R}} (Q^y_{\textit{stat}}[\Psi])^\prime	dr^\star| +\int_{\mathbb{R}}\left|\Re\left(\frac{y\Psi^\prime\overline{H^{(a\omega)}_{m\ell}}}{\sqrt{r^{2}+a^{2}}} \right)\right|dr^\star.
	\end{aligned}
\end{equation} 
In view of the fact that this frequency regime is axisymmetric and thus non-superradiant, we integrate the energy estimate of Lemma~\ref{lem: subsec: sec: frequency localized multiplier estimates, subsec 2, lem 1}, associated with the current ~$-E_0 Q^{K^+}[u]-E_0 Q^{\bar{K}^+}[u]$ for a sufficiently large~$E_0>0$, in view of Proposition~\ref{prop: subsec: energy identity, prop 1}, and then we sum it in~\eqref{eq: proof prop: energy estimate for the bounded stationary frequencies, 1, eq 0}, and conclude the desired~\eqref{eq: prop: energy estimate for the bounded stationary frequencies, 1, eq 1}.

 \underline{In the second case} we can again use the same multipliers that we used in the previous Section~\ref{subsubsec: sec: bounded: stationary}, for any sufficiently small~$\epsilon^\prime>0$,~$\omega_{low}>0$. Specifically, see Lemmata~\ref{lem: stationary frequencies for y, 1},~\ref{lem: h multiplier in the middle region in stationary frequencies, 2} ,~\ref{lem: stationary frequencies for y, 4} for the construction of the multipliers. To ease the comparison, recall that in Section~\ref{subsubsec: sec: bounded: stationary} we used that~$\lambda\gg\omega^2$ to construct the multipliers~$h,y$. The same arguments can be carried over to the present proof with the assumption of axisymmetry~$m=0$, instead of slow rotation~$a_0\ll 1$. 
 
At this point we fix the~$\epsilon^\prime>0$.

\underline{In the third case} we proceed as follows. Again we define 
\begin{equation}
	y = (r^2+a^2)^3.
\end{equation}
We integrate the energy identity of Lemma~\ref{lem: subsec: currents, lem 2}, associated to~$Q^y_{stat}$ to write 
\begin{equation}\label{eq: proof prop: energy estimate for the bounded stationary frequencies, 1, eq 1}
	\begin{aligned}
			\int_{\mathbb{R}}\left(\Big( y^\prime -y\frac{4r\Delta}{(r^{2}+a^{2})^{2}}  \Big)|\Psi^\prime|^2 +\Big(y^\prime \omega^2 -(y\tilde{V})^\prime\Big)|\Psi|^2\right)dr^\star = -\int_{\mathbb{R}} (Q^y_{\textit{stat}}[\Psi])^\prime	dr^\star -\int_{\mathbb{R}}\Re\left(\frac{y\Psi^\prime\overline{H^{(a\omega)}_{m\ell}}}{\sqrt{r^{2}+a^{2}}} \right)dr^\star.
	\end{aligned}
\end{equation} 
By integration by parts and by using that~$\tilde{V}(r_+)=\tilde{V}(\bar{r}_+)=0$ we obtain 
\begin{equation}
	\begin{aligned}
		b \int_{\mathbb{R}}\left( \Delta  |\Psi^\prime|^2 + y^\prime \omega^2 |\Psi|^2 \right)dr^\star	&	\leq  -Q^y_{stat}[\Psi](+\infty)+Q^y_{stat}[\Psi](-\infty) + \epsilon\int_{\mathbb{R}}\Delta |\Psi^\prime|^2dr^\star+ \frac{1}{\epsilon} \int_{\mathbb{R}}\Delta \tilde{\lambda}^2|\Psi|^2dr^\star\\
		&	\qquad-\int_{\mathbb{R}}\Re\left(\frac{y\Psi^\prime\overline{H^{(a\omega)}_{m\ell}}}{\sqrt{r^{2}+a^{2}}} \right)dr^\star,
	\end{aligned}
\end{equation}
where we also used a Young's ineqality. By taking~$\epsilon>0,\omega_{low}>0$ both sufficiently small, since~$\tilde{\lambda}^2\sim \omega^4$, we obtain the result 
\begin{equation}
b \int_{\mathbb{R}} \left(\Delta |\Psi^\prime|^2 + y^\prime \omega^2 |\Psi|^2 \right) dr^\star \leq  -Q^y_{stat}[\Psi](+\infty)+Q^y_{stat}[\Psi](-\infty)  -\int_{\mathbb{R}}\Re\left(\frac{y\Psi^\prime\overline{H^{(a\omega)}_{m\ell}}}{\sqrt{r^{2}+a^{2}}} \right)dr^\star.
\end{equation}

Finally, in view of the fact that this frequency regime is axisymmetric and thus non-superradiant, we integrate the energy identities of Lemmata~\ref{lem: subsec: currents, lem 2},~\ref{lem: subsec: sec: frequency localized multiplier estimates, subsec 2, lem 1} associated with the following current
	\begin{equation}
	Q_{\textit{stat}}^y[\Psi]-E_0 Q^{K^+}[u]-E_0 Q^{\bar{K}^+}[u],
	\end{equation}
	 in view of Proposition~\ref{prop: subsec: energy identity, prop 1},  for a sufficiently large~$E_0>0$, and conclude
	\begin{equation}
		\begin{aligned}
			\int_{\mathbb{R}} & \Delta\left(|\Psi^\prime|^2+(\omega^2+\lambda^{(a\omega)}_{0\ell})|\Psi|^2 \right)dr^\star\leq  B \int_{\mathbb{R}}\Bigg( \left|\omega\Im(\bar{\Psi}H)\right|+ \left|\Re\left(\frac{y\Psi^{\prime}\bar{H}}{\sqrt{r^{2}+a^{2}}} \right)\right|\Bigg)dr^\star.\\
		\end{aligned}
	\end{equation}		
We conclude the result of the Proposition. 	
\end{proof}

\subsubsection{\textbf{Near stationary case} \texorpdfstring{$\{|\omega|\leq \omega_{\textit{low}}\}\cap\{|m|\geq 1\}$ and \texorpdfstring{$\{|a|\geq a_0\}$}{g}}{g}}\label{subsubsec: bounded stationary freqs, large a}

In view of the definition of the frequency regimes, see Definition~\ref{def: subsec: sec: frequencies, subsec 3, def 1}, this frequency regime is manifestly non-superradiant for a sufficiently small~$\omega_{\textit{low}}>0$.

The main Proposition of this Section is the following

\begin{proposition}\label{prop: subsubsec: bounded stationary freqs, large a}

		Let~$l>0$,~$(a,M)\in \mathcal{B}_l$,~$\mu^2_{KG}\geq 0$. Let~$\omega_{high},\lambda_{low}>0$. Let~$|a|\geq a_0$. Then, for
	\begin{equation}
		-\infty< r^\star_{-\infty}<0<r^\star_{+\infty}<+\infty
	\end{equation}
	sufficiently small~(and negative) and sufficiently large~(and positive) respectively, and for any
	\begin{equation}
	\omega_{\textit{low}}>0,
	\end{equation}
	sufficiently small, then for~$(\omega,m,\ell)\in\mathcal{F}_\flat(\omega_{high},\lambda_{low}) \cap\{|\omega|\leq \omega_{\textit{low}}\}$ there exists a multiplier~$y$ satisfying the uniform bound 
	\begin{equation}
	|y|+|y^\prime|\leq B(\lambda_{\textit{low}},\omega_{\textit{high}})
	\end{equation}
	such that for all smooth solutions~$u$ of Carter's radial ode~\eqref{eq: ode from carter's separation} satisfying the outgoing boundary conditions~\eqref{eq: lem: sec carters separation, subsec boundary behaviour of u, boundary beh. of u, eq 3}, we have 
	\begin{equation}
	\begin{aligned}
	\int_{r^\star_{-\infty}}^{r^\star_{+\infty}}\big(|u^\prime|^2 + (1+\lambda^{(a\omega)}_{m\ell}+\omega^2)|u|^2\big)dr^\star \leq B \int_{\mathbb{R}}\left(\left|2y\Re (u^\prime H)\right|+\left|\Im(\bar{u}H)\right|\right) dr^\star.    
	\end{aligned}
	\end{equation} 
\end{proposition}

First, note the following Lemma 
\begin{lemma}
	Let the assumptions of Proposition~\ref{prop: subsubsec: bounded stationary freqs, large a} hold. Then the following hold 
	\begin{equation}
	\left(\omega-\frac{am\Xi}{r^2+a^2}\right)^2\geq b |m|^2,\qquad 
	\omega^2-V(r_+)\geq b|m|^2,\qquad \omega^2-V(\bar{r}_+)\geq b|m|^2. 
	\end{equation}
\end{lemma}
\begin{proof}
	This is straightforward. 
\end{proof}

Now we are ready to prove Proposition~\ref{prop: subsubsec: bounded stationary freqs, large a}.

\begin{proof}[\textbf{Proof of Proposition~\ref{prop: subsubsec: bounded stationary freqs, large a}}]

	For simplicity, we define the potential $\tilde{V}$ such that 
	\begin{equation}
	\omega^2-V=\left(\omega-\frac{am\Xi}{r_+^2+a^2}\right)^2-\tilde{V}
	\end{equation}
	
	We integrate the energy identity of Lemma~\ref{lem: subsec: currents, lem 1}, associated to~$Q^y$, to get 
	\begin{equation}\label{eq: proof: prop: subsubsec: bounded stationary freqs, large a, eq 1}
	\begin{aligned}
	\int_{\mathbb{R}}\left( y^\prime |u^\prime|^2+y^\prime\left(\omega-\frac{am\Xi}{r_+^2+a^2}\right)^2 |u|^2\right)dr^\star &	=	Q^y[u](\infty)-Q^y[u](-\infty)	\\
	&	\qquad  +\int_{\mathbb{R}}\left( \left(y\tilde{V}\right)^\prime |u|^2 +2y\Re (u^\prime \bar{H})\right)dr^\star .
	\end{aligned}
	\end{equation}
	Consider the multiplier
	\begin{equation}
	y=-\exp (-C\int^{r^\star}_{-\infty} \chi_{r^\star_{\infty}}),
	\end{equation}
	such that~$\chi_{r^\star_{\infty}}$ satisfies the following
	\begin{equation}
	\chi_{r^\star_\infty}(r) = 
	\begin{cases}
	\text{$1$} &\quad\text{if}\:r\in (r_+,r_\infty -1)\\
	\text{$0$} &\quad\text{if}\: r\in(r^\star_\infty,\bar{r}_{+}), \\
	\end{cases}
	\end{equation}
	for $r^\star_\infty$ sufficiently large. 
	
	We note the following straightforward computation
	\begin{equation}\label{eq: proof: prop: subsubsec: bounded stationary freqs, large a, eq 2}
	\begin{aligned}
	\int_{\mathbb{R}}\left(y\tilde{V}\right)^\prime |u|^2dr^\star \leq \epsilon \int_{\mathbb{R}}y^\prime |u^\prime|^2dr^\star  +B\epsilon^{-1}C^{-2}\left(\omega-\frac{am\Xi}{r_+^2+a^2}\right)^{-2}\int_{\mathbb{R}}y^\prime \left(\omega-\frac{am\Xi}{r_+^2+a^2}\right)^2 |u|^2dr^\star .
	\end{aligned}
	\end{equation}
	
	Therefore, from~\eqref{eq: proof: prop: subsubsec: bounded stationary freqs, large a, eq 1},~\eqref{eq: proof: prop: subsubsec: bounded stationary freqs, large a, eq 2} we obtain 
	\begin{equation}\label{eq: proof: prop: subsubsec: bounded stationary freqs, large a, eq 3}
	\begin{aligned}
	b\int_{\mathbb{R}}\left(y^\prime |u^\prime|^2+y^\prime \left(\omega-\frac{am\Xi}{r_+^2+a^2}\right)^2|u|^2\right)dr^\star 	&	\leq Q^y[u](r^\star=-\infty)-\int_{\mathbb{R}}2 y \Re (u^\prime \bar{H})dr^\star \\
	&	\leq y(-\infty) \left(|u^\prime|^2+\left(\omega^2-V\right)\right)(r^\star=-\infty)\\
	&	\qquad -\int_{\mathbb{R}}2 y \Re (u^\prime \bar{H})dr^\star .
	\end{aligned}
	\end{equation}

	As noted earlier, for~$|a|\geq a_0$, the present frequency regime 
	\begin{equation}
	\{|\omega|\leq \omega_{\textit{low}}\}\cap\{|m|\geq 1\}
	\end{equation}
	is non-superradiant, namely for~$a_0>0$ as in Proposition~\ref{prop: energy estimate for the bounded stationary frequencies} and for a sufficiently small
	\begin{equation}
	\omega_{\textit{low}}(a,M,l,\mu_{KG})
	\end{equation}
	we obtain 
	\begin{equation}
		\left(\omega-\frac{am\Xi}{\bar{r}_+^2+a^2}\right)\cdot \left(\omega-\frac{am\Xi}{r_+^2+a^2}\right)>0.
	\end{equation}

	Therefore,  in view of Proposition~\ref{prop: subsec: energy identity, prop 1}, we integrate the energy identity of Lemma~\ref{lem: subsec: sec: frequency localized multiplier estimates, subsec 2, lem 1} associated with the current 
	\begin{equation}
		-E_0Q^{K^+}-E_0 Q^{\bar{K}^+}
	\end{equation}
	for a sufficiently large~$E_0$, we sum it in~\eqref{eq: proof: prop: subsubsec: bounded stationary freqs, large a, eq 3} and we obtain the result. 
\end{proof}

\subsubsection{\textbf{Non stationary case} $\{ |\omega|\geq \omega_{\textit{low}} \}$}\label{subsec: bounded non stationary frequencies}

The main Proposition of this Section is the following

\begin{proposition}\label{prop: energy estimate in the bounded non stationary frequency regime}

		Let~$l>0$~$(a,M)\in \mathcal{B}_l$,~$\mu^2_{KG}\geq 0$. Let
	\begin{equation}
	\omega_{\textit{low}}>0,\qquad \lambda_{\textit{low}}>0,\qquad \omega_{\textit{high}}>0
	\end{equation}
	be any positive real numbers and let~$\mathcal{C}>1$ be sufficiently large. Then, for~$(\omega,m,\ell)\in\mathcal{F}_\flat(\omega_{low},\omega_{high},\lambda_{low}) \cap\{|\omega|\geq \omega_{\textit{low}}\}$ and for any 
	\begin{equation}
	-\infty<r^\star_{-\infty}<0<r^\star_{+\infty}<+\infty
	\end{equation}
	sufficiently small~(and negative) and sufficiently large~(and positive) respectively, there exists a multiplier~$y$ satisfying the uniform bound
	\begin{equation}
	|y|+|y^\prime|\leq B(\lambda_{\textit{low}},\omega_{\textit{high}},\mathcal{C})
	\end{equation}
	such that for all smooth solutions~$u$ of Carter's radial ode~\eqref{eq: ode from carter's separation} satisfying the outgoing boundary conditions~\eqref{eq: lem: sec carters separation, subsec boundary behaviour of u, boundary beh. of u, eq 3}, we have the following:

	If in addition~$(\omega,m,\ell)\in\mathcal{F}_\flat\cap\mathcal{SF}\cap \left(\{|\omega-\omega_+ m|\geq \mathcal{C}^{-1}\}\cup \{|\omega-\bar{\omega}_+ m|\geq \mathcal{C}^{-1}\}\right)$ then we have
	\begin{equation}\label{eq: prop: energy estimate in the bounded non stationary frequency regime, eq 1}
	\begin{aligned}
	\int_{r^\star_{-\infty}}^{r^\star_{+\infty}}\big(|u^\prime|^2 + (1+\lambda^{(a\omega)}_{m\ell}+\omega^2)|u|^2\big)dr^\star &	\leq B(r^\star_{\pm\infty},\mathcal{C})\Big( \int_{\mathbb{R}} \left(\left|2y\Re (u^\prime H)\right|+\left|\Im(\bar{u}H)\right|\right)dr^\star\\
	&	\qquad\qquad\qquad+\left|\left(\omega-\frac{am\Xi}{r_+^2+a^2}\right)\left(\omega-\frac{am\Xi}{\bar{r}_+^2+a^2}\right)\right||u|^2(-\infty)\Big),
	\end{aligned}
	\end{equation} 
	and if in addition~$(\omega,m,\ell)\in\mathcal{F}_\flat\cap \left(\{|\omega-\omega_+ m|\leq \mathcal{C}^{-1}\}\cup \{|\omega-\bar{\omega}_+ m|\leq \mathcal{C}^{-1}\}\cup \mathcal{SF}^c\right)$ then we have
	\begin{equation}\label{eq: prop: energy estimate in the bounded non stationary frequency regime, eq 1.1}
	\begin{aligned}
		\int_{r^\star_{-\infty}}^{r^\star_{+\infty}}\big(|u^\prime|^2 + (1+\lambda^{(a\omega)}_{m\ell}+\omega^2)|u|^2\big)dr^\star	&	\leq B(r^\star_{\pm \infty},\mathcal{C}) \int_{\mathbb{R}} \Big(\left|2y\Re (u^\prime H)\right|+\left|\Im(\bar{u}H)\right|\Big) dr^\star.
	\end{aligned}
	\end{equation} 
\end{proposition}

\begin{proof}
	
    We integrate the energy identity of Lemma~\ref{lem: subsec: currents, lem 1}, associated to~$Q^y$, to get 
	\begin{equation}\label{eq: proof: prop: energy estimate in the bounded non stationary frequency regime, eq 1}
	\int_{\mathbb{R}}\left( y^\prime |u^\prime|^2+y^\prime\omega^2 |u|^2\right)dr^\star =Q^y[u](\infty)-Q^y[u](-\infty) +\int_{\mathbb{R}}\left( (yV)^\prime |u|^2 +2y\Re (u^\prime \bar{H})\right)dr^\star .
	\end{equation}
	Let~$\epsilon>0$ be sufficiently small. Consider the multiplier
	\begin{equation}\label{eq: prop: energy estimate in the bounded non stationary frequency regime, eq 2}
	y=-\exp (-C\int^{r^\star}_{-\infty} \chi_{r^\star_{\infty}}),
	\end{equation}
	such that~$\chi_{r^\star_{\infty}}$ satisfies the following
	\begin{equation}
	\chi_{r^\star_\infty}(r) = 
	\begin{cases}
	\text{$1$} &\quad\text{if}\:r\in (r_+,r_\infty-\epsilon))\\
	\text{$0$} &\quad\text{if}\: r\in(r_\infty,\bar{r}_{+}), \\
	\end{cases}
	\end{equation}
	for $r^\star_\infty>0$ sufficiently large. By using the definition of the multiplier $y$, in equation \eqref{eq: prop: energy estimate in the bounded non stationary frequency regime, eq 2}, we note the following 
	\begin{equation}\label{eq: prop: energy estimate in the bounded non stationary frequency regime, eq 3}
	\begin{aligned}
	\int_{\mathbb{R}} (yV)^{\prime}|u|^2dr^\star &=\big[yV|u|^2 \big]_{-\infty}^{\infty}  -\int_{\mathbb{R}} yV(u\bar{u}^{\prime}+u^{\prime}\bar{u}) dr^\star\\
	&   \leq \left(yV|u|^{2}\right)(r^\star=\infty)-\left(yV|u|^{2}\right)(r^\star=-\infty)-\int_{\mathbb{R}}|yV|(u\bar{u}^{\prime}+u^{\prime}\bar{u})dr^\star\\
	& \leq B \left(\omega^2-\left(\omega-\frac{am\Xi}{\bar{r}_+^2+a^2}\right)^2\right)\left(y|u|^2\right)(\infty) -B\left(\omega^2-\left(\omega-\frac{am\Xi}{r_+^2+a^2}\right)^2\right)\left(y|u|^2\right)(-\infty) \\
	&   \quad+ \epsilon\int_{\mathbb{R}} y^{\prime}|u^{\prime}|^{2}dr^\star+\frac{1}{\epsilon}\int_{\mathbb{R}} y^{2}\frac{V^{2}}{y^{\prime}}|u|^{2}dr^\star. \\
	\end{aligned}
	\end{equation}

	Observe now that $y^\prime>0$. So, the first bulk term on the on the right hand side of the inequality~\eqref{eq: prop: energy estimate in the bounded non stationary frequency regime, eq 3} will be absorbed, in the left hand side of equation \eqref{eq: proof: prop: energy estimate in the bounded non stationary frequency regime, eq 1}, by taking~$\epsilon>0$ sufficiently small. The last bulk term on the right hand side of~\eqref{eq: prop: energy estimate in the bounded non stationary frequency regime, eq 3}, will be absorbed, in the left hand side of equation~\eqref{eq: proof: prop: energy estimate in the bounded non stationary frequency regime, eq 1}, by taking $C>0$ sufficiently large. We thus obtain the energy estimate 
	\begin{equation}\label{eq: prop: energy estimate in the bounded non stationary frequency regime, eq 4}
	\begin{aligned}
	&	\int_{\mathbb{R}} y^\prime\big(|u^\prime|^2 + (1+\lambda^{(a\omega)}_{m\ell}+\omega^2)|u|^2\big)dr^\star \\
	&\qquad \leq  B Q^y(\infty)- BQ^{y}(-\infty)\\
	&   \qquad\qquad+B \left(\omega^2-\left(\omega-\frac{am\Xi}{\bar{r}_+^2+a^2}\right)^2\right)\left(y|u|^2\right)(\infty) -B\left(\omega^2-\left(\omega-\frac{am\Xi}{r_+^2+a^2}\right)^2\right)\left(y|u|^2\right)(-\infty)\\
	&   \qquad\qquad +\int_{\mathbb{R}} 2y\Re (u^\prime H)dr^\star.
	\end{aligned}
	\end{equation}
	
	\begin{remark}
			Note that the intersection~$\mathcal{F}_\flat\cap \{am\omega\in\left(\frac{a^2m^2\Xi}{\bar{r}_+^2+a^2},\frac{a^2m^2\Xi}{r_+^2+a^2} \right)\}$ is in general not empty, see Remark~\ref{rem: subsec: sec: frequencies, subsec 3, rem 1}, and therefore the boundary terms on the right hand side of~\eqref{eq: prop: energy estimate in the bounded non stationary frequency regime, eq 4}, in general, have different signs.
	\end{remark}

	We proceed as follows. 	In the superradiant frequencies
	\begin{equation}
		(\omega,m,\ell)\in \mathcal{F}_\flat \cap\mathcal{SF}
	\end{equation}
	we integrate the energy identity of Lemma~\ref{lem: subsec: sec: frequency localized multiplier estimates, subsec 2, lem 1} associated with the current~$-E Q^{K^+}-EQ^{\bar{K}^+}$, where $E$ is sufficiently large. We sum it to we use estimate~\eqref{eq: prop: energy estimate in the bounded non stationary frequency regime, eq 4}, and take~$E$ sufficiently large so that we do absorb the boundary term $Q^{y}(\infty)$ and we conclude the following
	\begin{equation}\label{eq: prop: energy estimate in the bounded non stationary frequency regime, eq 5}
	\begin{aligned}
	\int_{\mathbb{R}} y^\prime\big(|u^\prime|^2 + (1+\lambda^{(a\omega)}_{m\ell}+\omega^2)|u|^2\big)dr^\star& \leq B \left|\left(\omega-\frac{am\Xi}{r_+^2+a^2}\right)\left(\omega-\frac{am\Xi}{\bar{r}_+^2+a^2}\right)\right||u|^2(-\infty)\\
	&   \quad +B\int_{\mathbb{R}}\left( \left|2y\Re (u^\prime H)\right|+\left|\Im(\bar{u}H)\right|\right)dr^\star.
	\end{aligned}
	\end{equation}
	In the non superradiant frequencies
	\begin{equation}
		(\omega,m,\ell)\in \mathcal{F}_\flat \cap\mathcal{SF}^c
	\end{equation}
	we can remove the boundary term of~\eqref{eq: prop: energy estimate in the bounded non stationary frequency regime, eq 5} by summing in~\eqref{eq: prop: energy estimate in the bounded non stationary frequency regime, eq 4} an energy identity associates with a current of the form~$E_0Q^{K^+}+E_0Q^{\bar{K}^+}$ and obtain
		\begin{equation}\label{eq: prop: energy estimate in the bounded non stationary frequency regime, eq 6}
		\begin{aligned}
			\int_{\mathbb{R}} y^\prime\big(|u^\prime|^2 + (1+\lambda^{(a\omega)}_{m\ell}+\omega^2)|u|^2\big)dr^\star& \leq B\int_{\mathbb{R}}\left( \left|2y\Re (u^\prime H)\right|+\left|\Im(\bar{u}H)\right|\right)dr^\star.
		\end{aligned}
	\end{equation}

	Now, assume that
	\begin{equation}
		(\omega,m,\ell)\in \mathcal{F}_\flat \cap\left(\{|\omega-\omega_+ m|\leq \mathcal{C}^{-1}\}\cup \{|\omega-\bar{\omega}_+ m|\leq \mathcal{C}^{-1}\}\cup \mathcal{SF}\right)
	\end{equation}
	for a sufficiently large~$\mathcal{C}>0$. In what follows we only discuss the case~$|\omega-\omega_+ m|\leq \mathcal{C}^{-1}$, since the case~$|\omega-\bar{\omega}_+ m|\leq \mathcal{C}^{-1}$ admits an almost identical treatment. Let~$\chi$ be a smooth function with
	 \begin{equation}
		\chi=
		\begin{cases}
			1,\qquad r\in [r_+,2r_{-\infty}]\\
			0,\qquad r\in [3r_{-\infty},\bar{r}_+]. 
		\end{cases}
	\end{equation}
	Let~$Q^{\bar{K}^+}$ be the currents of Definition~\ref{def: sec: currents: def 1, QT, QK currents}. By using the fundamental theorem of calculus for~$\chi\cdot1_{|\omega-\omega_+m|\leq \mathcal{C}^{-1}}Q^{\bar{K}^+}$  we obtain
\begin{equation}\label{eq: prop: energy estimate in the bounded non stationary frequency regime, eq 7}
	|\omega-\omega_+ m||\omega-\bar{\omega}_+ m||u|^2(-\infty) \leq \int_{2r_{-\infty}^\star}^{3r_{-\infty}^\star} |\omega-
	\omega_+ m| |\chi^\prime \Im (u^\prime \bar{u})|dr^\star + \int_{-\infty}^{3r^\star_{\infty}} |\chi||\omega-\omega_+ m| |\Im (H \bar{u})|dr^\star . 
\end{equation}
	
Now, by taking~$\mathcal{C}$ sufficiently large, we use~\eqref{eq: prop: energy estimate in the bounded non stationary frequency regime, eq 7} in~\eqref{eq: prop: energy estimate in the bounded non stationary frequency regime, eq 5} and obtain the desired result
	\begin{equation}\label{eq: prop: energy estimate in the bounded non stationary frequency regime, eq 8}
	\begin{aligned}
		\int_{\mathbb{R}} y^\prime\big(|u^\prime|^2 + (1+\lambda^{(a\omega)}_{m\ell}+\omega^2)|u|^2\big)dr^\star& \leq B\int_{\mathbb{R}}\left( \left|2y\Re (u^\prime H)\right|+\left|\Im(\bar{u}H)\right|\right)dr^\star.
	\end{aligned}
\end{equation}
	
Therefore, from~\eqref{eq: prop: energy estimate in the bounded non stationary frequency regime, eq 8},~\eqref{eq: prop: energy estimate in the bounded non stationary frequency regime, eq 6},~\eqref{eq: prop: energy estimate in the bounded non stationary frequency regime, eq 5} we conclude~\eqref{eq: prop: energy estimate in the bounded non stationary frequency regime, eq 1},~\eqref{eq: prop: energy estimate in the bounded non stationary frequency regime, eq 1.1} and the proof.

	We conclude the proof.
\end{proof}

\subsection{Multipliers for the $\lambda$~dominated frequency regime $F_{\lessflat}$}\label{subsec: subsection on angular dominated frequencies}

In view of the definition of the frequency regimes, see Definition~\ref{def: subsec: sec: frequencies, subsec 3, def 1}, this frequency regime is manifestly non-superradiant.

The main Proposition of this Section is the following 

\begin{proposition}\label{prop: energy estimate in the angular dominated frequency regime}

	Let~$l>0$,~$(a,M)\in \mathcal{B}_l$,~$\mu^2_{KG}\geq 0$. For 
	\begin{equation}
	\alpha^{-1}>0,\qquad \omega_{\textit{high}}>0,\qquad	\lambda_{\textit{low}}^{-1}>0
	\end{equation}
	all sufficiently large, where moreover~$\lambda_{low}\ll \alpha^2$, then for~$(\omega,m,\ell)\in\mathcal{F}_{\lessflat}(\lambda_{low},\omega_{high},\alpha)$ there exist sufficiently regular functions~$f,h$ satisfying the uniform bounds
	\begin{equation}
	|f|+|f^\prime|+|f^{\prime\prime}|+|h|+|h^\prime|+|h^{\prime\prime}|\leq B(\lambda_{\textit{low}},\omega_{\textit{high}})
	\end{equation}
	such that for all smooth solutions~$u$ of Carter's radial ode~\eqref{eq: ode from carter's separation} satisfying the outgoing boundary conditions~\eqref{eq: lem: sec carters separation, subsec boundary behaviour of u, boundary beh. of u, eq 3}, we have 
	\begin{equation}
	\begin{aligned}
	\int_{\mathbb{R}}  \Delta\Bigg( |u^{\prime}|^{2}+\Delta(1+\lambda^{(a\omega)}_{m\ell}+\omega^{2})|u|^{2}  \Bigg)dr^\star &\leq B \int_{\mathbb{R}} \Big( \left|\left(\omega-\frac{am\Xi}{\bar{r}_+^2+a^2}\right)\Im (u\bar{H})\right|+\left|\left(\omega-\frac{am\Xi}{r_+^2+a^2}\right)\Im (u\bar{H})\right|\\
	&\quad\quad\quad+ \left|2fRe(u^{\prime}\bar{H})\right|+(|f^{\prime}|+|h|)\left|Re(\bar{H}u)\right|\Big)dr^\star.
	\end{aligned}
	\end{equation}
\end{proposition}

Note the following lemma for the potential $V$.
\begin{lemma}\label{lem: angular dominated frequencies, lem 1}
	Let~$l>0$ and~$(a,M)\in\mathcal{B}_l$  and~$\mu^2_{KG}\geq 0$. 
	\begin{equation}
		\omega_{\textit{high}}>0,\qquad \lambda^{-1}_{\textit{low}}>0,\qquad \alpha^{-1}>0
	\end{equation}
	be sufficiently large, where~$\lambda_{low}\ll\alpha^2$. Then, for~$(\omega,m,\ell)\in\mathcal{F}_{\lessflat}(\lambda_{low},\omega_{high},\alpha)$ the potential~$V$ attains a unique critical point, a maximum, at a value 
	\begin{equation}
	r_{V,\textit{max}}\in (r_+,\bar{r}_+).
	\end{equation}
	and moreover there exist~$r^\star_{-\infty}<0<r^\star_{+\infty}$ sufficiently negative and sufficiently positive respectively, such that 
	\begin{equation}\label{eq: lem: angular dominated frequencies, lem 1, eq 1}
		\begin{aligned}
			&	b\Delta \tilde{\lambda}\leq V^\prime \leq B\Delta \tilde{\lambda},\quad r^\star<r^\star_{-\infty},\\
			&b\Delta \tilde{\lambda} \leq - V^\prime \leq  B\Delta \tilde{\lambda},\quad r^\star> r^\star_\infty.
		\end{aligned} 
	\end{equation}
	Finally, the following holds
	\begin{equation}\label{eq: lem: angular dominated frequencies, lem 1, eq 1.0}
		|\max_{r\in [r_+,\bar{r}_+]}V-\omega^2  |\geq b \tilde{\lambda}.
	\end{equation}
\end{lemma}
\begin{proof}

	First, by recalling the definition of the potential $V$, see \eqref{eq: the potential V}, we note that 
	\begin{equation}
	\frac{dV}{dr}=\frac{dV_0}{dr}+\frac{d V_{\textit{SL}}}{dr}+\frac{dV_\mu}{dr},
	\end{equation}
	where 
	\begin{equation}\label{eq: lem: angular dominated frequencies, lem 1, eq -4}
	\begin{aligned}
	\frac{dV_0}{dr} &   =\frac{d}{dr}\left(\frac{\Delta}{(r^2+a^2)^2}\right)\left(\tilde{\lambda}-2m\omega a \Xi\right)-2\left(\omega-\frac{am\Xi}{r^2+a^2}\right)\frac{2r}{(r^2+a^2)^2}am\Xi\\
	&   =\frac{d}{dr}\left(\frac{\Delta}{(r^2+a^2)^2}\right)\left(\tilde{\lambda}-2m\omega a \Xi\right)-\frac{4r}{(r^2+a^2)^2}(am\omega\Xi)+\frac{4r}{(r^2+a^2)^3}(am\Xi)^2.
	\end{aligned}
	\end{equation}

	Note that for a sufficiently small~$\alpha(a,M,l)>0$ and for~$(\omega,m,\ell)\in\mathcal{F}_{\lessflat}$ we may assume
	\begin{equation}\label{eq: lem: angular dominated frequencies, lem 1, eq -3}
	a m\omega \leq 0.
	\end{equation}
	Suppose the opposite, namely that $a m\omega>0$. Then, we have
	\begin{equation}\label{eq: lem: angular dominated frequencies, lem 1, eq -2}
	a m\omega \geq \frac{a^2m^2\Xi^2}{r_+^2+a^2}+|a|\alpha\tilde{\lambda}\geq |a| \alpha \lambda_{\textit{low}}^{-1}\omega^2\implies |m|\geq \alpha \lambda_{\textit{low}}^{-1}|\omega|.
	\end{equation}
	On the other hand the following holds
	\begin{equation}\label{eq: lem: angular dominated frequencies, lem 1, eq -1}
	am\omega \geq \frac{a^2m^2\Xi}{r_+^2+a^2}+|a|\alpha\tilde{\lambda} \implies |\omega|\geq \alpha \frac{\tilde{\lambda}}{|m|}\geq \alpha \Xi |m|,
	\end{equation}
	where we also used~\eqref{eq: lem: inequality for lambda, eq 1} from Lemma~\ref{lem: inequality for lambda}. Combining~\eqref{eq: lem: angular dominated frequencies, lem 1, eq -2},~\eqref{eq: lem: angular dominated frequencies, lem 1, eq -1} we obtain 
	\begin{equation}
	|\omega|\geq \alpha^2\lambda_{\textit{low}}^{-1}\Xi|\omega|.\implies \lambda_{low}\geq \alpha^2 \Xi.
	\end{equation}
	Therefore, for a sufficiently small $\alpha>0$ if we take~$\lambda_{low}(\alpha)\sim \alpha^3>0$ sufficiently small we obtain a contradiction. Therefore, we conclude that~$a m\omega\leq 0$.

	Now, since $am\omega \leq 0$, then from~\eqref{eq: lem: angular dominated frequencies, lem 1, eq -4} we obtain 
	\begin{equation}\label{eq: lem: angular dominated frequencies, lem 1, eq 0.4}
	\frac{dV_0}{dr}(r_+)>b\tilde{\lambda}\implies (r^2+a^2)^3\frac{dV_0}{dr}(r=r_+)>b\tilde{\lambda}. 
	\end{equation}
	which immediately implies that~$\frac{dV}{dr}(r_+)\geq b\tilde{\lambda}$, by taking~$\lambda_{low}^{-1}$ sufficiently large.

	Therefore, by taking~$\lambda_{low}^{-1}$ sufficiently large, and in view of the fact that~$\tilde{\lambda}>\lambda_{low}^{-1}(\omega^2+a^2m^2)$ we obtain that indeed~$V$ attains a unique maximum~$r_{V,max}$ and moreover~\eqref{eq: lem: angular dominated frequencies, lem 1, eq 1},~\eqref{eq: lem: angular dominated frequencies, lem 1, eq 1.0} hold.	
\end{proof}

\begin{proof}[\textbf{Proof of Proposition \ref{prop: energy estimate in the angular dominated frequency regime}}]
	
We integrate the energy identity of Lemma~\ref{lem: subsec: currents, lem 1}, associated to~$Q^f+Q^h$, to get 
	\begin{equation}\label{eq: proof prop: energy estimate angular dominated, eq 1}
	\begin{aligned}
	&   \int_{\mathbb{R}} \left(2f^\prime +h\right)|u^\prime|^2dr^\star+\int_{\mathbb{R}}\left(h(V-\omega^2)-\frac{1}{2}h^{\prime\prime}-fV^\prime-\frac{1}{2}f^{\prime\prime\prime}\right)|u|^2dr^\star\\
	&   \quad = Q^f(\infty)-Q^f(-\infty)+Q^h(\infty)-Q^h(-\infty)\\
	&   \quad\quad -\int_{\mathbb{R}} \left(2f \Re (u^\prime\bar{H})+f^\prime\Re(u\bar{H})+h\Re(u\bar{H})\right)dr^\star.
	\end{aligned}
	\end{equation}
	Specifically, for a sufficiently small~$p>0$ we consider the following regions
	\begin{equation}
	-\infty<e^{p^{-1}}r^\star_{-\infty}<r^\star_{-\infty}<0<r^\star_{+\infty}<e^{p^{-1}}r^\star_{+\infty}<\infty,
	\end{equation}
	where $r^\star_{-\infty}<0<r^\star_{+\infty}$ are from Lemma \ref{lem: angular dominated frequencies, lem 1}.

	First, we construct the multiplier 
	\begin{equation}\label{eq: proof prop: energy estimate angular dominated, eq 2}
	f=\frac{2}{\pi}\arctan (r-r_{V,\textit{max}})   
	\end{equation}
	for all $r\in [r_+,\bar{r}_+]$. Note that $\frac{d f}{dr}>0$ in $[r_+,\bar{r}_+]$ and 
	\begin{equation}
	-fV^\prime\geq b \Delta (r-r_{V,\textit{max}})^2\left(\lambda^{(a\omega)}_{m\ell}+a^2\omega^2\right),\:r\in [r_+,\bar{r}_+].
	\end{equation}
	
	Now, to deal with the region away from the horizons, we construct the multiplier $h$ such that
	\begin{equation}\label{eq: proof prop: energy estimate angular dominated, eq 3}
	h=
	\begin{cases}
	0,\quad r^\star\in(-\infty,e^{p^{-1}}r^\star_{-\infty})\cup (e^{p^{-1}}r^\star_{+\infty},\bar{r}_{+}) \\
	1,\quad r^\star\in (r_{-\infty},r_{+\infty}),
	\end{cases}    
	\end{equation}
	and $h(r)\geq 0$ for~$r\in [r_+,\bar{r}_+]$.

	Now, the integrand in the left hand side of \eqref{eq: proof prop: energy estimate angular dominated, eq 1} satisfies the following
	\begin{equation}
	-fV^\prime+h(V-\omega^2)-\frac{1}{2}h^{\prime\prime}-\frac{1}{2}f^{\prime\prime\prime}\geq b \Delta\left(\lambda^{(a\omega)}_{m\ell}+a^2\omega^2+\omega^2+1\right),\:r\in[r_+,\bar{r}_+],
	\end{equation}
	by requiring~$\lambda_{\textit{low}}$ smaller if needed, in view of Lemma~\ref{lem: angular dominated frequencies, lem 1}. Moreover, note~$h(\pm \infty)=0$, therefore we obtain 
	\begin{equation}\label{eq: proof prop: energy estimate angular dominated, eq 4}
	\begin{aligned}
	&   \int_{\mathbb{R}} \Delta\Bigg( |u^{\prime}|^{2}+(1+\lambda^{(a\omega)}_{m\ell}+a^{2}\omega^{2}+\omega^{2})|u|^{2}  \Bigg)dr^\star  \\
	&   \quad \leq B Q^{f}(\infty)-BQ^{f}(-\infty)+B\int_{\mathbb{R}} \left(2fRe(u^{\prime}\bar{H})+(f^{\prime}+h)Re(\bar{H}u)\right)dr^\star\\
	&   \quad \leq  Bf(\infty)\left(\omega-\frac{am\Xi}{\bar{r}_+^2+a^2}\right)^2|u|^2(\infty)-Bf(-\infty)\left(\omega-\frac{am\Xi}{r_+^2+a^2}\right)^2|u|^2(-\infty)\\
	&   \quad\quad +B\int_{\mathbb{R}} \left(2fRe(u^{\prime}\bar{H})+(f^{\prime}+h)Re(\bar{H}u)\right)dr^\star.
	\end{aligned}
	\end{equation}
	Finally, in view of the fact that the $\lambda$--dominated frequency regime~$\mathcal{F}_{\lessflat}$ is non-superradiant, see Section~\ref{sec: frequencies}, then,  in view of Proposition~\ref{prop: subsec: energy identity, prop 1}, we integrate the energy identity of Lemma~\ref{lem: subsec: sec: frequency localized multiplier estimates, subsec 2, lem 1} associated with the current
	\begin{equation}
		-E_0 Q^{K^+}-E_0 Q^{\bar{K}^+}
	\end{equation}
	and sum it to inequality \eqref{eq: proof prop: energy estimate angular dominated, eq 4} to conclude the desired result. 
\end{proof}

\subsection{Multipliers for the high~\texorpdfstring{$\omega\sim \lambda$}{g}~frequency regime $F_{\natural}$}\label{subsec: trapped frequencies}

In view of the definition of the frequency regimes, see Definition~\ref{def: subsec: sec: frequencies, subsec 3, def 1}, this frequency regime is manifestly non-superradiant.

The main Proposition of this Section is the following

\begin{proposition}\label{prop: energy estimate for the trapped frequency regime}

	Let~$l>0$~$(a,M)\in \mathcal{B}_l$,~$\mu^2_{KG}\geq 0$. Let~$\alpha>0$. For any
	\begin{equation}
	\lambda_{low}^{-1}>0,\qquad\omega_{\textit{high}}>0,\qquad E>0
	\end{equation}
	all sufficiently large, then for~$(\omega,m,\ell)\in\mathcal{F}_\natural(\omega_{high},\lambda_{low},\alpha)$ there exist smooth multipliers~$f(\omega,m,\tilde{\lambda},r),~h(\omega,m,\tilde{\lambda},r)$ and there exists a piecewise~$C^1$ multiplier~$y(\omega,m,\tilde{\lambda},r)$ that satisfy the uniform bound
	\begin{equation}
	|h|+|h^\prime|+|h^{\prime\prime}|+|y|+|y^\prime|+|f|+|f^\prime|+|f^{\prime\prime}|+|f^{\prime\prime\prime}|+|h|+|h^\prime|+|h^{\prime\prime}|\leq B,
	\end{equation}
	and there exitsts an
	\begin{equation}
		r_{trap}(\omega,m,\ell)\in (r_++\epsilon_{away},\bar{r}_+-\epsilon_{away})\cup \{0\}
	\end{equation}
	where~$\epsilon_{away}(a,M,l)>0$ is any sufficiently small real number, such that for all smooth solutions~$u$ of Carter's radial ode~\eqref{eq: ode from carter's separation} satisfying the outgoing boundary conditions~\eqref{eq: lem: sec carters separation, subsec boundary behaviour of u, boundary beh. of u, eq 3}, we have 
	\begin{equation}\label{eq: prop: energy estimate for the trapped frequency regime, eq 1}
	\begin{aligned}
	&   \int_{\mathbb{R}} \Delta\left(|u^{\prime}|^2+|u|^2+ \left(1-\frac{r_{\textit{trap}}}{r}\right)^2(\omega^2+\tilde{\lambda}) |u|^2\right)dr^\star  \\
	&\quad \leq 
	B\int_{\mathbb{R}} \left(|2f\Re(u^{\prime}\bar{H})|+|f^{\prime}\Re(\bar{H}u)|+|y\Re(u^\prime\bar{H})|+|h \Re(u\bar{H})|+\left|  2y\Re(u^{\prime}H) \right|\right)dr^\star\\
	&	\qquad  + E\int_{\mathbb{R}}\left( \left(\omega-\frac{am\Xi}{r_+^2+a^2}\right)\Im (\bar{u}H)+\left(\omega-\frac{am\Xi}{r_+^2+a^2}\right)\Im(\bar{u}H)\right)dr^\star .
	\end{aligned}
	\end{equation}
\end{proposition}

\begin{remark}
If the solution~$\psi$ of the Klein--Gordon equation~\eqref{eq: kleingordon} is axisymmetric~$\partial_{\varphi^\star}\psi=0$ then we note that in Proposition~\ref{prop: energy estimate for the trapped frequency regime}, we have~$r_{\textit{trap}}=r_{\Delta,\textit{trap}}$, where~$r_{\Delta,\textit{frac}}$ is the value of the unique local maximum of~$\frac{\Delta}{(r^2+a^2)^2}$, see Lemma~\ref{lem: sec: general properties of Delta, lem 3}. 
\end{remark}

We note that for the present frequency regime~$(\omega,m,\bar{\lambda})\in \mathcal{F}_\natural$ the following hold:
\begin{itemize}
	\item if~$am\omega<0$ then~$\frac{dV_0}{dr}(r_+)>b\tilde{\lambda}$,\\
	\item  if~$am\omega> \frac{a^2m^2\Xi}{r_+^2+a^2} +\alpha|a|\tilde{\lambda}$ then~$\frac{dV_0}{dr}(\bar{r}_+)<-b\tilde{\lambda}$. 
\end{itemize}
Therefore, in view of Lemma~\ref{lem: subsec: sec: trapping, subsec 1, lem 1}, it would be convenient to keep in mind that the potential~$V_0$ satisfies only one of the following 
\begin{equation}\label{eq: sec: trapped frequencies, eq 1} 
	\begin{aligned}
		&	(V_0~case~1)~\text{There exists a unique critical point of}~V_0~\text{that satisfies}~r_+<r_{V_0,max}<\bar{r}_+\\
		&		(V_0~case~2)~\text{There exist exist two critical points of}~V_0~\text{that satisfy}~r_+<r_{V_0,min}< r_{V_0,max}<\bar{r}_+\\
		&		(V_0~case~3)~\text{There exist exist two critical points of}~V_0~\text{that satisfy}~r_+<r_{V_0,min}< r_{V_0,max}= \bar{r}_+\\
		&		(V_0~case~4)~ \frac{dV_0}{dr}>0~ \text{or}~\frac{dV_0}{dr}<0~\textit{for all}~r\in (r_+,\bar{r}_+). \\
	\end{aligned}
\end{equation}
For a graphic representation of the above cases see Figure~\ref{fig: V0}. Note that in the~$(Potential~case~3)$ we do not exclude that~$V_0$ attains a critical point at~$\bar{r}_+$.

\begin{figure}[htbp]
	\begin{multicols}{4}
		\includegraphics[scale=0.17]{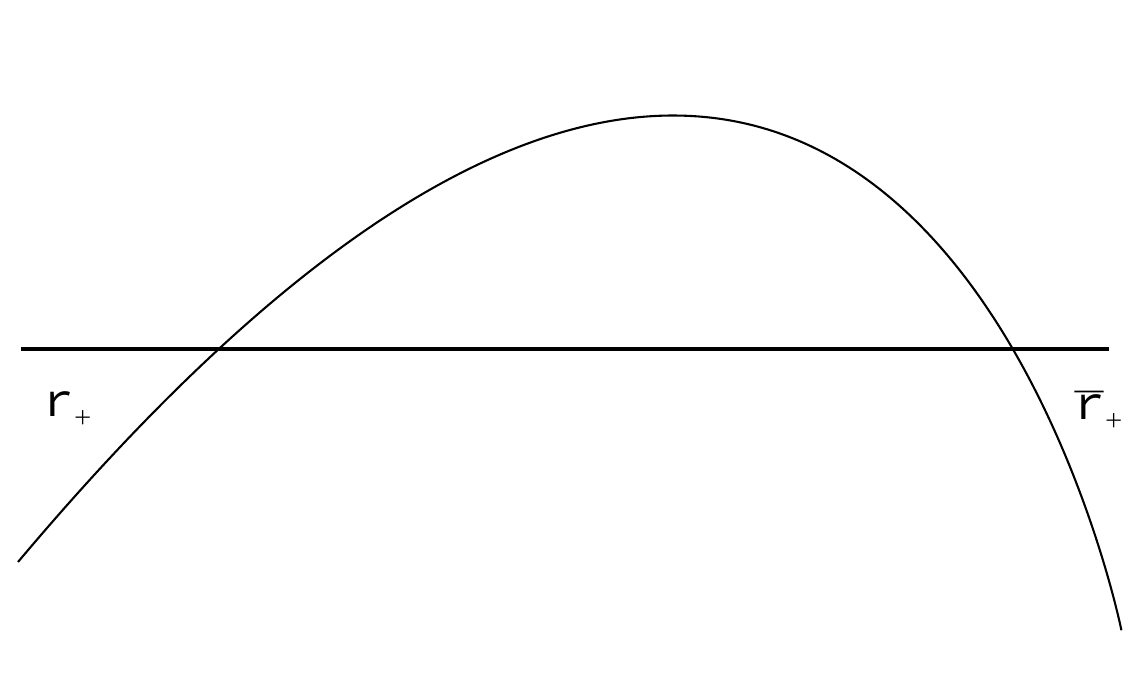}\par
		\includegraphics[scale=0.17]{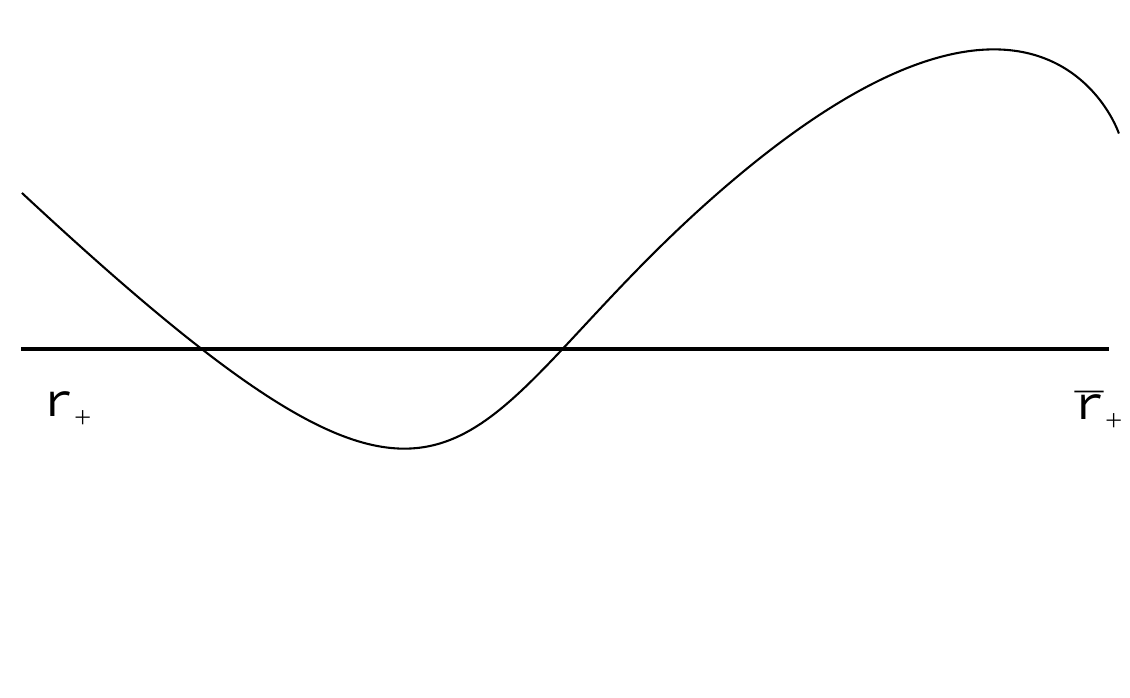}\par 
		\includegraphics[scale=0.17]{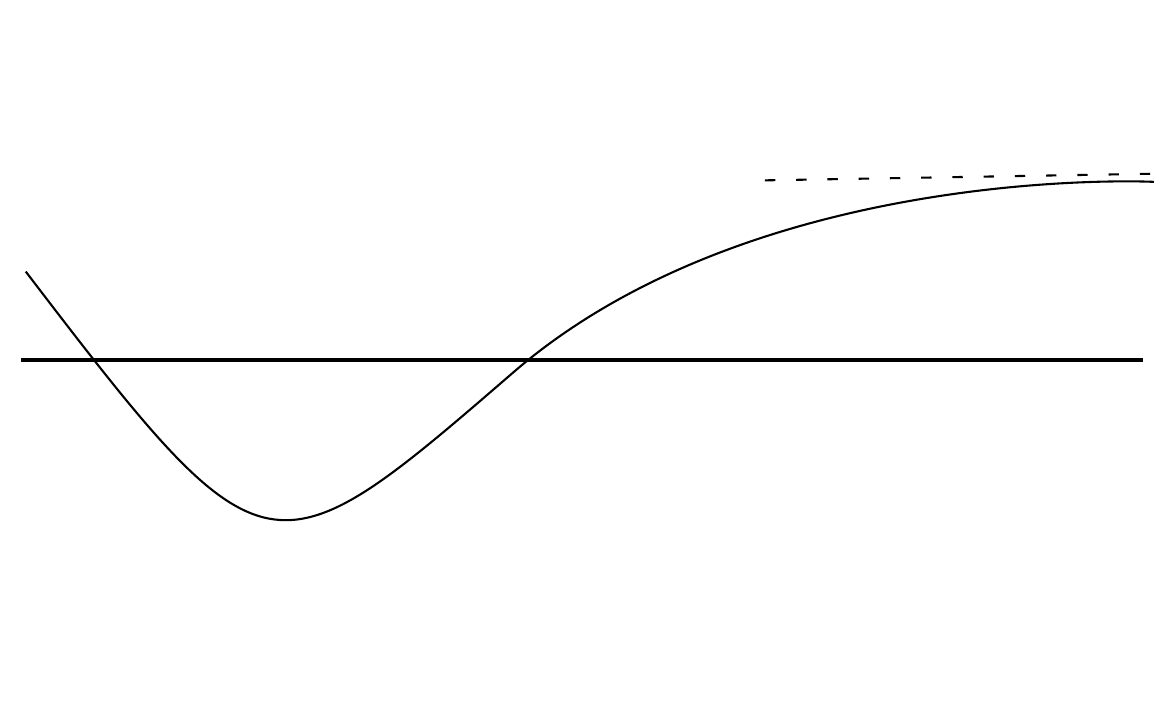}\par
		\includegraphics[scale=0.17]{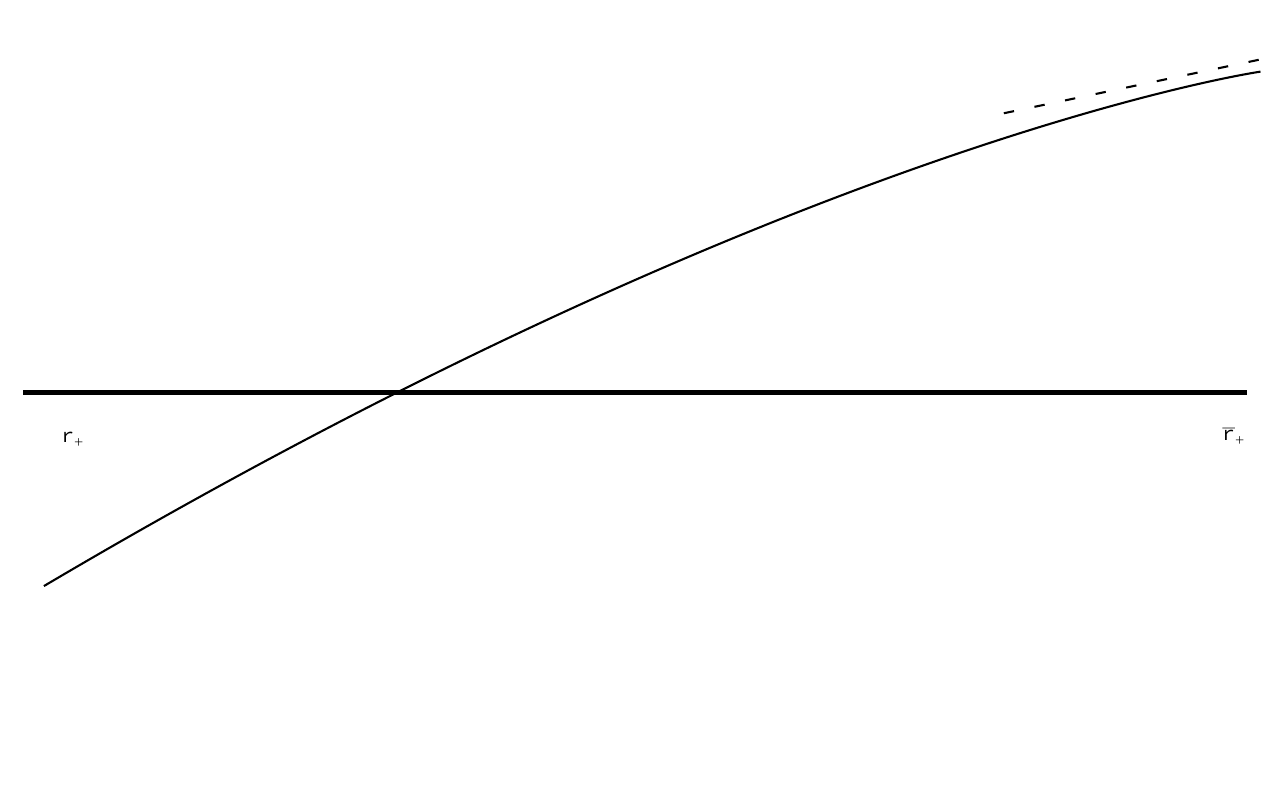}\par
	\end{multicols}
	\caption{The cases of~\eqref{eq: sec: trapped frequencies, eq 1}  respectively.}
	\label{fig: V0}
\end{figure}

Now, we construct the $Q^f,Q^y,Q^h$ currents, see Definition~\ref{def: currents}.

\begin{lemma}\label{lem: f current in trapped frequencies}
	Let $l>0$ and $(a,M)\in\mathcal{B}_l$ and~$\mu^2_{KG}\geq 0$. Let~$\alpha>0$. For any
	\begin{equation}
		\omega_{\textit{high}}>0,\qquad (\lambda_{low})^{-1}>0,\qquad \epsilon_{away}^{-1}>0
	\end{equation}
	all sufficiently large then for~$(\omega,m,\ell)\in \mathcal{F}_\natural(\omega_{high},\lambda_{low},\alpha)$ there exists a trapping parameter
	\begin{equation}
		r_{trap} (\omega,m,\ell) \in [r_++\epsilon_{away},\bar{r}_+-\epsilon_{away}]\cup \{0\},
	\end{equation}
	where~$\epsilon_{away}>0$, such that the following hold.

	There exist smooth multiplier~$f=f(\omega,m,\ell),h=h(\omega,m,\ell)$ and there exists a piecewise~$C^1$ multiplier $y=y(\omega,m,\ell)$ such that the following uniform bound holds piecewise
	\begin{equation}
		|h|+|h^\prime|+|h^{\prime\prime}|+|f|+|f^\prime|+|f^{\prime\prime}|+|f^{\prime\prime\prime}| +|y|+|y^\prime|\leq B(\omega_{high},\lambda_{low}),
	\end{equation}
and we have
	\begin{equation}\label{eq: prop: f current in trapped frequencies, eq 0}
		\begin{aligned}
			&	(P1)\quad  f(r_\textit{trap})=0,\\
			&	(P2)\quad h+ y^\prime +2f^\prime \geq b\Delta,~r\in (r_+,\bar{r}_+)\\
			&	(P3)\quad  \left(-fV^{\prime}-\frac{1}{2}f^{\prime\prime\prime}\right) +\left(y^\prime(\omega^2-V)-yV^\prime\right) +\left(h(V-\omega^2)-\frac{1}{2}h^{\prime\prime}\right)\geq b\Delta\cdot(r-r_{\textit{trap}})^2\left(\tilde{\lambda}+\omega^2\right)+b\Delta, \: r\in (r_+,\bar{r}_+).\\
		\end{aligned}
	\end{equation}	
\end{lemma}

\begin{proof}	
	
	We note that 
	\begin{equation}\label{eq: proof prop: f current in trapped frequencies, eq 0}
		\omega^2-V(r_+)= \left(\omega-\frac{am\Xi}{r_+^2+a^2}\right)^2\geq b\tilde{\lambda},\qquad \omega^2-V(\bar{r}_+)= \left(\omega-\frac{am\Xi}{\bar{r}_+^2+a^2}\right)^2\geq b\tilde{\lambda}
	\end{equation}
	where in the inequalities above we used that~$(\omega,m,\ell)\in \mathcal{F}_\natural$ for sufficiently large~$\omega_{high}>0$.

	Let~$\epsilon_{trap}>0$ be sufficiently small. In view of the definition of the high~$\omega\sim\lambda$ frequency regime~$\mathcal{F}_\natural$, see Definition~\ref{def: subsec: sec: frequencies, subsec 3, def 1}, we separate the frequency space~$(\omega,m,\ell)\in \mathcal{F}_\natural$ in the following sets 
	\begin{enumerate}\label{eq: proof prop: f current in trapped frequencies, eq 1}
		\item $|\max_{r\in [r_+,\bar{r}_+]}V-\omega^2 |\leq \epsilon_{trap} \tilde{\lambda}$,~$am\omega<0$\\
		\item $|\max_{r\in [r_+,\bar{r}_+]}V-\omega^2 |\leq \epsilon_{trap} \tilde{\lambda}$,~$am\omega\geq \frac{a^2m^2\Xi}{r_+^2+a^2}+\alpha|a|\tilde{\lambda}$\\
		\item $|\max_{r\in [r_+,\bar{r}_+]}V-\omega^2 |> \epsilon_{trap} \tilde{\lambda}$,~$am\omega<0$\\
		\item $|\max_{r\in [r_+,\bar{r}_+]}V-\omega^2 |> \epsilon_{trap} \tilde{\lambda}$,~$am\omega\geq \frac{a^2m^2\Xi}{r_+^2+a^2}+\alpha|a|\tilde{\lambda}$\\
	\end{enumerate}
As will become apparent below, the first two sets above are subject to trapping, namely~$r_{trap}\neq 0$.

$\bullet$ We begin with the first case~$(1)$. In this case we define the multipliers~$y,h\equiv 0$. First, we note that because~$am\omega<0$ we obtain
\begin{equation}
	\frac{dV_0}{dr} (r_+)> b\tilde{\lambda}.
\end{equation}
Moreover, we easily conclude that
\begin{equation}
	\frac{dV_0}{dr}(\bar{r}_+)<-b\tilde{\lambda},
\end{equation}
 in view of the fact that~$|\max_{r\in [r_+,\bar{r}_+]}V-\omega^2 |\leq \epsilon_{trap} \tilde{\lambda}$ and therefore the maximum value of~$V_0$ cannot be attained at the cosmological horizon~$\bar{r}_+$, also see the graphs of Figure~\ref{fig: V0}. Now, in view of Lemma~\ref{lem: subsec: sec: trapping, subsec 1, lem 1}, on the critical points of the potential~$V_0$, we conclude that indeed there exists a unique critical point~$r_{V_0,max}\in (r_++\epsilon_{away},\bar{r}_+-\epsilon_{away})$ of~$V_0$ which is a maximum and moreover the following holds~$\frac{d^2V_0}{dr^2}(r)<-b\tilde{\lambda}$ for~$r\in (r_{V_0}-\epsilon_{away},r_{V_0}+\epsilon_{away})$. Now, we take~$\omega_{high}>0$ sufficiently large and~$\epsilon_{away}>0$ sufficiently small to conclude that 
 \begin{equation}
 	\begin{aligned}
 		&	\frac{dV}{dr}>b\tilde{\lambda},\qquad r\in [r_+,r_{V_0}-2^{-1}\epsilon_{away}]\\
 		&		\frac{dV}{dr}<-b\tilde{\lambda},\qquad r\in [r_{V_0}+2^{-1}\epsilon_{away},\bar{r}_+]
 	\end{aligned}
 \end{equation}
 and moreover~$\frac{d^2V}{dr^2}(r)<0$ for~$r\in [r_{V_0}-2^{-1}\epsilon_{away},r_{V_0}+2^{-1}\epsilon_{away}]$. Therefore, we obtain that the potential~$V$ attains a unique maximum at~$r_{V,max}\in (r_++\epsilon_{away},\bar{r}_+-\epsilon_{away})$. In this case we define 
\begin{equation}
	r_{trap}=r_{V,max}.
\end{equation}

In view of the above, we construct a multiplier~$f$ as follows
\begin{equation}\label{eq: proof prop: f current in trapped frequencies, eq 2}
	-\frac{1}{2} f^{\prime\prime\prime} = \Delta,\qquad f(r_{trap})=0,\qquad f^\prime(r_{trap})=1,\qquad f^{\prime\prime}(r_{trap})=0,
\end{equation}	
in a neighborghood of~$r_{trap}$. We extend~$f$ smoothly in the entire interval~$[r_+,\bar{r}_+]$ such that we have
\begin{equation}
	f^\prime \geq b\Delta,\qquad -fV^\prime-\frac{1}{2}f^{\prime\prime\prime}\geq b\tilde{\lambda}\Delta (r-r_{trap})^2+b\Delta
\end{equation}
	and conclude~\eqref{eq: prop: f current in trapped frequencies, eq 0}.

$\bullet$ We continue with the second case~$(2)$. We first define the multiplier~$h\equiv 0$. Since~$am\omega\geq \frac{a^2m^2\Xi}{r_+^2+a^2}+\alpha|a|\tilde{\lambda}$ then we obtain~$\frac{dV_0}{dr}(\bar{r}_+)< -b\tilde{\lambda}$. Moreover, in view of~\eqref{eq: proof prop: f current in trapped frequencies, eq 0} we can find an~$r_3(\omega_{high},\lambda_{low},\epsilon_{trap})\in [r_+,\bar{r}_+]$ such that 
\begin{equation}
	\omega^2 -V \geq b \tilde{\lambda},\qquad r\in [r_+,r_3]. 
\end{equation}
It is an easy consequence of the trapping condition~$|\max_{r\in[r_+,\bar{r}_+]} V -\omega^2|\leq \epsilon_{trap}\tilde{\lambda}$ and of Lemma~\ref{lem: subsec: sec: trapping, subsec 1, lem 1}, on the critical points of~$V_0$, that for any~$\omega_{high},\lambda_{low}^{-1}>0$ sufficiently large and for any~$\epsilon_{trap}>0$ sufficiently small we can have
\begin{equation}\label{eq: proof prop: f current in trapped frequencies, eq 3}
	\frac{dV_0}{dr}(r_3)\geq b\tilde{\lambda},
\end{equation}
the reader should also look at the second Figure~\ref{fig: V0}. 
	
It is obvious now that, for sufficiently positive~$\omega_{high}$, the critical point~$r_{V,max}$ is attained inside the interval~$[r_3,\bar{r}_+]$. We define 
\begin{equation}
	r_{trap}=r_{V,max}
\end{equation}	
Moreover, for sufficiently small~$\epsilon_{trap},~\epsilon_{away},\lambda_{low}>0$ we conclude that~$r_{trap}\in [r_++\epsilon_{away},\bar{r}_+-\epsilon_{away}]$ and of course~$r_3<\bar{r}_+-\epsilon_{away}$.

In view of the above, we will now construct two multipliers~$y,f$. First, our multiplier~$y$ will satisfy the following
\begin{equation}\label{eq: proof prop: f current in trapped frequencies, eq 4}
	\begin{aligned}
		y	&	= 1- e^{C_{large}(r_3-r)},\quad r\in [r_+,r_3],\\
		y	&	=0,\qquad r\in [r_3,\bar{r}_+],
	\end{aligned}
\end{equation}
where~$C_{large}>0$ will be chosen sufficiently large later. Note that
\begin{equation}
	y^\prime(\omega^2-V)-yV^\prime =  \frac{\Delta}{r^2+a^2}C_{large} e^{C_{large}(r_3-r)} (\omega^2-V)-yV^\prime \geq b(C_{large}) \tilde{\lambda}
\end{equation}
in the interval~$(r_+,r_3)$ ,where in the last inequality we used that~$\omega^2\sim \tilde{\lambda}$ and we took~$C_{large}(\omega_{high},\lambda_{low})>0$ sufficiently large.

We construct the multiplier~$f$ again by solving the initial value problem~\eqref{eq: proof prop: f current in trapped frequencies, eq 2}. We choose~$C_{large}$ sufficiently positive to obtain the following 
\begin{equation}
	\begin{aligned}
	&	2f^\prime +y^\prime \geq b\Delta,\\
	&	\left( y^\prime (\omega^2-V)-y V^\prime \right)+\left( -fV^\prime -\frac{1}{2}f^{\prime\prime}\right)\geq b\tilde{\lambda}(r-r_{trap})^2\Delta + b\Delta,
	\end{aligned}
\end{equation}
	which concludes~\eqref{eq: prop: f current in trapped frequencies, eq 0}. 
	
$\bullet$~We continue with the third case~$(3)$. In this case we set
\begin{equation}
	r_{trap}=0.
\end{equation}
We first define~$y\equiv 0$. We note argue similarly to case~$(1)$. Specifically, since~$am\omega<0$ we have that
\begin{equation}
	\frac{dV_0}{dr}(r_+)\geq b\tilde{\lambda}. 
\end{equation}
Therefore, by  Lemma~\ref{lem: subsec: sec: trapping, subsec 1, lem 1}, we take~$\omega_{high}>0$ sufficiently large and we find that there exists a maximum of the potential~$V$, which we here call~$r_{V,max}$ such that~$r_{V,max}\in (r_+,\bar{r}_+]$. Note that~$r_{V,max}$ is not necessarily a critical point and that it can be the case that~$r_{V,max}=\bar{r}_+$.

We define a multiplier~$f$ by solving the initial value problem
\begin{equation}
	-\frac{1}{2} f^{\prime\prime\prime} = \Delta,\qquad f(r_{V,max})=0,\qquad f^\prime(r_{V,max})=1,\qquad f^{\prime\prime}(r_{V,max})=0,
\end{equation}	
in a neighborhood of~$r_{V,max}$. Recall that we do not exclude the case~$r_{V,max}=\bar{r}_+$. We extend~$f$ smoothly in the entire interval~$[r_+,\bar{r}_+]$ such that
\begin{equation}
	\begin{aligned}
		f^\prime \geq b\Delta,\qquad -fV^\prime -\frac{1}{2}f^{\prime\prime\prime}\geq b\Delta (r-r_{trap})^2 \tilde{\lambda}+b\Delta.
	\end{aligned}
\end{equation}
after taking~$\omega_{high}>0$ sufficiently large. 	

We use the non trapping condition~$|\max_{r}V-\omega^2|\geq \epsilon_{trap} \tilde{\lambda}$ so that we can construct a multiplier~$h$ such that for sufficiently large~$\omega_{high},\lambda_{low}^{-1}$ and sufficiently small~$\epsilon_{trap}>0$ we find that 
\begin{equation}\label{eq: proof prop: f current in trapped frequencies, eq 5}
	\left(h(V-\omega^2)-\frac{1}{2}h^{\prime\prime}\right) + \left(-fV^\prime -\frac{1}{2}f^{\prime\prime}\right)\geq b \Delta \tilde{\lambda},
\end{equation}	
	which concludes~\eqref{eq: prop: f current in trapped frequencies, eq 0}.

$\bullet$ Finally, we study that fourth case~$(4)$. We split the frequency space of case $(4)$ into the two frequency regimes
\begin{equation}\label{eq: proof prop: f current in trapped frequencies, eq 6}
	\begin{aligned}
		&	\{\omega^2-\max_{r\in[r_+,\bar{r}_+]}V_0\geq \epsilon_{trap} \tilde{\lambda} \}\\
		&	\{\max_{r\in[r_+,\bar{r}_+]}V_0-\omega^2\geq \epsilon_{trap} \tilde{\lambda} \}.
	\end{aligned}
\end{equation}

For the first case of~\eqref{eq: proof prop: f current in trapped frequencies, eq 6} we note that~$\omega^2-V(r)\geq \epsilon_{trap}\tilde{\lambda}$ for all~$r\in [r_+,\bar{r}_+]$. We define~$h,f\equiv 0$ and we define a multiplier of the form~$y=e^{C_{large}r}$. For a sufficiently large~$C_{large}>0$ we obtain that 
\begin{equation}
	y^\prime(\omega-V)-yV^\prime \geq b\Delta \tilde{\lambda}.
\end{equation}
We fix the~$C_{large}>0$.

We now study the second case of~\eqref{eq: proof prop: f current in trapped frequencies, eq 6}. First, we note that since~$am\omega> \frac{a^2m^2\Xi}{r_+^2+a^2}>\frac{a^2m^2\Xi}{\bar{r}_+^2+a^2}$ and~$am\omega \geq \alpha|a|\tilde{\lambda}$ then we have that
\begin{equation}\label{eq: proof prop: f current in trapped frequencies, eq 6.1}
	\frac{dV_0}{dr}(\bar{r}_+)<-b\tilde{\lambda}. 
\end{equation}
We note that the maximum value of~$V_0$ cannot be attained at the event horizon since then we would have~$-\left(\omega-\frac{am\Xi}{r_+^2+a^2}\right)^2\geq \epsilon_{trap}\tilde{\lambda}$. Therefore, for any sufficiently small~$\epsilon_{away}>0$ and for any~$\omega_{high}>0$ sufficiently large we have that the potential~$V$ attaines a maximim 
\begin{equation}
	r_{V,max}\in (r_++\epsilon_{away},\bar{r}_+-\epsilon_{away}). 
\end{equation}
Now, for any~$\epsilon_1>0$ sufficiently small we have that~$\frac{dV_0}{dr}(r_{V_0,max}-\epsilon_1)>b(\epsilon_1)\tilde{\lambda}$. We define~$r_3=r_{V_0,max}-\epsilon_1$. We fix~$\epsilon_1$.

We now construct three multipliers~$y,f,h$ as follows. First, we define~$y$ to be the following piecewise smooth function
\begin{equation}
	y=
	\begin{cases}
		&	1-e^{C_1(r_3-r)},~r\in [r_+,r_3]\\
		&	0,~r\in [r_3,\bar{r}_+],
	\end{cases}
\end{equation}
for some~$C_1>0$ that will be chosen sufficiently large later. We define a multiplier~$f$ by solving the initial value problem
\begin{equation}
	-\frac{1}{2} f^{\prime\prime\prime} = \Delta,\qquad f(r_{V_0,max})=0,\qquad f^\prime(r_{V_0,max})=1,\qquad f^{\prime\prime}(r_{V_0,max})=0,
\end{equation}	
in a neighborhood of~$r_{V,max}$. We extend~$f$ smoothly in the entire interval~$[r_+,\bar{r}_+]$ such that
\begin{equation}
	\begin{aligned}
		f^\prime \geq b\Delta,\qquad -fV^\prime-\frac{1}{2}f^{\prime\prime\prime}\geq b\Delta(r-r_{V,max})^2\tilde{\lambda}.
	\end{aligned}
\end{equation}
We define the smooth multiplier~$h$ such that it satisfies the following 
\begin{equation}
	h=
	\begin{cases}
		&	1,\qquad r\in [r_{V,max}-2\epsilon_1,r_{V,max}+2\epsilon_1]\\
		&	0,\qquad r\in [r_+,\bar{r}_+]\setminus[r_{V,max}-2\epsilon_1,r_{V,max}+2\epsilon_1].
	\end{cases}
\end{equation}

Therefore, for a sufficiently large~$C_1>0$ then in view of the non trapping condition~$\max V_0-\omega^2\geq \epsilon_{trap}\tilde{\lambda}$ we obtain that for any~$\omega_{high}>0$ sufficiently large the following hold
\begin{equation}
	\begin{aligned}
		&	y^\prime+2f^\prime+h\geq b\Delta\\
		&	h(V-\omega^2)-\frac{1}{2}h^{\prime\prime}+y^\prime(\omega^2-V)-yV^\prime +fV^\prime-\frac{1}{2}f^{\prime\prime\prime}\geq b\Delta\tilde{\lambda}. 	\end{aligned}
\end{equation}

	We conclude Lemma~\ref{lem: f current in trapped frequencies}.
\end{proof}

\begin{proof}[\textbf{Proof of Proposition~\ref{prop: energy estimate for the trapped frequency regime}}]
	
We integrate the energy identity of Lemma~\ref{lem: subsec: currents, lem 1} associated to the current
	\begin{equation}
	Q^f+Q^y+Q^h
	\end{equation}
	see Definition~\ref{def: currents}, and use the construction of the~$f,y$ multipliers of Lemma~\ref{lem: f current in trapped frequencies} to obtain the energy inequality
	\begin{equation}\label{eq: proof prop: energy estimate for the trapped frequency regime, eq 1}
	\begin{aligned}
	&   \int_{\mathbb{R}}\Delta\left(|u^{\prime}|^2+|u|^2+ \left((r-r_{\textit{trap}})^2(\lambda^{(a\omega)}_{m\ell}+a^{2}\omega^{2}+\omega^{2})\right) |u|^{2}\right) dr^\star  \\
	&   \quad\leq Q^f(\infty)-Q^f(-\infty) -B\int_{\mathbb{R}}\left( 2f\Re(u^{\prime}\bar{H})+f^{\prime}\Re(\bar{H}u)\right)dr^\star +Q^y(\infty)-Q^y(-\infty) +B\int_{\mathbb{R}}\Big( 2y\Re(u^{\prime}H) \Big)dr^\star \\
	&   \quad\leq B \left(\omega-\frac{am\Xi}{\bar{r}_+^2+a^2}\right)^2|u|^2(\infty)+B\left(\omega-\frac{am\Xi}{r_+^2+a^2}\right)^2|u|^2(-\infty)\\
	&	\qquad\qquad +B\int_{\mathbb{R}}\left( \left|2f\Re(u^{\prime}\bar{H})\right|+\left|f^{\prime}\Re(\bar{H}u)\right|+\left|2y\Re(u^{\prime}H\right|+\left|2y\Re(u^{\prime}H) \right|+|h\Re (u\bar{H})|\right)dr^\star.
	\end{aligned}
	\end{equation}

	Finally, since the high~$\omega\sim\lambda$ frequency regime is non-superradiant, see Section~\ref{sec: frequencies}, then, in view of Proposition~\ref{prop: subsec: energy identity, prop 1}, we integrate the energy identity of Lemma~\ref{lem: subsec: sec: frequency localized multiplier estimates, subsec 2, lem 1} associated with the current 
	\begin{equation}
		-E_0 Q^{K^+}-E_0Q^{\bar{K}^+}
	\end{equation}
	 where~$E_0>0$ is a sufficiently large constant, and then we sum in equation~\eqref{eq: proof prop: energy estimate for the trapped frequency regime, eq 1} to conclude the proof of Proposition~\ref{prop: energy estimate for the trapped frequency regime}. 
\end{proof}

\subsection{Multipliers for the \texorpdfstring{$\omega$}{g}~dominated frequency regime $F_{\sharp}$}\label{subsec: time dominated frequencies}

In view of the definition of the frequency regimes, see Definition~\ref{def: subsec: sec: frequencies, subsec 3, def 1}, this frequency regime is manifestly non-superradiant.

The main Proposition of this Section is the following 

\begin{proposition}\label{prop: energy estimate for the time dominated frequency regime}

		Let~$l>0$~$(a,M)\in \mathcal{B}_l$,~$\mu^2_{KG}\geq 0$. Let~$\alpha>0$. For any 
	\begin{equation}
	\lambda_{\textit{low}}^{-1}>0,\qquad \omega_{\textit{high}}>0
	\end{equation}
	both sufficiently large, then for~$(\omega,m,\ell)\in F_\sharp(\omega_{high},\lambda_{low},\alpha)$ there exists a sufficiently regular function~$y(r^\star)$ that satisfies the uniform bound 
	\begin{equation}
	|y|+|y^\prime|\leq B(\omega_{\textit{high}},\lambda_{\textit{low}})
	\end{equation}
	such that for all smooth solutions~$u$ of Carter's radial ode~\eqref{eq: ode from carter's separation} satisfying the outgoing boundary conditions~\eqref{eq: lem: sec carters separation, subsec boundary behaviour of u, boundary beh. of u, eq 3}, we have 
	\begin{equation}
	\begin{aligned}
	&   \int_{\mathbb{R}} \Delta\Big(|u^{\prime}|^{2}+(\lambda^{(a\omega)}_{m\ell}+\omega^{2}+1)|u|^{2}  \Big)dr^\star \\
	&   \leq B\Big(\int_{\mathbb{R}}\left( \left|\left(\omega-\frac{am\Xi}{\bar{r}_+^2+a^2}\right)\Im (\bar{H}u)\right|+\left|\left(\omega-\frac{am\Xi}{r_+^2+a^2}\right)\Im (\bar{H}u)\right|\right)dr^\star+   \int_{\mathbb{R}}\left|  2y\Re(u^{\prime}H) \right|dr^\star\Big).
	\end{aligned}
	\end{equation}
\end{proposition}

We need the following Lemma 

\begin{lemma}\label{lem: time dominated regime, lem 1}
Let~$l>0$ and~$(a,M)\in\mathcal{B}_l$ and~$\mu^2_{KG}\geq 0$. Let~$\alpha>0$. For any
\begin{equation}
	\lambda_{\textit{low}}^{-1}>0,\qquad \omega_{\textit{high}}>0
\end{equation}
both sufficiently large, then for~$(\omega,m,\ell)\in\mathcal{F}_\sharp$ we have  
	\begin{equation}
	\omega^{2}-V(r)\geq b \omega^{2},\qquad r\in [r_+,\bar{r}_+]. 
	\end{equation}
\end{lemma}

\begin{proof}
	The potential $V$, of \eqref{eq: the potential V} can be written as 
	\begin{equation}\label{eq: lem: time dominated regime, lem 1, eq 1}
	\begin{aligned}
	V  &=   \Delta\left( \frac{r^{3}\frac{d}{dr}\Delta -2r^{2}\Delta+a^{2}r\frac{d}{dr}\Delta}{(r^{2}+a^{2})^{4}} +\frac{\Delta a^{2}}{(r^{2}+a^{2})^{4}}+\frac{\lambda^{(a\omega)}_{m\ell}+a^2\omega^2}{(r^{2}+a^{2})^{2}} +\mu_{\textit{KG}}^2\cdot\frac{r^2+a^{2} }{(r^{2}+a^{2})^2}\right) \\
	&   \quad -\frac{\Xi^{2}a^{2}m^{2}+2m\omega a\Xi (\Delta -(r^{2}+a^{2}))}{(r^{2}+a^{2})^{2}}.
	\end{aligned}
	\end{equation}

	We consider the term
	\begin{equation}
	\frac{\Delta a^2}{(r^2+a^2)^2} \omega^2
	\end{equation}
	of \eqref{eq: lem: time dominated regime, lem 1, eq 1} and we note
	\begin{equation}
	\frac{\Delta a^2}{(r^2+a^2)^2}=\frac{a^2}{r^2+a^2}-\frac{2Mr}{(r^2+a^2)^2}a^2-\frac{a^2r^2}{l^2(r^2+a^2)}=\frac{a^2}{r^2+a^2}\left(1-\frac{2Mr}{r^2+a^2}-\frac{r^2}{l^2}\right)<1.
	\end{equation}
	Therefore, we obtain 
	\begin{equation}\label{eq: lem: time dominated regime, lem 1, eq 2}
	\omega^2-\frac{\Delta a^2}{(r^2+a^2)^2}\omega^2\geq b\omega^2.
	\end{equation}
	By utilizing~\eqref{eq: lem: time dominated regime, lem 1, eq 2} that~$\omega^2+(am)^2\geq \lambda_{low}^{-1}|\tilde{\lambda}|$, and~\eqref{eq: lem: inequality for lambda, eq 3} from Lemma~\ref{lem: inequality for lambda}, we conclude that for~$\lambda_{\textit{low}}>0$ sufficiently small and~$\omega_{\textit{high}}>0$ sufficiently large we obtain 
	\begin{equation}
	\omega^2-V\geq b\omega^2.
	\end{equation}
	We conclude the result. 
\end{proof}

We employ a $Q^{y}$ current, see Definition \ref{def: currents}. 

\begin{lemma}\label{lem: y current on time dominated frequencies}
	Let~$l>0$,~$(a,M)\in \mathcal{B}_l$ and~$\mu_{KG}\geq 0$. Let~$\alpha>0$. For any
	\begin{equation}
		\lambda^{-1}_{\textit{low}}>0,\qquad \omega_{\textit{high}}>0
	\end{equation}
	both sufficiently large, then for~$(\omega,m,\ell)\in\mathcal{F}_\sharp(\lambda_{low},\omega_{high},\alpha)$ we obtain the following.
	
	There exists a smooth bounded function~$y$ that satisfies
	\begin{equation}
	y^\prime\geq b\Delta,\qquad y^{\prime}(\omega^{2}-V)-yV^{\prime}\geq b\Delta\omega^2,\qquad r\in[r_+,\bar{r}_+]. 
	\end{equation}
\end{lemma}
\begin{proof}
	
	Let~$c_1>0$ be an arbitrary strictly positive real number. We choose a bounded smooth function~$y$ such that 
	\begin{equation}\label{eq: proof lem: y current on time dominated frequencies, eq 1}
		\begin{aligned}
			y(r)	&	=(r-r_{\Delta,\textit{frac}})\cdot\frac{dy}{dr} (r_{\Delta,\textit{frac}})+ \mathcal{O}(r-r_{\Delta,\textit{frac}})^2,\quad \text{in a small neightborhood of }r_{\Delta,\textit{frac}}\\
			\frac{dy}{dr}	&	> c_1>0,\quad r\in [r_+,\bar{r}_+],		
		\end{aligned}
	\end{equation}
	where~$r_{\Delta,\textit{frac}}$ is the unique critical point of~$\frac{\Delta}{(r^2+a^2)^2}$. Now, by utilizing the multiplier $y$ constructed in \eqref{eq: proof lem: y current on time dominated frequencies, eq 1}, we obtain 
	\begin{equation}\label{eq: proof lem: y current on time dominated frequencies, eq 2}
	\begin{aligned}
	y^\prime (\omega^2-V)-yV^\prime	&=\frac{\Delta}{r^2+a^2}\left(\frac{d y}{dr}(\omega^2-V)-y\frac{d V}{dr}\right)\\
	& \geq \frac{\Delta}{r^2+a^2}\left( \frac{d y}{dr}(b\omega^2)-y\frac{dV}{dr} \right)\\
	& \geq \frac{\Delta}{r^2+a^2}\left(c_1 b \omega^2-y\frac{dV}{dr}\right),
	\end{aligned}
	\end{equation}
	where we used Lemma~\ref{lem: time dominated regime, lem 1}. Finally, note that for the last term of equation~\eqref{eq: proof lem: y current on time dominated frequencies, eq 2}, the following holds 
	\begin{equation}
	\begin{aligned}
	-y\frac{d V}{dr} &=-y\frac{d}{dr}\left(V_{\textit{SL}} +\frac{\Delta}{(r^2+a^2)^2}(\lambda^{(a\omega)}_{m\ell}+a^2\omega^2-2m\omega a \Xi)-\left(\omega-\frac{am\Xi}{r^2+a^2}\right)^2+V_{\mu}\right)\\
	&=-y\frac{d V_{\textit{SL}}}{dr} -y\frac{d}{dr}\left(\frac{\Delta}{(r^2+a^2)^2}\right)(\lambda^{(a\omega)}_{m\ell}+a^2\omega^2-2m\omega a \Xi)-y\frac{d}{dr}\left(\omega-\frac{am\Xi}{r^2+a^2}\right)^2 -y\frac{d V_{\mu_{KG}}}{dr} \\
	&  \geq -y\frac{d V_{\textit{SL}}}{dr}-y\frac{d}{dr}\left(\omega-\frac{am\Xi}{r^2+a^2}\right)^2,
	\end{aligned}
	\end{equation}
	where in the last inequality we used that $y$ changes sign at the maximum of $\frac{\Delta}{(r^2+a^2)^2}$. For $V_{\textit{SL}},V_{\mu}$ see Section \ref{sec: carter separation}. Therefore, note that for~$\lambda_{\textit{low}}^{-1},\omega_{\textit{high}}>0$ sufficiently large we obtain 
	\begin{equation}
	c_1b\omega^2-y\frac{dV}{dr}\geq c_1b\omega^2-y\frac{d}{dr}\left(\omega-\frac{am\Xi}{r^2+a^2}\right)^2\geq b \omega^2,
	\end{equation}
	after appropriate Young's inequalities. 

	 We conclude the proof of the Lemma. 
\end{proof}

\begin{proof}[\textbf{Proof of Proposition \ref{prop: energy estimate for the time dominated frequency regime}}]
	
We integrate the energy identity of Lemma~\ref{lem: subsec: currents, lem 1}, associated to~$Q^y$ current, to get 
	\begin{equation}\label{eq: proof prop: energy estimate for the time dominated frequency regime, eq 1}
	\begin{aligned}
	&   \int_{\mathbb{R}}dr^\star\left(y^\prime |u^\prime|^2+\left(y^\prime(\omega^2-V)-yV^\prime\right)|u|^2\right)\\
	&   \quad =Q^y(\infty)-Q^y(-\infty)+\int_{\mathbb{R}} 2y\Re (u^\prime\bar{H})dr^\star.
	\end{aligned}
	\end{equation}
	Now, we use \eqref{eq: proof prop: energy estimate for the time dominated frequency regime, eq 1}, in conjunction with the construction of the $y$ multiplier, see Lemma \ref{lem: y current on time dominated frequencies}, to obtain 
	\begin{equation}\label{eq: proof prop: energy estimate for the time dominated frequency regime, eq 2}
	\int_{\mathbb{R}}\Delta\Big( |u^{\prime}|^{2}+(\lambda^{(a\omega)}_{m\ell}+\omega^{2}+1)|u|^{2}   \Big)dr^\star\leq B
	\int_{\mathbb{R}}\Big( 2y\Re(u^{\prime}H) \Big)dr^\star +Q^{y}(\infty)-Q^{y}(-\infty).
	\end{equation}
	Finally, in inequality \eqref{eq: proof prop: energy estimate for the time dominated frequency regime, eq 2}, in view of Proposition~\ref{prop: subsec: energy identity, prop 1}, we integrate the energy identity of Lemma~\ref{lem: subsec: sec: frequency localized multiplier estimates, subsec 2, lem 1} associated with the current 
	\begin{equation}
		-E_0 Q^{K^+}-E_0 Q^{\bar{K}^+}
	\end{equation}
 where $E_0$ is a sufficiently large constant, and we sum it in~\eqref{eq: proof prop: energy estimate for the time dominated frequency regime, eq 2} to conclude the desired result. 
\end{proof}

\subsection{The choices of \texorpdfstring{$\omega_{\textit{high}},\omega_{\textit{low}},\lambda_{\textit{low}},\alpha,E$}{l}}\label{subsec: choice of omega1, lambda2}

We will now fix the constants
\begin{equation}
E>0,\qquad \alpha^{-1}>0,\qquad 
\omega_{\textit{high}}>0,\qquad \omega_{\textit{low}}^{-1}>0 \qquad \lambda_{\textit{low}}^{-1}>0
\end{equation}
so that all of the Propositions~\ref{prop: energy estimate in the de Sitter freqs, l larger that omega},~\ref{prop: energy estimate in superradiant frequencies}, \ref{prop: energy estimate for the bounded stationary frequencies}, \ref{prop: energy estimate for the bounded stationary frequencies, 1}, \ref{prop: subsubsec: bounded stationary freqs, large a}, \ref{prop: energy estimate in the bounded non stationary frequency regime}, \ref{prop: energy estimate in the angular dominated frequency regime}, \ref{prop: energy estimate for the trapped frequency regime}, \ref{prop: energy estimate for the time dominated frequency regime} hold.

Specifically: First we fix~$E$ so that Propositions~\ref{prop: energy estimate for the trapped frequency regime},~\ref{prop: energy estimate in the de Sitter freqs, l larger that omega} hold. Then, we fix~$\omega_{high}$ so that Propositions~\ref{prop: energy estimate in the de Sitter freqs, l larger that omega},~\ref{prop: energy estimate in superradiant frequencies},~\ref{prop: energy estimate in the angular dominated frequency regime},~\ref{prop: energy estimate for the trapped frequency regime},~\ref{prop: energy estimate for the time dominated frequency regime} hold, and note that the remaining Propositions hold for any~$\omega_{high}>0$. We fix~$\omega_{low}$ so that Propositions~\ref{prop: energy estimate for the bounded stationary frequencies},~\ref{prop: energy estimate for the bounded stationary frequencies, 1},~\ref{prop: subsubsec: bounded stationary freqs, large a},~\ref{prop: energy estimate in the bounded non stationary frequency regime} hold and note that the remaining Propositions hold for any~$\omega_{low}>0$. Finally, we fix~$\lambda_{low},\alpha$ from Propositions~\ref{prop: energy estimate in superradiant frequencies}~\ref{prop: energy estimate in the angular dominated frequency regime},~\ref{prop: energy estimate for the trapped frequency regime}~\ref{prop: energy estimate for the time dominated frequency regime}, and note that the remaining  Propositions hold for any~$\lambda_{low},\alpha>0$.

\subsection{Proof of Theorem~\ref{thm: sec: proofs of the main theorems}}\label{subsec: sec: proof of main prop, subsec 1}

Now, it is immediate to prove the following 
\begin{proof}[\textbf{Proof of Theorem~\ref{thm: sec: proofs of the main theorems}}]

	We use the results of Propositions \ref{prop: energy estimate in the de Sitter freqs, l larger that omega}, \ref{prop: energy estimate in superradiant frequencies}, \ref{prop: energy estimate for the bounded stationary frequencies}, \ref{prop: energy estimate for the bounded stationary frequencies, 1}, \ref{prop: subsubsec: bounded stationary freqs, large a}, \ref{prop: energy estimate in the bounded non stationary frequency regime}, \ref{prop: energy estimate in the angular dominated frequency regime}, \ref{prop: energy estimate for the trapped frequency regime}, \ref{prop: energy estimate for the time dominated frequency regime} and choose any~$\mathcal{C}>0$ such that 
	\begin{equation}
		\mathcal{C}\geq \lambda_{low}^{-1}\omega^2_{high},
	\end{equation}
	and obtain 
	\begin{equation}
		\begin{aligned}
			& 1_{\left((\omega,m,\ell)\not{\in}\{\mathcal{F}_\flat\cap \{m=0\}\}\cup \mathcal{F}_{dS}\cup \mathcal{F}_\natural \right)} \int_{-\infty}^{\infty}  \Delta\left( |u^\prime|^2+\left(1+\omega^2+\tilde{\lambda}\right)|u|^2\right)dr^\star\\
			&	\qquad\qquad\qquad + 1_{(\mathcal{F}_\flat \cap \{m=0\})}\int_{\mathbb{R}}\Delta \left( |\Psi^\prime|^2 + (\omega^2+\tilde{\lambda})|\Psi|^2 \right)dr^\star +\int_{r^\star_{-\infty}}^{r^\star_\infty}  1_{|m|>0}\left(|u|^2 +\mu^2_{\textit{KG}} |u|^2\right)dr^\star \\
			&	\qquad\qquad\qquad 	+ 1_{\mathcal{F}_{\textit{dS}}\cup \mathcal{F}_\natural} \int_\mathbb{R}\Delta \left( |u^\prime|^2+|u|^2+\left(1-\frac{r_{\textit{trap}}}{r}\right)^2\left(\tilde{\lambda}+\omega^2\right)|u|^2\right)dr^\star \\
			&	\qquad \leq  B  1_{|m|>0} \int_{\mathbb{R}}\left(\left|\Re (u\bar{H})\right|+\left| \Re (u\bar{H})\right|+\left|\Re (u^\prime \bar{H})\right|\right)dr^\star \\
			&	\qquad\qquad +B \cdot  1_{\mathcal{F}_\flat \cap \{m=0\}}\int_{\mathbb{R}}\left( \big|\omega\Im (\bar{\Psi}H)\big|+\big|\omega\Im(\bar{\Psi}H)\big|+ |H|^2 + |\Psi^\prime H|^2\right)dr^\star\\
			&	\qquad\qquad +B\cdot  1_{ \mathcal{F}^\sharp\cup \mathcal{F}_\flat  \cup\mathcal{F}_{\lessflat}\cup \mathcal{F}_\sharp}\int_{\mathbb{R}}\left(\left|\omega-\frac{am\Xi}{r_+^2+a^2}\right||\Im (\bar{u}H)|+\left|\omega-\frac{am\Xi}{r_+^2+a^2}\right||\Im(\bar{u}H)|\right)dr^\star\\
			&	\qquad\qquad +E \int_{\mathbb{R}} \left(\left(\omega-\frac{am\Xi}{r_+^2+a^2}\right)\Im (\bar{u}H)+\left(\omega-\frac{am\Xi}{r_+^2+a^2}\right)\Im(\bar{u}H)\right)dr^\star\\
			&	\qquad\qquad +B 1_{\mathcal{F}_{\mathcal{SF},\mathcal{C}}}|\omega-\omega_+ m||\omega-\bar{\omega}_+ m| |u|^2(-\infty),
		\end{aligned}
	\end{equation}
	where for the set~$\mathcal{F}_{\mathcal{SF},\mathcal{C}}$ see~\eqref{eq: subsec: sec: carter separation, subsec 2, eq 1}. We conclude the proof.
\end{proof}

\subsection{Trapping parameters for fixed azimuthal frequency~\texorpdfstring{$m$}{m}}

The following lemma is used in the fixed azimuthal frequency Proposition~\ref{prop: sec: continuity argument, prop 2}. Note that in the following lemma we study trapping for a fixed azimuthal number~$m$.

We emphasize that the parameters~$\omega_{high},\lambda_{low}$ that will be displayed in the following lemma are different from the choices of Section~\ref{subsec: choice of omega1, lambda2}.

\begin{lemma}\label{lem: sec: continuity argument, lem 1}
	Let~$l>0$,~$(a,M)\in\mathcal{B}_l$ and~$\mu^2_{\textit{KG}}\geq 0$. Let the azimuthal frequency
	\begin{equation}
		|m|\geq 0
	\end{equation}
	be fixed.

	Then, we obtain that for any
	\begin{equation}
		\lambda_{\textit{low}}>0,\qquad \epsilon_{\textit{trap}}>0
	\end{equation}
	both sufficiently small there exists an
	\begin{equation}
		\omega_{\textit{high}}(m,\epsilon_{\textit{trap}},\lambda_{\textit{low}})>0
	\end{equation}
	sufficiently large such that the following holds
	\begin{equation}
		r_{\Delta,\textit{frac}}-\epsilon_{\textit{trap}}\leq r_{\textit{trap}}(\omega,m,\ell) \leq r_{\Delta,\textit{frac}}+\epsilon_{\textit{trap}},
	\end{equation} 
	where~$r_{\Delta,\textit{frac}}$ is the value of the unique maximum of 
	\begin{equation}
		\frac{\Delta}{(r^2+a^2)^2}.
	\end{equation}
\end{lemma}

\begin{proof}
	Recall that trapping can only occur in the frequency regimes 
	\begin{equation}
		\begin{aligned}
			\mathcal{F}_{\textit{dS}}	&	= \{(\omega,m,\ell)\: : \: \tilde{\lambda}\geq\left(\frac{|a|\Xi}{r_+^2+a^2}+\alpha\right)^{-1}\omega_{\textit{high}}\}\cap\{a m\omega\in\Big(0,\frac{a^2m^2\Xi}{\bar{r}_+^2+a^2}\Big)\} 	,\\
			\mathcal{F}_\natural &=  \{(\omega,m,\ell)\: : \:|\omega|\geq \omega_{\textit{high}},\:\lambda_{\textit{low}}\tilde{\lambda}\leq \omega^2+(am)^2\leq \lambda_{\textit{low}}^{-1}\tilde{\lambda}\}\cap\{am\omega\slashed{\in}\Big[0,\frac{a^2m^2\Xi}{r_+^2+a^2}+|a|\alpha\tilde{\lambda}\Big)\},\\
		\end{aligned}
	\end{equation}
	where the above frequency regimes where first defined in Section~\ref{sec: frequencies}.

	Moreover, recall that~$r_{\textit{trap}}$ is the maximum (when it is not~$0$) of the potential~$V$ in the frequency regimes~$\mathcal{F}_{\textit{dS}},~\mathcal{F}_\natural$ respectively. The derivative of the potential is 
	\begin{equation}
		\frac{dV}{dr}=\frac{dV_{\textit{SL}}}{dr}+\frac{dV_{\mu_{\textit{KG}}}}{dr}+\frac{dV_0}{dr}
	\end{equation}
	with 
	\begin{equation}
		\begin{aligned}
			\frac{dV_0}{dr}	&	= \frac{d}{dr}\left(\frac{\Delta}{(r^2+a^2)^2}\right)\left(\tilde{\lambda}-2m\omega a\Xi\right)-\frac{4r am\Xi}{(r^2+a^2)^2}\left(\omega-\frac{am\Xi}{r^2+a^2}\right)\\
			&	=\tilde{\lambda}\left(\frac{d}{dr}\left(\frac{\Delta}{(r^2+a^2)^2}\right)\left(1-\frac{2m\omega a\Xi}{\tilde{\lambda}}\right)-\frac{1}{\tilde{\lambda}}\frac{4r am\Xi}{(r^2+a^2)^2}\left(\omega-\frac{am\Xi}{r^2+a^2}\right)\right)
		\end{aligned}
	\end{equation}
	and the derivative terms~$\frac{dV_{\textit{SL}}}{dr}+\frac{dV_{\mu_{\textit{KG}}}}{dr}$ do not depend on the frequencies~$(\omega,m,\ell)$. 
	
	Therefore, for fixed azimuthal frequency~$m$ and for any sufficiently small 
	\begin{equation}
		\epsilon_{\textit{trap}}>0,
	\end{equation}
	we note that by taking
	\begin{equation}
		\omega_{\textit{high}}(m,\epsilon_{\textit{trap}},\lambda_{\textit{low}})>0
	\end{equation}
	sufficiently large then the equation
	\begin{equation}
		\frac{d}{dr}\left(\frac{\Delta}{(r^2+a^2)^2}\right)\left(1-\frac{2m\omega a\Xi}{\tilde{\lambda}}\right)-\frac{1}{\tilde{\lambda}}\frac{4r am\Xi}{(r^2+a^2)^2}\left(\omega-\frac{am\Xi}{r^2+a^2}\right)=0
	\end{equation}
	attains a unique solution in the interval~$(r_{\Delta,frac}-\epsilon_{trap},r_{\Delta,frac}+\epsilon_{trap})$. Therefore, the function~$\frac{dV_0}{dr}$ attains a unique zero in the desired interval. 
\end{proof}

\appendix

\section{Connectedness of the black hole parameter space~$\mathcal{B}_l$}\label{sec: appendix, sec 1}

We restate here the lemma we prove in the present Section

\begin{customLemma}{2.1}
	Let~$l>0$. Then, the set~$\mathcal{B}_l$, see Definition~\ref{def: subextremality, and roots of Delta}, is connected. 
\end{customLemma}

\begin{proof}
	We fix~$l>0$. The necessary and sufficient condition for the quartic polynomial~$\Delta$ to attain four distinct real roots can be inferred from~\cite{polynomials} to be the following two
	\begin{equation}\label{eq: proof lem: subsec: delta polynomial, lem 1, eq 1}
		\begin{aligned}
			&-16\left(\frac{a^6}{l^6}\right) 
			-8\left(\frac{a^4}{l^4}\right)\left(1-\frac{a^2}{l^2}\right)^2+36\left(\frac{M^2a^2}{l^4}\right)\left(1-\frac{a^2}{l^2}\right)\\
			&	\qquad\qquad\qquad-27\frac{M^4}{l^4}-\frac{a^2}{l^2}\left(1-\frac{a^2}{l^2}\right)^4+\frac{M^2}{l^2}\left(1-\frac{a^2}{l^2}\right)^3>0,
		\end{aligned}
	\end{equation}
	\begin{equation}\label{eq: proof lem: subsec: delta polynomial, lem 1, eq 2}
		|a|\leq l,~Μ>0.
	\end{equation}

	We rewrite~\eqref{eq: proof lem: subsec: delta polynomial, lem 1, eq 1} as follows 
	\begin{equation}\label{eq: proof lem: subsec: delta polynomial, lem 1, eq 3}
		\begin{aligned}
			P\left(\frac{a^2}{l^2},\frac{M^2}{l^2}\right) \dot{=}&-27\left(\frac{M^2}{l^2}\right)^2 +\frac{M^2}{l^2}\left(36\frac{a^2}{l^2}\left(1-\frac{a^2}{l^2}\right)+\left(1-\frac{a^2}{l^2}\right)^3\right)  -16\left(\frac{a^6}{l^6}\right) 
			-8\left(\frac{a^4}{l^4}\right)\left(1-\frac{a^2}{l^2}\right)^2 -\frac{a^2}{l^2}\left(1-\frac{a^2}{l^2}\right)^4>0,
		\end{aligned}
	\end{equation}
	
	We write
	\begin{equation}\label{eq: proof lem: subsec: delta polynomial, lem 1, eq 5}
		\begin{aligned}
			\mathcal{B}_l &	=	\{(a,M)\in \mathbb{R}\times  \mathbb{R}_{>0}:~0\leq \frac{a^2}{l^2}<1\}\cap \{(a,M)\in  \mathbb{R}\times\mathbb{R}_{>0}:~P(\frac{M^2}{l^2},\frac{a^2}{l^2})>0\}\\
			&	= \{(a,M)\in  \mathbb{R}\times \mathbb{R}_{>0}:~P(\frac{M^2}{l^2},\frac{a^2}{l^2})>0\},
		\end{aligned}
	\end{equation}
	where the last equality in~\eqref{eq: proof lem: subsec: delta polynomial, lem 1, eq 5} is proved as follows. Suppose that~$\frac{a^2}{l^2}\geq 1$. It is immediate that for any~$\frac{M^2}{l^2}>0$ we have~$P(\frac{a^2}{l^2},\frac{M^2}{l^2})<0$, since the second term of~$P(\frac{a^2}{l^2},\frac{M^2}{l^2})$ is strictly negative. This concludes the last equality of~\eqref{eq: proof lem: subsec: delta polynomial, lem 1, eq 5}.

	Since~$P\left(\frac{(-a)^2}{l^2},\frac{M^2}{l^2}\right)=P\left(\frac{a^2}{l^2},\frac{M^2}{l^2}\right)$ we conclude that the set~$\mathcal{B}_l$ is symmetric with respect to the line~$\{a=0\}$. Therefore, in order to prove that~$\mathcal{B}_l$ is connected it suffices to prove that the set
	\begin{equation}\label{eq: proof lem: subsec: delta polynomial, lem 1, eq 6}
		\begin{aligned}
			\widetilde{\mathcal{B}_l}  		&	= \{(a,M)\in  \mathbb{R}_{\geq 0}\times \mathbb{R}_{>0}:~P(\frac{a^2}{l^2},\frac{M^2}{l^2})>0\}.
		\end{aligned}
	\end{equation}
	is connected. The reader is refered to the following Figure~\ref{fig: subextremal}
	\begin{figure}[htbp]
		\centering
		\includegraphics[scale=1.5]{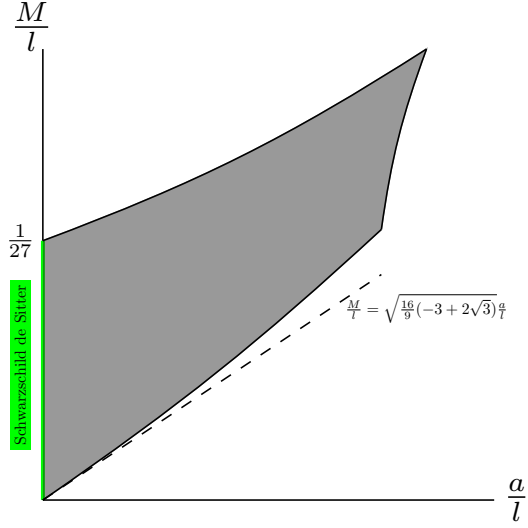}
		\caption{The shaded region is~$\tilde{\mathcal{B}}_l$}
		\label{fig: subextremal}
	\end{figure}

	The remaining of the proof focuses on proving that the set~$\widetilde{\mathcal{B}_l}$ is connected. In view of the identifications 
	\begin{equation}\label{eq: proof lem: subsec: delta polynomial, lem 1, eq 4}
		y=\frac{M^2}{l^2}, \qquad x=\frac{a^2}{l^2}
	\end{equation}
 	we write
	\begin{equation}
		P(x,y)= -27y^2 +y\left(36 x (1-x)+(1-x)^3 \right) -16 x^3 -8x^2(1-x)^2 -x(1-x)^4.
	\end{equation}
	We have that for~$x=0$ and~$0<y<\frac{1}{27}$ then~$P(y,0)>0$. 
	
	Before proving that the set~\eqref{eq: proof lem: subsec: delta polynomial, lem 1, eq 6} is connected we need a preliminary result for the polynomial~$P(x,y)$ for~$x,y\in\mathbb{R}$. In the radial directions
	\begin{equation}
		y= r x
	\end{equation}
	for fixed~$r$ the polynomial~$P(x,y)$ can be written as
	\begin{equation}
		\begin{aligned}
			P_r(x)=P(rx,x)	&	= rx -27(rx)^2 +(-1+33 rx)x +(-4-33 r x)x^2 +(-6-rx)x^3 -4x^4 -x^5 \\
			&	=	x\left(r-1+(-4+33r-27r^2)x+(-6-33r)x^2+(-4-r)x^3-x^4\right)\\
			&	\dot{=} x \tilde{P}_r(x).
		\end{aligned}
	\end{equation}
	Now, we will prove that for 
	\begin{equation}
		r \in (0,\frac{16}{9}(-3+2\sqrt{3})]
	\end{equation}
	we have that~$\tilde{P}_r(x)\leq 0$ for any~$x\in\mathbb{R}$, and for any
	\begin{equation}
		r> \frac{16}{9}(-3+2\sqrt{3})
	\end{equation} 
	the polynomial~$\tilde{P}_r(x)$ attains exactly two distinct real roots and two complex conjugate roots and~$\tilde{P}_r(x)>0$ in between the real roots.

	We study the discriminant of the quartic polynomial~$\tilde{P}_r(x)$, see~\cite{polynomials}, which we compute it to be 
	\begin{equation}
		\text{disc}= -4r^3\left(-256+9r(32+3r) \right)^3.
	\end{equation}
	We directly calculate that for
	\begin{equation}
		r\in \left(0, \frac{16}{9} (-3 + 2 \sqrt{3}) \right)
	\end{equation}
	then~$\text{disc}>0$ and for
	\begin{equation}
		r\in \left(\frac{16}{9} (-3 + 2 \sqrt{3}),+\infty\right)
	\end{equation}
	then~$\text{disc}<0$. We have that~$0<\frac{16}{9} (-3 + 2 \sqrt{3})<1$.

	Moreover we have the following auxilliary functions 
	\begin{equation}
		\begin{aligned}
			P	&	=	8(-1)(-6-33r)-3(4+r)^2=-3 (4 + (-58 + r) r)\\ 
			D	&=64(-1)^3(r-1)-16(-1)^2(-6-33r)^2+16(-1)(-4-r)^2(-6-33r)-16(-1)^2(-4-r)(-4+33r-27r^2)-3(4+r)^4\\
			&	=	r (4096 - 3 r (4864 + (-16 + r) r))
		\end{aligned}
	\end{equation}
	We immediately calculate that~$P,D$ cannot both be negative for~$r\in (0,1)$.

	By the conditions for the nature of the roots of quartic polynomials, see~\cite{polynomials}, we note that for
	\begin{equation}
		r\in \left(0, \frac{16}{9} (-3 + 2 \sqrt{3}) \right)
	\end{equation}
	since~$\text{disc}>0$ and that one of~$P,D$ is positive then~$\tilde{P}_r(x)<0$. For~$r=\frac{16}{9} (-3 + 2 \sqrt{3})$ the polynomial~$\tilde{P}_r(x)$ attains a double real root and two complex conjugate roots and~$\tilde{P}_r(x)\leq 0$ for all~$x \in \mathbb{R}$.

	Now for
	\begin{equation}
		r\in \left(\frac{16}{9} (-3 + 2 \sqrt{3}),+\infty\right)
	\end{equation}
	since~$\text{disc}<0$ then the polynomial~$\tilde{P}_r(x)$ attains two distinct real roots and two complex conjugate roots. Therefore, since~$\tilde{P}_r(x)$ is a polynomial of degree~$4$ then we conclude that~$\tilde{P}_r(x)>0$ in between these two real roots.

	Therefore, we have proved that for
	\begin{equation}
		r\in \big(0,\frac{16}{9}(-3+2\sqrt{3})\big]
	\end{equation}
	the following holds~$P_{r}(x)\leq 0$ for all~$x$ and for
	\begin{equation}
		r > \frac{16}{9}(-3+2\sqrt{3})
	\end{equation}
	the polynomial~$P_r(x)$ attains exactly two real roots
	\begin{equation}
		x_1(r)\leq 0<  x_2(r)
	\end{equation}
	and it is positive between the roots. Note that~$x_1(1)=0$. It is necessary that the following holds~$x_1(r)\leq 0$ because otherwise no neighborhood of~$(0,0)$ belongs in the subextremal set~$\tilde{B}_l$.

	Now, we are ready to prove that the set~\eqref{eq: proof lem: subsec: delta polynomial, lem 1, eq 6} is connected. A preliminary observation that will help is that part of the boundary of the set~\eqref{eq: proof lem: subsec: delta polynomial, lem 1, eq 6} is the following 
	\begin{equation}
		B=  \text{image}(\gamma_2)
	\end{equation}
	where
	\begin{equation}
		\begin{aligned}
			\gamma_2:[\frac{16}{9}(-3+2\sqrt{3}),+\infty) &	\rightarrow \mathbb{R}^2\\
			r	&	\mapsto (r x_2(r),x_2(r)). 
		\end{aligned}
	\end{equation}
	The set~$B$ is connected because the curve~$\gamma_2$ is uniformly continuous.

	Now, we consider the following (nonempty) cylinder around the set~$B$ and inside~$\widetilde{\mathcal{B}_l}$
	\begin{equation}
		\begin{aligned}
			B_{\epsilon} = \{(a,M)\in  \mathbb{R}_{\geq 0}\times \mathbb{R}_{>0}:~d \left((x,y),B\right)\leq \epsilon\} \cap \widetilde{\mathcal{B}_l}.
		\end{aligned}
	\end{equation}
	Since the boundary~$B$ is connected then for a sufficiently small~$\epsilon$ the set~$B_\epsilon$ is also connected.

	Now, take two points
	\begin{equation}
		(a_1,M_1),~(a_2,M_2)\in \{(a,M)\in\mathbb{R}_{\geq 0} \times \mathbb{R}_{> 0}:~ P(x,y)>0\}
	\end{equation}
	which we rewrite respectively as 
	\begin{equation}
		( r_1 b_1,b_1),~( r_2 b_2,b_2)
	\end{equation} 
	for two~$r_1,r_2>0$. Necessarily we have that
	\begin{equation}
		x_1(r_1)\leq b_1\leq x_2(r_1),\qquad 	x_1(r_2)\leq b_2\leq x_2(r_2)
	\end{equation}

	Let~$\tilde{\epsilon}(\epsilon)>0$ be sufficiently small. The line connecting
	\begin{equation}
		( r_1 b_1,b_1),
	\end{equation}
	to 
	\begin{equation}
		\left(r_1 \cdot (x_2(r_1)-\tilde{\epsilon}),x_2(r_1)-\tilde{\epsilon} \right) \in B_\epsilon
	\end{equation}
	lies entirely in~$\mathcal{B}_l$. Furthermore, the line connecting 
	\begin{equation}
		( r_2 b_2,b_2)
	\end{equation}
	to 
	\begin{equation}
		( r_2 \cdot (x_2(r_2)-\tilde{\epsilon}),x_2(r_2)-\tilde{\epsilon}) \in B_\epsilon
	\end{equation}
	lies entirely in~$\widetilde{\mathcal{B}_l}$. Since the set~$B_\epsilon$ is connected we conclude that the set~$\widetilde{\mathcal{B}_l}$, see~\eqref{eq: proof lem: subsec: delta polynomial, lem 1, eq 6}, is connected which, as discussed previously, concludes that the set~$\mathcal{B}_l$ is connected. 
\end{proof}

\section{Geodesic flow of Kerr--de~Sitter and \texorpdfstring{$\partial_t$}{t}-orthogonal trapped null geodesics}\label{sec: geodesics}

The purpose of this Section is to prove Proposition~\ref{prop: geodesics}, which has no analogue in the asymptotically flat~$\Lambda=0$ Kerr case.

We define trapped null geodesics as follows
\begin{definition}
	Let~$l>0$ and~$(a,M)\in\mathcal{B}_l$. Moreover, let~$(t,r,\theta,\varphi)$ be the Boyer--Lindquist coordinates, see Section~\ref{subsec: boyer Lindquist coordinates}. Then, a geodesic~$\gamma(v)$, parametrized by an affine time~$v$, is called trapped null if it satisfies~$g(\dot{\gamma},\dot{\gamma})=0$ and moreover 
	\begin{equation}
		\lim_{v\rightarrow \infty} r(\gamma(v))\in (r_+,\bar{r}_+). 
	\end{equation}
\end{definition}

Note the following proposition

\begin{proposition}\label{prop: geodesics}
	
	Let~$l>0$ and~$(a,M)\in\mathcal{B}_l$ be subextremal Kerr--de~Sitter black hole parameters such that 
	\begin{equation}\label{eq: prop: geodesics, eq 1}
		\max_{r\in[r_+,\bar{r}_+]}\frac{\Delta}{a^2}\leq 1,\quad |a|\not{=}0. 
	\end{equation}

	Then, there exists a trapped null geodesic 
	\begin{equation}
		\gamma (v)= (t(v),r_{\Delta,\textit{max}},\theta_0,\phi(v))
	\end{equation} 
	with~$g(\dot{\gamma},\partial_t)=0$, for some~$\theta_0\in [0,\pi]$, where~$r_{\Delta,\textit{max}}\in (r_+,\bar{r}_+)$ is the unique maximum of~$\Delta$, see Lemma~\ref{lem: sec: general properties of Delta, lem 4}.  
	
	Finally, if~\eqref{eq: prop: geodesics, eq 1} does not hold, then there exists no~$\partial_t$ orthogonal trapped null geodesic in the Kerr--de~Sitter background~$\mathcal{M}$.  
\end{proposition}

\begin{remark}	
	Note that if~\eqref{eq: prop: geodesics, eq 1} holds, then the ergoregion is connected, see Section~\ref{subsec: ergoregion}.

	Note that in the asympotically flat Kerr case there exists no trapped null geodesic with zero~$\partial_t$ energy, see~\cite{DR2}. This non-existence of a~$\partial_t$ orthogonal trapped null geodesic, in the~$\Lambda=0$ Kerr exterior, was crucial in the proof of rigidity of Kerr in a neighborhood of Kerr, see~\cite{alexakis1}.
\end{remark}

First, note the following Lemma for the null geodesic flow of Kerr--de~Sitter

\begin{lemma}\label{lem: geodesics equations}
	A curve~$\gamma=(t,r,\theta,\phi)$ on the Kerr--de~Sitter background, parametrized by a parameter~$v$, is a null geodesic if and only if it satisfies 
	\begin{equation}\label{eq: geodesic equations}
		\begin{aligned}
			\rho^2\dot{t}&=     \frac{a\Xi^2}{\Delta_\theta}\left(\mathcal{E}a\sin^2\theta-\Xi\mathcal{L} \right)+\Xi^2\frac{(r^2+a^2)\left(\Xi\mathcal{L}a-(r^2+a^2)\mathcal{E}\right)}{\Delta}\\
			\rho^2\dot{\phi}&  =    \frac{a\Xi^2}{\Delta_\theta}\frac{\mathcal{E}a\sin^2\theta-\Xi\mathcal{L}}{\sin^2\theta}+\Xi^2\frac{a\left(\Xi\mathcal{L}a-(r^2+a^2)\mathcal{E}\right)}{\Delta}\\
			\rho^4(\dot{\theta})^2& =   \Delta_\theta\mathcal{K}-\Xi^2\frac{\left(\Xi\mathcal{L}-a\mathcal{E}\sin^2\theta\right)^2}{\sin^2\theta} \\
			\rho^4(\dot{r})^2 & =   \Xi^2\left((r^2+a^2)\mathcal{E}-a\Xi\mathcal{L}\right)^2-\Delta\mathcal{K},
		\end{aligned}
	\end{equation}
	with~$\rho^2=r^2+a^2\cos^2\theta,~\Delta_\theta=1+\frac{a^2}{l^2}\cos^2\theta$ see Definition~\ref{def: delta+, and other polynomials}, where note we have not fully decoupled the geodesic equations. The constant~$\mathcal{K}$ is Carter's constant of motion, and note the conserved quantities 
	\begin{equation}
		\begin{aligned}
			\mathcal{E} &=    g(\dot{\gamma},\partial_t)=-\left(1-\frac{2Mr}{\rho^2}\right)\dot{t}-\frac{2Mra\sin^2\theta}{\rho^2}\dot{\phi},\\
			\mathcal{L}&    =-g(\dot{\gamma},\partial_\phi)=\frac{2Mra\sin^2\theta}{\rho^2}\dot{t}-\sin^2\theta\frac{(r^2+a^2)^2-a^2\sin^2\theta\Delta}{\rho^2}\dot{\phi},\\
			\mathcal{Q} &=    \rho^4\left(\dot{\theta}\right)^2+\frac{\Xi^2\mathcal{L}^2}{\sin^2\theta}-a^2E^2\cos^2\theta,\\
			\mathcal{K}	&=\Xi^2\left(\mathcal{Q}+a^2\mathcal{E}^2-2a\Xi\mathcal{L}\mathcal{E}\Xi\right).
		\end{aligned}
	\end{equation}
\end{lemma}
\begin{proof}
	This proof can be inferred from~\cite{Carter2} or~\cite{hackman-geodesics}. 
\end{proof}

Now we are ready for

\begin{proof}[\textbf{Proof of Proposition \ref{prop: geodesics}}]

	We substitute
	\begin{equation}
		\mathcal{E}=0
	\end{equation}
	in the geodesic equations~\eqref{eq: geodesic equations} and obtain
	\begin{equation}\label{eq: prop: geodesics, E=0 geodesics, eq 1}
		\begin{aligned}
			\rho^2\dot{t}&=     \frac{a\Xi^2}{\Delta_\theta}\left(-\Xi\mathcal{L} \right)+\Xi^2\frac{(r^2+a^2)\left(\Xi\mathcal{L}a\right)}{\Delta},\\
			\rho^2\dot{\phi}&  =    \frac{a\Xi^2}{\Delta_\theta}\frac{-\Xi\mathcal{L}}{\sin^2\theta}+\Xi^2\frac{a\left(\Xi\mathcal{L}a\right)}{\Delta},\\
			\rho^4(\dot{\theta})^2=\Theta(\theta)& \:\dot{=}   \Delta_\theta\Xi^2\mathcal{Q}-\Xi^2\frac{\left(\Xi\mathcal{L}\right)^2}{\sin^2\theta}, \\
			\rho^4(\dot{r})^2=R(r) & \:\dot{=}   \Xi^2\left(-a\Xi\mathcal{L}\right)^2-\Delta\Xi^2\mathcal{Q}=a^2\Xi^2\Xi^2\mathcal{L}^2-\Delta\Xi^2\mathcal{Q}=a^2\Xi^2\left(\Xi^2\mathcal{L}^2-\frac{\Delta}{a^2}\mathcal{Q}\right).
		\end{aligned}
	\end{equation}
	Let~$r_{\Delta,\textit{max}}$ be the value where~$\Delta$ attains its unique maximum and choose~$\mathcal{L},\mathcal{Q}$ such that
	\begin{equation}\label{eq: geodesics, eq 1}
		\Xi^2\mathcal{L}^2=\frac{\Delta(r_{\Delta,\textit{max}})}{a^2}\mathcal{Q}.
	\end{equation}
	
	We first solve the~$\theta$-equation of motion. By taking 
	\begin{equation}\label{eq: geodesics, eq 1.9}
		0<\frac{\Delta(r_{\Delta,\textit{max}})}{a^2}\leq 1,
	\end{equation}
	the $\theta$-equation of motion can be solved globally
	\begin{equation}\label{eq: geodesics, eq 2}
		(\dot{\theta})^2=\frac{1}{\rho^4}\left(\frac{\Delta_\theta\sin^2\theta-\frac{\Delta(r_{\Delta,\textit{max}})}{a^2}}{\sin^2\theta}\right)\Xi^2\mathcal{Q},
	\end{equation}
	where~$\Delta_\theta=1+\frac{a^2}{l^2}\cos^2\theta$, since there exists a~$\theta_0\in (0,\pi)$ such that
	\begin{equation}\label{eq: geodesics, eq 2.9}
		\Delta_{\theta_0}\sin^2\theta_0-\frac{\Delta (r_{\Delta,\textit{max}})}{a^2}=0.
	\end{equation}
	Therefore, for the rest of the proof we assume that~$\theta(v)\equiv\theta_0$ for all~$v\in\mathbb{R}$, where~$v$ is the affine time.

	Second, we solve the~$r$-equation of motion. We find the value of~$r$ that the trapped null geodesics will asymptote to. The~$r$ equations of motion read
	\begin{equation}\label{eq: geodesics, eq 3}
		\begin{aligned}
			R(r)   \equiv a^2\Xi^2\left(\Xi^2\mathcal{L}^2-\frac{\Delta}{a^2}\mathcal{Q}\right)\equiv a^2\Xi^2\left(\frac{\Delta(r_{\Delta,\textit{max}})-\Delta}{a^2}\right)\mathcal{Q} =0,\qquad 
			\frac{dR(r)}{dr}   =-\Xi^2\frac{d\Delta}{dr}\mathcal{Q}=0.
		\end{aligned}
	\end{equation}
	From the latter equation of~\eqref{eq: geodesics, eq 3} we obtain 
	\begin{equation}
		\frac{dR}{dr}(r)\Big|_{r=r_{\Delta,\textit{max}}}=0,
	\end{equation}
	where~$r_{\Delta,\textit{max}}\in(r_+,\bar{r}_+)$. Then, we solve the first equation of~\eqref{eq: geodesics, eq 3}. From the known well posedness theorem, the $r$-equation of motion can be solved locally
	\begin{equation}
		\begin{aligned}
			\rho^4(\dot{r})^2=a^2\Xi^2\left(\frac{\Delta(r_{\Delta,\textit{max}})}{a^2}-\frac{\Delta}{a^2}\right)\mathcal{Q},~\text{or equivalently}~\frac{\rho^2 dr}{\sqrt{a^2\Xi^2\left(\frac{\Delta(r_{\Delta,\textit{max}})}{a^2}-\frac{\Delta}{a^2}\right)\mathcal{Q}}}=\pm dv,
		\end{aligned}
	\end{equation}
	with initial conditions~$r(0)=r_{\Delta,\textit{max}}$,~$\dot{r}(0)=\dot{r}_0$. By~\eqref{eq: geodesics, eq 1} we note that~$r_{\Delta,\textit{max}}$ is a root of multiplicity~$2$ of the following
	\begin{equation}
		\Xi^2\mathcal{L}^2-\frac{\Delta}{a^2}\mathcal{Q}=\frac{\Delta(r_{\Delta,\textit{max}})-\Delta}{a^2}\mathcal{Q}
	\end{equation}
	and therefore the affine parameter~$v$ takes values up to infinity since we can write
	\begin{equation*}
		v-v_0=\int_{r_{\Delta,\textit{max}}}^{r(v)}\frac{\left(r^2+a^2\cos{\theta_0}\right)dr}{\sqrt{R(r)}}
	\end{equation*}
	for all~$v$ in a small interval around~$v_0$. Therefore, if the curve
	\begin{equation}
		\gamma(v)
	\end{equation}
	is a trapped null geodesic, it needs to be complete.

	The~$t,\phi$ equations of motions, namely the first and second equations of respectively~\eqref{eq: prop: geodesics, E=0 geodesics, eq 1}, can be solved easily since they are linear.
	
	Therefore, if~\eqref{eq: geodesics, eq 1.9} holds, then the trapped null geodesic equations~\eqref{eq: prop: geodesics, E=0 geodesics, eq 1} correspond to a complete trapped null geodesic 
	\begin{equation}
		\gamma=(t(v),r_{\Delta,\textit{max}},\theta_0,\phi(v))
	\end{equation}
	where for~$\theta_0$ see~\eqref{eq: geodesics, eq 2.9} with 
	\begin{equation}
		g(\dot{\gamma},\partial_t)=0.
	\end{equation}

	Lastly, from the last two equations of motion of~\eqref{eq: prop: geodesics, E=0 geodesics, eq 1} we note that if~$\max_{[r_+,\bar{r}_+]}\frac{\Delta}{a^2}>1$ then there exists no~$\partial_t$ orthogonal trapped null geodesic, in view of the fact that the~$r$ coordinate of the geodesic needs to be~$r_{\Delta,\textit{max}}$. We proceed to prove this by first noting that if the contrary holds, then from the last equation of motion of~\eqref{eq: prop: geodesics, E=0 geodesics, eq 1} we obtain
	\begin{equation}\label{eq: geodesics, eq 4}
		\Xi^2\mathcal{L}^2-\frac{\Delta(r_{\Delta,\textit{max}})}{a^2}\mathcal{Q}=0\implies \Xi^2\mathcal{L}^2 >\mathcal{Q}.
	\end{equation}
	Now, we combine~\eqref{eq: geodesics, eq 4} with the second to last equation of motion of~\eqref{eq: prop: geodesics, E=0 geodesics, eq 1} and arrive at a contradiction, since we obtain~$(\dot{\theta})^2<0$.

	We conclude the proof.
\end{proof}

Note the following Remark

\begin{remark}\label{rem: sec: geodesics, rem 2}
	We note that if we formally identify
	\begin{equation}
		\begin{aligned}
			\mathcal{E}    \rightarrow \omega,\qquad 
			\mathcal{L}   \rightarrow m,\qquad
			\mathcal{Q}    \rightarrow \lambda^{(a\omega)}_{m\ell},
		\end{aligned}
	\end{equation}
	then the geodesic equation for the motion of~$\dot{r}$, see Lemma~\ref{lem: geodesics equations}, can be written in the familiar form 
	\begin{equation}
		\frac{\rho^4}{\Xi^2(r^2+a^2)^2}(\dot{r})^2= \mathcal{E}^2-V_0(\mathcal{Q},\mathcal{L},\mathcal{E},r)
	\end{equation}
	where~$V_0$ is the potential appearing in Carter's separation of variables, see Proposition~\ref{prop: Carters separation, radial part}.
\end{remark}

\bibliographystyle{plain}
\bibliography{MyBibliography}

\end{document}